%% file: main.tex
\documentclass[12pt,glossary]{dalthesis}

\usepackage{graphicx}
\usepackage{amsthm,amsmath,amssymb,amsbsy,amsfonts,enumitem,mathtools,csquotes}
\usepackage[hidelinks,backref=page]{hyperref}

\usepackage[numbers]{natbib}

\usepackage{etoolbox}
\makeatletter
\patchcmd{\BR@backref}{\newblock}{\newblock(cited on p.~}{}{}
\patchcmd{\BR@backref}{\par}{)\par}{}{}
\makeatother

\def\addsymbol #1: #2#3{$#1$ \> \parbox{5in}{#2}\\} 
\def\newnot#1{\label{#1}} 
\newtheorem{theorem}{Theorem}[chapter]
\newtheorem{lemma}[theorem]{Lemma}

\newtheorem{proposition}[theorem]{Proposition}

\newtheorem{condition}[theorem]{Condition}

\theoremstyle{definition}
\newtheorem{definition}[theorem]{Definition}

\newtheorem{remark}[theorem]{Remark}

\newtheorem*{conjecturea}{Saff-Varga Width Conjecture}
\newtheorem*{conjectureb}{Modified Saff-Varga Width Conjecture}

\newtheorem{rhp}[theorem]{Riemann-Hilbert Problem}

\numberwithin{section}{chapter}
\numberwithin{equation}{section}

\def\N{\mathbb{N}}
\def\C{\mathbb{C}}
\def\R{\mathbb{R}}

\DeclareMathOperator{\erfc}{erfc}
\DeclareMathOperator{\re}{Re}
\DeclareMathOperator{\im}{Im}

\DeclareMathOperator{\len}{length}
\DeclareMathOperator{\M}{\mathbf M}
\DeclareMathOperator{\Ai}{Ai}
\DeclareMathOperator{\Bi}{Bi}

\newcommand{\const}{\text{const}}
\renewcommand{\th}{\text{th}}
\newcommand{\dsim}{\mathrel{\dot\sim}}

\makeatletter
\def\blfootnote{\gdef\@thefnmark{}\@footnotetext}
\makeatother

\begin{document}

\title{Scaling Limits for Partial Sums of Power Series}
\author{Antonio R.\ Vargas}
\defenceday{15}
\defencemonth{July}
\defenceyear{2016}
\convocation{October}{2016}


\supervisor{Karl Dilcher}
\reader{Robert Milson}
\reader{Theodore Kolokolnikov}
\reader{Peter Miller}

\nolistoftables

\phd 
\frontmatter

\begin{abstract}
In this thesis it will be shown that the partial sums of the Maclaurin series for a certain class of entire functions possess scaling limits in various directions in the complex plane. In doing so we obtain information about the zeros of the partial sums. We will only assume that these entire functions have a certain asymptotic behavior at infinity.

With this information we will partially verify for this class of functions a conjecture on the location of the zeros of their partial sums known as the Saff-Varga Width Conjecture.
\end{abstract}

\chapter*{List of Symbols Used\hfill} \addcontentsline{toc}{chapter}{List of Symbols Used}
\markboth{}{List of Symbols Used} 
\input{Symbols/symbols} \clearpage


\begin{acknowledgements}

\vspace{1cm}

\noindent Thank you Karl for your sharp guidance. You knew there was a good problem here.

\vspace{1cm}

\noindent Thank you Rob for your unending generosity. You went far, far out of your way to help me.

\vspace{1cm}

\noindent Thank you Amelia for your infinite patience. We did it.
\end{acknowledgements}

\mainmatter

\include{chap_intro}

\include{chap_prelims}

\include{chap_limitcurves}

\include{chap_curvescaling}

\include{chap_cornerscaling}

\include{chap_applications}

\include{chap_conclusion}

\appendix

\include{appendix_laplace}

\bibliographystyle{amsplain}
\bibliography{biblio}

\end{document}

%% file: Symbols/symbols.tex
\begin{tabbing}
$\R$~~~~~~~~~~~~~~~~~~~~~~~~~~\=\parbox{5in}{Field of real numbers}\\
\addsymbol \C: {Field of complex numbers}{symbol:CC}
\addsymbol |z|: {Modulus of the complex number $z$}{symbol:modulus}
\addsymbol \re z: {Real part of the complex number $z$}{symbol:real}
\addsymbol \im z: {Imaginary part of the complex number $z$}{symbol:imag}
\addsymbol f^\pm: {Continuous extension of $f$ to the $\pm$ side of an oriented contour}{symbol:extensions}

\addsymbol p_n[f](z): {$n^{\text{th}}$ partial sum of the Maclaurin series for $f$}{symbol:section}

\addsymbol \#_n^{\angle}(\theta_1,\theta_2): {Number of zeros in the sector $\theta_1 \leq \arg z \leq \theta_2$}{symbol:numsec}
\addsymbol \#_{n}^{\circ}(r): {Number of zeros in the disk $|z| \leq r$}{symbol:numdisk}

\addsymbol \binom{n}{k}: {Coefficient of $x^k$ in the expansion of $(1+x)^n$}{symbol:binom}
\addsymbol \log: {Natural logarithm}{symbol:log}
\addsymbol \erfc: {Complementary error function}{symbol:erfc}

\addsymbol \approx: {Is approximately}{symbol:approx}
\addsymbol O,\Theta,\Omega,o,\omega,\sim,\dsim: {Asymptotic notation}{symbol:bigo}

\end{tabbing}

%% file: chap_intro.tex
\graphicspath{{images/chap_intro/}}

\chapter{Introduction}
\label{chap_intro}

Consider the complex function
\[
	f(z) = \frac{(1+3z)(z-\zeta_1)(z-\zeta_2)}{(1-z)^2},
\]
where
\[
	\zeta_1 = -\frac{1}{5} + \frac{i}{3} \qquad \text{and} \qquad \zeta_2 = \frac{35}{100} - \frac{i}{4}.
\]
The particular values of $\zeta_1$ and $\zeta_2$ aren't so important---just note that the function $f$ has three zeros $z=-1/3,\zeta_1,\zeta_2$, each inside the circle $|z| = 1$\newnot{symbol:modulus}, and a singularity located at $z=1$. Due to this singularity at $z=1$ the radius of convergence of the Maclaurin series for $f$ is $1$ and thus the series has the unit circle as its circle of convergence.

\begin{figure}[!htb]
	\centering
	\begin{minipage}[c]{0.45\textwidth}
		\includegraphics[width=\textwidth]{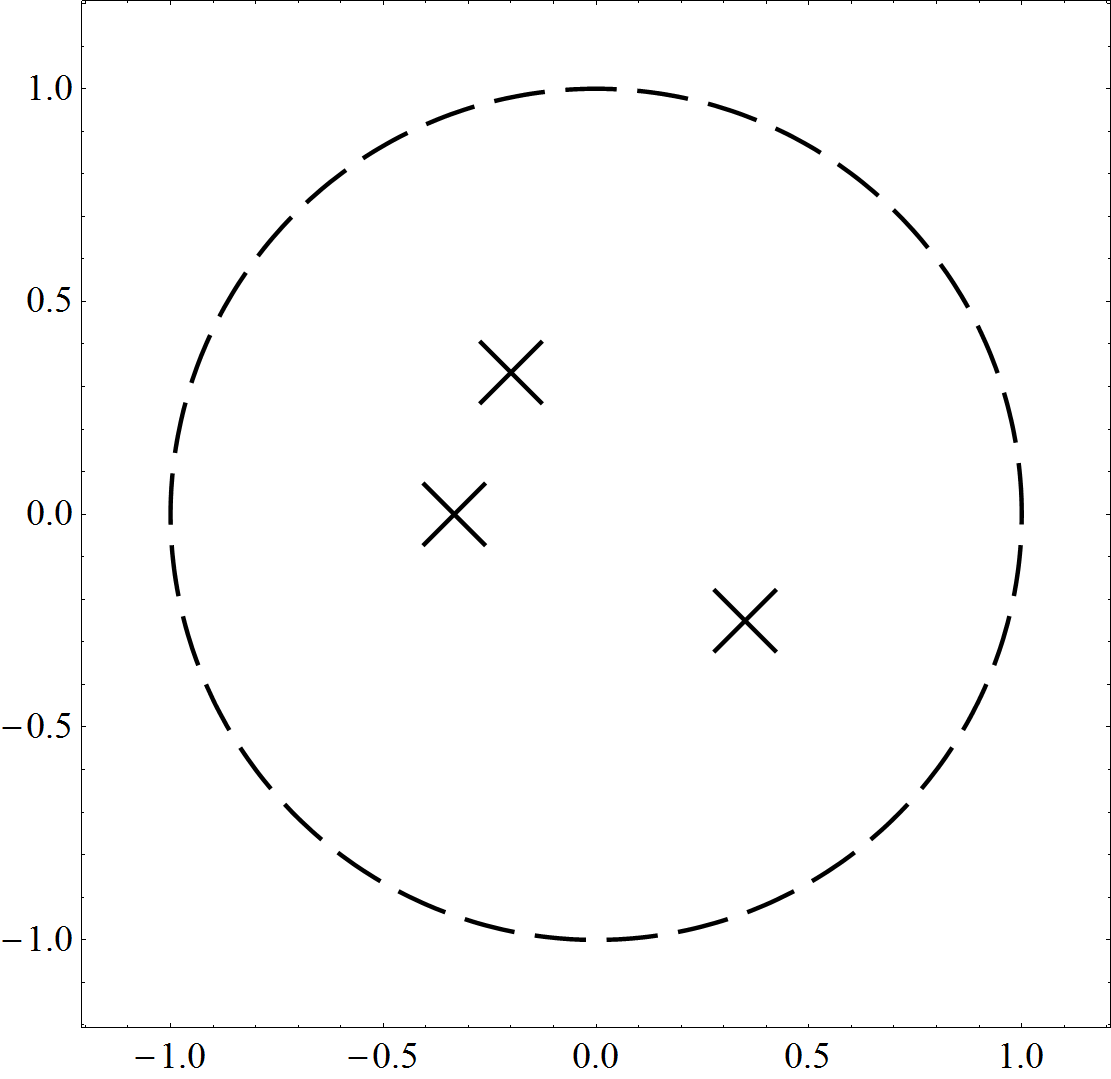}
	\end{minipage}
	\hspace{0.3cm}
	\begin{minipage}[c]{0.45\textwidth}
		\caption[The zeros of $f$.]{The zeros of {$f$}, marked with black crosses. The dashed curve is the unit circle.}
		\label{introintrofig_finitefzeros}
	\end{minipage}
\end{figure}

Let $p_n(z)$ be the $n^\th$ partial sum of the Maclaurin series for $f$. The first few of these are
\begin{align*}
	&p_0(z) = \zeta_1 \zeta_2, \\
	&p_1(z) = \zeta_1 \zeta_2 + (5\zeta_1\zeta_2 - \zeta_1 - \zeta_2)z, \\
	&p_2(z) = \zeta_1 \zeta_2 + (5\zeta_1\zeta_2 - \zeta_1 - \zeta_2)z + (9\zeta_1\zeta_2 - 5\zeta_1 - 5\zeta_2 + 1)z^2,
\end{align*}
and in general
\begin{align*}
	p_n(z) &= \zeta_1 \zeta_2 + (5\zeta_1\zeta_2 - \zeta_1 - \zeta_2)z \\
	&\qquad + \sum_{k=2}^{n} \bigl[(4k+1)\zeta_1\zeta_2 - (4k-3)(\zeta_1+\zeta_2) + 4k-7\bigr]z^k.
\end{align*}
Consider the zeros of the partial sum $p_5$.

\begin{figure}[!htb]
	\centering
	\begin{minipage}[c]{0.45\textwidth}
		\includegraphics[width=\textwidth]{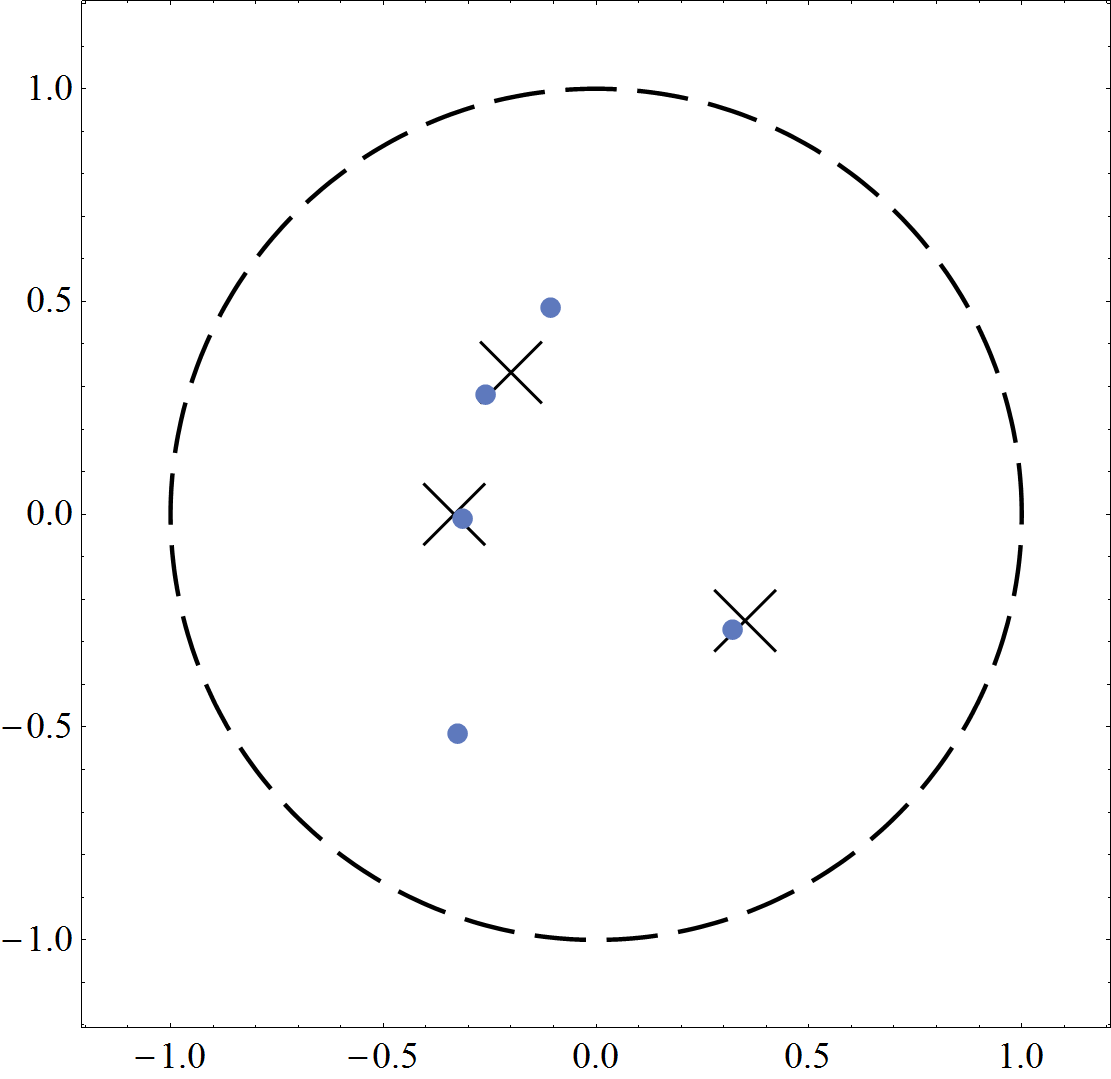}
	\end{minipage}
	\hspace{0.3cm}
	\begin{minipage}[c]{0.45\textwidth}
		\caption[The zeros of $p_5(z)$ and the zeros of $f$.]{The zeros of $p_5(z)$ as blue dots and the zeros of $f$ as black crosses. The dashed curve is the unit circle.}
		\label{introintrofig_finitefp5}
	\end{minipage}
\end{figure}

\noindent As $p_5$ is a fifth degree polynomial it has $5$ complex zeros. Three of them are close to the three zeros of $f$, but it also has two extra zeros. As $n\to\infty$ the polynomials $p_n$ converge to $f$ uniformly on compact subsets of $|z| < 1$, so we expect that each $p_n$ will have a zero near each zero of $f$ and, further, that $p_n$ \textit{won't} have a zero where $f$ doesn't. This is essentially a special case of Hurwitz's theorem \cite[p.\ 4]{marden:geom}:

\begin{theorem}[Hurwitz's Theorem]
\label{introthm_hurwitz}
Let $f_n(z)$ be a sequence of functions which are analytic in a region $R$ and which converge uniformly to a function $f(z) \not\equiv 0$ in every closed subregion of $R$. Let $\zeta$ be an interior point of $R$. If $\zeta$ is a limit point of the zeros of the $f_n(z)$, then $\zeta$ is a zero of $f(z)$. Conversely, if $\zeta$ is an $m$-fold zero of $f(z)$, every sufficiently small neighborhood $K$ of $\zeta$ contains exactly $m$ zeros (counting multiplicities) of each $f_n(z)$ if $n$ is large enough.
\end{theorem}

The partial sum $p_5$ has two more zeros than $f$, $p_6$ has three more, and so on. The limit function $f$ can have only finitely-many zeros in any compact subset of $|z| < 1$, so it follows from Hurwitz's theorem that almost all of the zeros of $p_n$ must leave any fixed compact subset of $|z| < 1$ as $n \to \infty$.

\begin{figure}[h!tb]
	\centering
	\begin{tabular}{ccc}
		\includegraphics[width=0.3\textwidth]{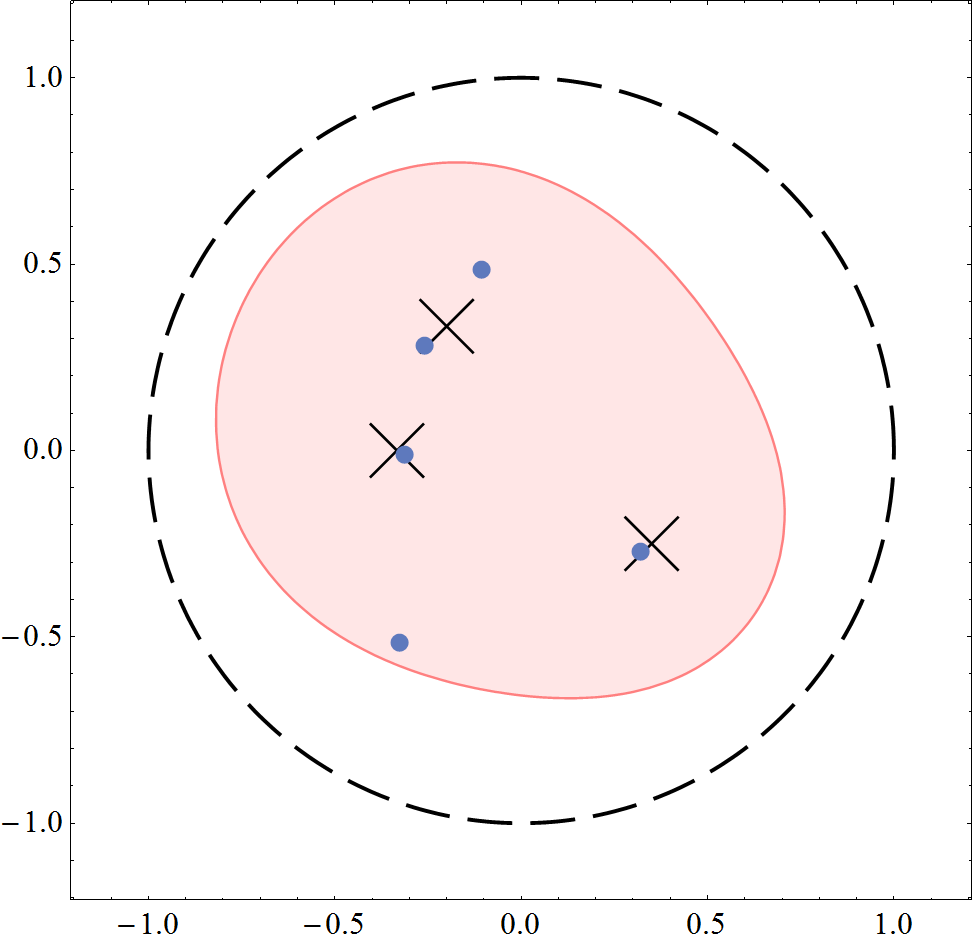}
			& \includegraphics[width=0.3\textwidth]{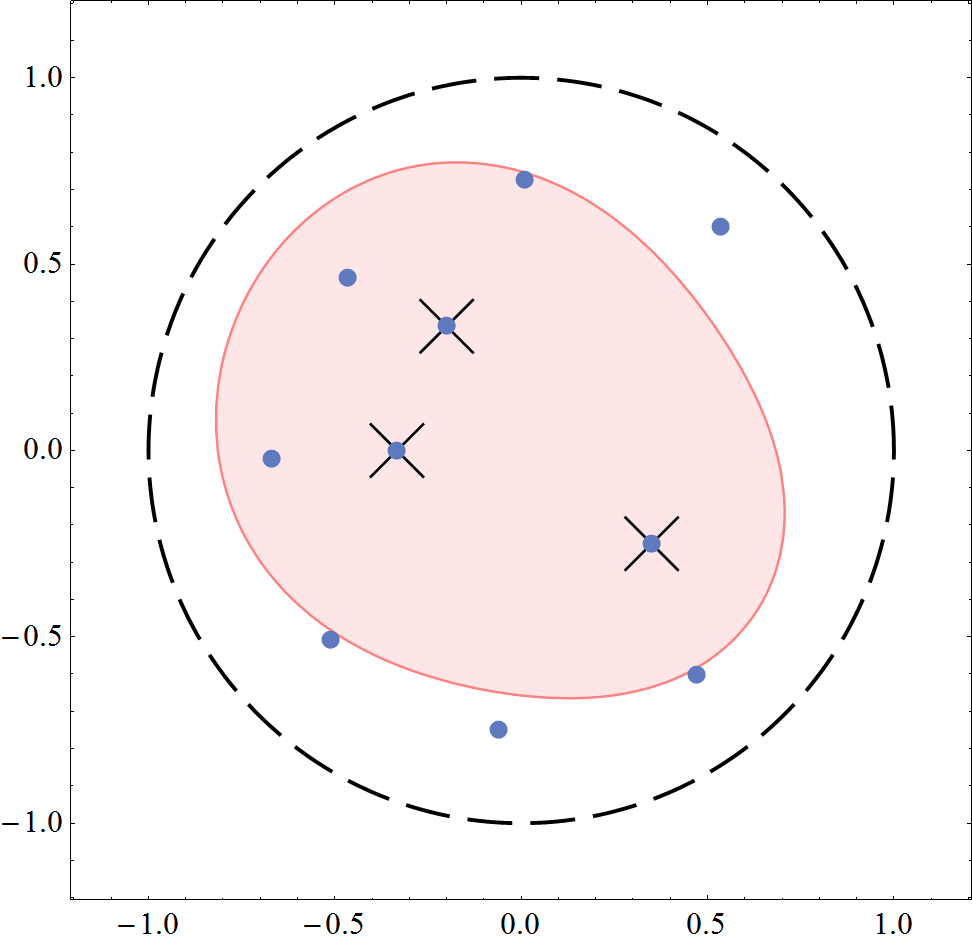}
			& \includegraphics[width=0.3\textwidth]{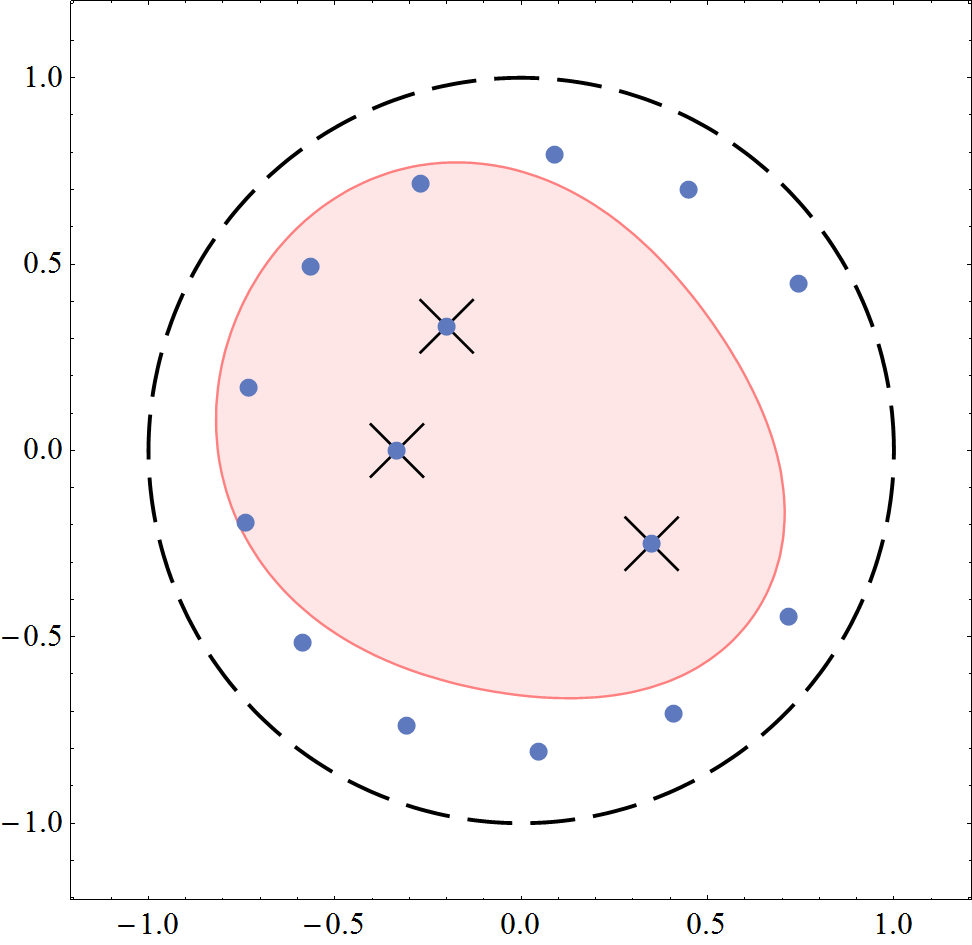} \\
		\includegraphics[width=0.3\textwidth]{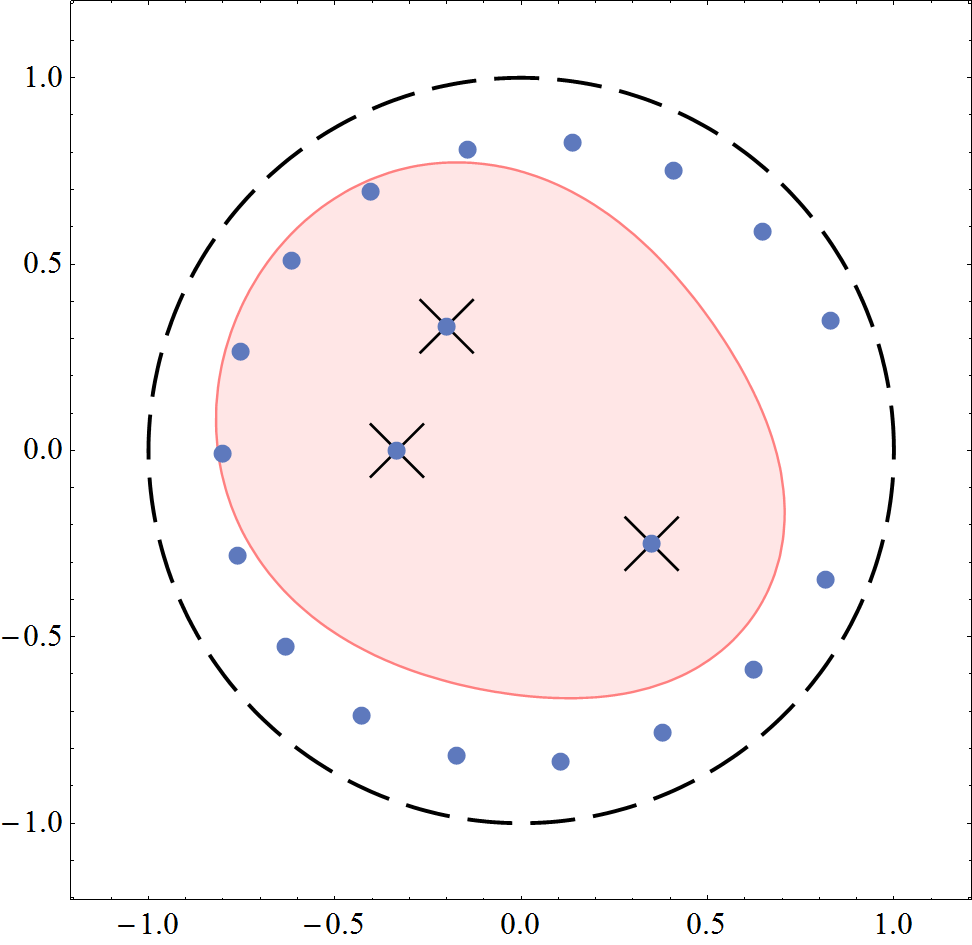}
			& \includegraphics[width=0.3\textwidth]{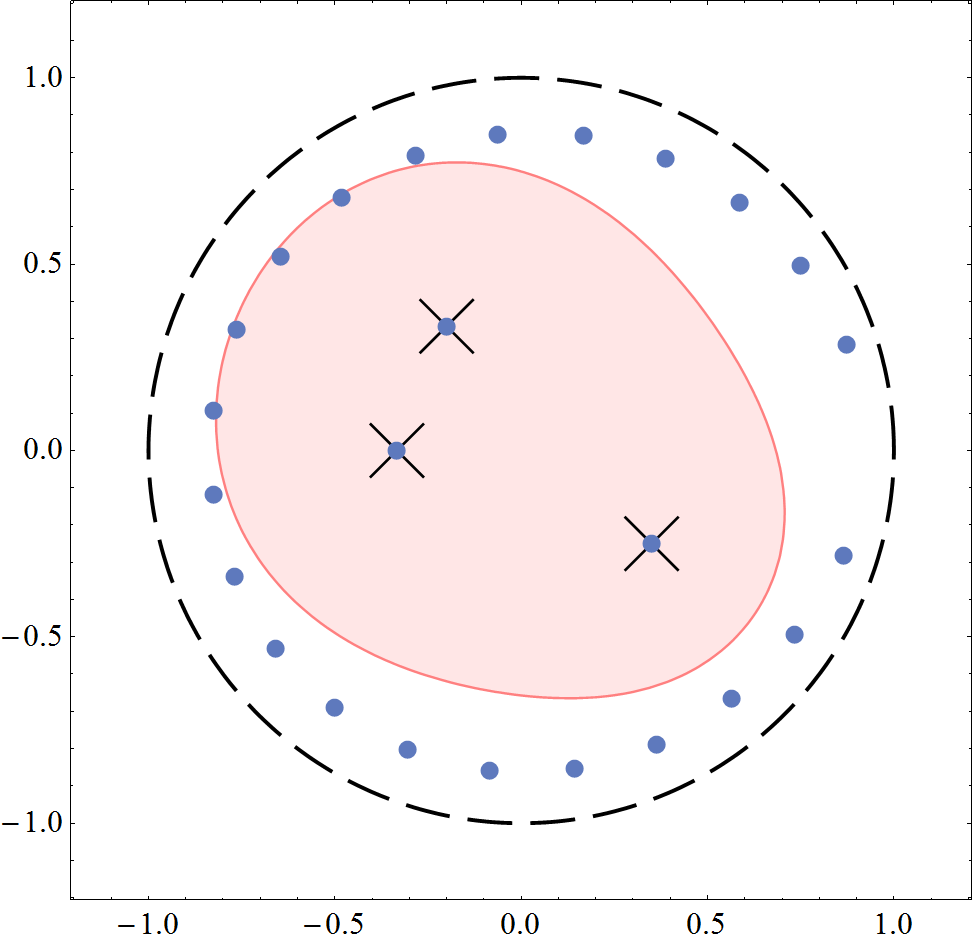}
			& \includegraphics[width=0.3\textwidth]{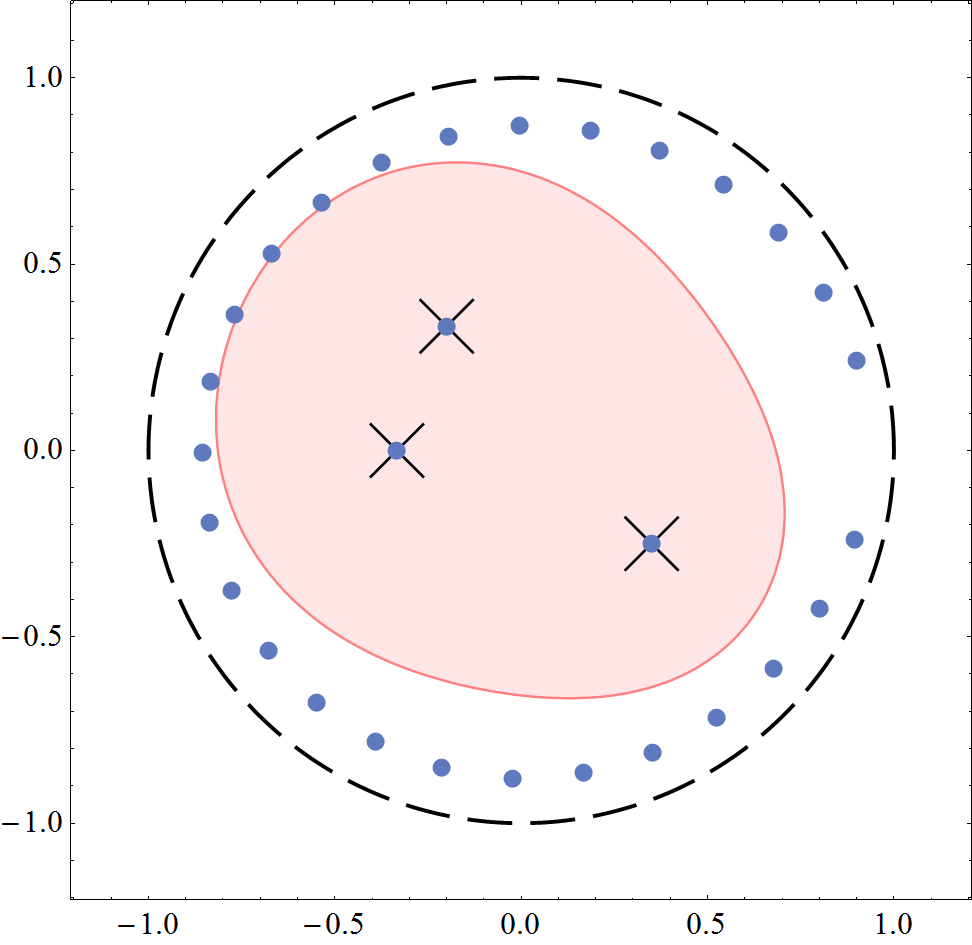} \\
		\includegraphics[width=0.3\textwidth]{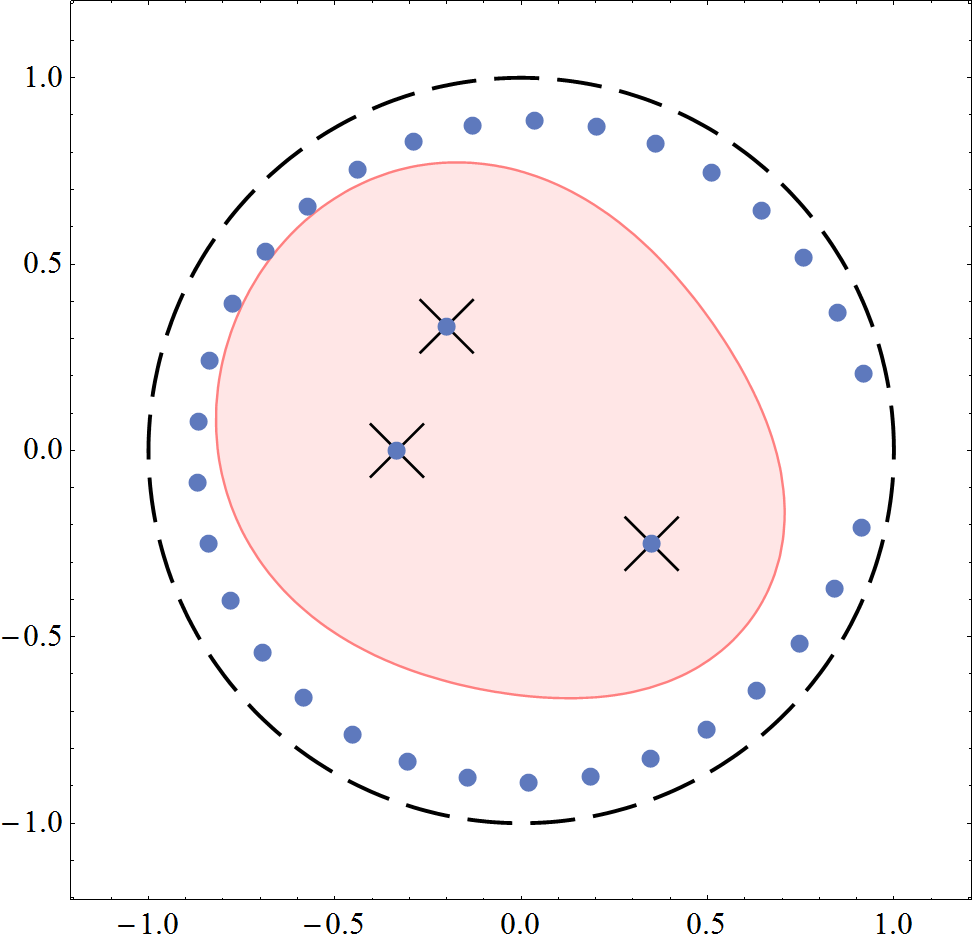}
			& \includegraphics[width=0.3\textwidth]{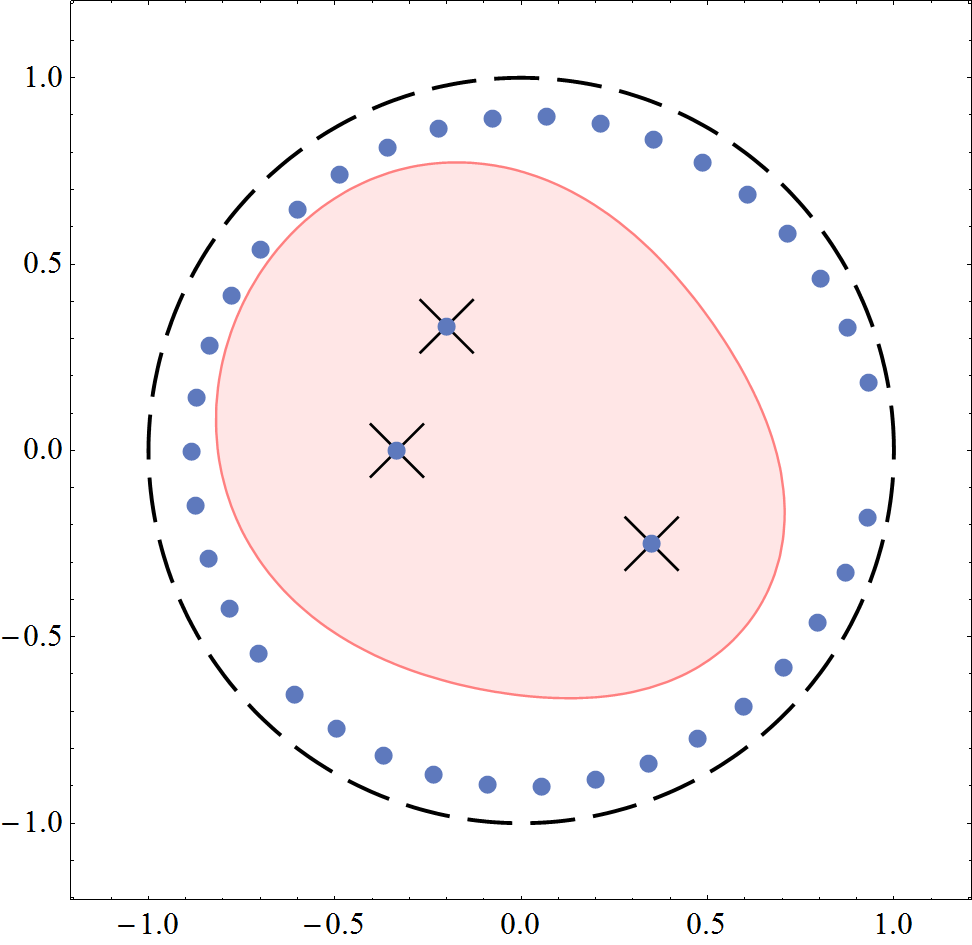}
			& \includegraphics[width=0.3\textwidth]{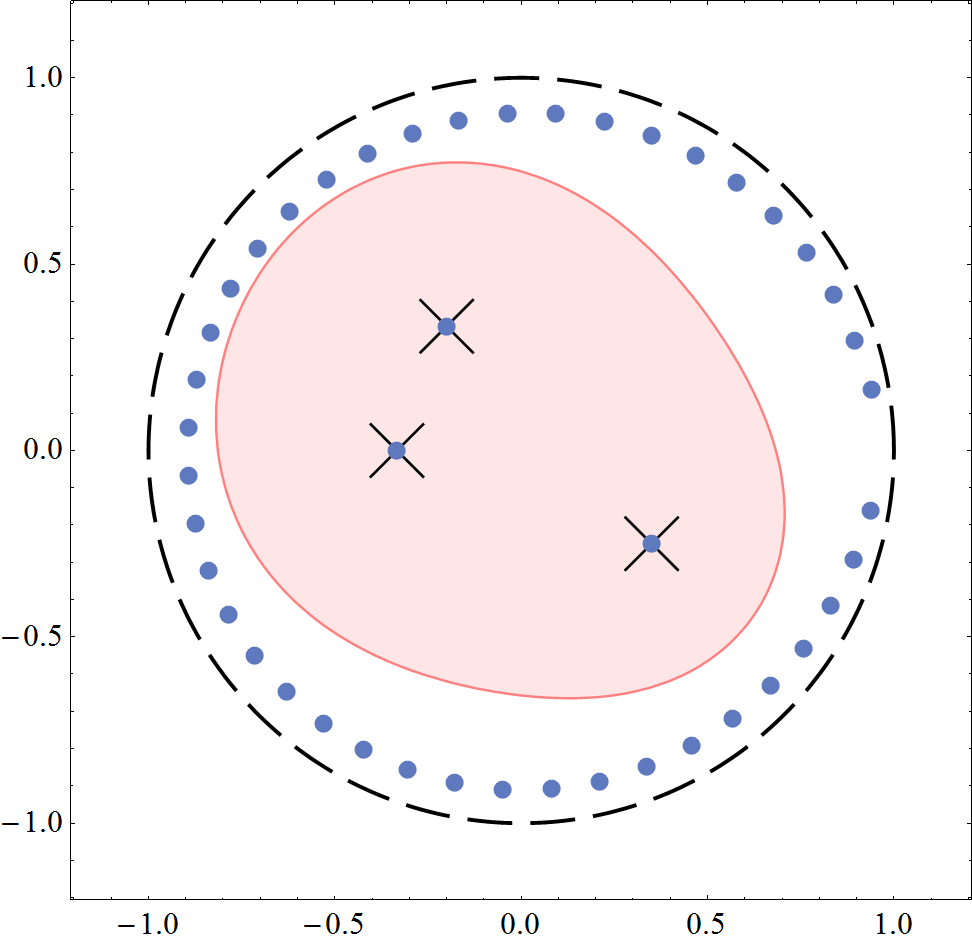}
	\end{tabular}
	\caption[From left-to-right and top-to-bottom, the zeros of the partial sums $p_n$ with $n=5,10,15,20,25,30,35,40,45$, respectively, shown with an arbitrary compact subset of $|z| < 1$ and the unit circle.]{From left-to-right and top-to-bottom, the zeros of the partial sums $p_n$ with $n=5,10,15,20,25,30,35,40,45$, respectively, shown with an arbitrary compact subset of $|z| < 1$ (pink) and the unit circle (dashed).}
	\label{introfig_finiteconvergence}
\end{figure}

In Figure \ref{introfig_finiteconvergence} we can see that all but three zeros of $p_n$ leave the pink set (an arbitrary compact subset of $|z| < 1$) as $n$ grows. In fact it appears that, as $n \to \infty$, these zeros of $p_n$ which don't converge to zeros of $f$ will converge to the unit circle, which is the circle of convergence of the power series. This is no coincidence: Jentzsch showed in 1916 \cite{jentzsch,jentzsch2} that, given any power series with positive, finite radius of convergence, every point on the circle of convergence of the power series will be a limit point of the zeros of the partial sums of the power series. This result was strengthened by Szeg\H{o} in 1922 \cite{szego:jentzsch} when he showed that there is a subsequence $(n_k)$ for which the zeros of $p_{n_k}$ are asymptotically uniformly distributed in angle.

Together these results have come to be known as the Jentzsch-Szeg\H{o} theorems and have sparked a large field of research, especially in finding analogues of the results for other types of approximating polynomials. A modern treatment of this topic can be found in \cite{andrievskiiblatt:discrepancybook}.

The situation is quite different if we instead suppose that the power series has infinite radius of convergence. In this case there is no circle of convergence---or, rather, now the point at $z=\infty$ plays the role of the circle of convergence. Indeed, as $n$ grows it is again true that almost all of the zeros of $p_n$ must leave any fixed compact subset of the region of convergence of the power series, only in this case the region of convergence is all of $\C$\newnot{symbol:CC}. Consequently almost all of the zeros of $p_n$ must tend to $z=\infty$ as $n \to \infty$.

This raises several questions, the simplest of which are probably
\begin{enumerate}[label={(\alph*)}]
\item How quickly do the zeros tend to $\infty$?
\item In the finite-radius case the extra zeros asymptotically formed a circle. What kind of geometry do the zeros have in this case?
\end{enumerate}

Let's take a look at the partial sums of the exponential function as an example. Define
\[
	p_n[\exp](z) = \sum_{k=0}^{n} \frac{z^k}{k!}.
\]
As $n$ grows so do the zeros of $p_n[\exp]$\newnot{symbol:section}, and from Figure \ref{introfig_expunscaled} it looks like they do so pretty quickly.

\begin{figure}[!htb]
	\centering
	\begin{tabular}{ccc}
		\includegraphics[width=0.3\textwidth]{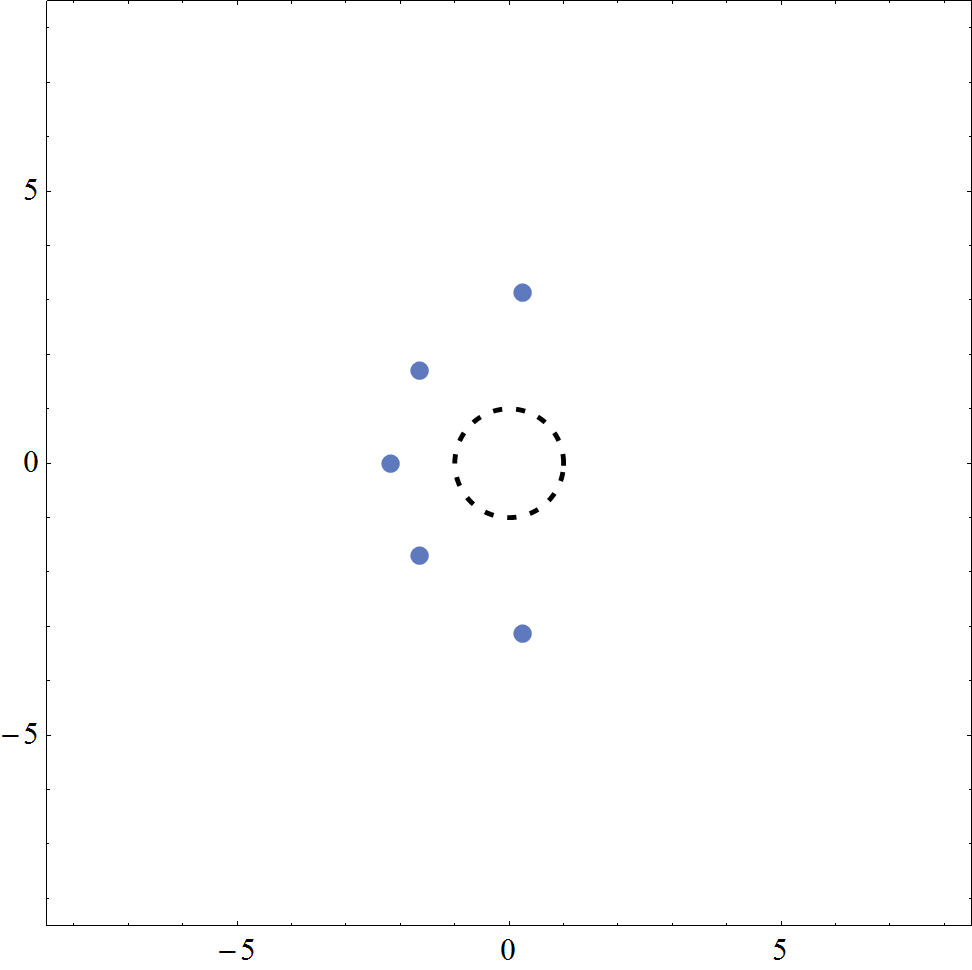}
			& \includegraphics[width=0.3\textwidth]{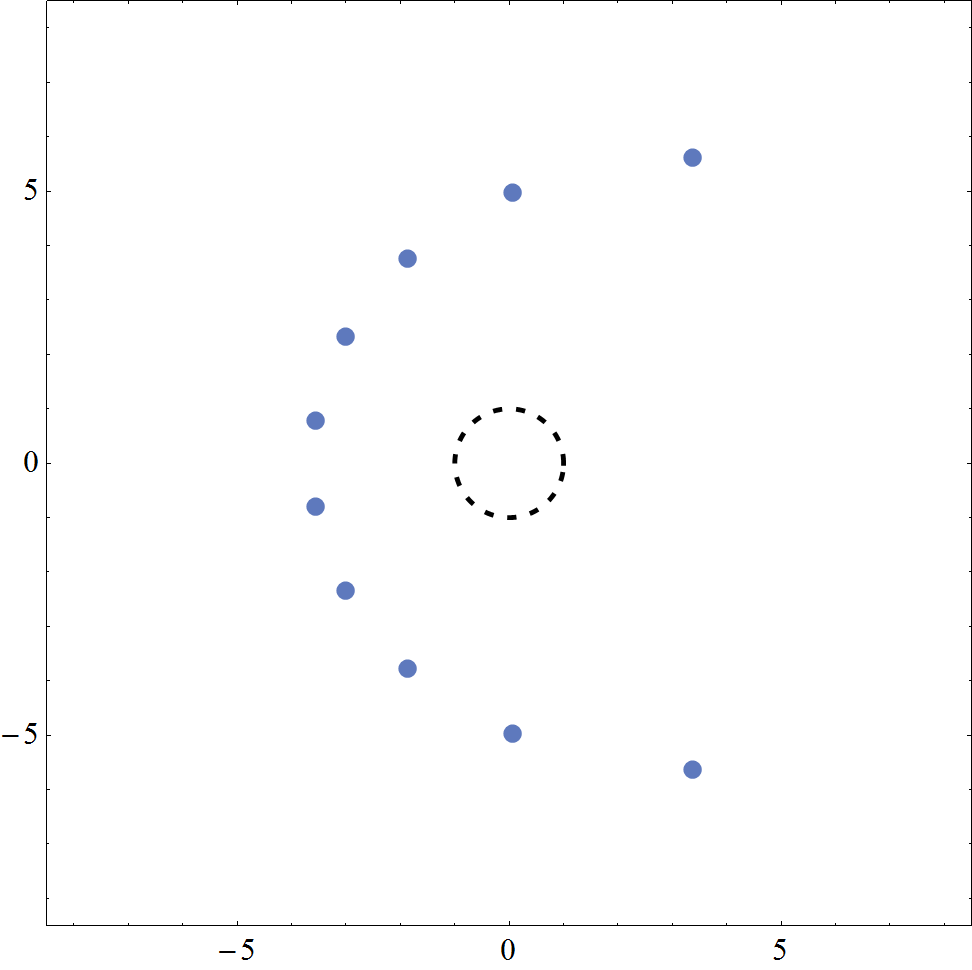}
			& \includegraphics[width=0.3\textwidth]{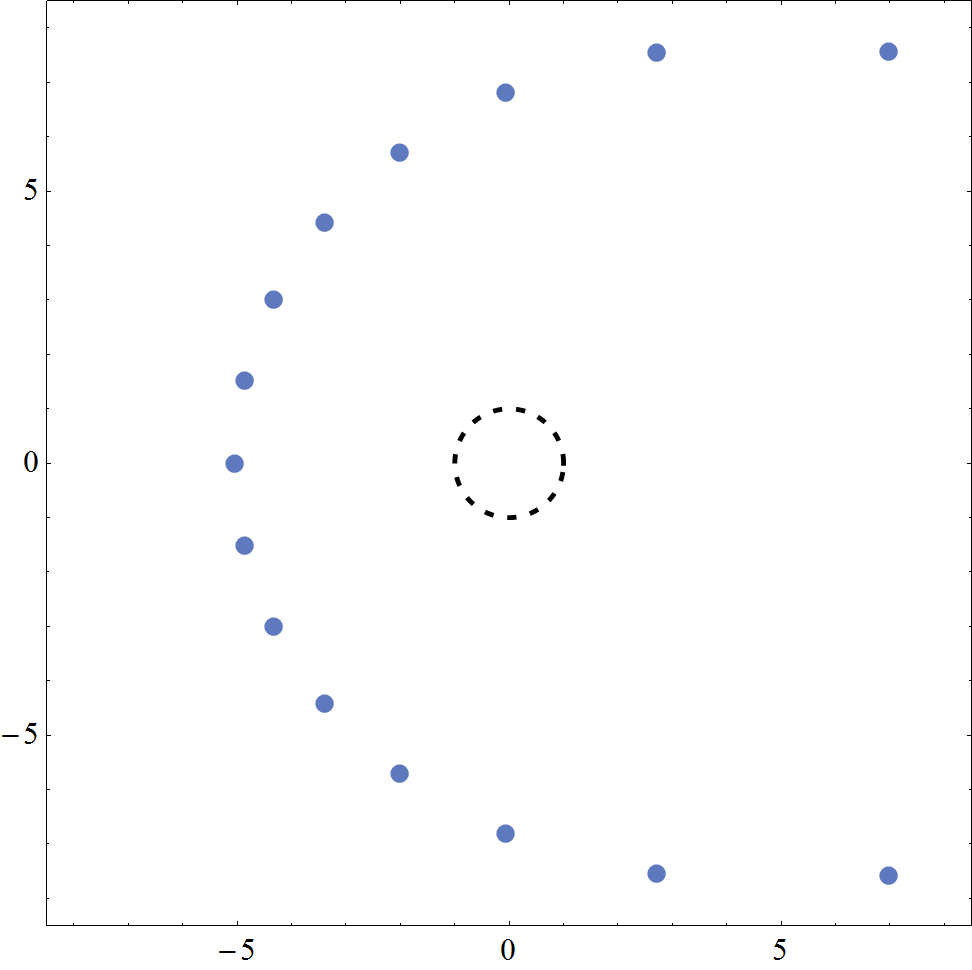}
	\end{tabular}
	\caption[From left-to-right, the zeros of the partial sums {$p_n[\exp]$} with {$n=5,10,15$}, respectively, shown with the unit circle.]{From left-to-right, the zeros of the partial sums {$p_n[\exp]$} with {$n=5,10,15$}, respectively, shown with the unit circle (dashed).}
	\label{introfig_expunscaled}
\end{figure}

These polynomials were studied by Szeg\H{o} in 1924 \cite{szego:exp}. Szeg\H{o} showed that the zeros of $p_n[\exp]$ which tend to $\infty$ do so at a rate comparable to $n$, and further that the zeros of the scaled partial sums $p_n[\exp](nz)$ converge to the piecewise-smooth curve given by
\begin{equation}
\label{introeq_szegocurve}
	S = \left\{ z \in \C : \left|ze^{1-z}\right| = 1, \text{ } |z| \leq 1 \right\}.
\end{equation}

\begin{figure}[!htb]
	\centering
	\begin{tabular}{ccc}
		\includegraphics[width=0.3\textwidth]{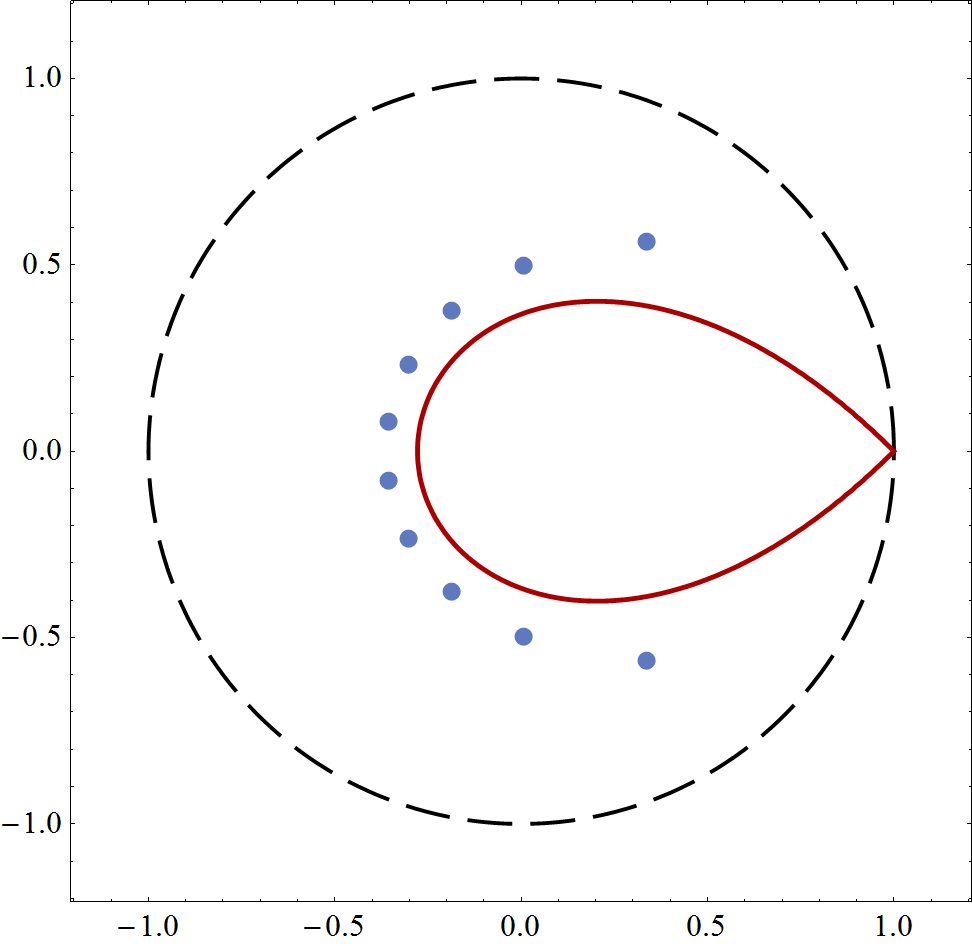}
			& \includegraphics[width=0.3\textwidth]{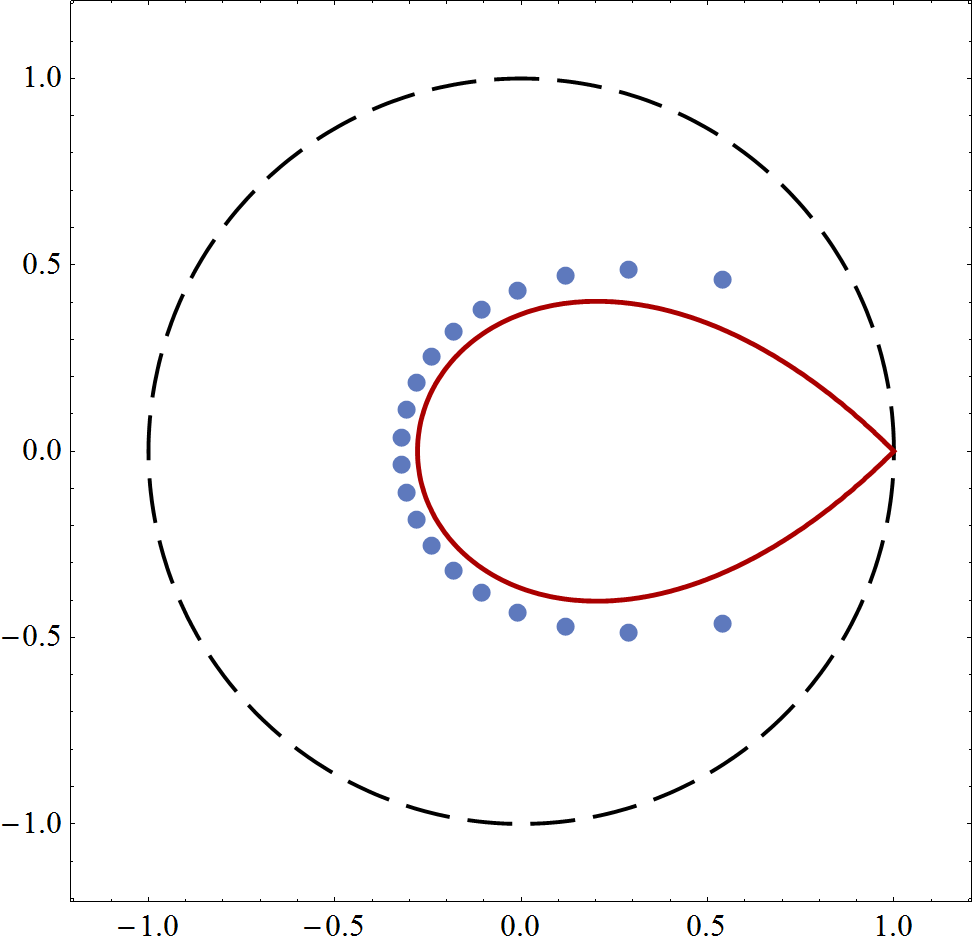}
			& \includegraphics[width=0.3\textwidth]{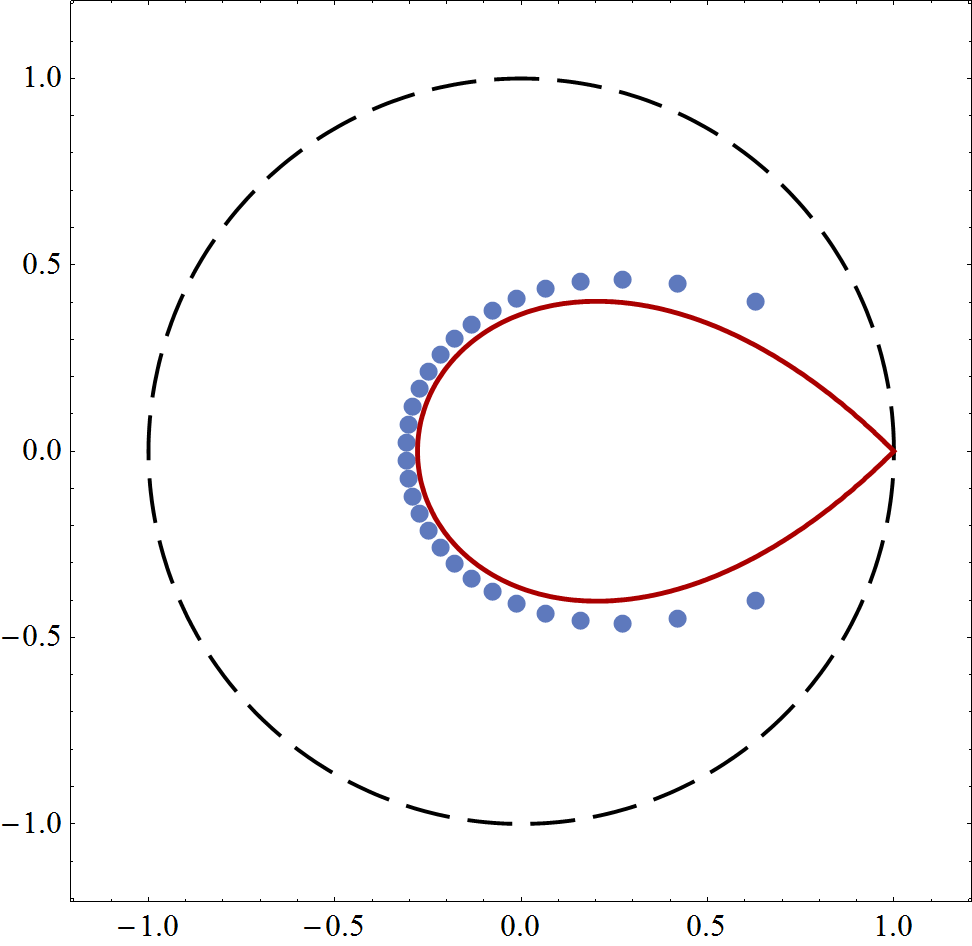}
	\end{tabular}
	\includegraphics[width=0.9\textwidth]{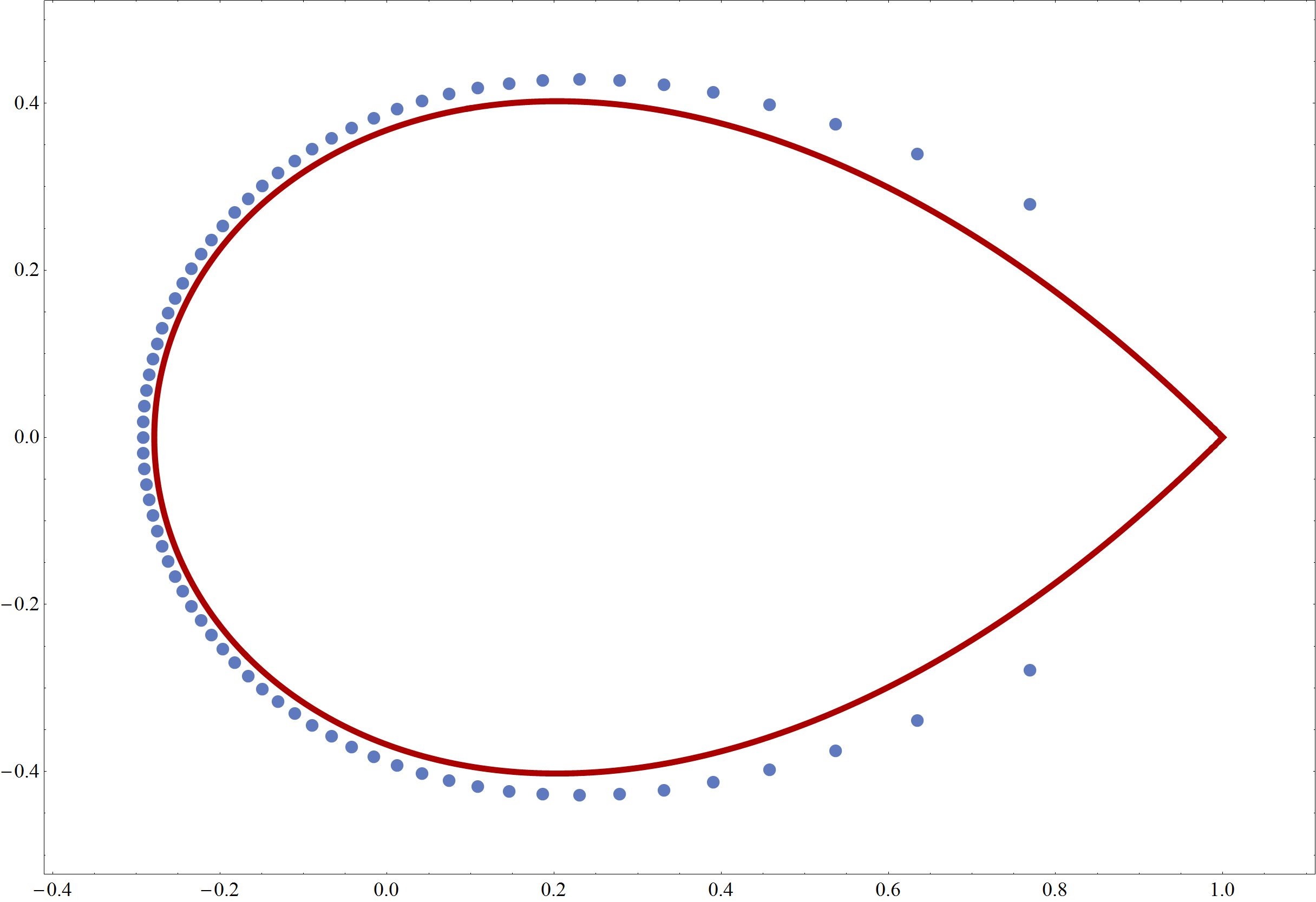}
	\caption[From left-to-right and top-to-bottom, the zeros of the \textbf{scaled} partial sums {$p_n[\exp](nz)$} with {$n=10,20,30,75$}, respectively, shown with the curve {$S$} and the unit circle.]{From left-to-right and top-to-bottom, the zeros of the \textbf{scaled} partial sums {$p_n[\exp](nz)$} with {$n=10,20,30,75$}, respectively, shown with the curve {$S$} (solid red) and the unit circle (dashed).}
	\label{introfig_expscaled}
\end{figure}

\noindent This curve has come to be known as the Szeg\H{o} curve.

Since Szeg\H{o}'s work many authors have studied various aspects of this phenomenon. Of note is Dieudonn\'e, who independently rediscovered many of Szeg\H{o}'s results on the partial sums of the exponential function \cite{dieu:expsections}. In the following sections we will explore what has been studied in detail for these partial sums, the partial sums of other entire functions, as well as what is known in terms of general theory.

\section{Limit Curves and Distribution of the Zeros}
\label{introsec_limitcurves}

In his 1924 paper Szeg\H{o} also studied the angular distribution of the zeros of the partial sums $p_n[\exp]$. He observed that the map $\phi(z) = ze^{1-z}$ maps points of the limit curve $S$ to points on the unit circle in a monotonic fashion, and that the arguments of the zeros of the partial sums are asymptotically uniformly distributed modulo weighting by $\phi$.

To be precise, let $\theta_1,\theta_2$ be such that $0 < \theta_1 < \theta_2 < 2\pi$ and let $z_j = \phi^{-1}(e^{i\theta_j})$, $j=1,2$, so that $z_1,z_2 \in S$. If
\[
	W = \{ z \in \C : \arg z_1 \leq \arg z \leq \arg z_2 \},
\]
\begin{equation}
\label{introeq_zndef}
	Z_n = \{ z \in \C : p_n[\exp](z) = 0 \},
\end{equation}
and $\#(W \cap Z_n)$ is the number of points in the set $W \cap Z_n$, then Szeg\H{o} showed that there is a constant $C_1$ depending only on $\theta_1$ and $\theta_2$ such that
\begin{equation}
\label{introeq_szegoangular}
	\left|\frac{\#(W \cap Z_n)}{n} - \frac{\theta_2 - \theta_1}{2\pi}\right| \leq \frac{C_1}{n}.
\end{equation}
Note, however, that this result doesn't apply to sectors which contain the positive real axis. Taking another look at Figures \ref{introfig_expunscaled} and \ref{introfig_expscaled} gives us a clue as to why this might be: the zeros of the scaled partial sums $p_n[\exp](nz)$ seem to approach the point of the limit curve $S$ which lies on the positive real axis (namely the point $z=1$) much more slowly than they approach the rest of the curve. We might suspect then that the $1/n$ bound on the right-hand side of \eqref{introeq_szegoangular} wouldn't hold near the positive real axis. Indeed, Szeg\H{o}'s result was strengthened by Andrievskii, Carpenter, and Varga in 2006 \cite{acv:angulardistribution} to include such sectors. They showed that there are absolute positive constants $C_2$ and $a$ such that, for any choice of $\theta_1$ and $\theta_2$ satisfying $0 < \theta_2 - \theta_1 < 2\pi$,
\[
	\left|\frac{\#(W \cap Z_n)}{n} - \frac{\theta_2 - \theta_1}{2\pi}\right| \leq \frac{C_2}{n^a}.
\]
We also note that Szeg\H{o}'s result has been reformulated in the modern language of weak*-convergence of measures \cite{pritskervarga:weighted}.

There has also been a great deal of interest in the radial distribution of zeros. This is most often framed as the study of the distance from the zeros of the scaled partial sums $p_n[\exp](nz)$ to their limit curve $S$.

In 1966, Buckholtz \cite{buckholtz:expcharacter} (building on his earlier work \cite{buckholtz:copapprox}) obtained two striking results in this vein. First, Buckholtz showed using a very simple argument that none of the zeros of $p_n[\exp](nz)$ lie on the interior of the curve $S$, which is characterized by the inequalities
\[
	\left| ze^{1-z} \right| \leq 1, \quad |z| \leq 1.
\]
In other words, the zeros of the scaled partial sums approach $S$ from the outside. This can be seen clearly in Figure \ref{introfig_expscaled}. Second, Buckholtz showed that all zeros of $p_n[\exp](nz)$ lie within a distance of $2e/\sqrt{n}$ of $S$: if $p_n[\exp](nz_0) = 0$ then
\[
	\max_{s \in S} |s-z_0| \leq \frac{2e}{\sqrt{n}}.
\]

\begin{figure}[!htb]
	\centering
	\begin{minipage}[c]{0.50\textwidth}
		\includegraphics[width=\textwidth]{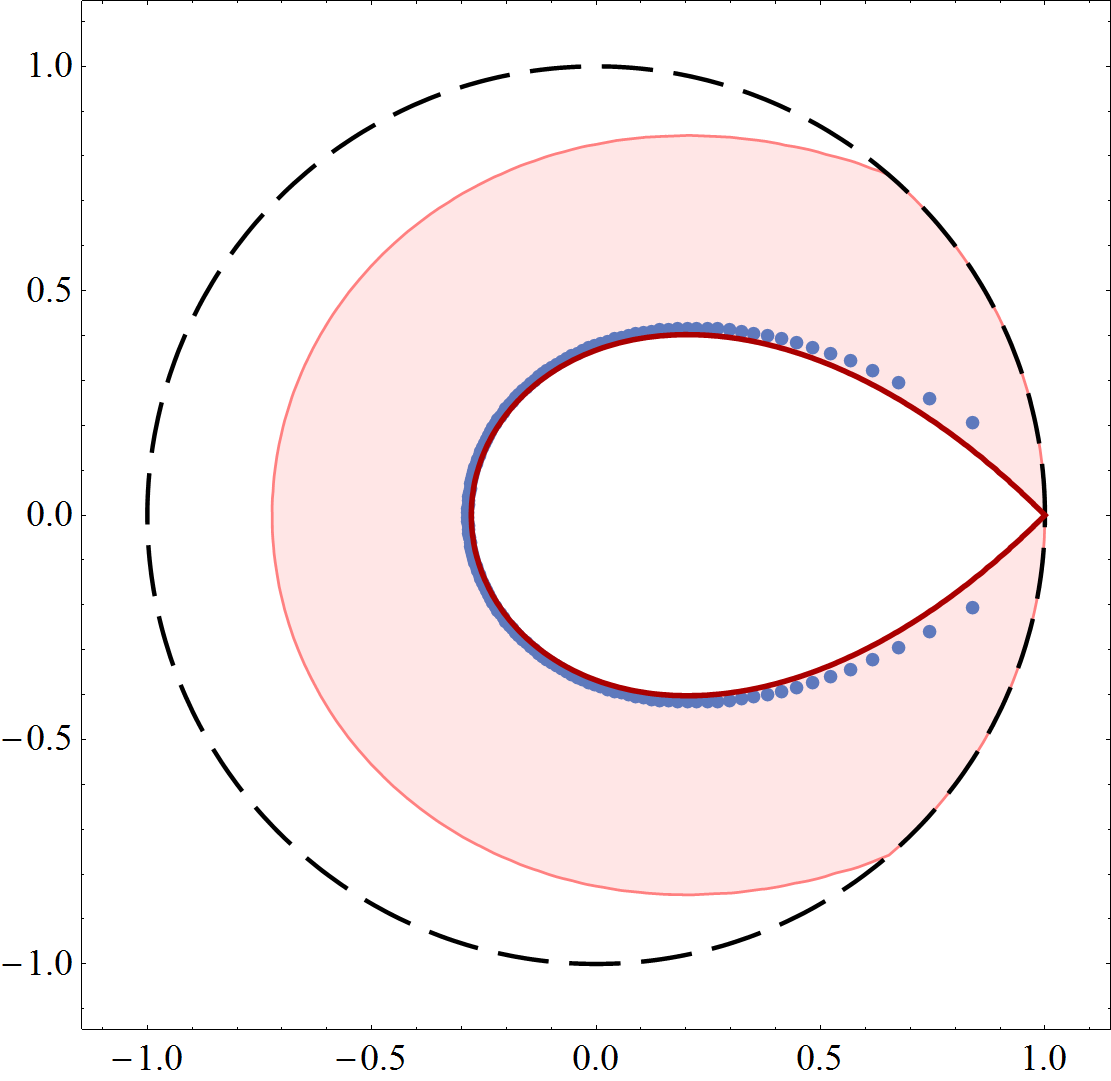}
	\end{minipage}
	\hspace{0.3cm}
	\begin{minipage}[c]{0.4\textwidth}
		\caption[The zeros of {$p_{150}[\exp](150z)$}, shown with the bounding region found by Buckholtz, the Szeg\H{o} curve {$S$}, and the unit circle.]{The zeros of $p_{150}[\exp](150z)$ as blue dots, shown with the bounding region found by Buckholtz (pink), the Szeg\H{o} curve $S$ (solid red), and the unit circle (dashed). All zeros of $p_n[\exp](nz)$ lie in the disk $|z| \leq 1$ by the Enestr\"om-Kakeya theorem \cite{asv:ke}, so Buckholtz's region is restricted to this disk in the image.}
		\label{introfig_buckholtz}
	\end{minipage}
\end{figure}

It turns out that, except for the constant $2e$, Buckholtz's estimate is asymptotically best-possible in the sense that
\[
	\max_{z \in Z_n} \min_{s \in S} |s-z| = \Theta\!\left(n^{-1/2}\right)
\]
as $n \to \infty$, where $Z_n$ is as defined in \eqref{introeq_zndef} and $\Theta(\cdots)$ is Big Theta notation (see Definition \ref{prelimsdef_bigtheta}). In fact it was shown by Carpenter, Varga, and Waldvogel in \cite{cvw:expasympi} that
\[
	\liminf_{n \to \infty} \sqrt{n} \cdot \max_{z \in Z_n} \min_{s \in S} |s-z| \geq \re w_1 + \im w_1 \approx 0.636657,
\]
where $w_1 \approx -1.35481 + 1.99147i$\newnot{symbol:real}\newnot{symbol:imag} is the zero of the complementary error function $\erfc$ with smallest modulus in the upper half-plane. (For information about $\erfc$ and its zeros we refer the reader to \cite{fettis:erfczeros}.)\newnot{symbol:erfc}

A major study of these types of detailed asymptotics was carried out by Edrei, Saff, and Varga, who published a monograph on the topic in 1983 \cite{esv:sections}. In the monograph the authors studied the partial sums of the Mittag-Leffler function $E_{1/\lambda}$, which can be considered a generalization of the exponential function $e^z$ and is defined for $z \in \C$ and $\lambda > 1$ by
\[
	E_{1/\lambda}(z) = \sum_{k=0}^{\infty} \frac{z^k}{\Gamma(k/\lambda + 1)}.
\]
When $\lambda = 1$ the usual exponential function is recovered. When $\lambda > 1$, however, the Mittag-Leffler function is an entire function of order $\lambda$. (See Section \ref{prelimssec_entirefunctions} for the definition of the order of an entire function.)

Let
\[
	p_n[E](z) = \sum_{k=0}^{n} \frac{z^k}{\Gamma(k/\lambda + 1)}.
\]
denote the $n^\th$ partial sum of $E_{1/\lambda}$. Edrei, Saff, and Varga found that the zeros of $p_n[E]$ which do not converge to zeros of $E_{1/\lambda}$ grow at a rate comparable to
\[
	r_n := e^{1/(2n)} \left(\frac{n}{\lambda}\right)^{1/\lambda},
\]
and, more precisely, that the corresponding zeros of the scaled partial sums $p_n[E](r_n z)$ converge to the curve
\begin{align*}
S_E &= \left\{ z \in \C : \left|z^\lambda \exp\!\left(1-z^\lambda\right)\right| = 1, \text{ } |z| \leq 1, \text{ and } |{\arg z}| \leq \frac{\pi}{2\lambda} \right\} \\
	&\qquad \cup \left\{ z \in \C : |z| = e^{-1/\lambda} \text{ and } |{\arg z}| \geq \frac{\pi}{2\lambda} \right\}.
\end{align*}
Further, the authors found that the zeros of the scaled partial sums which approach points $\xi \in S_E$ with $\xi \neq 1$ and $|{\arg \xi}| \neq \pi/(2\lambda)$ do so at a rate of $\Theta(\log n/n)$\newnot{symbol:log} and are separated from each other by a distance of $\Theta(n^{-1})$, and those which approach the point $\xi = 1$ do so at a rate of $\Theta(n^{-1/2})$ and are separated by a distance of $\Theta(n^{-1/2})$. So, not only do the zeros approach the smooth arcs of the limit curve more quickly than they approach the corner at $z=1$, but they are also clustered more closely together near the arcs than they are near the corner.

\begin{figure}[!htb]
	\centering
	\includegraphics[width=0.9\textwidth]{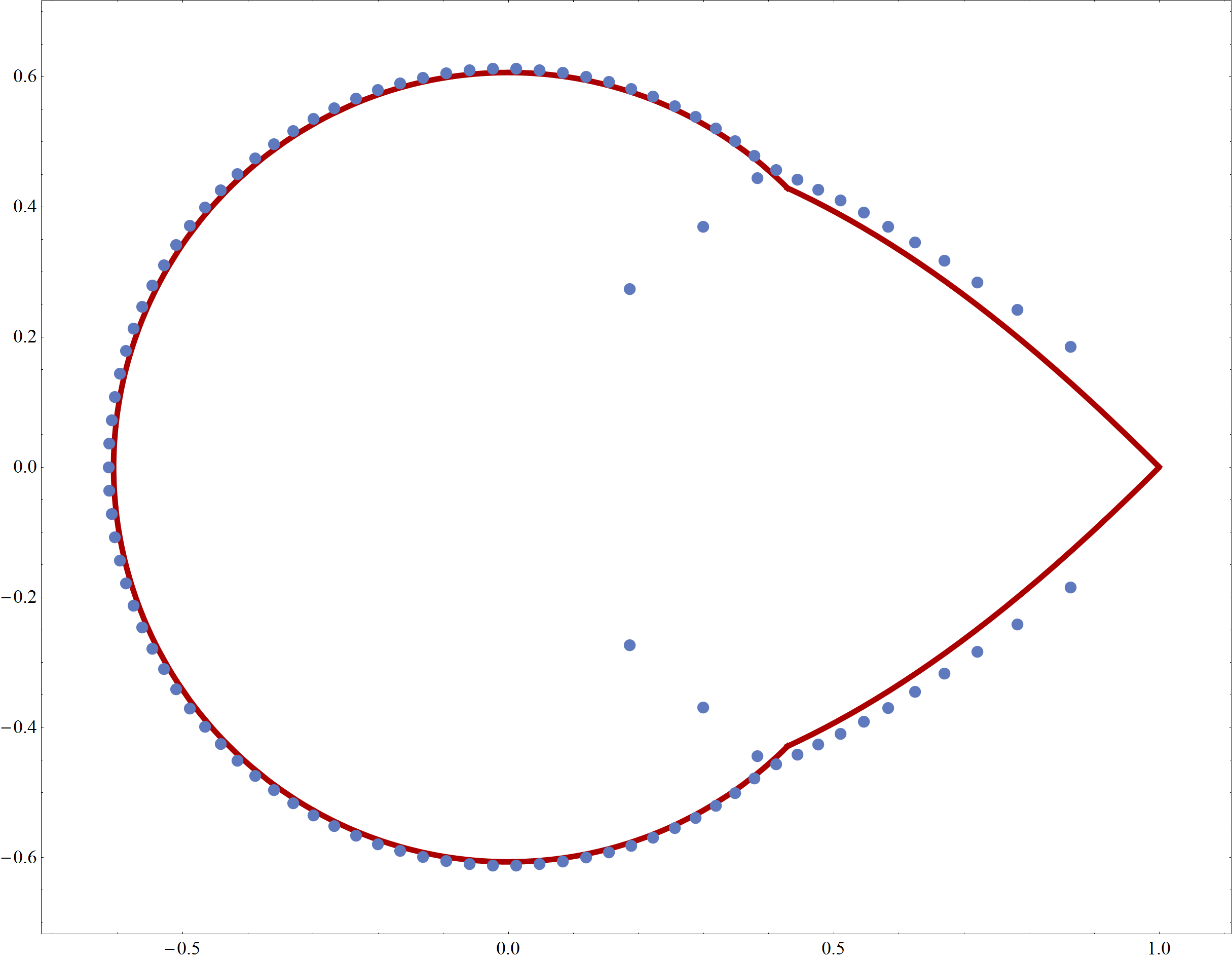}
	\caption[The zeros of the \textbf{scaled} partial sums {$p_{105}[E](r_{105} z)$} with {$\lambda = 2$}, shown with the curve {$S_E$}.]{The zeros of the \textbf{scaled} partial sums {$p_{105}[E](r_{105} z)$} with {$\lambda = 2$}, shown with the curve {$S_E$} (solid red).}
	\label{introfig_mittagscaled}
\end{figure}

Even though it was assumed that $\lambda > 1$, these facts are still true for the zeros of the partial sums of the exponential function ($\lambda = 1$): the zeros of $p_n[\exp](nz)$ which approach points $\xi \in S$ with $\xi \neq 1$ do so at a rate of $\Theta(\log n/n)$ and, as we indicated above, the zeros which approach $\xi = 1$ do so at a rate of $\Theta(n^{-1/2})$. See, for example, \cite{cvw:expasympi,vc:expasympii} for detailed discussions.

We will go into much more detail about the results of Edrei, Saff, and Varga in Section \ref{introsec_width}. The content of their monograph \cite{esv:sections} inspired much of the work that appears in this thesis, and the Mittag-Leffler function $E_{1/\lambda}$ is essentially the archetype of the functions we will consider.

Other functions whose partial sums have been studied include the sine and cosine functions in Szeg\H{o}'s original paper \cite{szego:exp} as well as more recently by Kappert \cite{kappert:sincos} and by Varga and Carpenter \cite{vc:sincosasympi,vc:sincosasympii,vc:sincosasympiii}, the confluent hypergeometric functions ${}_1F_1(1,b,z)$ by Norfolk \cite{norfolk:1f1}, finite sums of exponentials by Bleher and Mallison \cite{mallison:expsums}, and exponential integrals of the form
\[
	f(z) = \int_a^b g(t) e^{zt}\,dt
\]
by both Norfolk \cite{norfolk:transforms} and the current author \cite{vargas:limitcurves}.

A general study of these types of results was undertaken by Rosenbloom in his 1944 PhD thesis \cite{rosen:thesis}, later summarized in \cite{rosen:distrib}. One basic yet fundamental contribution of his is that, if
\[
	f(z) := \sum_{k=0}^{\infty} a_k z^k
\]
is an entire function of positive, finite order, then the zeros of
\[
	p_n(z) := \sum_{k=0}^{n} a_k z^k
\]
which do not converge to zeros of $f$ grow at a rate of approximately $|a_n|^{-1/n}$. Before giving the precise statement of this result let's take a moment to build some intuition for it.

First, if $a_0,a_n \neq 0$ then the product of the moduli of the zeros of $p_n$ is equal to $|a_0/a_n|$ in absolute value, and so the geometric mean of their moduli satisfies
\[
	\left| \frac{a_0}{a_n} \right|^{1/n} \sim |a_n|^{-1/n}
\]
as $n \to \infty$. If we were to scale the zeros by this factor then the new geometric mean of their moduli would be approximately equal to $1$, and we might expect then that the scaled zeros would be bounded. This interpretation is due to Dupuy \cite{dupuy:partialsumsanswer}.

We can also draw a comparison with the case when the power series has a finite radius of convergence $R$. We mentioned in the previous section that the the zeros in this case cluster on the circle of convergence for the power series. There is a subsequence $(n_k)$ such that
\[
	\lim_{k \to \infty} |a_{n_k}|^{-1/n_k} = R,
\]
so it follows that the zeros of the scaled partial sums $p_{n_k}(|a_{n_k}|^{-1/n_k} z)$ cluster on the unit circle. When, on the other hand, the power series has infinite radius of convergence, Rosenbloom's result below shows that this scaling indeed still essentially fixes the moduli of the zeros of the partial sums at $1$. From this perspective Rosenbloom's results are a direct generalization of the Jentzsch-Szeg\H{o} theorems to entire power series with positive, finite order.

Let $\#_n^\angle(\theta_1,\theta_2)$\newnot{symbol:numsec} denote the number of zeros of the partial sum $p_n$ in the sector $\theta_1 \leq \arg z \leq \theta_2$ for some appropriate determination of $\arg z$ and let $\#_n^\circ(r)$\newnot{symbol:numdisk} denote the number of zeros of $p_n$ in the disk $|z| \leq r$. We say that a sequence of partial sums $(p_{n_k})$ has a positive fraction of zeros in any sector with vertex at the origin if
\[
	\liminf_{k \to \infty} \frac{\#_{n_k}^\angle(\theta_1,\theta_2)}{n_k} > 0
\]
for any $\theta_1 < \theta_2$.

\begin{theorem}[Rosenbloom]
\label{introthm_rosen1}
Let $f$ be an entire function of positive, finite order $\lambda$, let $p_n$ denote the $n^\th$ partial sum of the Maclaurin series for $f$, and let $r_n = |a_n|^{-1/n}$. Then there is a subsequence $(n_k)$ such that the sequence of partial sums $p_{n_k}$ has a positive fraction of zeros in any sector with vertex at the origin and, for every $\epsilon > 0$, the number of zeros of $p_{n_k}$ satisfying
\[
	|z| \geq \left(e^{1/\lambda} + \epsilon\right) r_{n_k}
\]
is bounded and all zeros lie in the disk
\[
	|z| \leq \left(2e^{1/\lambda} + \epsilon\right) r_{n_k}
\]
for $k$ large enough. Further, for $0 \leq t < 1$ and $\epsilon > 0$ we have
\[
	\liminf_{k \to \infty} \frac{\#_{n_k}^{\circ}\left((e^{1/\lambda} + \epsilon)r_{n_k}\right) - \#_{n_k}^{\circ}(t r_{n_k})}{n_k} \geq 1 - t^{\lambda} > 0.
\]
\end{theorem}

The means by which such a sequence of indices $(n_k)$ can be constructed was given by Norfolk in \cite{norfolk:widthconj}.  In doing so, Norfolk furnished a constructive proof of the above result.

Not only did Rosenbloom obtain the rate at which the zeros of $p_n$ grow, but he also showed that the zeros of the scaled partial sums $p_n(|a_n|^{-1/n} z)$ will indeed converge to some kind of limit curve, and, further, that the shape of this curve determines the asymptotic density of the zeros along any segment of it.

\begin{theorem}[Rosenbloom]
\label{introthm_rosen2}
Let $f$, $\lambda$, $p_n$, $r_n$, and $(n_k)$ be as in Theorem \ref{introthm_rosen1} and suppose that the following conditions hold:
\begin{enumerate}[label=(\arabic*)]
\item the quantity $f(r_{n_k} z)^{1/n_k}$ converges uniformly to a single-valued analytic function $g$ in some subdomain $X$ of the disk $|z| \leq e^{1/\lambda}$;
\item $w = g(z)/z$ maps $X$ univalently onto a domain $X_1$;
\item no limit function of the sequence
\[
	T_{n_k}(z) = \frac{f(r_{n_k} z) - p_{n_k}(r_{n_k} z)}{z^{n_k}}
\]
is identically zero in $X$; and
\item $T_{n_k}(z) \neq 0$ in $X$ for $k$ large enough.
\end{enumerate}
Then the only limit points of the zeros of $p_{n_k}(r_{n_k} z)$ in $X$ are the points on the curve $|g(z)/z| = 1$, and their images in $X_1$ under the mapping $w = g(z)/z$ are equidistributed about the unit circle $|w| = 1$; that is, the number of zeros of $p_{n_k}(r_{n_k} z)$ which accumulate on any arc in $X$ whose image under $w=g(z)/z$ has length $\alpha$ is asymptotically $n_k \alpha/2\pi$.
\end{theorem}

These two remarkable theorems encompass what is essentially the (asymptotic) first-order theory of the zeros of the partial sums of power series for entire functions of positive, finite order: this is what we need to scale the zeros by to make them converge and this is what they converge to.

\section{A Parabolic Zero-Free Region for the Partial Sums of the Exponential Function}
\label{introsec_zerofree}

In contrast to Rosenbloom's results for entire functions of positive, finite order, Carlson showed in 1948 \cite{carlson:entiresector} that, supposing $f$ is a function analytic near the origin and $p_n$ is the $n^\th$ partial sum of its Maclaurin series, if there is a sector with vertex at the origin in which $p_n$ has $o(n)$ zeros as $n \to \infty$ then $f$ is an entire function of order $0$. (Here $o(\cdots)$ is Little O notation, see Definition \ref{prelimsdef_littleo}.) Because the exponential function is entire of order $1$, it follows that there cannot be any such zero-free sector for its partial sums.

An insight into the next part of the story appears in Szeg\H{o}'s collected papers \cite{szego:collected1} just after its reproduction of \cite{szego:exp}:

\begin{displayquote}
Varga was interested in the location of the zeros of the partial sums of the power series of $e^z$ because of applications to the analysis of stability for some numerical methods for solving systems of ordinary differential equations. Fast computing machines were then available and he asked Iverson to compute the zeros for $n$ up to 20.
\end{displayquote}

Possibly inspired by Iverson's findings, Varga showed in 1953 \cite{varga:strips} that the partial sums of $e^z$ have no zeros in the semi-infinite strip $|{\im z}| < \sqrt{6}$, $\re z > 0$.

\begin{figure}[!htb]
	\centering
	\includegraphics[width=0.7\textwidth]{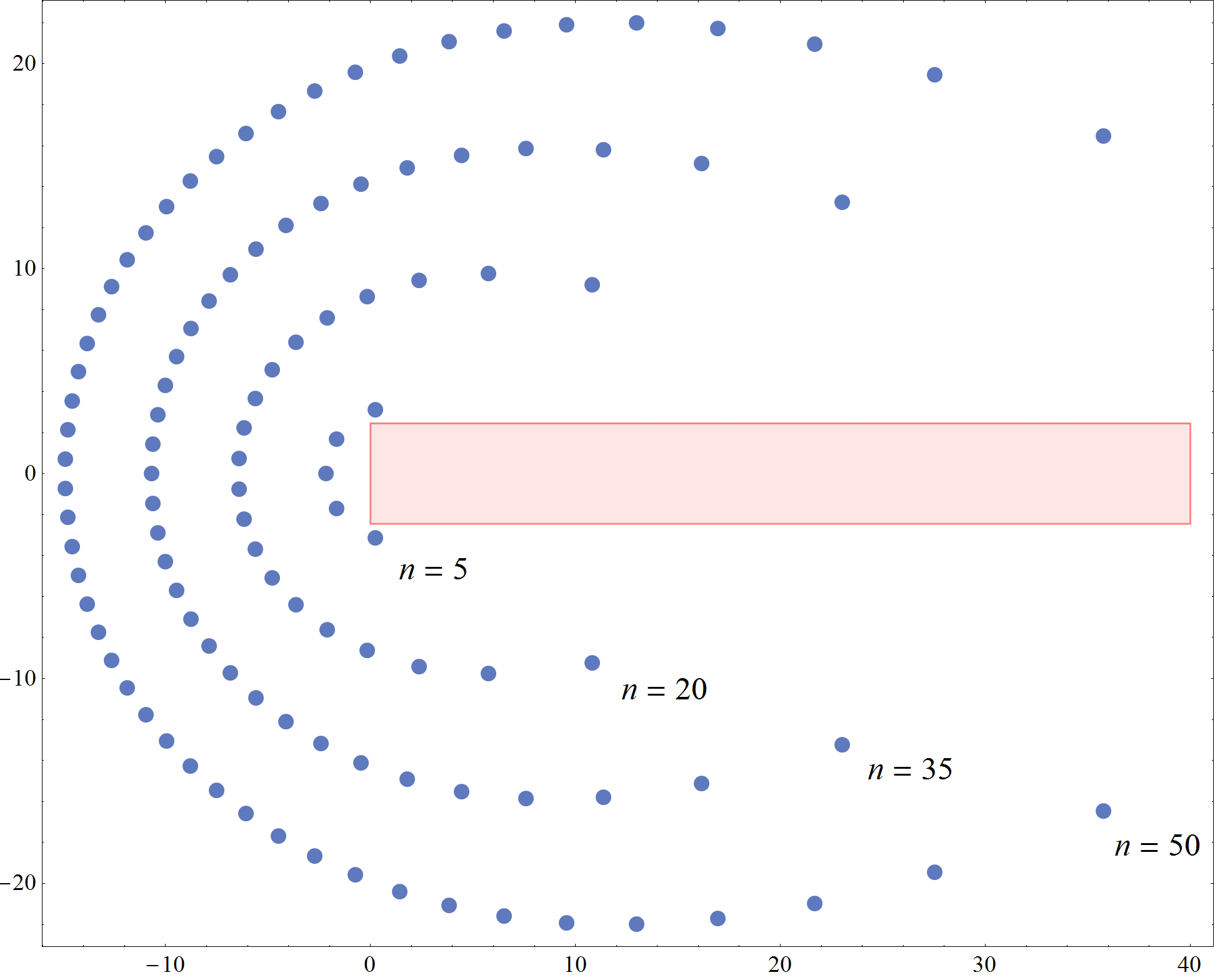}
	\caption{The zeros of the partial sums $p_n[\exp](z)$ for $n=5,20,35,50$, shown with the zero-free strip $|y| \leq \sqrt{6}$, $x > 0$ found by Varga.}
	\label{introfig_vargastrip}
\end{figure}

\noindent The next year Iverson published a paper \cite{iverson:zeros} containing his numerical findings in which he remarked that there seemed to be a large zero-free region surrounding the positive real axis---larger than the zero-free strip found by Varga---which was not yet described by the available literature.  Some twenty years later this zero-free region was investigated by Newman and Rivlin in \cite{newriv:expzeros,newriv:expzeroscorrect} and in a more general setting by Saff and Varga in \cite{sv:zerofree}, and in the latter paper it was shown that no partial sum has a zero in the parabolic region
\begin{equation}
\label{introeq_expsvparab}
	\{x+iy : y^2 \leq 4(x+1) \text{ and } x > -1\}.
\end{equation}

\begin{figure}[!htb]
	\centering
	\includegraphics[width=0.7\textwidth]{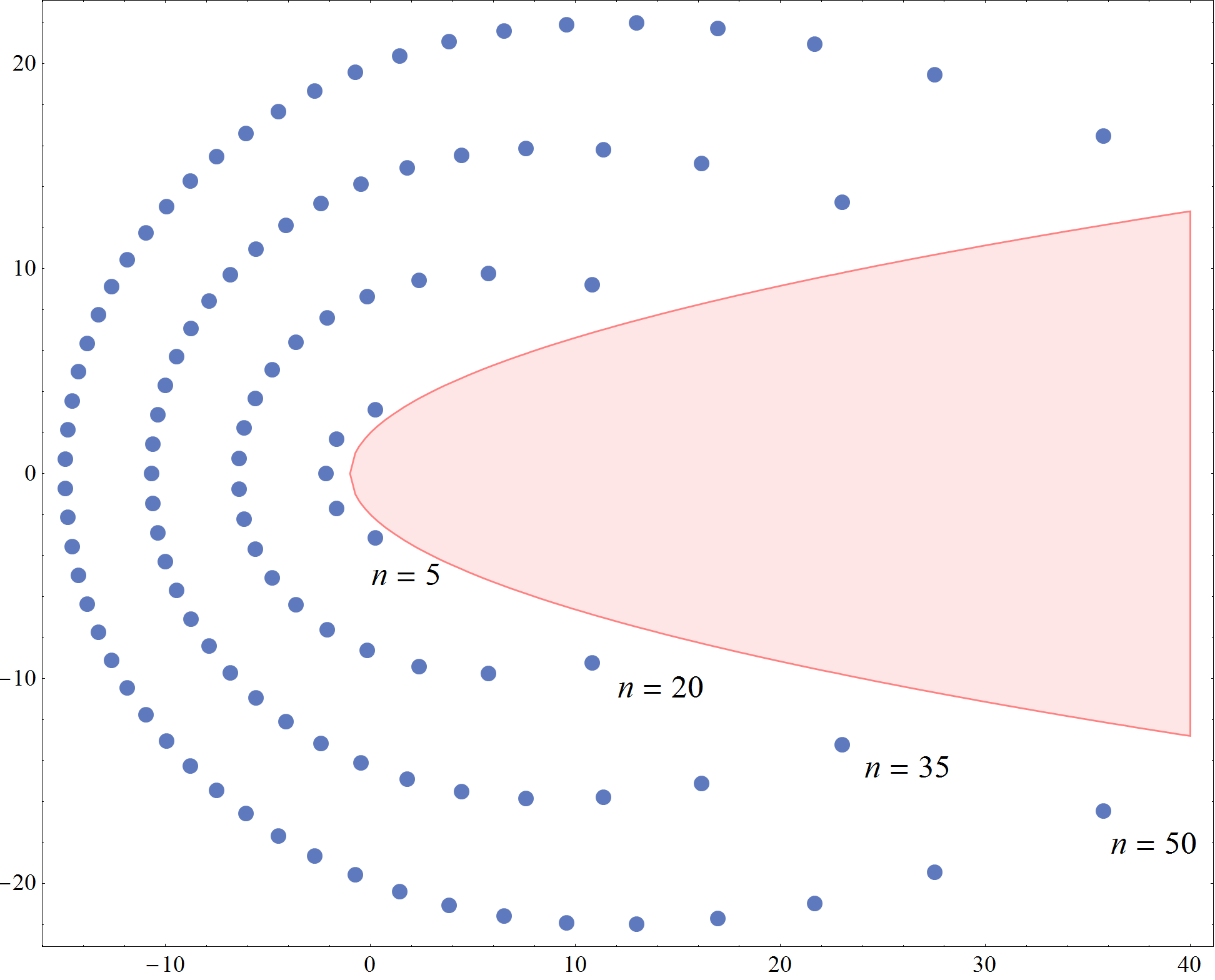}
	\caption{The zeros of the partial sums {$p_n[\exp](z)$} for $n=5,20,35,50$, shown with the zero-free parabola $y^2 \leq 4(x+1)$, $x > -1$ found by Saff and Varga.}
	\label{introfig_saffvargaparab}
\end{figure}

\noindent Conversely, in their first paper \cite{newriv:expzeros} Newman and Rivlin calculated a certain scaling limit for the partial sums $p_n[\exp]$ in which the argument traces out a parabolic arc in the plane.

\begin{theorem}[Newman and Rivlin]
\label{introthm_newriv}
\[
	\lim_{n\to\infty} \frac{p_n[\exp](n+w\sqrt{n})}{\exp(n+w\sqrt{n})} = \frac{1}{2} \erfc\!\left(\frac{w}{\sqrt{2}}\right)
\]
uniformly for $w$ restricted to any compact subset of $\C$.
\end{theorem}

The function $\erfc$ in the theorem statement above is known as the complementary error function and is defined by
\[
	\erfc(z) = \frac{2}{\sqrt{\pi}} \int_z^\infty e^{-s^2}\,ds,
\]
where the contour of integration is the horizontal line starting at $s=z$ and extending to the right to $s = z+\infty$. This function has infinitely many zeros, all of which lie in the sectors $\pi/2 < |{\arg z}| < 3\pi/4$ and approach asymptotically the rays $\arg z = \pm 3\pi/4$ (see, e.g., \cite{fettis:erfczeros}).

If $w_0 = s_0 + it_0$ is any zero of $\erfc(w)$ then, by Hurwitz's theorem (Theorem \ref{introthm_hurwitz}), Theorem \ref{introthm_newriv} says that $p_n[\exp](z)$ has a zero $z_0 = x_0 + iy_0$ of the form
\[
	z_0 = n + w_0 \sqrt{2n} + o\!\left(\sqrt{n}\right).
\]
In particular we have $x_0 = n + s_0 \sqrt{2n} + o(\sqrt{n})$ and $y_0 = t_0 \sqrt{2n} + o(\sqrt{n})$, so that
\[
	x_0 = \frac{1}{2}\left(\frac{y_0}{t_0}\right)^2 + O(y_0)
\]
as $n \to \infty$. In other words, for large $n$ the partial sum $p_n[\exp]$ has a zero that lies near the parabola $x = (y/t_0)^2/2$.

\begin{figure}[!htb]
	\centering
	\includegraphics[width=0.9\textwidth]{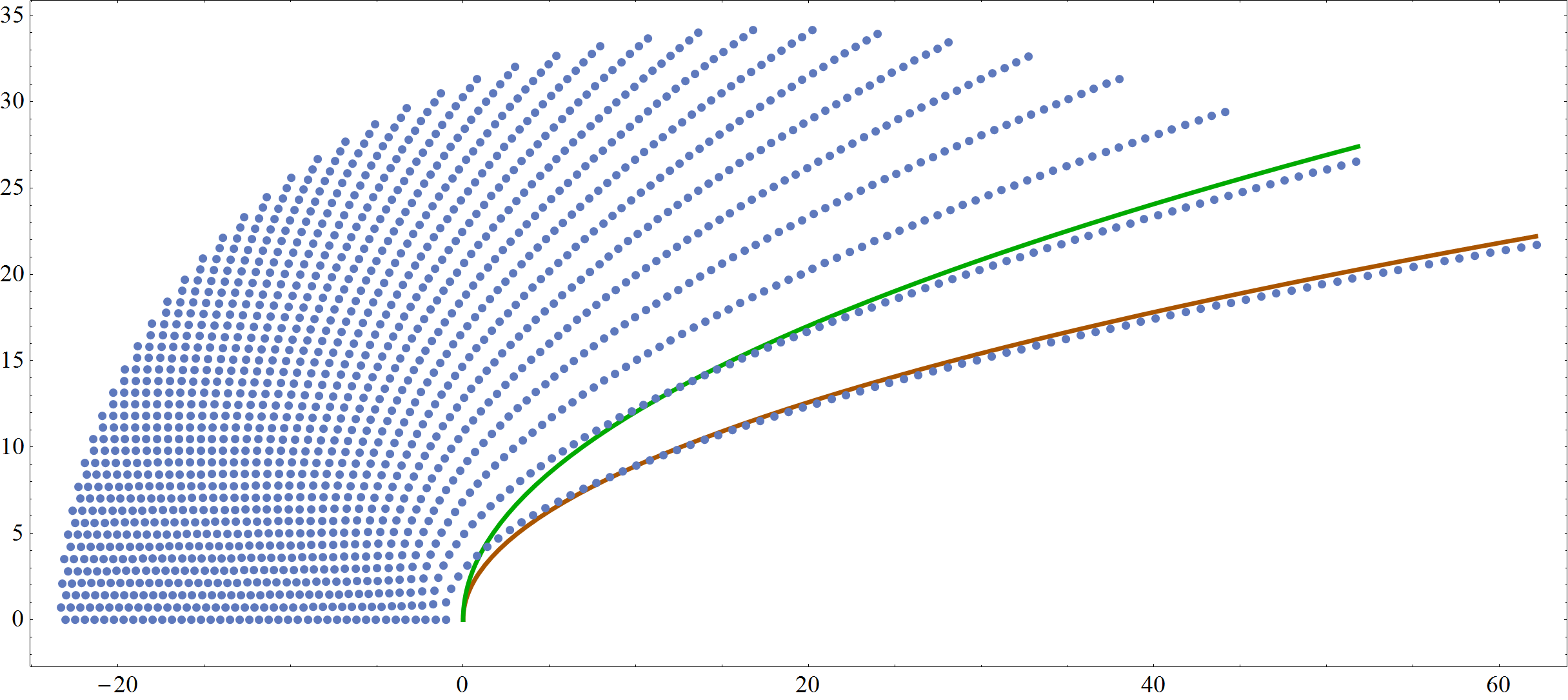}
	\caption[The zeros of the partial sums {$p_n[\exp](z)$} for {$n=1,\ldots,80$} in the upper half plane, shown with the parabolas {$x = (y/t_1)^2$} and {$x = (y/t_2)^2$}, where {$w_1 = s_1 + it_1$} is the smallest zero of {$\erfc$} in the upper half-plane and {$w_2 = s_2 + it_2$} is the next smallest.]{The zeros of the partial sums {$p_n[\exp](z)$} for {$n=1,\ldots,80$} in the upper half plane, shown with the parabolas {$x = (y/t_1)^2$} in orange and {$x = (y/t_2)^2$} in green, where {$w_1 = s_1 + it_1 \approx -1.35481 + 1.99147i$} is the smallest zero of {$\erfc$} in the upper half-plane and {$w_2 = s_2 + it_2 \approx -2.17704 + 2.69115i$} is the next smallest.}
	\label{introfig_newrivparab}
\end{figure}

\noindent The function $\erfc$ has infinitely many zeros, so this analysis allows us to conclude that the number of zeros of $p_n[\exp]$ in any region which is larger than a parabola, say of the form
\[
	|y| \leq Ax^{1/2+\epsilon}, \quad x \geq B,
\]
is unbounded as $n \to \infty$.

Theorem \ref{introthm_newriv} was later studied in considerably more detail in \cite{mallison:expsums}.

\section{The Saff-Varga Width Conjecture and Scaling Limits}
\label{introsec_width}

The exponential function is not the only entire function to have a parabolic zero-free region as in \eqref{introeq_expsvparab}. For a given entire function
\[
	f(z) := \sum_{k=0}^{\infty} a_kz^k,
\]
denote its $n^\th$ partial sum by $p_n(z)$. Let $W$ denote the set of all nondecreasing functions $h\colon [0,\infty) \to (0,\infty)$ such that neither $f$ nor any $p_n$ have zeros in
\[
	\{z = x+iy : x \geq 0 \text{ and } |y| \leq h(x) \}
\]
and define
\[
	H(x) = \sup \{ h(x) : h \in W \}.
\]

\begin{theorem}[Saff, Varga \cite{sv:overconvergence}]
If $f$ is entire of positive order and its power series coefficients $(a_k)$ are positive and satisfy
\[
	\inf_{k > 0} \frac{a_k}{k^2 a_{k+1}} > 0
\]
then
\[
	c\sqrt{x} \leq H(x) < \infty, \qquad x \geq 0
\]
for some positive constant $c$.
\end{theorem}

Prompted by the this result and by additional computations, Saff and Varga made the following conjecture \cite{sv:overconvergence,sv:openprobs,esv:sections}.

\begin{conjecturea}
\label{introconj_widthconjecture1}
Consider the ``parabolic region''
\[
	S_0(\tau) = \left\{z = x+iy : |y| \leq K x^{1-\tau/2}, \, x \geq x_0\right\},
\]
where $K$ and $x_0$ are fixed positive constants, and consider also the regions $S_\theta(\tau)$ obtained by rotations of $S_0(\tau)$:
\[
	S_\theta(\tau) = e^{i\theta} S_0(\tau).
\]
Given any entire function $f$ of positive finite order $\lambda > \tau$, denote its $n^\text{th}$ partial sum by $p_n(z)$.  There exists an infinite sequence $M$ of positive integers such that there is no $S_\theta(\tau)$ which is devoid of all zeros of all partial sums $p_m(z)$, $m \in M$.
\end{conjecturea}

Essentially what the conjecture is saying is that \textit{if} the zeros of the partial sums avoid some region, then that region may not be too wide.  In the particular case of an entire function of order $\lambda$ with a zero-free region symmetric about the positive real axis, the conjecture posits that the width of the region (i.e. the range of $y$ values) must be $O(x^{1-\lambda/2+\epsilon})$ as $x \to \infty$ for all $\epsilon > 0$. So for the exponential function, for example, an entire function of order $1$, the zero-free region would not be wider than $O(x^{1/2+\epsilon})$. Compare this with the discussion following Theorem \ref{introthm_newriv} in Section \ref{introsec_zerofree}.

\begin{figure}[!htb]
	\centering
	\includegraphics[width=0.9\textwidth]{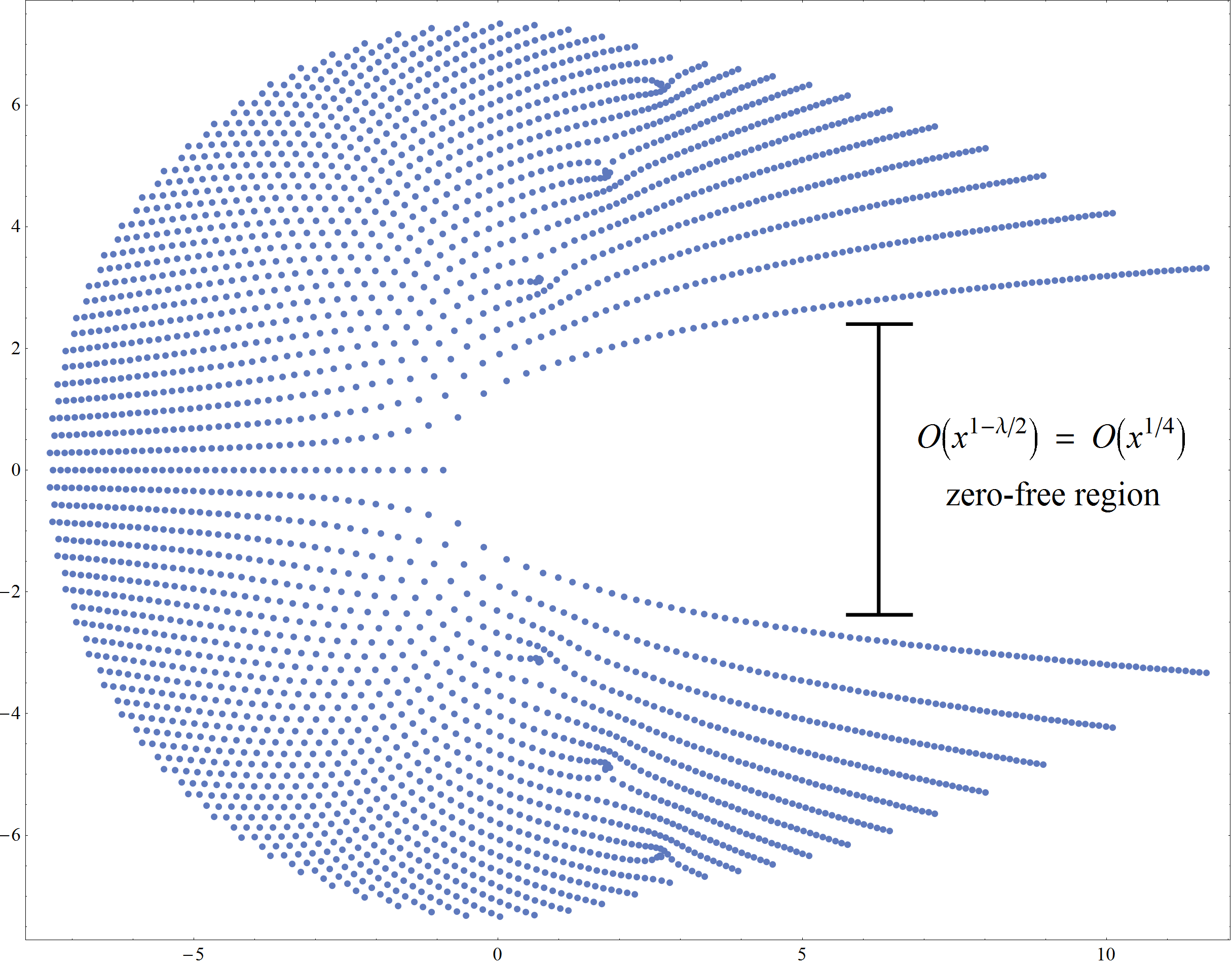}
	\caption[Zeros of the partial sums {$p_n[E](z)$} for {$n=1,\ldots,80$} with \newline {$\lambda = 3/2$}.]{Zeros of the partial sums $p_n[E](z)$ for $n=1,\ldots,80$ with $\lambda = 3/2$. A zero-free region can be seen opening to the right whose width is in agreement with the Saff-Varga Width Conjecture.}
	\label{introfig_mittagzerofree}
\end{figure}

According to the authors, the monograph \cite{esv:sections} arose from an attempt to settle this conjecture. In it they studied the Mittag-Leffler function $E_{1/\lambda}$, defined for $z \in \C$ and $\lambda > 1$ by
\[
	E_{1/\lambda}(z) = \sum_{k=0}^{\infty} \frac{z^k}{\Gamma(k/\lambda + 1)},
\]
and proved the following.

\begin{theorem}[Edrei, Saff, and Varga]
\label{introthm_esvcornerscaling}
Let
\[
	p_n[E](z) = \sum_{k=0}^n \frac{z^k}{\Gamma(k/\lambda + 1)}
\]
and let $r_n = (n/\lambda)^{1/\lambda} e^{1/(2n)}$.  Then
\[
	\lim_{n \to \infty} \frac{p_n\!\left(r_n \left(1+w\sqrt{2/(\lambda n)} \right)\right)}{\left(1+w\sqrt{2/(\lambda n)}\right)^n E_{1/\lambda}(r_n)} = \frac{1}{2} \exp\!\left(w^2\right) \erfc(w)
\]
uniformly for $w$ restricted to any compact subset of $\C$.
\end{theorem}

Let
\begin{align*}
S_E &= \left\{ z \in \C : \left|z^\lambda \exp\!\left(1-z^\lambda\right)\right| = 1, \text{ } |z| \leq 1, \text{ and } |{\arg z}| \leq \frac{\pi}{2\lambda} \right\} \\
	&\qquad \cup \left\{ z \in \C : |z| = e^{-1/\lambda} \text{ and } |{\arg z}| \geq \frac{\pi}{2\lambda} \right\}.
\end{align*}

\begin{theorem}[Edrei, Saff, and Varga]
\label{introthm_esvcurvescaling}
Let $\xi \in S_E$ with $0 < |\arg \xi| < \pi/(2\lambda)$ be fixed and set
\[
	\tau = |\xi|^\lambda \sin(\lambda \arg \xi) - \lambda \arg \xi.
\]
Define the sequence $(\tau_n)$ by the condition
\[
	\frac{\tau}{\lambda} n \equiv \tau_n \pmod{2\pi}, \qquad -\pi < \tau_n \leq \pi.
\]
Let
\[
	\zeta_0 = -\frac{1}{2} \log(2\pi \lambda) + \frac{1}{2}\left(1-\xi^\lambda\right) + \log\!\left(\frac{\xi}{1-\xi}\right),
\]
where $\log(2\pi\lambda)$ is real and the determination of the last logarithm is such that
\[
	-\pi < \im \log\!\left(\frac{\xi}{1-\xi}\right) \leq \pi.
\]
Put
\[
	P_n = \left(1 + \frac{\log n + 2i \tau_n - 2\zeta_0}{2(1-\xi^\lambda)n}\right) r_n \xi
\]
and consider all zeros of the polynomial in $\zeta$ given by
\[
	\psi_n(\zeta) := p_n[E]\!\left(P_n - \frac{r_n \xi}{(1-\xi^\lambda)n} \zeta\right),
\]
where $r_n$ and $p_n[E]$ are as in Theorem \ref{introthm_esvcornerscaling}. Given $t > 0$ ($t$ not a multiple of $2\pi$), the polynomial $\psi_n$ has, in the disk $|\zeta| \leq t$, exactly
\[
	2 \left\lfloor \frac{t}{2\pi} \right\rfloor + 1 =: 2 \ell + 1
\]
zeros, all of them simple.  Denoting these zeros of $\psi_n$ by $\zeta_{n,j}$, $j = 0,\pm 1, \ldots, \pm \ell$, then
\[
	\zeta_{n,j} = 2j\pi i + \eta_{n,j}, \qquad j = 0,\pm 1, \ldots, \pm \ell,
\]
where for fixed $t$
\[
	\lim_{n \to \infty} \eta_{n,j} = 0.
\]

A similar result holds for $\xi \in S_E$ with $|{\arg \xi}| > \pi/(2\lambda)$.
\end{theorem}

Together these theorems say that the zeros of the scaled partial sums $p_n[E](r_n z)$ of the Mittag-Leffler function which approach points $\xi \in S_E$ with $\xi \neq 1$ and $|{\arg \xi}| \neq \pi/(2\lambda)$ do so at a rate of $\Theta(\log n/n)$ and are separated from each other by a distance of $\Theta(n^{-1})$, and those which approach the point $\xi = 1$ do so at a rate of $\Theta(n^{-1/2})$ and are separated by a distance of $\Theta(n^{-1/2})$. This is typical of the functions which have been studied to date: most the zeros of the partial sums cluster densely together and ``fill'' up most of the plane, and there are only a finite number of exceptional arguments where the zeros are widely spaced and zero-free regions like those discussed in the previous section exist. 

To capture these observations, Edrei, Saff, and Varga proposed a modified Width Conjecture \cite[p.\ 6]{esv:sections}.

\begin{conjectureb}
\label{introconj_widthconjecture2}
Let $f$ be an entire function of positive, finite order $\lambda$ and let $p_n(z)$ denote its $n^\th$ partial sum.  We can find an infinite sequence $M$ of positive integers and a finite number of exceptional arguments $\theta_1,\theta_2,\ldots,\theta_q$ such that
\begin{enumerate}[label=\textnormal{(\alph*)}]
\item For any argument $\theta \neq \theta_j$, $j = 1,2,\ldots,q$, it's possible to find a positive sequence $(\rho_m)_{m\in M}$ with $\rho_m \to \infty$ and $\rho_m = O(m^{2/\lambda})$ such that, for every fixed $\epsilon > 0$, the number of zeros of the partial sum $p_m(z)$ in the disk
\[
	\left|z - \rho_m e^{i\theta}\right| \leq \rho_m m^{-1+\epsilon}
\]
tends to infinity as $m \to \infty$, $m \in M$.

\item For any exceptional argument $\theta_j$ it's possible to find an integer $k \geq 2$ and a positive sequence $(\rho_m)_{m\in M}$ with $\rho_m \to \infty$ and $\rho_m = O(m^{2/(\lambda k)})$ such that, for every fixed $\epsilon > 0$, the number of zeros of the partial sum $p_m(z)$ in the disk
\[
	\left|z - \rho_m e^{i\theta_j}\right| \leq \rho_m m^{-1/k + \epsilon}
\]
tends to infinity as $m \to \infty$, $m \in M$.
\end{enumerate}
\end{conjectureb}

One can check that a verification of the Modified Width Conjecture would imply the truth of the standard Width Conjecture.  The benefit of this second conjecture is that it makes an attempt to distinguish between the two distinct behaviors observed of the scaled zeros of the partial sums: heavy clustering along smooth arcs of their limit curve (the non-exceptional arguments) and an aversion to approaching any of the curve's corners (the exceptional arguments).

For the particular case of the Mittag-Leffler function, Theorem \ref{introthm_esvcornerscaling} verifies part (b) of the Modified Width Conjecture at the exceptional argument $\theta = 0$ with $k=2$ and Theorem \ref{introthm_esvcurvescaling} verifies part (a) along any argument $\theta \neq 0,\pm \pi/(2\lambda)$. Together with the fact that the zeros of $E_{1/\lambda}$ lie asymptotically on the rays $\arg z = \pm \pi/(2\lambda)$ (and an application of Hurwitz's theorem), this verifies the original Saff-Varga Width Conjecture for this function. The authors also proved versions of Theorems \ref{introthm_esvcornerscaling} and \ref{introthm_esvcurvescaling} for $\mathcal{L}$-functions \cite[p.\ 21]{esv:sections}.

Similarly, Theorem \ref{introthm_newriv} verifies part (b) of the Modified Width Conjecture at the exceptional argument $\theta = 0$ with $\lambda = 1$, $\rho_n = n$, and $k=2$ in the particular case of the exponential function $f(z) = \exp(z)$. Norfolk obtained the following analogue for the case of the confluent hypergeometric function \cite{norfolk:1f1} which also verifies the Modified Width Conjecture at the exceptional argument $\theta = 0$ with $\lambda = 1$, $\rho_n = n$, and $k=2$.

\begin{theorem}[Norfolk]
\label{introthm_norfolkcorner}
Let
\[
	{}_1F_1(1;b;z) = \Gamma(b) \sum_{k=0}^{\infty} z^k/\Gamma(k+b)
\]
and let $p_n[{}_1F_1](z)$ denote its $n^\th$ partial sum. If $b \in \R$\newnot{symbol:RR} with $b \neq 1,0,-1,-2,\ldots$ then
\[
	\lim_{n\to\infty} \frac{p_n[{}_1F_1](n+w\sqrt{n})}{e^{w\sqrt{n}}\, {}_1F_1(1;b;n)} = \frac{1}{2} \erfc\!\left(\frac{w}{\sqrt{2}}\right)
\]
uniformly for $w$ restricted to any compact subset of $\C$.
\end{theorem}

\begin{figure}[!htb]
	\centering
	\includegraphics[width=0.9\textwidth]{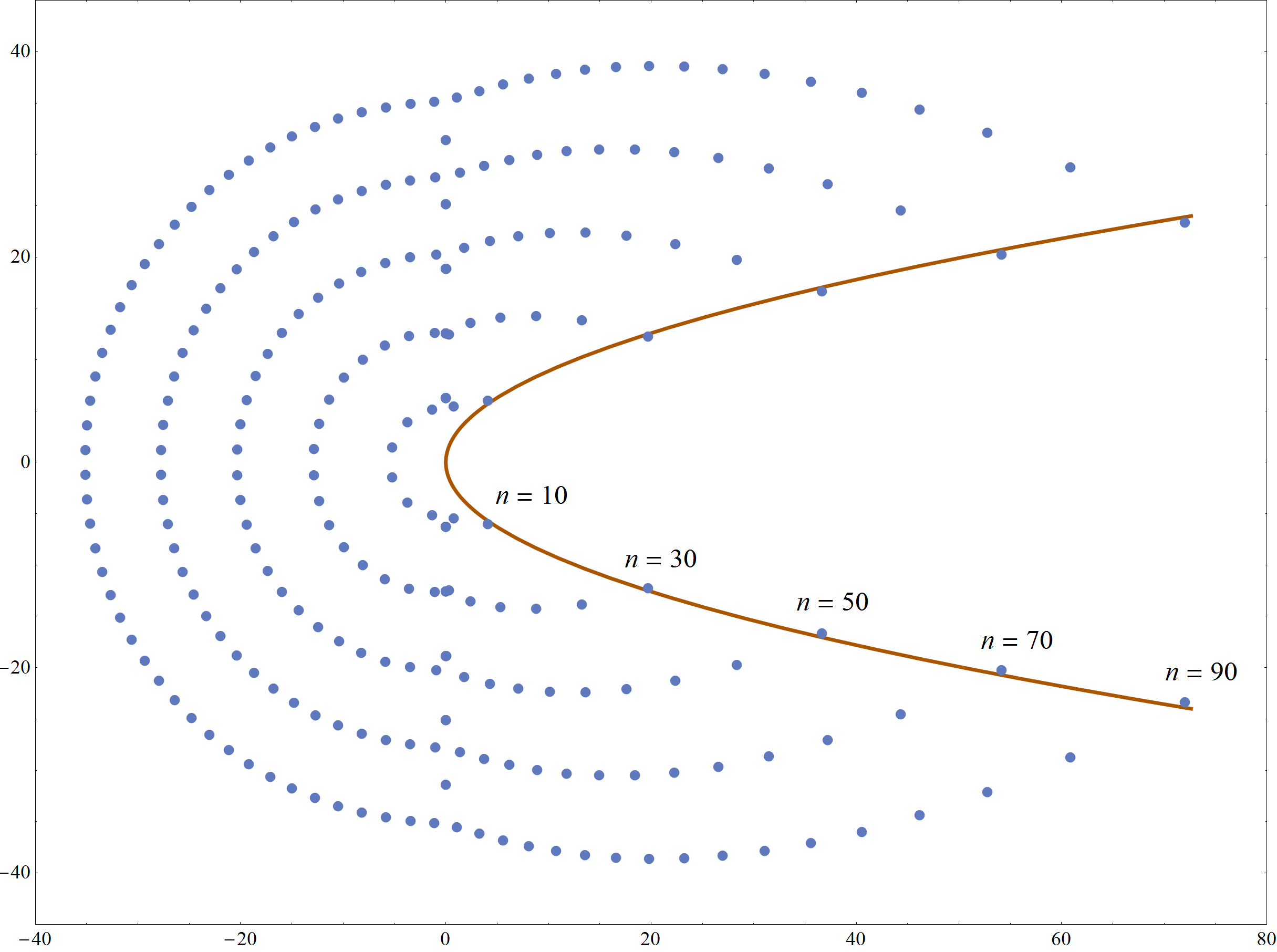}
	\caption[Zeros of the partial sums {$p_n[{}_1F_1](z)$} for {$n=10,30,50,70,90$}, shown with the parabola {$x = (y/t_1)^2$}, where {$w_1 = s_1 + it_1$} is the smallest zero of {$\erfc$} in the upper half-plane.]{Zeros of the partial sums $p_n[{}_1F_1](z)$ for $n=10,30,50,70,90$, shown with the parabola $x = (y/t_1)^2$ in orange, where $w_1 = s_1 + it_1 \approx -1.35481 + 1.99147i$ is the smallest zero of $\erfc$ in the upper half-plane.}
	\label{introfig_1f1parab}
\end{figure}

\noindent Other results in this same vein have been proved for binomial expansions \cite{norfolk:binom} and for linear combinations of sections and tails of Mittag-Leffler functions \cite{zhel:mlsectails} and of classical Lindel\"of functions \cite{zhel:lindelof}.

Though we are primarily concerned with the above statements of the standard and modified Width Conjectures we note that an alternate formulation of the standard conjecture was given by Norfolk in \cite{norfolk:widthconj}.

\section{A Riemann-Hilbert Approach}

In 2008 Kriecherbauer, Kuijlaars, McLaughlin, and Miller published a paper \cite{mclaughlin:exprh} in which they undertook an analysis of the partial sums of the exponential function using a version of the Riemann-Hilbert method of asymptotic analysis. Among other things the authors obtained a complete asymptotic expansion of $p_n[\exp](nz)$ in the regime $n \to \infty$, $z \to 1$.

They began by defining the function
\[
	F_n(z) = \frac{1}{2\pi i} \int_\gamma \frac{(se^{1-s})^{-n}}{s-z}\,ds,
\]
where $\gamma$ is a simple closed loop around the origin passing through the point $s=1$. This function is related to the partial sums of the exponential function by the formula
\[
	F_n(z) = \begin{cases}
		-(ez)^{-n} p_{n-1}[\exp](nz) & \text{for } z \text{ outside } \gamma, \\
		(ze^{1-z})^{-n} - (ez)^{-n} p_{n-1}[\exp](nz) & \text{for } z \neq 0 \text{ inside } \gamma.
		\end{cases}
\]
Taking into account its decay as $|z| \to \infty$ and the above jump as $z$ crosses the curve $\gamma$ the authors formulated a Riemann-Hilbert problem (see Section \ref{prelimssec_rhps}) solved by $F_n(z)$. By applying the Riemann-Hilbert method to this problem they obtained the following asymptotic expansion for $p_{n-1}[\exp]$.

\begin{theorem}[Kriecherbauer, Kuijlaars, McLaughlin, and Miller]
\label{introthm_kkmm}
There exists an $\epsilon > 0$ such that, for any $J \in \N$,
\[
	(ez)^{-n} p_{n-1}[\exp](nz) = \frac{1}{2} e^{n\varphi(z)} \erfc\!\left(\sqrt{n\varphi(z)}\right) - \frac{1}{\sqrt{2\pi n}} \left[\sum_{j=0}^{J-1} \frac{g_j(z)}{n^j} + O\!\left(n^{-J}\right)\right],
\]
where
\[
	\varphi(z) = z - 1 - \log z,
\]
\[
	g_j(z) = \frac{(-1)^j \Gamma\!\left(j+\frac{1}{2}\right)}{(2\pi)^{3/2} i} \int_{|s-1| = 2\epsilon} \frac{\varphi(s)^{-j-1/2}}{s-z}\,ds,
\]
and the error term is uniform for $|z-1| < \epsilon$ with $z$ inside $\gamma$.
\end{theorem}

This is a considerable improvement over Newman and Rivlin's scaling limit in Theorem \ref{introthm_newriv}. From this result the authors were able to obtain complete asymptotic expansions of the zeros of the scaled partial sums $p_{n-1}[\exp](nz)$ which approach the point $z=1$.

\begin{theorem}[Kriecherbauer, Kuijlaars, Mclaughlin, and Miller]
Let $(w_k)_{k\in\N}$ be the sequence of zeros of $\erfc$ in the upper half-plane ordered so that $|w_k| < |w_{k+1}|$. There exist polynomials $q_j$ of degree $j$ such that, for $0 < \beta < 1$, $1 < k < n^\beta$, and $r \in \N$, the scaled partial sum $p_{n-1}[\exp](nz)$ has a zero $z_{k,n}$ satisfying
\[
	z_{k,n} = 1 + \sum_{j=1}^{r-1} \frac{q_j(w_k)}{n^{j/2}} + O\!\left(\left(\frac{k}{n}\right)^{r/2}\right),
\]
where the constant in the error term depends only on the choice of $r$ and $\beta$. The polynomials $q_j$ can be computed explicitly.
\end{theorem}

\noindent Writing $z = z_{k,n}$ and $w = w_k$ for simplicity, the first few terms of this series are
\begin{equation}
\label{introeq_rhzapprox}
	z = 1 + \sqrt{2} w n^{-1/2} + \frac{2w^2-1}{3} n^{-1} + \frac{2w^3-7w}{18\sqrt{2}} n^{-3/2} - \frac{6w^4 + 7w^2 - 8}{405} n^{-2} + \cdots.
\end{equation}

\begin{figure}[!htb]
	\centering
	\includegraphics[width=0.95\textwidth]{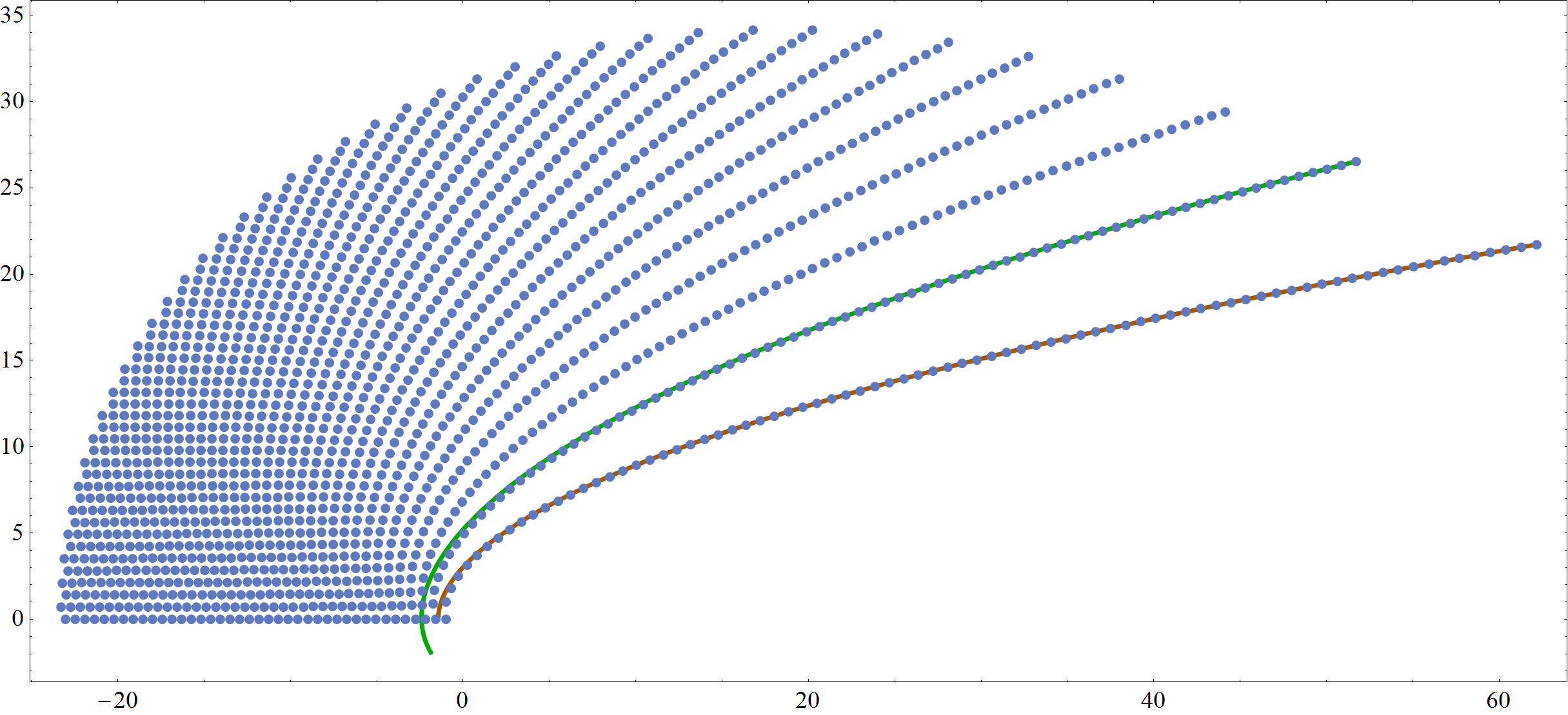}
	\caption[The zeros of the partial sums {$p_n[\exp](z)$} for {$n=1,\ldots,80$} in the upper half plane, shown with the approximations given by the first four terms of equation \eqref{introeq_rhzapprox} using {$w=w_1$} and {$w=w_2$}, where {$w_1$} is the smallest zero of {$\erfc$} in the upper half-plane and {$w_2$} is the next smallest.]{The zeros of the partial sums $p_n[\exp](z)$ for $n=1,\ldots,80$ in the upper half plane, shown with the approximations given by the first four terms of equation \eqref{introeq_rhzapprox} using $w=w_1$ in orange and $w=w_2$ in green, where $w_1 \approx -1.35481 + 1.99147i$ is the smallest zero of $\erfc$ in the upper half-plane and $w_2 \approx -2.17704 + 2.69115i$ is the next smallest. Compare with Figure \ref{introfig_newrivparab}.}
	\label{introfig_rhexpparab}
\end{figure}

\begin{remark}
In \cite{newriv:expzeros} the authors claim that if $w = u+iv$ is any zero of $\erfc(w/\sqrt{2})$ then $p_n[\exp]$ has a zero arbitrarily close to the parabola
\[
	x = \frac{1}{v^2}y^2 + \frac{u}{v}y.
\]
The expansion in \eqref{introeq_rhzapprox} indicates that the constant on the $y$ term is not correct. In fact the correct parabola is evidently
\[
	x = \frac{1}{v^2} y^2 - \frac{u}{3v} y + \frac{1-u^2-5v^2}{18}.
\]
\end{remark}

The methods in this thesis, and in particular in Chapter \ref{chap_posfinite}, are based on the ones Kriecherbauer, Kuijlaars, Mclaughlin, and Miller used to obtain the above results. By generalizing their function $F_n(z)$ we are able to apply their methods to the study of partial sums of power series of a wide class of entire functions.


\section{This Thesis}

In this thesis we will verify the Modified Saff-Varga Width Conjecture for a class of entire functions of positive, finite order with a certain asymptotic character in their sectors of maximal growth.

We will focus on functions $f$ where $f(z) \approx \exp(z^\lambda)$ in one or two sectors in the plane, and which are bounded by a smaller exponential $\exp(\mu |z|^\lambda)$, $\mu < 1$ otherwise. For the sake of generality we will allow $f$ to have some subexponential factors in its asymptotic---in the case of a single sector of maximal growth we will assume that $f(z) \sim z^a (\log z)^b \exp(z^\lambda)$ for some $a,b \in C$, and in the case of two sectors of maximal growth we will restrict this a little and only assume that $f(z) \sim z^a \exp(z^\lambda)$. Specifically we assume the following.

Let $a,b,A \in \C$, $0 < \lambda < \infty$, $\mu < 1$, and $\zeta \in \C$ with $|\zeta| = 1$, $\zeta \neq 1$. Let $\theta \in (0,\pi)$ be small enough so that the sectors $|{\arg z}| \leq \theta$ and $|{\arg(z/\zeta)}| \leq \theta$ are disjoint. We say that an entire function $f$ has a single direction of maximal exponential growth if
\[
f(z) = \begin{cases}
	z^a (\log z)^b \exp(z^\lambda)\,\bigl[1+o(1)\bigr] & \text{for } |{\arg z}| \leq \theta, \\
	O\!\left(\exp(\mu |z|^\lambda)\right) & \text{for } |{\arg z}| > \theta
	\end{cases}			
\]
as $|z| \to \infty$, and that the function $f$ has two directions of maximal exponential growth if
\[
f(z) = \begin{cases}
	z^a \exp\!\left(z^\lambda\right) \bigl[1+o(1)\bigr] & \text{for } |{\arg z}| \leq \theta, \\
	A(z/\zeta)^b \exp\!\left((z/\zeta)^\lambda\right) \bigl[1+o(1)\bigr] & \text{for } |{\arg(z/\zeta)}| \leq \theta, \\
	O\!\left(\exp\!\left(\mu |z|^\lambda\right)\right) & \text{otherwise}
	\end{cases}			
\]
as $|z| \to \infty$, with all estimates holding uniformly in their sector. Under either of these assumptions $f$ is of order $\lambda$.

\begin{remark}
Both the assumption that $f(z) \sim z^a (\log z)^b \exp(z^\lambda)$ in the case of one direction of maximal growth and that $f(z) \sim z^a \exp(z^\lambda)$ in the case of two directions of maximal growth can be generalized to include other reasonably simple subexponential factors such as $(\log\log z)^c$. The conditions as they are were chosen to try to strike a balance between generality and simplicity.
\end{remark}

We will begin in Chapter \ref{chap_limitcurves} by deriving the limit curves for the appropriately-scaled zeros of the partial sums of these functions $f$ in the sector $|{\arg z}| < \theta$. Then in Chapter \ref{chap_curvescaling} we will prove analogues of Theorem \ref{introthm_esvcurvescaling}, which we refer to as ``scaling limits at the arcs of the limit curve''. We will then show that these scaling limits verify part (a) of the Modified Saff-Varga Width Conjecture for these functions $f$ in the sector $0 < |{\arg z}| < \theta$. Finally in Chapter \ref{chap_posfinite} we will prove analogues of Theorems \ref{introthm_newriv}, \ref{introthm_esvcornerscaling}, and \ref{introthm_norfolkcorner}, which we refer to as ``scaling limits at the corner of the limit curve''. We will show that these scaling limits verify part (b) of the Modified Saff-Varga Width Conjecture at the exceptional argument $\arg z = 0$ in the case of a single direction of maximal exponential growth, and also when $f$ has two directions of maximal exponential growth as long as $\re b - \re a < \lambda/2$.

In total we verify the Modified Saff-Varga Width Conjecture for these functions $f$ in the full sector $|{\arg z}| < \theta$.

In Chapter \ref{chap_applications} we will apply the results in the preceding chapters to several common special functions. Among these will be the sine and cosine functions \cite{szego:exp,kappert:sincos,vc:sincosasympi,vc:sincosasympii,vc:sincosasympiii}, the confluent hypergeometric functions \cite{norfolk:1f1}, the Bessel functions of the first kind \cite{vargas:limitcurves}, and certain exponential integrals \cite{norfolk:transforms,vargas:limitcurves}, all of which have been studied before in some way in the listed citations (though some in less generality, e.g.\ with tighter restrictions on the ranges of their parameters). Among these functions, scaling limits of the form in this thesis have only been obtained for the confluent hypergeometric functions in \cite{norfolk:1f1}. We will also study the Airy functions and the parabolic cylinder functions, neither of which have, to my knowledge, been examined in this way.

We note that \cite{vc:sincosasympii} refers to the behavior of the zeros of the partial sums of sine and cosine near the convex corners of their limit curve as an open problem. This is addressed in the relevant section of Chapter \ref{chap_applications} by applying Theorem \ref{secunbalcauchy_maintheo}.

Chapter \ref{chap_prelims} contains several preliminary facts and definitions which we will need in our analysis, including a discussion of Riemann-Hilbert problems which appear in Chapter \ref{chap_posfinite}. Finally, Appendix \ref{appendix_laplace} contains an overview of the Laplace method, a strategy for estimating integrals which we use throughout the thesis.

%% file: chap_prelims.tex
\chapter{Technical Preliminaries}
\label{chap_prelims}

\section{Entire Functions}
\label{prelimssec_entirefunctions}

A function $f \colon \C \to \C$ is said to be entire if it is analytic on all of $\C$ or, equivalently, if its Taylor series converges on all of $\C$.

\begin{definition}
\label{orderdef}
The \textit{order} (sometimes \textit{exponential order}) of an entire function $f$ is defined as the infimum over all positive numbers $\lambda$ such that
\[
	|f(z)| < \exp\left(|z|^\lambda\right)
\]
for all $z$ large enough.  If this isn't satisfied for any positive number $\lambda$ then we say $f$ has order $\infty$.  If $f$ is defined by the power series
\[
	f(z) = \sum_{k=0}^{\infty} a_k z^k
\]
then its order can be calculated by
\[
	\lambda = \limsup_{k \to \infty} \frac{\log k}{\log\left(|a_k|^{-1/k}\right)}.
\]
(See, for example, \cite{levin:entirelec}.)
\end{definition}

\section{The Notation of Asymptotic Analysis}

Asymptotic analysis uses a variety of notation whose purpose is to hide certain information about the quantities we handle. It allows us to ignore details we're not interested in and to greatly simplify otherwise complicated formulas.

It's important to note that the various notations we will introduce below sometimes have different meanings depending on the conventions of the field of mathematics in which they appear. One symbol may mean different things in different situations, and two different authors may use two different symbols for the same purpose. Suffice it to say that the notation is only somewhat standardized. Our usage of the notation follows the conventions of the subfield of classical analysis which is concerned with the study of special functions. We will try to give a relatively complete description of these conventions here.

\subsection{Definitions}

The principal player in the notation of asymptotic analysis is the ``Big O'' (pronounced ``big oh'').\newnot{symbol:bigo}

\begin{definition}[Big O Notation]
\label{bigodef}
Given a set $U$ and two functions $f$ and $g$ whose domains include $U$, the statement
\begin{equation}
\label{bigodefeq}
	f(x) = O(g(x)) \quad \text{for all } x \in U
\end{equation}
is defined to mean that there exists a constant $C > 0$ such that
\[
	|f(x)| \leq C|g(x)| \quad \text{for all } x \in U.
\]
\end{definition}

There is some flexibility in the quantifier in this definition, ``for all $x \in U$''. For example, if we instead write
\[
	f(x) = O(g(x)) \quad \text{as } x \to \infty,
\]
we mean that there exists an $M \in \R$ such that
\[
	f(x) = O(g(x)) \quad \text{for all } x > M.
\]
Here we have taken $U = \{x \in \R : x > M\}$. Similarly we may write something like
\[
	f(x) = O(g(x)) \quad \text{as } x \to a,
\]
where $a$ is some limit point of the domains of $f$ and $g$, to mean that there exists some open neighborhood $U$ of $a$ in the domains of $f$ and $g$ such that $f(x) = O(g(x))$ for all $x \in U$.

Notice that by writing ``$f(x) = O(g(x))$ as $x \to \infty$'' we've suddenly hidden away the values of \textit{two} constants, namely the $C$ in the definition of ``$=O(\cdots)$'' and the $M$ in the discussion above. This is a strength of the notation---for our purposes we won't need to know what their values are, only that they exist.

\begin{definition}[Big Omega Notation]
Given a set $U$ and two functions $f$ and $g$ whose domains include $U$, the statement
\[
	f(x) = \Omega(g(x)) \quad \text{for all } x \in U
\]
is defined to mean that there exists a constant $C > 0$ such that
\[
	|f(x)| \geq C|g(x)| \quad \text{for all } x \in U.
\]
\end{definition}

\begin{definition}[Big Theta Notation]
\label{prelimsdef_bigtheta}
Given a set $U$ and two functions $f$ and $g$ whose domains include $U$, the statement
\[
	f(x) = \Theta(g(x)) \quad \text{for all } x \in U
\]
is defined to mean that there exists a constant $C > 0$ such that
\[
	f(x) = O(g(x)) \quad \text{and} \quad f(x) = \Omega(g(x)) \quad \text{for all } x \in U.
\]
\end{definition}

The quantifiers in Big Omega and Big Theta notation have the same flexibility we described above for Big O notation.

\begin{definition}[Little O Notation]
\label{prelimsdef_littleo}
Given a point $a$ (with the possibility that $a=\infty$) and two functions $f$ and $g$ defined in a neighborhood of $a$, the statement
\[
	f(x) = o(g(x)) \quad \text{as } x \to a
\]
is defined to mean that for any $\epsilon > 0$ we can find a neighborhood $U$ of $a$ such that
\[
	|f(x)| \leq \epsilon |g(x)| \quad \text{for all } x \in U.
\]
\end{definition}

If $g(x) \neq 0$ for all $x$ near $a$, $x\neq a$, then this definition is equivalent to the statement that
\[
	\lim_{x \to a} \frac{f(x)}{g(x)} = 0.
\]

\begin{definition}[Little Omega Notation]
Given a point $a$ (with the possibility that $a=\infty$) and two functions $f$ and $g$ defined in a neighborhood of $a$, the statement
\[
	f(x) = \omega(g(x)) \quad \text{as } x \to a
\]
is defined to mean that for any $M > 0$ we can find a neighborhood $U$ of $a$ such that
\[
	|f(x)| \geq M |g(x)| \quad \text{for all } x \in U.
\]
\end{definition}

If $g(x) \neq 0$ for all $x$ near $a$, then this definition is equivalent to the statement that
\[
	\lim_{x \to a} \left|\frac{f(x)}{g(x)}\right| = \infty.
\]
It should also be noted that this Little Omega notation is not in very common use and should be explained whenever it appears.

So, we have collected all of our basic notations: $O$, $\Omega$, $\Theta$, $o$, and $\omega$. Now we will expand the scope of their definitions to allow them to appear in formulas. This is the most useful aspect of the notations though it does require an understanding of the underlying conventions. The following discussion applies to all of these notations, so for simplicity we'll just use Big O notation as our example.

There is an inherent asymmetry in the meaning of the equals sign $=$ in formulas involving $O$. The primary convention we use is that we are always simplifying left-to-right. A complicated expression may appear on the left of $=$, and its simplified  or refined form should appear on the right. Because of this the equals sign should be viewed as part of the new notation rather than as usual equality.

Specifically, when we write a statement involving $O$ there are hidden quantifiers. On the left side of the equals sign $=$ there is a ``for all'', and on the right side there is a ``there exists''. For example, in the statement
\begin{equation}
\label{bigoexample}
	\log(1+x) = x - \frac{x^2}{2} + O(x^3) \quad \text{as } x \to 0
\end{equation}
the $O$ notation appears on the right, and is therefore associated with a ``there exists'': there exists a function $g$ defined in a neighborhood of $x=0$ such that
\[
	\log(1+x) = x - \frac{x^2}{2} + g(x)
\]
and which satisfies
\[
	g(x) = O(x^3) \quad \text{as } x \to 0,
\]
where this last line is to be interpreted precisely as in Definition \ref{bigodef}.

Similarly, the expression
\[
	\log(1+x) + O(x^4) = x - \frac{x^2}{2} + O(x^3) \quad \text{as } x \to 0
\]
should be interpreted as follows: there is an $O(\cdots)$ on the left of the $=$ and another on the right, so we understand the expression to mean that \textbf{for all} functions $f(x) = O(x^4)$ as $x \to 0$ \textbf{there exists} a function $g(x)$ which satisfies the equation
\[
	\log(1+x) + f(x) = x - \frac{x^2}{2} + g(x)
\]
and the asymptotic $g(x) = O(x^3)$ as $x \to 0$. Indeed, taking such a $g$ defined by this equation we can deduce from \eqref{bigoexample} and the assumption $f(x) = O(x^4)$ as $x \to 0$ that
\begin{align*}
	|g(x)| &= \left|\log(1+x) - x + \frac{x^2}{2} + f(x)\right| \\
	&\leq \left|\log(1+x) - x + \frac{x^2}{2}\right| + |f(x)| \\
	&\leq C_1 x^3 + C_2 x^4 \\
	&\leq C_3 x^3
\end{align*}
for all $x$ small enough. This proves that $g(x) = O(x^3)$ as $x \to 0$, as desired.

\begin{remark}
If it's not clear, it would be a valuable exercise to check that by reversing the quantifiers the above statement is not true. It is not true that for any function $g(x) = O(x^3)$ as $x \to 0$ there exists a function $f(x) = O(x^4)$ as $x \to 0$ such that
\[
	\log(1+x) + f(x) = x - \frac{x^2}{2} + g(x).
\]
Consider, for example, $g(x) = x^3/6$.
\end{remark}

\begin{definition}[Asymptotic Equivalence]
\label{defasympequiv}
Given a point $a$ (with the possibility that $a = \infty$) and two functions $f$ and $g$, we say that $f$ is asymptotic to $g$, and write
\[
	f(x) \sim g(x) \quad \text{as } x \to a,
\]
if
\[
	\lim_{x \to a} \frac{f(x)}{g(x)} = 1.
\]
Alternatively we may say that $f$ is asymptotically equivalent to $g$.
\end{definition}

\begin{definition}[Asymptotic Scale]
Given a point $a$ (with the possibility that $a = \infty$) and a sequence of functions $\{\varphi_n(x)\}_{n \in \N}$, we say that this sequence is an asymptotic scale as $x \to a$ if
\[
	\varphi_{n+1}(x) = o(\varphi_n(x)) \quad \text{as } x \to a
\]
for all $n \in \N$.
\end{definition}

\begin{definition}[Asymptotic Expansion]
Given a function $f$, a point $a$, and an asymptotic scale $\{\varphi_n(x)\}_{n \in \N}$ as $x \to a$, we say that the formal series
\[
	\sum_{n=1}^{\infty} \varphi_n(x)
\]
is an asymptotic expansion for $f(x)$ as $x \to a$, and write
\[
	f(x) \dsim \sum_{n=1}^{\infty} \varphi_n(x) \quad \text{as } x \to a,
\]
if
\[
	f(x) = \sum_{n=1}^{N} \varphi_n(x) + o(\varphi_N(x)) \quad \text{as } x \to a
\]
for all $N \in \N$.
\end{definition}

\begin{remark}
The dotted tilde notation $\dsim$ above is non-standard. In some texts the plain tilde $\sim$ is instead used in the notation for an asymptotic expansion, and in others a wavy equals $\approx$ is used. In my experience using the plain $\sim$ can cause confusion because it already has another meaning in asymptotic analysis; see Definition \ref{defasympequiv}. The reason I don't use $\approx$\newnot{symbol:approx} is slightly different: I prefer that it be left undefined so that it may be used informally. It can be a very helpful tool for communicating exploratory, non-rigorous calculations. These are useful for building intuition. The precise meaning of $\approx$ can then be determined afterward with a more detailed, rigorous analysis. (We use $\approx$ in this way in Section \ref{curscasec_exploratory}.)
\end{remark}

\begin{definition}[Uniformity]
Given sets $U$ and $V$ and two functions $f(x,y)$ and $g(x)$ whose $x$-domains include $U$, the statement
\[
	f(x,y) = O(g(x)) \quad \text{uniformly for all } x \in U \text{ and } y \in V
\]
is defined to mean that there exists a constant $C > 0$ depending only on $U$ and $V$ such that
\[
	|f(x,y)| \leq C|g(x)| \quad \text{for all } x \in U \text{ and } y \in V.
\]
\end{definition}

Of course similar definitions of uniformity can be made for the other notations $\Omega$, $\Theta$, $o$ $\omega$, $\sim$, and $\dsim$. Essentially a claim of uniformity in an asymptotic expression is a statement that none of the hidden constants depend on the specified additional parameters, though it is possible that they depend on the \textit{ranges} of the parameters.

The main way that the concept of uniform estimates will be used in this thesis is in the description of the behavior of functions defined on the complex plane when $|z| \to \infty$. For example, if $\theta \in (0,\pi/2)$ then
\[
	\left| \frac{-2i \sin(iz)}{e^z} - 1 \right| \leq e^{-2|z|\cos\theta}, \qquad |{\arg z}| \leq \theta.
\]
There is a constant $C(\theta) > 0$ which depends on $\theta$ such that
\[
	e^{-2x\cos\theta} \leq \frac{1}{x}
\]
for all $x > C(\theta)$, so the above formula implies that if $\theta \in (0,\pi/2)$ then
\[
	\left| \frac{-2i \sin(iz)}{e^z} - 1 \right| \leq \frac{D}{|z|}, \qquad |{\arg z}| \leq \theta \text{ and } |z| > C(\theta)
\]
with $D=1$. We can restate this in asymptotic terms:
\[
	-2i \sin(iz) = e^z \left[1 + O\!\left(z^{-1}\right)\right]
\]
as $|z| \to \infty$ uniformly for $|{\arg z}| \leq \theta$, and we interpret this to mean that the hidden constants associated with the $O(\cdots)$ term (in this case $C(\theta)$ and $D$) do not depend on $\arg z$ as long as $|{\arg z}| \leq \theta$.

\section{Cauchy Integrals}

In this section we will state several important facts about Cauchy integrals. Proofs of these as well as further details can be found in \cite{miller:rhnotes,musk:singularintegrals,gakhov:bvp}.

\begin{definition}[Cauchy Integrals]
Let $L$ be a finite, smooth, oriented curve (which may be a closed contour) and suppose that $g\colon \C \to \C$ is integrable with respect to arc length on $L$, i.e.
\[
	\int_L |g(t)|\,|dt| < \infty.
\]
For $z \in C\setminus L$ the Cauchy integral of $g$ is defined as
\[
	\mathcal C_L[g](z) = \frac{1}{2\pi i} \int_L \frac{g(t)}{t-z}\,dt.
\]
\end{definition}

\begin{proposition}
\label{prelimsprop_cauchyproperties}
The function $\mathcal C_L[g]$ is analytic on $\C \setminus L$ and $\mathcal C_L[g](z) = O(z^{-1})$ as $|z| \to \infty$.
\end{proposition}

\begin{definition}[H\"older Continuity]
Let $0 < \alpha \leq 1$. A function $g$ defined on some connected set $U \subset \C$ is said to be H\"older continuous on $U$ if there exists a constant $K$ such that
\[
	|g(z_1) - g(z_2)| \leq K|z_1-z_2|^\alpha \quad \text{for all } z_1,z_2 \in U.
\]
\end{definition}

\begin{proposition}[Plemelj Formulas]
\label{prelimsprop_plemelj}
If $g$ is H\"older continuous on $L$ then $\mathcal C_L[g]$ has extensions from the left of $L$ onto $L$ and from the right of $L$ onto $L$ denoted by $\mathcal C_L^+[g]$ and $\mathcal C_L^-[g]$\newnot{symbol:extensions}, respectively, which are continuous except possibly in arbitrarily small neighborhoods of the endpoints of $L$. If $\zeta \in L$ is not an endpoint of $L$, or if it is an endpoint and $g(\zeta) = 0$, then
\[
	\mathcal C_L^\pm[g](\zeta) = \frac{1}{2\pi i} \operatorname{P.V.} \int_L \frac{g(t)}{t-\zeta}\,dt \pm \frac{g(\zeta)}{2},
\]
where $\operatorname{P.V.} \int$ is a principal value integral, and hence
\[
	\mathcal C_L^+[g](\zeta) - \mathcal C_L^-[g](\zeta) = g(\zeta).
\]
If $L$ is an arc which connects $a$ to $b$, $a,b\in \C$, then there exists a function $H_a$ defined in a neighborhood of $a$ which is analytic on $L^c$ (the complement of $L$), continuous at $a$, and has continuous extensions onto $L\setminus\{a\}$ from the left and the right such that
\[
	\mathcal C_L[g](z) = \frac{g(a)}{2\pi i} \log \frac{1}{z-a} + H_a(z)
\]
for $z$ near $a$, where the branch cut of the logarithm coincides with $L$. For $z$ near $b$ there is an analogous function $H_b$ such that
\[
	\mathcal C_L[g](z) = -\frac{g(b)}{2\pi i} \log \frac{1}{z-b} + H_b(z),
\]
where again the branch cut of the logarithm coincides with $L$.
\end{proposition}

\section{Scalar Riemann-Hilbert Problems}
\label{prelimssec_rhps}

Let $L$ be an oriented curve in the plane and let $g$ be a function defined on $L$. If $L$ is a closed contour, set $L^* = L$ and if $L$ is an arc which begins at $z=a$ and ends at $z=b$, set $L^* = L \setminus \{a,b\}$. A scalar Riemann-Hilbert problem is a problem of the following form.

\begin{rhp}
\label{prelimsrhp_example}
Seek an analytic function $\Phi\colon \C \setminus L \to \C$ such that
\begin{enumerate}
	\item $\Phi$ has continuous extensions $\Phi^+$ and $\Phi^-$ from the left and right of $L^*$, respectively, onto $L^*$,
	\item $\Phi^+(\zeta) - \Phi^-(\zeta) = g(\zeta)$ for $\zeta \in L^*$,
	\item $\Phi(z) \to 0$ as $|z| \to \infty$,
	\item if $c$ is an endpoint of $L$ then $\Phi(z) = O(|z-c|^{-1+\epsilon})$ as $z \to c$ with $z \in \C \setminus L$ for some $\epsilon > 0$.
\end{enumerate}
\end{rhp}

\begin{proposition}
If $L$ has finite length and if $g$ is H\"older continuous on $L$ then $\Phi = \mathcal C_L[g]$ is the unique solution to Riemann Hilbert Problem \ref{prelimsrhp_example}.
\end{proposition}

\begin{proof}
It follows immediately from Propositions \ref{prelimsprop_cauchyproperties} and \ref{prelimsprop_plemelj} that $\mathcal C_L[g]$ solves the Riemann-Hilbert problem.

To show that $\mathcal C_L[g]$ it is the only solution, suppose that $\Psi$ also solves the Riemann-Hilbert problem and define
\[
	 F = C_L[g] - \Psi.
\]
Then for $\zeta \in L^*$ we have
\[
	F^+(\zeta) = F^-(\zeta),
\]
Consequently $F$ is continuous across $L^*$ and hence analytic on $\C$ except possibly at the endpoints of $L$. These endpoints are at most isolated singularities of $F$ of order strictly less than $1$, so in fact $F$ is bounded near these points and is thus analytic there as well. In total $F$ is entire, and since $F(z) \to 0$ as $|z| \to \infty$ it must be true that $F \equiv 0$ by Liouville's theorem.
\end{proof}

\noindent This is essentially the same proof given in \cite[sec. 78]{musk:singularintegrals}.

We will make use of this proposition at key points of Chapter \ref{chap_posfinite}.

%% file: chap_limitcurves.tex
\chapter{Limit Curves for the Zeros of the Partial Sums}
\label{chap_limitcurves}

In this chapter we will calculate the set of limit points in a certain sector of the (appropriately scaled) zeros of the partial sums of the Maclaurin series for the functions $f$ we are interested in. Just as in the examples discussed in Chapter \ref{chap_intro}, as the degree of the partial sum tends to infinity its scaled zeros converge to a piecewise-smooth curve in the plane.

\section{One Direction of Maximal Exponential Growth}
\label{seccurve1_onedirsec}

Let $a,b \in \C$, $0 < \lambda < \infty$, $0 < \theta < \pi$, and $\mu < 1$. We suppose that $f$ is an entire function with the asymptotic behavior
\begin{equation}
\label{seccurve1_fgrowth}
f(z) = \begin{cases}
	z^a (\log z)^b \exp\!\left(z^\lambda\right) \bigl[1+o(1)\bigr] & \text{for } |{\arg z}| \leq \theta, \\
	O\!\left(\exp(\mu |z|^\lambda)\right) & \text{for } |{\arg z}| > \theta
	\end{cases}			
\end{equation}
as $|z| \to \infty$, with each estimate holding uniformly in its sector.

Note that without loss of generality we assume that the sector of maximal growth is bisected by the positive real line---if a function grows maximally in some other direction we can replace $z$ by $\zeta z$ for some appropriate $\zeta \in \C$ with $|\zeta| = 1$ so that the maximal growth sector of the rotated function is oriented as desired. Similarly note that we assume $f$ has been normalized so that the leading coefficient in its asymptotic, as well as the coefficient of $z^\lambda$ in the exponential, are both equal to $1$.

For this $f$, let
\[
	p_n(z) = \sum_{k=0}^{n} \frac{f^{(k)}(0)}{k!} z^k
\]
and define
\begin{equation}
\label{seccurve1_rndef}
	r_n = \left(\frac{n}{\lambda}\right)^{1/\lambda}.
\end{equation}

\begin{theorem}
\label{seccurve1_maintheo}
All limit points of the zeros of the scaled partial sums $p_{n-1}(r_n z)$ in the sector $|{\arg z}| < \theta$, $z \neq 0$ lie on the curve
\[
	\left| z^\lambda \exp\!\left(1-z^\lambda\right) \right| = 1, \qquad |z| \leq 1.
\]
Further, these zeros approach this curve from the region $|z^\lambda \exp(1-z^\lambda)| > 1$.
\end{theorem}

In fact we will show that if $z = z(n)$ is a zero of $p_{n-1}(r_n z)$ which converges to a point $z_0$ with $|{\arg z_0}| < \theta$, $z_0 \neq 0,1$ then
\[
	\left| z^\lambda \exp\!\left(1-z^\lambda\right) \right| = 1 + \frac{\lambda\log n}{2n} + O\!\left(n^{-1}\right)
\]
as $n \to \infty$.

\subsection{Definitions and Preliminaries}
\label{seccurve1_defsandprelims}

For $|{\arg z}| \leq \theta$ it follows from \eqref{seccurve1_fgrowth} that
\begin{align}
\frac{f(r_n z)}{r_n^a (\log r_n)^b (e^{1/\lambda}z)^n} &\sim z^a \left(z^\lambda e^{1-z^\lambda}\right)^{-n/\lambda} \nonumber \\
&= z^a e^{n \varphi(z)} \label{seccurve1_integrandasymp}
\end{align}
as $n \to \infty$, where
\begin{equation}
\label{seccurve1_phidef}
	\varphi(z) := \left.\left(z^\lambda - 1 - \lambda \log z\right)\right/\lambda.
\end{equation}

\begin{definition}
\label{seccurve1_admisscontour}
A contour $\gamma$ is said to be \textit{admissible} for a function $\varphi$ if
\begin{enumerate}
\item $\gamma$ is a smooth Jordan curve winding counterclockwise around the origin.
\item In the sector $|{\arg z}| \leq \theta$, $\gamma$ is a positive distance from the curve $\re \varphi(z) = 0$ except for a part that lies in some neighborhood $U_\gamma$ of $z = 1$.  In this set $U_\gamma$ the contour $\gamma$ coincides with the path of steepest decent of the function $\re \varphi(z)$ passing through the point $z = 1$.
\item In the sector $|{\arg z}| \geq \theta$, $\gamma$ coincides with the unit circle.
\end{enumerate}
\end{definition}

Let $\gamma$ be an admissible contour for $\varphi$ and suppose for now that $z \neq 0$ is inside the scaled contour $r_n \gamma$.  The function
\[
	\frac{f(z) - p_{n-1}(z)}{z^n} =: \Phi(z)
\]
is entire, so by Cauchy's integral formula
\begin{equation}
\label{seccurve1_fnbuild}
	\Phi(z) = \frac{1}{2\pi i} \int_{r_n \gamma} \zeta^{-n} f(\zeta) \frac{d\zeta}{\zeta - z} - \frac{1}{2\pi i}\int_{r_n \gamma} \zeta^{-n} p_{n-1}(\zeta) \frac{d\zeta}{\zeta - z}.
\end{equation}
Since
\[
	\int_{r_n \gamma} \zeta^{-m} \frac{d\zeta}{\zeta - z} = 0
\]
for all integers $m \geq 1$, the second integral in \eqref{seccurve1_fnbuild} is zero.  Making the substitution $\zeta = r_n s$ and replacing $z$ by $r_n z$ yields the identity
\[
	\frac{f(r_n z) - p_{n-1}(r_n z)}{(r_n z)^n} = \frac{1}{2\pi i} \int_\gamma (r_n s)^{-n} f(r_n s) \frac{ds}{s-z},
\]
which holds for $z \neq 0$ inside $\gamma$.  (This construction is a special case of the one in \cite[p.\ 436]{edrei:paderem} for an integral representation of the error of a Pad\'e approximation.)

Define the function
\begin{equation}
\label{seccurve1_fdef}
	F_n(z) = \frac{r_n^{-a} (\log r_n)^{-b}}{2\pi i} \int_{\gamma} \left(e^{1/\lambda}s\right)^{-n}\! f(r_n s) \frac{ds}{s-z}
\end{equation}
for $z \notin \gamma$, $z \neq 0$. For $z$ inside $\gamma$ with $z \neq 0$ it follows from the derivation above that
\[
	F_n(z) = \frac{f(r_n z) - p_{n-1}(r_n z)}{r_n^a (\log r_n)^b (e^{1/\lambda} z)^n}.
\]
The value of $F_n(z)$ for $z$ outside $\gamma$ can be calculated using a similar derivation or by using the residue theorem, and in total
\begin{equation}
\label{seccurve1_fnexplicit}
	F_n(z) = \frac{1}{r_n^a (\log r_n)^b (e^{1/\lambda} z)^n} \times \begin{cases}
		-p_{n-1}(r_n z) & \text{for } z \text{ outside } \gamma, \\
		f(r_n z) - p_{n-1}(r_n z) & \text{for } z \neq 0 \text{ inside } \gamma.
		\end{cases}
\end{equation}

A straightforward calculation shows that $\varphi(1) = \varphi'(1) = 0$ and $\varphi''(1) = \lambda$, so
\[
	\varphi(s) = \frac{\lambda}{2}(s-1)^2 + O\!\left((s-1)^3\right)
\]
in a neighborhood of $s=1$.  The inverse function theorem ensures the existence of a neighborhood $V$ of the origin, a neighborhood $U \subset U_\gamma$ of $s=1$, and a biholomorphic map $\psi \colon V \to U$ which satisfies
\[
	(\varphi \circ \psi)(x) = x^2
\]
for $x \in V$. This function $\psi$ maps a segment of the imaginary axis onto the path of steepest descent of the function $\re \varphi(z)$ going through $z=1$ with $\psi(0) = 1$, and we make the choice that $\psi'(0) = \sqrt{2/\lambda}$.

\subsection{Proof of Theorem \ref{seccurve1_maintheo}}
\label{seccurve1_maintheoproof}

Split the integral for $F_n$ into the two pieces
\begin{equation}
\label{seccurve1_fnsplit}
	F_n(z) = \frac{r_n^{-a} (\log r_n)^{-b}}{2\pi i} \left( \int_{\gamma_\theta} \left(e^{1/\lambda}s\right)^{-n}\! f(r_ns) \frac{ds}{s-z} + \int_{\gamma \setminus \gamma_\theta} \left(e^{1/\lambda}s\right)^{-n}\! f(r_ns) \frac{ds}{s-z} \right),
\end{equation}
where $\gamma_\theta$ is the portion of $\gamma$ in the sector $|{\arg z}| \leq \theta$.

\begin{lemma}
\label{seccurve1_gminusgtlem}
If $z$ is restricted to any sector $|{\arg z}| \leq \theta - \epsilon$ with $\epsilon > 0$ then
\[
	\int_{\gamma \setminus \gamma_\theta} \left(e^{1/\lambda}s\right)^{-n}\! f(r_ns) \frac{ds}{s-z} = O\!\left(e^{(\mu-1)n/\lambda}\right)
\]
as $n \to \infty$ uniformly in $z$.
\end{lemma}

\begin{proof}
For any $\epsilon > 0$ there exists a constant $K > 0$ such that $|s-z| \geq K$ for all $s \in \gamma\setminus\gamma_\theta$ and all $z$ with $|{\arg z}| \leq \theta - \epsilon$. Also, it follows from the asymptotic assumption on $f$ in \eqref{seccurve1_fgrowth} that for all $z$ large enough with $|{\arg z}| \geq \theta$ there is a constant $M$ such that
\[
	|f(z)| \leq M \exp\!\left(\mu |z|^\lambda\right).
\]
Since the contour $\gamma \setminus \gamma_\theta$ lies on the unit circle in the sector $|{\arg z}| \geq \theta$, it therefore follows from the above assumptions that
\[
	\left| \int_{\gamma \setminus \gamma_\theta} \left(e^{1/\lambda}s\right)^{-n}\! f(r_ns) \frac{ds}{s-z} \right| \leq K^{-1} M \len(\gamma\setminus\gamma_\theta) e^{(\mu-1)n/\lambda}
\]
for all $n$ large enough.
\end{proof}

Define the function $\delta(z)$ for $|{\arg z}| \leq \theta$ by
\begin{equation}
\label{seccurve1_deltadef}
	f(z) = z^a (\log z)^b \exp(z^\lambda) \bigl[ 1+\delta(z) \bigr].
\end{equation}
This implies
\[
	\frac{f(r_n s)}{r_n^a (\log r_n)^b (e^{1/\lambda}s)^n} = s^a e^{n \varphi(s)} \left(1 + \frac{\log s}{\log r_n}\right)^b \bigl[ 1 + \delta(r_n s) \bigr]
\]
for $s \in \gamma_\theta$. By then defining
\begin{equation}
\label{seccurve1_deltatdef}
	\tilde{\delta}(r_n, s) = \left(1 + \frac{\log s}{\log r_n}\right)^b \bigl[ 1 + \delta(r_n s) \bigr] - 1
\end{equation}
the integrand of the first integral in \eqref{seccurve1_fnsplit} can be rewritten as
\[
	s^a e^{n\varphi(s)} + s^a e^{n\varphi(s)} \tilde{\delta}(r_n, s).
\]
It then follows from \eqref{seccurve1_fnsplit} and Lemma \ref{seccurve1_gminusgtlem} that, for some constant $c > 0$,
\begin{equation}
\label{seccurve1_fnsplit2}
	F_n(z) = \frac{1}{2\pi i} \int_{\gamma_\theta} s^a e^{n\varphi(s)} \frac{ds}{s-z} + \frac{1}{2\pi i} \int_{\gamma_\theta} s^a e^{n\varphi(s)}\tilde{\delta}(r_n, s) \frac{ds}{s-z} + O(e^{-cn})
\end{equation}
as $n \to \infty$ uniformly for $z$ restricted to any sector $|{\arg z}| \leq \theta - \epsilon$ with $\epsilon > 0$.

For $\epsilon > 0$ define $N_\epsilon$ to be the set of all points within a distance of $\epsilon$ of the curve~$\gamma_\theta$.

\begin{lemma}
\label{seccurve1_fintapprox1}
\[
	\int_{\gamma_\theta} s^a e^{n \varphi(s)} \frac{ds}{s-z} = \frac{i}{1-z} \sqrt\frac{2\pi}{\lambda n} + O\!\left(n^{-1}\right)
\]
as $n \to \infty$ uniformly for $z \in \C \setminus N_\epsilon$ with $\epsilon > 0$.
\end{lemma}

\begin{proof}
Fix $\epsilon > 0$. Using the same method as in the proof of Lemma \ref{seccurve1_gminusgtlem} it can be shown that the integral over the part of the contour outside of $U$ is exponentially small, so that
\[
	\int_{\gamma_\theta} s^a e^{n\varphi(s)} \frac{ds}{s-z} = \int_{\gamma_\theta \cap U} s^a e^{n\varphi(s)} \frac{ds}{s-z} + O(e^{-cn})
\]
as $n \to \infty$ uniformly for $z \in \C \setminus N_\epsilon$, where $c$ is some positive constant not depending on $n$ or $z$. Making the substitution $s = \psi(it)$ yields
\begin{equation}
\label{seccurve1_g2sub}
	\int_{\gamma_\theta} s^a e^{n\varphi(s)} \frac{ds}{s-z} = \int_{-\alpha_1}^{\alpha_2} e^{-nt^2} \psi(it)^a \frac{i\psi'(it)}{\psi(it)-z}\,dt + O(e^{-cn})
\end{equation}
for some $\alpha_1,\alpha_2 > 0$.

Now write
\[
	\psi(it)^a \frac{i\psi'(it)}{\psi(it)-z} = \frac{i}{1-z} \sqrt\frac{2}{\lambda} + \tilde\psi(t,z),
\]
where
\[
	\tilde\psi(t,z) := \psi(it)^a \frac{i\psi'(it)}{\psi(it)-z} - \psi(0)^a \frac{i\psi'(0)}{\psi(0)-z}.
\]
By taking $U$ smaller if necessary it is tedious though straightforward to show by Taylor's theorem that for all $z \in \C \setminus N_\epsilon$ and for all $t \in [-\alpha_1,\alpha_2]$ that
\[
	|\tilde\psi(t,z)| \leq M|t|
\]
for some constant $M$ not depending on $z$, and therefore that
\begin{align*}
\left|\int_{-\alpha_1}^{\alpha_2} e^{-nt^2} \tilde\psi(t,z)\,dt \right| &\leq M \int_{-\alpha_1}^{\alpha_2} e^{-nt^2} |t|\,dt \\
&< 2M \int_0^\infty e^{-nt^2} t\,dt \\
&= \frac{M}{n}.
\end{align*}
Substituting this into \eqref{seccurve1_g2sub} and using the fact that
\[
	\int_{-\alpha_1}^{\alpha_2} e^{-nt^2} \,dt = \sqrt\frac{\pi}{n} + O(e^{-c'n})
\]
for some constant $c' > 0$ yields the estimate
\begin{align}
\int_{\gamma_\theta} s^a e^{n\varphi(s)} \frac{ds}{s-z} &= \frac{i}{1-z} \sqrt\frac{2}{\lambda} \int_{-\alpha_1}^{\alpha_2} e^{-nt^2} \,dt + O\!\left(n^{-1}\right) \nonumber \\
	&= \frac{i}{1-z} \sqrt\frac{2\pi}{\lambda n} + O\!\left(n^{-1}\right)
	\label{seccurve1_g2error}
\end{align}
as $n \to \infty$ uniformly for $z \in \C \setminus N_\epsilon$.
\end{proof}

\begin{lemma}
\label{seccurve1_fintapprox2}
\[
	\int_{\gamma_\theta} s^a e^{n\varphi(s)} \tilde\delta(r_n,s) \frac{ds}{s-z} = o\!\left(n^{-1/2}\right)
\]
as $n \to \infty$ uniformly for $z \in \C \setminus N_\epsilon$ with $\epsilon > 0$.
\end{lemma}

\begin{proof}
Fix $\epsilon > 0$ and let $U' \subset U$ be a neighborhood of $z=1$ such that
\[
	\sup_{z \in U'} |z-1| < \epsilon.
\]
As in Lemma \ref{seccurve1_fintapprox2},
\[
	\int_{\gamma_\theta} s^a e^{n\varphi(s)} \tilde\delta(r_n,s) \frac{ds}{s-z} = \int_{\gamma_\theta \cap U'} s^a e^{n\varphi(s)} \tilde\delta(r_n,s) \frac{ds}{s-z} + O(e^{-cn})
\]
as $n \to \infty$ uniformly for $z \in \C \setminus N_\epsilon$, where $c$ is some positive constant not depending on $n$ or $z$, and the substitution $s=\psi(it)$ yields
\[
	\int_{\gamma_\theta} s^a e^{n\varphi(s)} \tilde\delta(r_n,s) \frac{ds}{s-z} = \int_{-\alpha_1'}^{\alpha_2'} e^{-nt^2} \psi(it)^a \frac{i\psi'(it)}{\psi(it)-z} \tilde\delta(r_n,\psi(it))\,dt + O(e^{-cn})
\]
for some $\alpha_1',\alpha_2' > 0$. If $z \in \C \setminus N_\epsilon$ and $t \in [-\alpha_1',\alpha_2']$ then $|\psi(it) - z| \geq K$ for some positive constant $K$, so that
\begin{align*}
&\left| \int_{-\alpha_1'}^{\alpha_2'} e^{-nt^2} \psi(it)^a \frac{i\psi'(it)}{\psi(it)-z} \delta(n\psi(it))\,dt \right| \\
&\qquad \leq K^{-1} \int_{-\alpha_1'}^{\alpha_2'} e^{-nt^2}\,dt \sup_{-\alpha_1' < t < \alpha_2'} \left| \psi(it)^a \psi'(it) \tilde\delta(r_n,\psi(it)) \right| \\
&\qquad < \frac{M}{\sqrt{n}} \sup_{-\alpha_1' < t < \alpha_2'} \left|\tilde\delta(r_n,\psi(it))\right|
\end{align*}
for some positive constant $M$. Note that $U$ may need to be made smaller to ensure that $\psi'(it)$ is bounded---doing this doesn't cause any issues.

By definition $\delta(z) \to 0$ as $|z| \to \infty$ uniformly for $|{\arg z}| \leq \theta$, and so $\delta(r_n s) \to 0$ as $n \to \infty$ uniformly for $s \in \gamma_\theta$. By extension this holds for $\tilde\delta(r_n,s)$ as well, and thus
\[
	\lim_{n \to \infty} \sup_{-\alpha_1' < t < \alpha_2'} \left|\tilde\delta(r_n,\psi(it))\right| = 0.
\]
\end{proof}

Combining Lemmas \ref{seccurve1_fintapprox1} and \ref{seccurve1_fintapprox2} and equation \eqref{seccurve1_fnsplit2} yields the asymptotic
\begin{equation}
\label{seccurve1_fnasymp}
F_n(z) = \frac{1}{(1-z)\sqrt{2\pi\lambda n}} + o\!\left(n^{-1/2}\right)
\end{equation}
as $n \to \infty$ uniformly for $z \in \C \setminus N_\epsilon$ with $|{\arg z}| \leq \theta - \epsilon$ for any fixed $\epsilon > 0$.

Suppose that $z = z(n)$ is a zero of $p_{n-1}(r_n z)$, i.e.\ that $p_{n-1}(r_n z) = 0$. Suppose further that, as $n \to \infty$, $z$ tends to a limit point inside $\gamma$ and in the set
\[
	\{z \in \C : |{\arg z}| < \theta \text{ and } z \neq 0,1\}.
\]
Then for $n$ large enough there is an $\epsilon > 0$ such that $|{\arg z}| \leq \theta - \epsilon$ and $z \in \C \setminus N_\epsilon$, and so from the definition of $F_n(z)$ it follows from \eqref{seccurve1_fnasymp} that
\[
	\frac{f(r_n z)}{r_n^a (\log r_n)^b (e^{1/\lambda} z)^n} \sim \frac{1}{(1-z)\sqrt{2\pi\lambda n}} = \frac{e^{-(\log n)/2}}{(1-z)\sqrt{2\pi\lambda}}
\]
as $n \to \infty$, and from the asymptotic assumption on $f$ in \eqref{seccurve1_fgrowth} that
\begin{equation}
\label{seccurve1_zeroest}
	z^a \left[ z^\lambda \exp\!\left( 1 - z^\lambda \right) \right]^{-n/\lambda} \sim \frac{e^{-(\log n)/2}}{(1-z)\sqrt{2\pi\lambda}}
\end{equation}
as $n \to \infty$. If $g(n) = \Theta(1)$ then
\[
	|g(z)|^{1/n} = 1 + O\!\left(n^{-1}\right)
\]
as $n \to \infty$, so taking absolute values and raising both sides of \eqref{seccurve1_zeroest} to the power $-\lambda/n$ yields
\begin{align}
\left| z^\lambda \exp\!\left( 1 - z^\lambda \right) \right| &= \exp\!\left(\frac{\lambda \log n}{2n}\right) \left[ 1 + O\!\left(n^{-1}\right) \right] \nonumber \\
	&= 1 + \frac{\lambda \log n}{2n} + O\!\left(n^{-1}\right)
\label{seccurve1_insideasymp}
\end{align}
as $n \to \infty$.

It follows from the above asymptotic that the limit points $z_0$ of the zeros of $p_{n-1}(r_n z)$ inside $\gamma$ with $|{\arg z_0}| < \theta$, $z_0 \neq 0$ lie on the curve
\[
	\left|z_0^\lambda \exp\!\left(1-z_0^\lambda\right)\right| = 1,
\]
or, equivalently,
\[
	\re \varphi(z_0) = 0.
\]
Further, the exterior of this curve in this region is characterized by the inequality
\[
	 \left|z^\lambda \exp\!\left(1-z^\lambda\right)\right| > 1,
\]
so the zeros approach this limit curve from the exterior. Finally, as an admissible contour may be taken to lie as close to the lines $\im \varphi(z) = 0$ as desired, the only such limit points $z_0$ in the whole set
\[
	\{z \in \C : \re \varphi(z) > 0 \} \cup \{ z \in \C : \re \varphi(z) \leq 0 \text{ and } |z| \leq 1 \}
\]
must lie on the curve $\re \varphi(z_0) = 0$.

Suppose now that $z$ is a zero of $p_{n-1}(r_n z)$ which lies in a sector $|{\arg z}| \leq \theta-\epsilon$ with $\epsilon > 0$ and in the set
\[
	\{ z \in \C : \re \varphi(z) \leq 0 \text{ and } |z| > 1 \} \setminus N_\epsilon
\]
for $n$ large enough. Then $z$ eventually lies outside of $\gamma$, and so $F_n(z) = 0$ by \eqref{seccurve1_fnexplicit}. From equation \eqref{seccurve1_fnasymp} it follows that
\[
	0 = \frac{1}{(1-z)\sqrt{2\pi\lambda n}} + o\!\left(n^{-1/2}\right)
\]
as $n \to \infty$, and, on multiplying this by $\sqrt{n}$, that
\[
	0 = \frac{1}{(1-z)\sqrt{2\pi\lambda}} + o(1)
\]
as $n \to \infty$. The only way this is possible is if $|z| \to \infty$.

Since $\epsilon > 0$ was arbitrary, it follows that the zeros of $p_{n-1}(r_n z)$ have no limit point in the set
\[
	\{ z \in \C : \re \varphi(z) \leq 0 \text{ and } |z| > 1 \}.
\]
This completes the proof.

\section{Two Directions of Maximal Exponential Growth}
\label{seccurvetwo_twodirsec}

Let $a,b,A \in \C$, $0 < \lambda < \infty$, $\mu < 1$, and $\zeta \in \C$ with $|\zeta| = 1$, $\zeta \neq 1$. Let $\theta \in (0,\pi)$ be small enough so that the sectors $|{\arg z}| \leq \theta$ and $|{\arg(z/\zeta)}| \leq \theta$ are disjoint.  We suppose that $f$ is an entire function with the asymptotic behavior
\begin{equation}
\label{seccurvetwo_fgrowth}
f(z) = \begin{cases}
	z^a \exp\!\left(z^\lambda\right) \bigl[1+o(1)\bigr] & \text{for } |{\arg z}| \leq \theta, \\
	A(z/\zeta)^b \exp\!\left((z/\zeta)^\lambda\right) \bigl[1+o(1)\bigr] & \text{for } |{\arg(z/\zeta)}| \leq \theta, \\
	O\!\left(\exp\!\left(\mu |z|^\lambda\right)\right) & \text{otherwise}
	\end{cases}			
\end{equation}
as $|z| \to \infty$, with each estimate holding uniformly in its sector.

Without loss of generality we assume that one sector of maximal growth is bisected by the positive real line---if neither of the function's directions of maximal growth are bisected by the positive real line then we can replace $z$ by $\omega z$ for some $\omega \in \C$ with $|\omega| = 1$ so that one of those directions is as desired. Similarly we assume that $f$ has been normalized so that the leading coefficient in its asymptotic in this direction bisected by the positive real line, as well as the coefficient of $z^\lambda$ in the first exponential and the coefficient of $(z/\zeta)^\lambda$ in the second exponential, are all equal to $1$.

For this $f$, let
\[
	p_n(z) = \sum_{k=0}^{n} \frac{f^{(k)}(0)}{k!} z^k
\]
and define
\begin{equation}
\label{seccurvetwo_rndef}
	r_n = \left(\frac{n}{\lambda}\right)^{1/\lambda}.
\end{equation}

\begin{theorem}
\label{seccurvetwo_maintheo}
All limit points of the zeros of the scaled partial sums $p_{n-1}(r_n z)$ in the sector $|{\arg z}| < \theta$, $z \neq 0$ lie on the curve
\[
	\left| z^\lambda \exp\!\left(1-z^\lambda\right) \right| = 1, \qquad |z| \leq 1.
\]
If $\re a \neq \re b$, define $\alpha = \min\{\re a,\re b\}$. In this case the zeros approach this limit curve from the region $|z^\lambda \exp(1-z^\lambda)| > 0$ if $\alpha - \re b + \lambda/2 > 0$ and from the region $|z^\lambda \exp(1-z^\lambda)| < 0$ if $\alpha - \re b + \lambda/2 < 0$.
\end{theorem}

In fact we will show that if $\re a \neq \re b$ and if $z = z(n)$ is a zero of $p_{n-1}(r_n z)$ which converges to a point $z_0$ with $|{\arg z_0}| < \theta$, $z_0 \neq 0$ then
\[
	\left| z^\lambda \exp\!\left(1-z^\lambda\right) \right| = 1 + \left(\alpha - \re b + \frac{\lambda}{2}\right)\frac{\log n}{n} + O\!\left(n^{-1}\right)
\]
as $n \to \infty$. This formula also holds when $\re a = \re b$ as long as
\[
	z_0 \notin \{ w \in \C : |\zeta - w| = |A(1-w)| \}.
\]

\subsection{Definitions and Preliminaries}
\label{seccurvetwo_defsandprelims}

Let
\[
	\varphi(z) = \left.\left(z^\lambda - 1 - \lambda\log z\right)\right/z
\]
and let $\gamma$ be an admissible contour for $\varphi$ (see Definition \ref{seccurve1_admisscontour}).

Define
\begin{equation}
\label{seccurvetwo_fndef}
	F_n(z) = \frac{r_n^{-a}}{2\pi i} \int_\gamma \left(e^{1/\lambda} s\right)^{-n}\! f(r_n s) \frac{ds}{s-z}
\end{equation}
for $z \notin \gamma$, $z \neq 0$. As in Section \ref{seccurve1_defsandprelims},
\begin{equation}
\label{seccurvetwo_fnexplicit}
	F_n(z) = \frac{1}{r_n^a (e^{1/\lambda} z)^n} \times \begin{cases}
		-p_{n-1}(r_n z) & \text{for } z \text{ outside } \gamma, \\
		f(r_n z) - p_{n-1}(r_n z) & \text{for } z \neq 0 \text{ inside } \gamma.
		\end{cases}
\end{equation}

\subsection{Proof of Theorem \ref{seccurvetwo_maintheo}}
\label{seccurvetwo_maintheoproof}

Call $\gamma_1$ the part of $\gamma$ in the sector $|{\arg z}| \leq \theta$, call $\gamma_2$ the part of $\gamma$ in the sector $|{\arg(z/\zeta)}| \leq \theta$, and call $\gamma_3$ the part of $\gamma$ outside of either of those sectors. This allows us to divide the integral in the definition of $F_n(z)$ in \eqref{seccurvetwo_fndef} into three parts
\[
	\int_\gamma = \int_{\gamma_1} + \int_{\gamma_2} + \int_{\gamma_3}
\]
which will be analyzed separately.

Using a method identical to the proof of Lemma \ref{seccurve1_gminusgtlem} it can be shown that
\begin{equation}
\label{seccurvetwo_g3asymp}
	\int_{\gamma_3} \left(e^{1/\lambda} s\right)^{-n}\! f(r_n s) \frac{ds}{s-z} = O\!\left(e^{(\mu-1)n/\lambda}\right)
\end{equation}
as $n \to \infty$ uniformly for $z$ restricted to any sector $|{\arg z}| \leq \theta - \epsilon$ with $\epsilon > 0$, and using a method identical to the proofs of Lemmas \ref{seccurve1_fintapprox1} and \ref{seccurve1_fintapprox2} that
\begin{equation}
\label{seccurvetwo_g1asymp}
	\int_{\gamma_1} \left(e^{1/\lambda} s\right)^{-n}\! f(r_n s) \frac{ds}{s-z} = \frac{ir_n^a}{1-z} \sqrt\frac{2\pi}{\lambda n} + o\!\left(r_n^a n^{-1/2}\right)
\end{equation}
as $n \to \infty$ uniformly for $z$ in any set $\C \setminus N_\epsilon$ with $\epsilon > 0$. Here $N_\epsilon$ is defined to be the set of all points within a distance of $\epsilon$ from $\gamma_1$.

For $\epsilon > 0$ define $\tilde N_\epsilon$ to be the set of all points within a distance of $\epsilon$ of the curve~$\gamma_2$.

\begin{lemma}
\label{seccurvetwo_g2lemma}
\[
	\int_{\gamma_2} \left(e^{1/\lambda} s\right)^{-n}\! f(r_n s) \frac{ds}{s-z} = \frac{iA\zeta^{1-n} r_n^b}{\zeta-z} \sqrt\frac{2\pi}{\lambda n} + o(r_n^b n^{-1/2})
\]
as $n \to \infty$ uniformly for $z \in \C \setminus \tilde N_\epsilon$ with $\epsilon > 0$.
\end{lemma}

\begin{proof}
By the asymptotic assumption on $f$ in \eqref{seccurvetwo_fgrowth}, for $|{\arg(z/\zeta)}| \leq \theta$ we can write
\begin{equation}
\label{seccurvetwo_deltadef}
	f(z) = A (z/\zeta)^b \exp\!\left((z/\zeta)^\lambda\right) \bigl[ 1+\delta(z) \bigr],
\end{equation}
where $\delta(z) \to 0$ uniformly as $|z| \to \infty$ in this sector.  This implies
\[
	\frac{f(r_n s)}{(e^{1/\lambda}s)^n} = A \zeta^{-n} r_n^b (s/\zeta)^b e^{n \tilde\varphi(s)} \bigl[ 1 + \delta(r_n s) \bigr]
\]
for $s \in \gamma_2$, where
\[
	\tilde\varphi(s) = \left.\left[(s/\zeta)^\lambda-1-\lambda\log(s/\zeta)\right]\right/\lambda,
\]
allowing us to split the integral in question like
\begin{align}
&\int_{\gamma_2} \left(e^{1/\lambda} s\right)^{-n}\! f(r_n s) \frac{ds}{s-z} \nonumber \\
	&\qquad = A \zeta^{-n} r_n^b \left( \int_{\gamma_2} (s/\zeta)^b e^{n\tilde\varphi(s)} \frac{ds}{s-z} + \int_{\gamma_2} (s/\zeta)^b e^{n\tilde\varphi(s)} \delta(r_ns) \frac{ds}{s-z} \right). \label{seccurvetwo_g2split}
\end{align}
By the inverse function theorem there exists a neighborhood $\tilde V$ of the origin, a neighborhood $\tilde U$ of $s=\zeta$, and a biholomorphic map $\tilde\psi \colon \tilde V \to \tilde U$ which satisfies
\[
	(\tilde\varphi \circ \tilde\psi)(x) = x^2
\]
for $x \in \tilde V$. It follows that $\tilde\psi(0) = \zeta$, and we make the choice that $\psi'(0) = \zeta\sqrt{2/\lambda}$. This function $\tilde\psi$ maps a segment of the imaginary axis onto the path of steepest descent of the function $\re \tilde\varphi(z)$ going through $z=\zeta$.

The rest of the proof proceeds just as the proofs of Lemmas \ref{seccurve1_fintapprox1} and \ref{seccurve1_fintapprox2} by using $\tilde\varphi$ and $\tilde\psi$ in place of $\varphi$ and $\psi$, respectively.
\end{proof}

Combining equations \eqref{seccurvetwo_g3asymp}, \eqref{seccurvetwo_g1asymp}, and Lemma \ref{seccurvetwo_g2lemma} yields the asymptotic
\begin{equation}
\label{seccurvetwo_fnasymp}
	F_n(z) = \frac{1}{(1-z)\sqrt{2\pi\lambda n}} + \frac{A\zeta^{1-n} r_n^{b-a}}{(\zeta-z)\sqrt{2\pi\lambda n}} + o\!\left(n^{-1/2}\right) + o\!\left(r_n^{b-a} n^{-1/2}\right)
\end{equation}
as $n \to \infty$ uniformly for $z \in \C \setminus N_\epsilon$ restricted to the sector $|{\arg z}| \leq \theta - \epsilon$, where $\epsilon > 0$ is arbitrary but fixed.

Suppose that $z = z(n)$ is a zero of $p_{n-1}(r_n z)$ which, as $n \to \infty$, tends to a limit point inside $\gamma$ and in the set
\[
	\{ z \in \C : |{\arg z}| < \theta \text{ and } z \neq 0,1 \}.
\]
Then for $n$ large enough there is an $\epsilon > 0$ such that $|{\arg z}| \leq \theta - \epsilon$ and $z \in \C \setminus N_\epsilon$, and so from the definition of $F_n(z)$ it follows from \eqref{seccurvetwo_fnasymp} that
\[
	\frac{f(r_n z)}{r_n^a (e^{1/\lambda} z)^n} = \frac{1}{(1-z)\sqrt{2\pi\lambda n}} + \frac{A\zeta^{1-n} r_n^{b-a}}{(\zeta-z)\sqrt{2\pi\lambda n}} + o\!\left(n^{-1/2}\right) + o\!\left(r_n^{b-a} n^{-1/2}\right)
\]
as $n \to \infty$, and from the asymptotic assumption on $f$ in \eqref{seccurvetwo_fgrowth} that
\begin{align}
&z^a \left[ z^\lambda \exp\!\left(1-z^\lambda\right) \right]^{-n/\lambda} \bigl[1 + o(1) \bigr] \nonumber \\
&\qquad = \frac{1}{(1-z)\sqrt{2\pi\lambda n}} + \frac{A\zeta^{1-n} r_n^{b-a}}{(\zeta-z)\sqrt{2\pi\lambda n}} + o\!\left(n^{-1/2}\right) + o\!\left(r_n^{b-a} n^{-1/2}\right)
\label{seccurvetwo_incurveasymp1}
\end{align}
as $n \to \infty$.

If $\re a > \re b$ then
\[
	z^a \left[ z^\lambda \exp\!\left(1-z^\lambda\right) \right]^{-n/\lambda} \sim \frac{1}{(1-z)\sqrt{2\pi\lambda n}}
\]
and hence, just as in \eqref{seccurve1_insideasymp},
\begin{equation}
\label{seccurvetwo_incurveasymp2}
	\left| z^\lambda \exp\!\left(1-z^\lambda\right) \right| = 1 + \frac{\lambda \log n}{2n} + O\!\left(n^{-1}\right)
\end{equation}
as $n \to \infty$. If instead $\re a < \re b$ then
\[
	z^a \left[ z^\lambda \exp\!\left(1-z^\lambda\right) \right]^{-n/\lambda} \sim \frac{A\zeta^{1-n} r_n^{b-a}}{(\zeta-z)\sqrt{2\pi\lambda n}}
\]
and hence
\begin{equation}
\label{seccurvetwo_incurveasymp3}
	\left| z^\lambda \exp\!\left(1-z^\lambda\right) \right| = 1 + \left(\re a - \re b + \frac{\lambda}{2}\right)\frac{\log n}{n} + O\!\left(n^{-1}\right)
\end{equation}
as $n \to \infty$. Finally suppose $\re b = \re a$, so that
\begin{align}
&z^a \left[ z^\lambda \exp\!\left(1-z^\lambda\right) \right]^{-n/\lambda} \bigl[1+o(1)\bigr] \nonumber \\
&\qquad = \left( A^{-1}\frac{\zeta - z}{1-z} + \zeta^{1-n} r_n^{b-a} + o(1) \right) \frac{1}{(\zeta-z)\sqrt{2\pi\lambda n}}
\label{seccurvetwo_aeqb}
\end{align}
as $n \to \infty$. Note that $|\zeta^{1-n} r_n^{b-a}| = 1$ in this case, so for $n$ large enough the quantity in parentheses is bounded below in absolute value by a positive constant unless
\[
	\left| \frac{\zeta - z}{1-z} \right| \to |A|.
\]
If $z$ does not tend to a point on the circle
\[
	\{w \in \C : |\zeta - w| = |A(1-w)| \}
\]
then taking absolute values and raising both sides of the equation to the power $-\lambda/n$ yields
\begin{equation}
\label{seccurvetwo_incurveasymp4}
	\left| z^\lambda \exp\!\left(1-z^\lambda\right) \right| = 1 + \frac{\lambda \log n}{2n} + O\!\left(n^{-1}\right)
\end{equation}
as $n \to \infty$. Suppose now that $z$ does tend to a point on the circle
\[
	\{w \in \C : |\zeta - w| = |A(1-w)| \},
\]
define
\[
	M = \limsup_{n \to \infty} \left| A^{-1}\frac{\zeta - z}{1-z} + \zeta^{1-n} r_n^{b-a} \right|,
\]
and note that $M > 0$. It follows that
\begin{equation}
\label{seccurvetwo_oscliminf}
	\liminf_{n \to \infty} \left| A^{-1}\frac{\zeta - z}{1-z} + \zeta^{1-n} r_n^{b-a} \right|^{-\lambda/n} = 1.
\end{equation}
If
\[
	\limsup_{n \to \infty} \left| A^{-1}\frac{\zeta - z}{1-z} + \zeta^{1-n} r_n^{b-a} \right|^{-\lambda/n} \neq 1
\]
then
\[
	\lim_{n \to \infty} \left| A^{-1}\frac{\zeta - z}{1-z} + \zeta^{1-n} r_n^{b-a} \right|^{-\lambda/n}
\]
does not exist. However, by the above assumption on the convergence of $z$ as $n \to \infty$ it's true that
\[
	\lim_{n \to \infty} \left|z^{-a\lambda/n}\right|,\ \lim_{n \to \infty} \left|z^\lambda \exp\!\left(1-z^\lambda\right)\right|,\ \text{and } \lim_{n \to \infty} \left|\frac{1}{(\zeta-z)\sqrt{2\pi\lambda n}}\right|^{-\lambda/n}
\]
all exist. Consequently it follows from \eqref{seccurvetwo_aeqb} that
\[
	\lim_{n \to \infty} \left| A^{-1}\frac{\zeta - z}{1-z} + \zeta^{1-n} r_n^{b-a} \right|^{-\lambda/n}
\]
exists, and then from \eqref{seccurvetwo_oscliminf} that this limit is equal to $1$. Altogether, this implies that
\begin{equation}
\label{seccurvetwo_incurveasymp5}
	\left|z^\lambda \exp\!\left(1-z^\lambda\right)\right| = 1 + o(1)
\end{equation}
as $n \to \infty$.

It follows from the asymptotics in \eqref{seccurvetwo_incurveasymp2}, \eqref{seccurvetwo_incurveasymp3}, \eqref{seccurvetwo_incurveasymp4}, and \eqref{seccurvetwo_incurveasymp5} that the limit points $z_0$ of the zeros of $p_{n-1}(r_n z)$ inside $\gamma$ with $|{\arg z_0}| < \theta$, $z_0 \neq 0$ lie on the curve
\[
	\left|z_0^\lambda \exp\!\left(1-z_0^\lambda\right)\right| = 1,
\]
or, equivalently,
\[
	\re \varphi(z_0) = 0.
\]
Further, the exterior of this curve in this region is characterized by the inequality
\[
	 \left|z^\lambda \exp\!\left(1-z^\lambda\right)\right| > 1,
\]
so if $\re a \neq \re b$ and we set $\alpha = \min\{\re a,\re b\}$ then the zeros approach this limit curve from the exterior if $\alpha - \re b + \lambda/2 > 0$ and from the interior if $\alpha - \re b + \lambda/2 < 0$. If $\re a = \re b$ then the zeros which approach points of the curve which are not on the circle
\[
	\{w \in \C : |\zeta - w| = |A(1-w)| \}
\]
do so from the exterior.

As an admissible contour may be taken to lie as close to the lines $\im \varphi(z) = 0$ as desired, the only such limit points $z_0$ in the whole set
\[
	\{z \in \C : \re \varphi(z) > 0 \} \cup \{ z \in \C : \re \varphi(z) \leq 0 \text{ and } |z| \leq 1 \}
\]
must lie on the curve $\re \varphi(z_0) = 0$.

Suppose now that $z$ is a zero of $p_{n-1}(r_n z)$ which lies in a sector $|{\arg z}| \leq \theta-\epsilon$ with $\epsilon > 0$ and in the set
\[
	\{ z \in \C : \re \varphi(z) \leq 0 \text{ and } |z| > 1 \} \setminus N_\epsilon
\]
for $n$ large enough. Then $z$ eventually lies outside of $\gamma$, and so $F_n(z) = 0$ by \eqref{seccurvetwo_fnexplicit}. From equation \eqref{seccurvetwo_fnasymp} it follows that
\[
	0 = \frac{1}{(1-z)\sqrt{2\pi\lambda n}} + \frac{A\zeta^{1-n} r_n^{b-a}}{(\zeta-z)\sqrt{2\pi\lambda n}} + o\!\left(n^{-1/2}\right) + o\!\left(r_n^{b-a} n^{-1/2}\right)
\]
as $n \to \infty$.

If $\re a > \re b$ then multiplying through by $\sqrt{n}$ yields
\[
	0 = \frac{1}{(1-z)\sqrt{2\pi\lambda}} + o(1)
\]
as $n \to \infty$, and the only way this is possible is if $|z| \to \infty$. If $\re a < \re b$ then multiplying through by $r_n^{a-b} \sqrt{n}$ instead yields
\[
	0 = \frac{A\zeta^{1-n}}{(\zeta-z)\sqrt{2\pi\lambda}} + o(1)
\]
as $n \to \infty$, and we again conclude that we must have $|z| \to \infty$. Finally if $\re a = \re b$ then multiplying through by $\sqrt{n}$ yields
\[
	0 = \left( A^{-1}\frac{\zeta - z}{1-z} + \zeta^{1-n} r_n^{b-a}\right) \frac{1}{(\zeta-z)\sqrt{2\pi\lambda}} + o(1)
\]
as $n \to \infty$. Again, the only way this holds is if $|z| \to \infty$.

Since $\epsilon > 0$ was arbitrary, it follows that the zeros of $p_{n-1}(r_n z)$ have no limit point in the set
\[
	\{ z \in \C : \re \varphi(z) \leq 0 \text{ and } |z| > 1 \}.
\]
This completes the proof.

\section{Generalization to More Directions of Maximal Growth}
\label{seccurvemany_sec}

It's not difficult to extend the results in this chapter to functions with more than two of directions of maximal exponential growth.

Let $a,b_1,\ldots,b_m,A_1,\ldots,A_m \in \C$, $0 < \lambda < \infty$, $\mu < 1$, and $\zeta_1,\ldots,\zeta_m \in \C$ with $|\zeta_k| = 1$, $\zeta_k \neq 1$ for all $k = 1,\ldots,m$ and $\zeta_j \neq \zeta_k$ for $j \neq k$. Let $\theta \in (0,\pi)$ be small enough so that all of the sectors $|{\arg z}| \leq \theta$, $|{\arg(z/\zeta_k)}| \leq \theta$, $k=1,\ldots,m$, are disjoint.  We suppose that $f$ is an entire function with the asymptotic behavior
\begin{equation}
\label{seccurvemany_fgrowth}
f(z) = \begin{cases}
	z^a \exp\!\left(z^\lambda\right) \bigl[1+o(1)\bigr] & \text{for } |{\arg z}| \leq \theta, \\
	A_1(z/\zeta_1)^{b_1} \exp\!\left((z/\zeta_1)^\lambda\right) \bigl[1+o(1)\bigr] & \text{for } |{\arg(z/\zeta_1)}| \leq \theta, \\
	\qquad \vdots \\
	A_m(z/\zeta_m)^{b_m} \exp\!\left((z/\zeta_m)^\lambda\right) \bigl[1+o(1)\bigr] & \text{for } |{\arg(z/\zeta_m)}| \leq \theta, \\
	O\!\left(\exp\!\left(\mu |z|^\lambda\right)\right) & \text{otherwise}
	\end{cases}			
\end{equation}
as $|z| \to \infty$, with each estimate holding uniformly in its sector. For this $f$, let, $p_n(z)$, $r_n$, and $F_n(z)$ be defined as in Section \ref{seccurvetwo_twodirsec}. For convenience of notation, define $b_0 = a$, $A_0 = 1$, and $\zeta_0 = 1$.

We can derive an analogue to equations \eqref{seccurve1_fnasymp} and \eqref{seccurvetwo_fnasymp}, specifically
\begin{equation}
\label{seccurvemany_fnapprox}
	r_n^a F_n(z) = \frac{1}{\sqrt{2\pi\lambda n}} \sum_{k=0}^{m} \left[\frac{A_k \zeta_{k}^{1-n} r_n^{b_k}}{\zeta_k-z} + o\!\left(r_n^{b_k}n^{-1/2}\right)\right]
\end{equation}
as $n \to \infty$ uniformly with respect to $z$ as long as $z$ remains in any sector $|{\arg(z/\zeta_k)}| \leq \theta - \epsilon$ with $\epsilon > 0$, $k=0,1,\ldots,m$, and remains bounded away from $\gamma$. Following this we would proceed just as before to conclude that the limit points of the zeros of the scaled partial sum $p_{n-1}(r_n z)$ in the sector $|{\arg z}| < \theta$, $z \neq 0$ all still lie on the curve
\[
	\left| z^\lambda \exp\!\left(1-z^\lambda\right)\right| = 1.
\]

If $\re a > \re b_k$ for all $k=1,\ldots,n$ then the zeros will again approach these limit points from the exterior of the curve. The main complication that may arise is if $\re a$ is equal to a number of the quantities $\re b_k$, and in that case it may be more difficult to determine where the analogue of formula \eqref{seccurvetwo_incurveasymp3},
\[
	\left| z^\lambda \exp\!\left(1-z^\lambda\right)\right| = 1 + \left( \re a - \re b_k + \frac{\lambda}{2} \right) \frac{\log n}{n} + O\!\left(n^{-1}\right)
\]
for some $k \in \{1,\ldots,m\}$, will hold.

%% file: chap_curvescaling.tex
\chapter{Scaling Limits at the Arcs of the Limit Curve}
\label{chap_curvescaling}

In this chapter we aim to study the zeros of the scaled partial sums $p_n(r_n z)$ which approach the smooth arcs of the limit curve
\[
	S = \left\{ z \in \C : \left|z^\lambda \exp\!\left(1-z^\lambda\right)\right| = 1 \text{ and } |z| \leq 1 \right\}
\]
in the sector $|{\arg z}| < \theta$. We will determine how quickly these zeros approach the curve, track their movement and the spacing between them, and ultimately calculate a certain limit of the partial sums depending on an argument which follows the zeros in their approach.

The results in Sections \ref{curscasec_onedir}, \ref{curscatwosec_sec}, and \ref{curscamanysec_sec} can be seen as generalizations of Theorem \ref{introthm_esvcurvescaling} which was originally obtained by Edrei, Saff, and Varga in their monograph \cite{esv:sections}.

\section{Exploratory Analysis to Estimate the Rate of Approach to the Limit Curve}
\label{curscasec_exploratory}

We will first assume that $f$ has a single direction of maximal growth, just as in Section \ref{seccurve1_onedirsec}. In the derivation of the limit curve for the zeros of the scaled partial sums $p_{n-1}(r_n z)$ we obtained equation \eqref{seccurve1_zeroest}, which essentially says that any such zero $z = z(n)$ which tends to a point $\xi$ on the limit curve (other than the point $\xi=1$) satisfies
\begin{equation}
\label{curscaexploreeq_oneasymp}
	z^a \exp\!\left\{ \frac{n}{\lambda}\left(z^\lambda-1 - \lambda \log z\right) \right\} \approx \frac{e^{-(\log n)/2}}{(1-z)\sqrt{2\pi\lambda}}
\end{equation}
as $n \to \infty$. If we set $z = \xi + \delta$ and assume that $\delta \to 0$ as $n \to \infty$ then
\[
	z^a \approx \xi^a, \qquad 1-z \approx 1-\xi,
\]
and
\begin{align*}
&z^\lambda - 1 - \lambda \log z \\
	&\qquad = (\xi+\delta)^\lambda - 1 - \lambda\log(\xi+\delta) \\
	&\qquad = \xi^\lambda \left(1 + \frac{\delta}{\xi}\right)^\lambda - 1 - \lambda\log\xi - \lambda\log\!\left(1 + \frac{\delta}{\xi}\right) \\
	&\qquad \approx \xi^\lambda \left( 1 + \frac{\lambda}{\xi}\delta \right) - 1 - \lambda\log\xi - \lambda \frac{\delta}{\xi} \\
	&\qquad = \xi^\lambda - 1 - \lambda\log\xi + \frac{\lambda(\xi^\lambda-1)}{\xi}\delta,
\end{align*}
so that the estimate in \eqref{curscaexploreeq_oneasymp} becomes
\[
	\xi^a \exp\!\left\{\frac{n}{\lambda} \left(\xi^\lambda - 1 - \lambda\log\xi\right) + \frac{\xi^\lambda-1}{\xi}n\delta \right\} \approx \frac{e^{-(\log n)/2}}{(1-\xi)\sqrt{2\pi\lambda}}
\]
as $n \to \infty$. Since $\xi$ lies on the limit curve we have
\[
	\re\!\left(\xi^\lambda - 1 - \lambda\log\xi\right) = 0,
\]
so the factor $\xi^a \exp\{\frac{n}{\lambda} (\xi^\lambda - 1 - \lambda\log\xi)\}$ on the left-hand side is $\Theta(1)$. In order to balance the decay of the factor $e^{-(\log n)/2}$ on the right-hand side we thus need
\[
	\frac{\xi^\lambda-1}{\xi}n\delta \approx -\frac{1}{2} \log n,
\]
and hence
\[
	\delta \approx \frac{\xi\log n}{2(1-\xi^\lambda)n}
\]
as $n \to \infty$.

Now, this estimate for $\delta$ doesn't include anything that differentiates two separate zeros of $p_{n-1}(r_n z)$, so it makes sense for us to try to look for this information in higher-order corrections to $\delta$. In \eqref{seccurve1_insideasymp} we found that $z$ satisfies
\[
	\left| z^\lambda \exp\!\left( 1 - z^\lambda \right) \right| = 1 + \frac{\lambda \log n}{2n} + O\!\left(n^{-1}\right)
\]
as $n \to \infty$, so we might guess that $\delta$ has the same $1/n$ correction term, i.e.
\[
	\delta = \frac{\xi\log n}{2(1-\xi^\lambda)n} + \frac{v}{n},
\]
where $v = O(1)$ as $n \to \infty$. Attempting to use such a $\delta$ to prove something like Theorem \ref{curscaonetheo_maintheo} below leads immediately to the refinement $v = \xi(w-i\tau_n)/(1-\xi^\lambda)$ that appears in the statement of the theorem.

Suppose now that $f$ has two directions of maximal exponential growth, just as in Section \ref{seccurvetwo_twodirsec}. If $\re a > \re b$ then we again have
\[
	\xi^a \exp\!\left\{ \frac{n}{\lambda}\left(z^\lambda-1 - \lambda \log z\right) \right\} \approx \frac{e^{-(\log n)/2}}{(1-\xi)\sqrt{2\pi\lambda}}
\]
as $n \to \infty$ for any zero $z = z(n)$ of $p_{n-1}(r_n z)$ which tends to a point $\xi \neq 1$ on the limit curve. By following the steps above we again get an estimate like
\[
	z \approx \xi + \frac{\xi\log n}{2(1-\xi^\lambda)n}
\]
as $n \to \infty$. On the other hand, if $\re a < \re b$ then we instead get from \eqref{seccurvetwo_incurveasymp1} that
\[
	z^a \exp\!\left\{ \frac{n}{\lambda}\left(z^\lambda-1 - \lambda \log z\right) \right\} \approx \frac{A\zeta^{1-n} \lambda^{(a-b)/\lambda}}{(\zeta-z)\sqrt{2\pi\lambda }} \exp\!\left\{\left(\frac{b-a}{\lambda} - \frac{1}{2}\right)\log n \right\}
\]
as $n \to \infty$. If we set $z = \xi + \delta$ and assume that $\delta \to 0$ as $n \to \infty$ then the estimate for the left-hand side is the same as before, transforming the above into
\begin{align*}
&\xi^a \exp\!\left\{\frac{n}{\lambda} \left(\xi^\lambda - 1 - \lambda\log\xi\right) + \frac{\xi^\lambda-1}{\xi}n\delta \right\} \\
	&\qquad \approx \frac{A\zeta^{1-n} \lambda^{(a-b)/\lambda}}{(\zeta-\xi)\sqrt{2\pi\lambda }} \exp\!\left\{-\left(\frac{a-b}{\lambda} + \frac{1}{2}\right)\log n \right\}.
\end{align*}
To balance the possible growth or decay of the exponential factor on the right-hand side we need to take
\[
	\frac{\xi^\lambda-1}{\xi}n\delta \approx -\left(\frac{a-b}{\lambda} + \frac{1}{2}\right)\log n
\]
and hence
\[
	\delta \approx \left(a-b + \frac{\lambda}{2}\right) \frac{\xi\log n}{\lambda(1-\xi^\lambda) n}.
\]
There are still extra oscillations coming from the $\zeta^{-n}$ factor on the right-hand side which this scaling doesn't account for. We handle these in the statement of Theorem \ref{curscatwotheo_maintheo} by introducing extra periodicity in the higher-order corrections to this $\delta$ that we didn't need when we had $\re a > \re b$.

\section{One Direction of Maximal Exponential Growth}
\label{curscasec_onedir}

Let $a,b \in \C$, $0 < \lambda < \infty$, $0 < \theta < \pi$, and $\mu < 1$. We suppose that $f$ is an entire function with the asymptotic behavior
\begin{equation}
\label{curscaoneeq_fgrowth}
f(z) = \begin{cases}
	z^a (\log z)^b \exp\!\left(z^\lambda\right) \bigl[1+o(1)\bigr] & \text{for } |{\arg z}| \leq \theta, \\
	O\!\left(\exp(\mu |z|^\lambda)\right) & \text{for } |{\arg z}| > \theta
	\end{cases}			
\end{equation}
as $|z| \to \infty$, with each estimate holding uniformly in its sector. For this $f$, let
\[
	p_n(z) = \sum_{k=0}^{n} \frac{f^{(k)}(0)}{k!} z^k
\]
and define
\begin{equation}
\label{curscaoneeq_rndef}
	r_n = \left(\frac{n}{\lambda}\right)^{1/\lambda}.
\end{equation}

We showed in Chapter \ref{chap_limitcurves} that the limit points of the zeros of the scaled partial sums $p_{n-1}(r_n z)$ in the sector $|{\arg z}| < \theta$, $z \neq 0$, lie on the curve
\[
	S = \left\{ z \in \C : \left|z^\lambda \exp\!\left(1-z^\lambda\right)\right| = 1 \text{ and } |z| \leq 1 \right\}.
\]
If $\xi$ is a point of $S$ then
\[
	\re\!\left(\xi^\lambda - 1 - \lambda \log \xi\right) = 0.
\]

\begin{theorem}
\label{curscaonetheo_maintheo}
Let $\xi$ be a point of $S$ with $|{\arg \xi}| < \theta$, $\xi \neq 1$ and define
\[
	\tau = \im\!\left(\xi^\lambda - 1 - \lambda \log \xi\right).
\]
Define the sequence $\tau_n$ by the conditions
\[
	\frac{\tau n}{\lambda} \equiv \tau_n \pmod{2\pi}, \qquad -\pi < \tau_n \leq \pi
\]
and let
\[
	z_n(w) = \xi \left(1 + \frac{\log n}{2(1-\xi^\lambda)n} - \frac{w-i\tau_n}{(1-\xi^\lambda) n}\right).
\]
Then
\[
	\lim_{n \to \infty} \frac{p_{n-1}(r_n z_n(w))}{f(r_n z_n(w))} = 1 - \frac{e^{-w}}{\xi^a (1-\xi) \sqrt{2\pi\lambda}}
\]
uniformly for $w$ restricted to any compact subset of $\C$.
\end{theorem}

\begin{remark}
Theorem \ref{curscaonetheo_maintheo} gives us precise asymptotics for individual zeros of the scaled partial sums $p_{n-1}(r_n z)$ near a given point $\xi$ on the arcs of the curve $S$. Details of this are given in Section \ref{curscasec_width}, where we use that information to verify part (a) of the Modified Saff-Varga Width Conjecture (see Section \ref{introsec_width}) for this class of functions.
\end{remark}

\subsection{Definitions and Preliminaries}

We will repeat here several relevant definitions from Section \ref{seccurve1_defsandprelims}.

Let $\gamma$ be an admissible contour for the function
\begin{equation}
\label{curscaoneeq_phidef}
	\varphi(z) := \left.\left(z^\lambda - 1 - \lambda \log z\right)\right/\lambda
\end{equation}
and define
\begin{equation}
\label{curscaoneeq_fndef}
	F_n(z) = \frac{r_n^{-a} (\log r_n)^{-b}}{2\pi i} \int_{\gamma} \left(e^{1/\lambda}s\right)^{-n}\! f(r_n s) \frac{ds}{s-z}
\end{equation}
for $z \notin \gamma$, $z \neq 0$. Just as in Section \ref{seccurve1_defsandprelims},
\begin{equation}
\label{curscaoneeq_fnexplicit}
	F_n(z) = \frac{1}{r_n^a (\log r_n)^b (e^{1/\lambda} z)^n} \times \begin{cases}
		-p_{n-1}(r_n z) & \text{for } z \text{ outside } \gamma, \\
		f(r_n z) - p_{n-1}(r_n z) & \text{for } z \neq 0 \text{ inside } \gamma.
		\end{cases}
\end{equation}

\subsection{Proof of Theorem \ref{curscaonetheo_maintheo}}

Let $\gamma_\theta$ be the portion of $\gamma$ in the sector $|{\arg z}| \leq \theta$ and for $\epsilon > 0$ define $N_\epsilon$ to be the set of all points within a distance of $\epsilon$ of $\gamma_\theta$. It was shown in the proof of Theorem \ref{seccurve1_maintheo} (see equation \eqref{seccurve1_fnasymp}) that
\begin{equation}
\label{curscaoneeq_fnasymp}
	F_n(z) = \frac{1}{(1-z)\sqrt{2\pi\lambda n}} + o\!\left(n^{-1/2}\right)
\end{equation}
as $n \to \infty$ uniformly for $z \in \C \setminus N_\epsilon$ with $|{\arg z}| \leq \theta - \epsilon$ for any fixed $\epsilon > 0$.

Because $\re \varphi(\xi) = 0$ and because $\gamma$ is an admissible contour for the function $\varphi$, $\epsilon > 0$ can be taken small enough so that
\[
	\inf_{s \in \gamma} |\xi - s| > \epsilon.
\]
Consequently if $w$ is restricted to a compact subset of $\C$ then $z_n(w) \notin N_\epsilon$ and $|{\arg z_n(w)}| \leq \theta - \epsilon$ for all such $w$ if $n$ is large enough. It follows from \eqref{curscaoneeq_fnasymp} that
\begin{align*}
F_n(z_n(w)) &= \frac{1}{(1-z_n(w))\sqrt{2\pi\lambda n}} + o\!\left(n^{-1/2}\right) \nonumber \\
	&\sim \frac{1}{(1-\xi)\sqrt{2\pi\lambda n}}
\end{align*}
as $n \to \infty$ uniformly for $w$ restricted to any compact subset of $\C$. Then, since $z_n(w)$ is inside $\gamma$ for $n$ large enough, \eqref{curscaoneeq_fnexplicit} implies that
\begin{equation}
\label{curscaoneeq_fnasymp2}
	\frac{f(r_n z_n(w))}{r_n^a (\log r_n)^b (e^{1/\lambda} z_n(w))^n} \left(\frac{p_{n-1}(r_n z_n(w))}{f(r_n z_n(w))} - 1 \right) \sim - \frac{1}{(1-\xi)\sqrt{2\pi\lambda n}}
\end{equation}
as $n \to \infty$ uniformly for $w$ restricted to any compact subset of $\C$.

It follows from the asymptotic assumption on $f$ in \eqref{curscaoneeq_fgrowth} that
\begin{align*}
\frac{f(r_n z_n(w))}{r_n^a (\log r_n)^b} &\sim \xi^a e^{r_n^\lambda z_n(w)^\lambda} \\
	&= \xi^a \exp\!\left\{\frac{n\xi^\lambda}{\lambda} \left(1 + \frac{\log n}{2(1-\xi^\lambda)n} - \frac{w-i\tau_n}{(1-\xi^\lambda) n}\right)^\lambda \right\} \\
	&\sim \xi^a \exp\!\left\{\frac{n\xi^\lambda}{\lambda} \left(1 + \frac{\lambda \log n}{2(1-\xi^\lambda)n} - \frac{\lambda(w-i\tau_n)}{(1-\xi^\lambda) n}\right) \right\} \\
	&= \xi^a n^{\xi^\lambda/[2(1-\xi^\lambda)]} \exp\!\left\{\frac{n\xi^\lambda}{\lambda} - \frac{\xi^\lambda(w-i\tau_n)}{1-\xi^\lambda}\right\}
\end{align*}
as $n \to \infty$ uniformly for $w$ restricted to any compact subset of $\C$. Since
\begin{align*}
z_n(w)^n &= \xi^n \left(1 + \frac{\log n}{2(1-\xi^\lambda)n} - \frac{w-i\tau_n}{(1-\xi^\lambda) n}\right)^n \\
	&= \xi^n \exp\!\left\{n\log\!\left(1 + \frac{\log n}{2(1-\xi^\lambda)n} - \frac{w-i\tau_n}{(1-\xi^\lambda) n}\right)\right\} \\
	&\sim \xi^n \exp\!\left\{\frac{\log n}{2(1-\xi^\lambda)} - \frac{w-i\tau_n}{1-\xi^\lambda}\right\} \\
	&= \xi^n n^{1/[2(1-\xi^\lambda)]} \exp\!\left\{ - \frac{w-i\tau_n}{1-\xi^\lambda}\right\}
\end{align*}
it then follows that
\begin{align}
\frac{f(r_n z_n(w))}{r_n^a (\log r_n)^b (e^{1/\lambda} z_n(w))^n} &\sim \xi^a n^{-1/2} \exp\!\left\{ \frac{n(\xi^\lambda-1-\lambda\log \xi)}{\lambda} + w - i\tau_n\right\} \nonumber \\
	&= \xi^a n^{-1/2} \exp\!\left\{ \frac{i\tau n}{\lambda} + w - i\tau_n\right\} \nonumber \\
	&= \xi^a n^{-1/2} e^w
\label{curscaoneeq_fnfracasymp}
\end{align}
as $n \to \infty$ uniformly for $w$ restricted to any compact subset of $\C$.

Substituting \eqref{curscaoneeq_fnfracasymp} into \eqref{curscaoneeq_fnasymp2} yields the limit
\[
	\frac{p_{n-1}(r_n z_n(w))}{f(r_n z_n(w))} \longrightarrow 1 - \frac{e^{-w}}{\xi^a (1-\xi)\sqrt{2\pi\lambda}}
\]
as $n \to \infty$ uniformly for $w$ restricted to any compact subset of $\C$, which is exactly the limit we desire.

\section{Maximal Growth in Two Directions and Its Effect on the Scaling Limit}
\label{curscatwosec_sec}

Let $a,b,A \in \C$, $0 < \lambda < \infty$, $\mu < 1$, and $\zeta \in \C$ with $|\zeta| = 1$, $\zeta \neq 1$. Let $\theta \in (0,\pi)$ be small enough so that the sectors $|{\arg z}| \leq \theta$ and $|{\arg(z/\zeta)}| \leq \theta$ are disjoint.  We suppose that $f$ is an entire function with the asymptotic behavior
\begin{equation}
\label{curscatwoeq_fgrowth}
f(z) = \begin{cases}
	z^a \exp\!\left(z^\lambda\right) \bigl[1+o(1)\bigr] & \text{for } |{\arg z}| \leq \theta, \\
	A(z/\zeta)^b \exp\!\left((z/\zeta)^\lambda\right) \bigl[1+o(1)\bigr] & \text{for } |{\arg(z/\zeta)}| \leq \theta, \\
	O\!\left(\exp\!\left(\mu |z|^\lambda\right)\right) & \text{otherwise}
	\end{cases}			
\end{equation}
as $|z| \to \infty$, with each estimate holding uniformly in its sector. For this $f$, let
\[
	p_n(z) = \sum_{k=0}^{n} \frac{f^{(k)}(0)}{k!} z^k
\]
and define
\begin{equation}
\label{curscatwoeq_rndef}
	r_n = \left(\frac{n}{\lambda}\right)^{1/\lambda}.
\end{equation}

We showed in Chapter \ref{chap_limitcurves} that the limit points of the zeros of the scaled partial sums $p_{n-1}(r_n z)$ in the sector $|{\arg z}| < \theta$, $z \neq 0$, lie on the curve
\[
	S = \left\{ z \in \C : \left|z^\lambda \exp\!\left(1-z^\lambda\right)\right| = 1 \text{ and } |z| \leq 1 \right\}.
\]
If $\xi$ is a point of $S$ then
\[
	\re\!\left(\xi^\lambda - 1 - \lambda \log \xi\right) = 0.
\]

\begin{theorem}
\label{curscatwotheo_maintheo}
Let $\xi$ be a point of $S$ with $|{\arg \xi}| < \theta$, $\xi \neq 1$ and define
\[
	\tau = \im\!\left(\xi^\lambda - 1 - \lambda \log \xi\right).
\]
Define the sequence $\tau_n$ by the conditions
\[
	\frac{\tau n}{\lambda} \equiv \tau_n \pmod{2\pi}, \qquad -\pi < \tau_n \leq \pi
\]
and let
\[
	z_n^1(w) = \xi \left(1 + \frac{\log n}{2(1-\xi^\lambda)n} - \frac{w-i\tau_n}{(1-\xi^\lambda) n}\right).
\]
Define the sequence $\sigma_n$ by the conditions
\[
	n\arg\zeta \equiv \sigma_n \pmod{2\pi}, \qquad -\pi < \sigma_n \leq \pi
\]
and let
\[
	z_n^2(w) = \xi \left[ 1 + \left(a-b+\frac{\lambda}{2}\right)\frac{\log n}{\lambda (1-\xi^\lambda) n} - \frac{w - i\sigma_n - i\tau_n}{(1-\xi^\lambda) n} \right].
\]
If $\re a > \re b$ then
\[
	\lim_{n \to \infty} \frac{p_{n-1}(r_n z_n^1(w))}{f(r_n z_n^1(w))} = 1 - \frac{e^{-w}}{\xi^a (1-\xi) \sqrt{2\pi\lambda}},
\]
if $\re a < \re b$ then
\[
	\lim_{n \to \infty} \frac{p_{n-1}(r_n z_n^2(w))}{f(r_n z_n^2(w))} = 1 - \frac{A\zeta \lambda^{(a-b)/\lambda} e^{-w}}{\xi^a (\zeta-\xi)\sqrt{2\pi\lambda}},
\]
and if $\re a = \re b$ then
\[
	\frac{p_{n-1}(r_n z_n^1(w))}{f(r_n z_n^1(w))} = 1 - \left(\frac{1}{1-\xi} + \frac{A\zeta^{1-n} r_n^{b-a}}{\zeta-\xi} \right) \frac{e^{-w}}{\xi^a \sqrt{2\pi\lambda}} + o(1)
\]
as $n \to \infty$. All three limits are uniform with respect to $w$ as long as $w$ is restricted to a compact subset of $\C$.
\end{theorem}

\begin{remark}
Depending on the balance between $\re a$ and $\re b$ there are three possible forms of the scaling limit in this case, compared to only one when $f$ has a single direction of maximal growth as in Section \ref{curscasec_onedir}. Examples of these extra scaling limits are given in Chapter \ref{chap_applications} (in all sections except \ref{applicationssec_confluent}).
\end{remark}

\begin{remark}
Just as with Theorem \ref{curscaonetheo_maintheo} in the case of one direction of maximal growth, Theorem \ref{curscatwotheo_maintheo} gives us precise asymptotics for individual zeros of the scaled partial sums $p_{n-1}(r_n z)$ near a given point $\xi$ on the arcs of the curve $S$ in the case of two directions of maximal growth. Details of this are given in Section \ref{curscasec_width}, where we use that information to verify part (a) of the Modified Saff-Varga Width Conjecture (see Section \ref{introsec_width}) for this class of functions.
\end{remark}

\subsection{Definitions and Preliminaries}

We will repeat here several relevant definitions from Section \ref{seccurvetwo_defsandprelims}.

Let $\gamma$ be an admissible contour for the function
\begin{equation}
\label{curscatwoeq_phidef}
	\varphi(z) := \left.\left(z^\lambda - 1 - \lambda \log z\right)\right/\lambda
\end{equation}
and define
\begin{equation}
\label{curscatwoeq_fndef}
	F_n(z) = \frac{r_n^{-a}}{2\pi i} \int_{\gamma} \left(e^{1/\lambda}s\right)^{-n}\! f(r_n s) \frac{ds}{s-z}
\end{equation}
for $z \notin \gamma$, $z \neq 0$. Just as in Section \ref{seccurvetwo_defsandprelims},
\begin{equation}
\label{curscatwoeq_fnexplicit}
	F_n(z) = \frac{1}{r_n^a (e^{1/\lambda} z)^n} \times \begin{cases}
		-p_{n-1}(r_n z) & \text{for } z \text{ outside } \gamma, \\
		f(r_n z) - p_{n-1}(r_n z) & \text{for } z \neq 0 \text{ inside } \gamma.
		\end{cases}
\end{equation}

\subsection{Proof of Theorem \ref{curscatwotheo_maintheo}}

Let $\gamma_\theta$ be the portion of $\gamma$ in the sector $|{\arg z}| \leq \theta$ and for $\epsilon > 0$ define $N_\epsilon$ to be the set of all points within a distance of $\epsilon$ of $\gamma_\theta$. It was shown in the proof of Theorem \ref{seccurvetwo_maintheo} (see equation \eqref{seccurvetwo_fnasymp}) that
\begin{equation}
\label{curscatwoeq_fnasymp}
	F_n(z) = \frac{1}{(1-z)\sqrt{2\pi\lambda n}} + \frac{A\zeta^{1-n} r_n^{b-a}}{(\zeta-z)\sqrt{2\pi\lambda n}} + o\!\left(n^{-1/2}\right) + o\!\left(r_n^{b-a} n^{-1/2}\right)
\end{equation}
as $n \to \infty$ uniformly for $z \in \C \setminus N_\epsilon$ with $|{\arg z}| \leq \theta - \epsilon$ for any fixed $\epsilon > 0$.

Fix $j \in \{1,2\}$. Because $\re \varphi(\xi) = 0$ and because $\gamma$ is an admissible contour for the function $\varphi$, $\epsilon > 0$ can be taken small enough so that
\[
	\inf_{s \in \gamma} |\xi - s| > \epsilon.
\]
Consequently if $w$ is restricted to a compact subset of $\C$ then $z_n^j(w) \notin N_\epsilon$ and $|{\arg z_n^j(w)}| \leq \theta - \epsilon$ for all such $w$ if $n$ is large enough. It follows from \eqref{curscatwoeq_fnasymp} that
\begin{align*}
F_n(z_n^j(w)) &= \frac{1}{(1-z_n^j(w))\sqrt{2\pi\lambda n}} + \frac{A\zeta^{1-n} r_n^{b-a}}{(\zeta-z_n^j(w))\sqrt{2\pi\lambda n}} + o\!\left(n^{-1/2}\right) + o\!\left(r_n^{b-a} n^{-1/2}\right) \nonumber \\
	&= \frac{1}{(1-\xi)\sqrt{2\pi\lambda n}} + \frac{A\zeta^{1-n} r_n^{b-a}}{(\zeta-\xi)\sqrt{2\pi\lambda n}} + o\!\left(n^{-1/2}\right) + o\!\left(r_n^{b-a} n^{-1/2}\right)
\end{align*}
as $n \to \infty$ uniformly for $w$ restricted to any compact subset of $\C$. Then, since $z_n^j(w)$ is inside $\gamma$ for $n$ large enough, \eqref{curscatwoeq_fnexplicit} implies that
\begin{align}
&\frac{f(r_n z_n^j(w))}{r_n^a (e^{1/\lambda} z_n^j(w))^n} \left(\frac{p_{n-1}(r_n z_n^j(w))}{f(r_n z_n^j(w))} - 1 \right) \nonumber \\
	&\qquad = -\frac{1}{(1-\xi)\sqrt{2\pi\lambda n}} - \frac{A\zeta^{1-n} r_n^{b-a}}{(\zeta-\xi)\sqrt{2\pi\lambda n}} + o\!\left(n^{-1/2}\right) + o\!\left(r_n^{b-a} n^{-1/2}\right)
\label{curscatwoeq_fnasymp2}
\end{align}
as $n \to \infty$ uniformly for $w$ restricted to any compact subset of $\C$.

Suppose that $\re a > \re b$. Then from \eqref{curscatwoeq_fnasymp2} it follows that
\[
	\frac{f(r_n z_n^1(w))}{r_n^a (e^{1/\lambda} z_n^1(w))^n} \left(\frac{p_{n-1}(r_n z_n^1(w))}{f(r_n z_n^1(w))} - 1 \right) \sim -\frac{1}{(1-\xi)\sqrt{2\pi\lambda n}}
\]
as $n \to \infty$ uniformly for $w$ restricted to any compact subset of $\C$. This is analogous to equation \eqref{curscaoneeq_fnasymp2} from the proof of Theorem \ref{curscaonetheo_maintheo}, and by proceeding as in that proof it can be shown that
\begin{equation}
\label{curscatwoeq_bleqalim}
	\frac{p_{n-1}(r_n z_n^1(w))}{f(r_n z_n^1(w))} \longrightarrow 1 - \frac{e^{-w}}{\xi^a (1-\xi)\sqrt{2\pi\lambda}}
\end{equation}
as $n \to \infty$ uniformly for $w$ restricted to any compact subset of $\C$. This is the first desired limit in Theorem \ref{curscatwotheo_maintheo}.

Now suppose $\re a < \re b$. From \eqref{curscatwoeq_fnasymp2} it follows that
\begin{equation}
\label{curscaoneeq_fnasymp3}
	\frac{f(r_n z_n^2(w))}{r_n^a (e^{1/\lambda} z_n^2(w))^n} \left(\frac{p_{n-1}(r_n z_n^2(w))}{f(r_n z_n^2(w))} - 1 \right) \sim - \frac{A\zeta^{1-n} r_n^{b-a}}{(\zeta-\xi)\sqrt{2\pi\lambda n}}
\end{equation}
as $n \to \infty$ uniformly for $w$ restricted to a compact subset of $\C$.

The asymptotic assumption on $f$ in \eqref{curscatwoeq_fgrowth} implies that
\begin{align*}
\frac{f(r_n z_n^2(w))}{r_n^a} &\sim \xi^a e^{r_n^\lambda z_n^2(w)^\lambda} \\
	&= \xi^a \exp\!\left\{ \frac{n\xi^\lambda}{\lambda} \left[1 + \left(a-b+\frac{\lambda}{2}\right)\frac{\log n}{\lambda (1-\xi^\lambda) n} - \frac{w - i\sigma_n - i\tau_n}{(1-\xi^\lambda) n} \right]^\lambda\right\} \\
	&\sim \xi^a \exp\!\left\{ \frac{n\xi^\lambda}{\lambda} \left[1 + \left(a-b+\frac{\lambda}{2}\right)\frac{\log n}{(1-\xi^\lambda) n} - \frac{\lambda(w - i\sigma_n - i\tau_n)}{(1-\xi^\lambda) n} \right]\right\} \\
	&= \xi^a \exp\!\left\{ \frac{n\xi^\lambda}{\lambda} + \left(a-b+\frac{\lambda}{2}\right)\frac{\xi^\lambda \log n}{\lambda(1-\xi^\lambda)} - \frac{\xi^\lambda(w - i\sigma_n - i\tau_n)}{1-\xi^\lambda} \right\}
\end{align*}
as $n \to \infty$ uniformly for $w$ restricted to any compact subset of $\C$. Since
\begin{align*}
z_n^2(w)^n &= \xi^n \left[1 + \left(a-b+\frac{\lambda}{2}\right)\frac{\log n}{\lambda (1-\xi^\lambda) n} - \frac{w - i\sigma_n - i\tau_n}{(1-\xi^\lambda) n} \right]^n \\
	&= \xi^n \exp\!\left\{ n \log\!\left[1 + \left(a-b+\frac{\lambda}{2}\right)\frac{\log n}{\lambda (1-\xi^\lambda) n} - \frac{w - i\sigma_n - i\tau_n}{(1-\xi^\lambda) n} \right]\right\} \\
	&\sim \xi^n \exp\!\left\{ n \left[\left(a-b+\frac{\lambda}{2}\right)\frac{\log n}{\lambda (1-\xi^\lambda) n} - \frac{w - i\sigma_n - i\tau_n}{(1-\xi^\lambda) n} \right]\right\} \\
	&= \xi^n \exp\!\left\{ \left(a-b+\frac{\lambda}{2}\right)\frac{\log n}{\lambda (1-\xi^\lambda)} - \frac{w - i\sigma_n - i\tau_n}{1-\xi^\lambda} \right\}
\end{align*}
it then follows that
\begin{align}
\frac{f(r_n z_n^2(w))}{r_n^a (e^{1/\lambda} z_n^2(w))^n} &\sim \xi^a n^{(b-a)/\lambda-1/2} \exp\!\left\{ \frac{n}{\lambda}\left(\xi^\lambda - 1 - \lambda\log\xi\right) + w - i\sigma_n - i\tau_n\right\} \nonumber \\
	&= \xi^a n^{(b-a)/\lambda-1/2} \exp\!\left\{ \frac{i\tau n}{\lambda} + w - i\sigma_n - i\tau_n\right\} \nonumber \\
	&= \xi^a n^{(b-a)/\lambda-1/2} e^{w - i\sigma_n}
\label{curscatwoeq_fnfracasymp}
\end{align}
as $n \to \infty$ uniformly for $w$ restricted to a compact subset of $\C$.

Substituting \eqref{curscatwoeq_fnfracasymp} into \eqref{curscaoneeq_fnasymp3} yields the limit
\[
	\frac{p_{n-1}(r_n z_n^2(w))}{f(r_n z_n^2(w))} \longrightarrow 1 - \frac{A\zeta^{1-n}\lambda^{(a-b)/\lambda} e^{i\sigma_n-w}}{\xi^a (\zeta-\xi)\sqrt{2\pi\lambda}} = 1 - \frac{A\zeta \lambda^{(a-b)/\lambda} e^{-w}}{\xi^a (\zeta-\xi)\sqrt{2\pi\lambda}}
\]
as $n \to \infty$ uniformly for $w$ restricted to any compact subset of $\C$. This is the second desired limit in Theorem \ref{curscatwotheo_maintheo}.

Finally suppose $\re a = \re b$. In this case, after setting $j=1$ equation \eqref{curscatwoeq_fnasymp2} becomes
\begin{align*}
&\frac{f(r_n z_n^1(w))}{r_n^a (e^{1/\lambda} z_n^1(w))^n} \left(\frac{p_{n-1}(r_n z_n^1(w))}{f(r_n z_n^1(w))} - 1 \right) \\
	&\qquad = -\left(\frac{1}{1-\xi} + \frac{A\zeta^{1-n} r_n^{b-a}}{\zeta-\xi} \right) \frac{1}{\xi^a \sqrt{2\pi\lambda n}} + o\!\left(n^{-1/2}\right)
\end{align*}
as $n \to \infty$ uniformly for $w$ restricted to any compact subset of $\C$. By following the same method as in the proof of Theorem \ref{curscaonetheo_maintheo} to obtain equation \eqref{curscaoneeq_fnfracasymp} it can be shown that
\[
	\frac{f(r_n z_n^1(w))}{r_n^a (e^{1/\lambda} z_n^1(w))^n} \sim \xi^a n^{-1/2} e^w
\]
as $n \to \infty$ uniformly for $w$ restricted to any compact subset of $\C$. Substituting this into the above yields the asymptotic
\[
	\frac{p_{n-1}(r_n z_n^1(w))}{f(r_n z_n^1(w))} = 1 - \left(\frac{1}{1-\xi} + \frac{A\zeta^{1-n} r_n^{b-a}}{\zeta-\xi} \right) \frac{e^{-w}}{\xi^a \sqrt{2\pi\lambda}} + o(1)
\]
as $n \to \infty$ uniformly for $w$ restricted to any compact subset of $\C$, which is the last desired item in Theorem \ref{curscatwotheo_maintheo}.

\section{Generalization to More Directions of Maximal Growth}
\label{curscamanysec_sec}

Let $a,b_1,\ldots,b_m,A_1,\ldots,A_m \in \C$, $0 < \lambda < \infty$, $\mu < 1$, and $\zeta_1,\ldots,\zeta_m \in \C$ with $|\zeta_k| = 1$, $\zeta_k \neq 1$ for all $k = 1,\ldots,m$ and $\zeta_j \neq \zeta_k$ for $j \neq k$. Let $\theta \in (0,\pi)$ be small enough so that all of the sectors $|{\arg z}| \leq \theta$, $|{\arg(z/\zeta_k)}| \leq \theta$, $k=1,\ldots,m$, are disjoint.  We suppose that $f$ is an entire function with the asymptotic behavior
\begin{equation}
\label{curscamanyeq_fgrowth}
f(z) = \begin{cases}
	z^a \exp\!\left(z^\lambda\right) \bigl[1+o(1)\bigr] & \text{for } |{\arg z}| \leq \theta, \\
	A_1(z/\zeta_1)^{b_1} \exp\!\left((z/\zeta_1)^\lambda\right) \bigl[1+o(1)\bigr] & \text{for } |{\arg(z/\zeta_1)}| \leq \theta, \\
	\qquad \vdots \\
	A_m(z/\zeta_m)^{b_m} \exp\!\left((z/\zeta_m)^\lambda\right) \bigl[1+o(1)\bigr] & \text{for } |{\arg(z/\zeta_m)}| \leq \theta, \\
	O\!\left(\exp\!\left(\mu |z|^\lambda\right)\right) & \text{otherwise}
	\end{cases}			
\end{equation}
as $|z| \to \infty$, with each estimate holding uniformly in its sector. For this $f$, let, $p_n(z)$, $r_n$, and $F_n(z)$ be defined as in Section \ref{curscatwosec_sec}. For convenience of notation, define $b_0 = a$, $A_0 = 1$, and $\zeta_0 = 1$.

As in Section \ref{seccurvemany_sec} we can derive an analogue to equations \eqref{seccurve1_fnasymp} and \eqref{seccurvetwo_fnasymp}, specifically
\begin{equation}
\label{curscamanyeq_fnapprox}
	F_n(z) = \frac{1}{\sqrt{2\pi\lambda n}} \sum_{k=0}^{m} \left[\frac{A_k \zeta_{k}^{1-n} r_n^{b_k-a}}{\zeta_k-z} + o\!\left(r_n^{b_k-a}n^{-1/2}\right)\right]
\end{equation}
as $n \to \infty$ uniformly with respect to $z$ as long as $z$ remains in any sector $|{\arg(z/\zeta_k)}| \leq \theta - \epsilon$ with $\epsilon > 0$, $k=0,1,\ldots,m$, and remains bounded away from $\gamma$.

The main difficulty in extending Theorem \ref{curscatwotheo_maintheo} to arbitrary $m \geq 3$ lies in the necessity of modifying the quantity $z_n(w)$ that appears in the scaling limit when the real parts of the $b_k$ balance in different ways. In the simplest case there is a $j$ such that $\re b_j > \re b_k$ for all $k \neq j$, which allows us to simplify \eqref{seccurvemany_fnapprox} into
\[
	F_n(z) \sim \frac{A_j \zeta_{j}^{1-n} r_n^{b_j-a}}{(\zeta_j-z)\sqrt{2\pi\lambda n}},
\]
after which the analysis proceeds just as in the relevant parts of the proof of Theorem \ref{curscatwotheo_maintheo}. This yields the following result.

\begin{theorem}
\label{curscamanytheo_agebb}
Let $m \geq 1$ and suppose that there is a $j \in \{0,1,\ldots,m\}$ such that $\re b_j > \re b_k$ for all $k \neq j$. Let $\xi$ be a point of
\[
	S = \left\{ z \in \C : \left|z^\lambda \exp\!\left(1-z^\lambda\right)\right| = 1 \text{ and } |z| \leq 1 \right\}
\]
with $|{\arg \xi}| < \theta$, $\xi \neq 1$ and define
\[
	\tau = \im\!\left(\xi^\lambda - 1 - \lambda \log \xi\right).
\]
Define the sequences $\tau_n$ and $\sigma_n$ by the conditions
\[
	\frac{\tau n}{\lambda} \equiv \tau_n \pmod{2\pi}, \qquad -\pi < \tau_n \leq \pi,
\]
\[
	n\arg\zeta_j \equiv \sigma_n \pmod{2\pi}, \qquad -\pi < \sigma_n \leq \pi
\]
and let
\[
	z_n(w) = \xi \left[ 1 + \left(a-b_j+\frac{\lambda}{2}\right)\frac{\log n}{\lambda (1-\xi^\lambda) n} - \frac{w - i\sigma_n - i\tau_n}{(1-\xi^\lambda) n} \right].
\]
Then
\[
	\lim_{n \to \infty} \frac{p_{n-1}(r_n z_n(w))}{f(r_n z_n(w))} = 1 - \frac{A_j\zeta_j \lambda^{(a-b_j)/\lambda} e^{-w}}{\xi^a (\zeta_j-\xi)\sqrt{2\pi\lambda}}
\]
uniformly for $w$ restricted to any compact subset of $\C$.
\end{theorem}

Now we will consider what happens when $\re a$ is not strictly larger than all of the other $\re b_k$. Because the analysis for general $m$ would be very complicated, we will restrict ourselves to the case $m=2$. Up to relabeling the $b_k$ there are three cases not included in the theorem above: (i) $\re a = \re b_1 > \re b_2$, (ii) $\re a = \re b_1 = \re b_2$, and (iii) $\re b_1 = \re b_2 > \re a$.

\textbf{Case (i): $\re a = \re b_1 > \re b_2$.} In this case \eqref{seccurvemany_fnapprox} becomes
\[
	F_n(z) = \frac{1}{(1-z)\sqrt{2\pi\lambda n}} + \frac{A_1\zeta_1^{1-n} r_n^{b_1-a}}{(\zeta_1-z)\sqrt{2\pi\lambda n}} + o\!\left(n^{-1/2}\right).
\]
If we define
\[
	z_n(w) = \xi \left(1 + \frac{\log n}{2(1-\xi^\lambda)n} - \frac{w-i\tau_n}{(1-\xi^\lambda) n}\right)
\]
then
\[
	F_n(z_n(w)) = \left(\frac{1}{1-\xi} + \frac{A_1\zeta_1^{1-n} r_n^{b_1-a}}{\zeta_1-\xi} \right) \frac{1}{\sqrt{2\pi\lambda n}} + o\!\left(n^{-1/2}\right)
\]
as $n \to \infty$ uniformly for $w$ restricted to any compact subset of $\C$. The remainder of the analysis proceeds just as in the analogous part of the proof of Theorem \ref{curscatwotheo_maintheo} and produces the following result.

\begin{theorem}
\label{curscamanytheo_aeqbgeb}
Let $m = 2$ and $\re a = \re b_1 > \re b_2$. Let $\xi$ be a point of $S$ with $|{\arg \xi}| < \theta$, $\xi \neq 1$ and define $\tau$ and $\tau_n$ as in Theorem \ref{curscamanytheo_agebb}. Let
\[
	z_n(w) = \xi \left(1 + \frac{\log n}{2(1-\xi^\lambda)n} - \frac{w-i\tau_n}{(1-\xi^\lambda) n}\right).
\]
Then
\[
	\frac{p_{n-1}(r_n z_n(w))}{f(r_n z_n(w))} = 1 - \left(\frac{1}{1-\xi} + \frac{A_1\zeta_1^{1-n} r_n^{b_1-a}}{\zeta_1-\xi} \right) \frac{e^{-w}}{\xi^a \sqrt{2\pi\lambda}} + o(1)
\]
as $n \to \infty$ uniformly for $w$ restricted to any compact subset of $\C$.
\end{theorem}

\textbf{Case (ii): $\re a = \re b_1 = \re b_2$.} This case is very similar to the previous one, the only difference being that we keep all three terms in the sum in \eqref{seccurvemany_fnapprox} instead of only two. Doing so yields the following result.

\begin{theorem}
\label{curscamanytheo_aeqbeqb}
Let $m = 2$ and $\re a = \re b_1 = \re b_2$. Let $\xi$ be a point of $S$ with $|{\arg \xi}| < \theta$, $\xi \neq 1$ and define $\tau$ and $\tau_n$ as in Theorem \ref{curscamanytheo_agebb} and $z_n(w)$ as in Theorem \ref{curscamanytheo_aeqbgeb}. Then
\[
	\frac{p_{n-1}(r_n z_n(w))}{f(r_n z_n(w))} = 1 - \left(\frac{1}{1-\xi} + \frac{A_1\zeta_1^{1-n} r_n^{b_1-a}}{\zeta_1-\xi} + \frac{A_2\zeta_2^{1-n} r_n^{b_2-a}}{\zeta_2-\xi} \right) \frac{e^{-w}}{\xi^a \sqrt{2\pi\lambda}} + o(1)
\]
as $n \to \infty$ uniformly for $w$ restricted to any compact subset of $\C$.
\end{theorem}

\textbf{Case (iii): $\re b_1 = \re b_2 > \re a$.} In this case we retain both the $b_1$ term and the $b_2$ term in \eqref{seccurvemany_fnapprox}, i.e.
\[
	F_n(z) = \left( \frac{A_1}{\zeta_1-z} + \frac{A_2(\zeta_2/\zeta_1)^{1-n} r_n^{b_2-b_1}}{\zeta_2-z} \right) \frac{\zeta_1^{1-n} r_n^{b_1-a}}{\sqrt{2\pi\lambda n}} + o\!\left(r_n^{b_1-a} n^{-1/2}\right).
\]

Unfortunately we can't easily cancel the oscillations coming from the factor in parentheses by adding extra periodicity to $z_n(w)$ like before. We will again have to settle for an asymptotic rather than a proper limit.

\begin{theorem}
\label{curscamanytheo_beqbgea}
Let $m=2$ and $\re b_1 = \re b_2 > \re a$. Let $\xi$ be a point of $S$ with $|{\arg \xi}| < \theta$, $\xi \neq 1$ and define $\tau$ and $\tau_n$ as in Theorem \ref{curscamanytheo_agebb}. Define the sequence $\sigma_n$ by the conditions
\[
	n\arg\zeta_1 \equiv \sigma_n \pmod{2\pi}, \qquad -\pi < \sigma_n \leq \pi
\]
and let
\[
	z_n(w) = \xi \left[ 1 + \left(a-b_1+\frac{\lambda}{2}\right)\frac{\log n}{\lambda (1-\xi^\lambda) n} - \frac{w - i\sigma_n - i\tau_n}{(1-\xi^\lambda) n} \right].
\]
Then
\[
	\frac{p_{n-1}(r_n z_n(w))}{f(r_n z_n(w))} = 1 - \left( \frac{A_1}{\zeta_1-\xi} + \frac{A_2(\zeta_2/\zeta_1)^{1-n} r_n^{b_2-b_1}}{\zeta_2-\xi} \right)\frac{\zeta_1 \lambda^{(a-b_1)/\lambda} e^{-w}}{\xi^a \sqrt{2\pi\lambda}} + o(1)
\]
as $n \to \infty$ uniformly for $w$ restricted to any compact subset of $\C$.
\end{theorem}

\section{Verification of the Modified Saff-Varga Width Conjecture in the Sector $0 < |{\arg z}| < \theta$}
\label{curscasec_width}

The theorems in this chapter allow us to verify part (a) of the Modified Saff-Varga Width Conjecture (see Section \ref{introsec_width}) for the function $f$ in the sector $0 < |{\arg z}| < \theta$.

When $f$ has a single direction of maximal exponential growth, Theorem \ref{curscaonetheo_maintheo} and Hurwitz's theorem imply that if $\xi$ is any point of the curve
\[
	S = \left\{ z \in \C : \left|z^\lambda \exp\!\left(1-z^\lambda\right)\right| = 1 \text{ and } |z| \leq 1 \right\}
\]
with $0 < |{\arg \xi}| < \theta$ and if $w_0$ is any solution to the equation
\begin{equation}
\label{curscawidtheq_weq}
	\xi^a (1-\xi) \sqrt{2\pi\lambda} = e^{-w}
\end{equation}
then $p_{n-1}(z)$ has a zero $z_0$ satisfying
\begin{equation}
\label{curscawidtheq_zeq}
	z_0 = r_n\xi \left(1 + \frac{\log n}{2(1-\xi^\lambda)n} - \frac{w_0-i\tau_n}{(1-\xi^\lambda) n} + o\!\left(n^{-1}\right)\right)
\end{equation}
as $n \to \infty$.

Fix $\phi \in (-\theta,\theta)$ with $\phi \neq 0$. There is a unique $\xi \in S$ such that $\xi = |\xi| e^{i\phi}$. Set $\rho_n = |\xi| r_n$. Then for the zero $z_0$ above,
\begin{align*}
z_0 - \rho_n e^{i\phi} &= z_0 - r_n\xi \\
	&\sim r_n\xi \frac{\log n}{2(1-\xi^\lambda)n} \\
	&= \rho_n e^{i\phi} \frac{\log n}{2(1-\xi^\lambda)n}
\end{align*}
as $n \to \infty$. It follows that, for any $\epsilon > 0$, $z_0$ will lie inside the disk
\[
	\left| z - \rho_n e^{i\phi} \right| \leq \rho_n n^{-1+\epsilon}
\]
for $n$ large enough. As equation \eqref{curscawidtheq_weq} has infinitely many solutions, the number of zeros of $p_{n-1}(z)$ in any such disk will tend to infinity as $n \to \infty$. Since $\rho_n = \Theta(n^{1/\lambda}) = O(n^{2/\lambda})$ this completes the verification of part (a) of the Modified Saff-Varga Width Conjecture for the sector $0 < |{\arg z}| < \theta$.

When $f$ has two directions of maximal exponential growth and $\re a \neq \re b$ then the situation is very similar to the one above. Theorem \ref{curscatwotheo_maintheo} implies that for any $\xi \in S$ with $0 < |{\arg \xi}| < \theta$ there is an equation of the form
\[
	u = e^w
\]
for some $u = u(\xi) \in \C$ such that if $w_0$ is any solution of the equation then $p_{n-1}(z)$ has a zero $z_0$ of the form
\[
	z_0 = r_n\xi + Cr_n\xi \frac{\log n}{n} + h_n(w_0),
\]
where $C \in \C$ and $h_n$ is a function satisfying $h_n(z) = O(n^{-1})$ as $n \to \infty$ for any fixed $z \in \C$. Part (b) of the Modified Saff-Varga Width Conjecture follows just as above.

When $f$ has two directions of maximal exponential growth and $\re a = \re b$ then Theorem \ref{curscatwotheo_maintheo} implies that for any $\xi \in S$ with $0 < |{\arg \xi}| < \theta$ it is possible to find a constant $D \in \C$ with $D \neq 0$ and a subsequence $M$ such that
\[
	\lim_{m \in M} \frac{p_{m-1}(r_m z_m^1(w))}{f(r_m z_m^1(w))} = 1 - De^{-w}
\]
uniformly on compact subsets of the $w$-plane. So if $w_0$ is any solution of the equation
\[
	1 = De^{-w}
\]
then by Hurwitz's theorem $p_{m-1}(z)$ has a zero $z_0$ of the form
\[
	z_0 = r_m\xi \left(1 + \frac{\log m}{2(1-\xi^\lambda)m} - \frac{w_0-i\tau_m}{(1-\xi^\lambda) m} + o\!\left(m^{-1}\right)\right)
\]
as $m \to \infty$ with $m \in M$. The rest of the verification of the Modified Saff-Varga Width Conjecture proceeds just as above, though with the indices restricted to the subsequence $M$ (as allowed in the Conjecture).

By using Theorems \ref{curscamanytheo_agebb}, \ref{curscamanytheo_aeqbgeb}, \ref{curscamanytheo_aeqbeqb}, and \ref{curscamanytheo_beqbgea} we can verify the conjecture for functions with three directions of maximal exponential growth, and though we haven't obtained any explicit results in the case that $f$ has more than three directions of maximal exponential growth the conjecture can be verified using a similar method.

%% file: chap_cornerscaling.tex
\graphicspath{{images/chap_posfinite/}}

\chapter{Scaling Limits at the Corner of the Limit Curve}
\label{chap_posfinite}

In this chapter we aim to study the zeros of the scaled partial sums $p_n(r_n z)$ which approach the corner of the limit curve
\[
	S = \left\{ z \in \C : \left|z^\lambda \exp\!\left(1-z^\lambda\right)\right| = 1 \text{ and } |z| \leq 1 \right\}
\]
located at $z=1$. To this end we will calculate a certain limit of the partial sums depending on an argument which follows the zeros as they approach this corner.

The results in Sections \ref{seconeexp}, \ref{secunbalcauchy_sec}, and \ref{corscasec_many} can be seen as generalizations of Theorem \ref{introthm_esvcornerscaling} which was obtained by Edrei, Saff, and Varga in their monograph \cite{esv:sections}.

\section{One Direction of Maximal Exponential Growth}
\label{seconeexp}

Let $a,b \in \C$, $0 < \lambda < \infty$, $0 < \theta < \pi$, and $\mu < 1$. We suppose that $f$ is an entire function with the asymptotic behavior
\begin{equation}
\label{seconeexp_fgrowth}
f(z) = \begin{cases}
	z^a (\log z)^b \exp\!\left(z^\lambda\right)\bigl[1+o(1)\bigr] & \text{for } |{\arg z}| \leq \theta, \\
	O\!\left(\exp(\mu |z|^\lambda)\right) & \text{for } |{\arg z}| > \theta
	\end{cases}			
\end{equation}
as $|z| \to \infty$, with each estimate holding uniformly in its sector. For this $f$, let
\[
	p_n(z) = \sum_{k=0}^{n} \frac{f^{(k)}(0)}{k!} z^k
\]
and define
\begin{equation}
\label{seconeexp_rndef}
	r_n = \left(\frac{n}{\lambda}\right)^{1/\lambda}.
\end{equation}

\begin{theorem}
\label{seconeexp_maintheo}
\[
	\lim_{n \to \infty} \frac{p_{n-1}(r_n(1+w/\sqrt{n}))}{f(r_n(1+w/\sqrt{n}))} = \frac{1}{2} \erfc\!\left(w\sqrt{\lambda/2}\,\right)
\]
uniformly for $w$ restricted to any compact subset of $\re w < 0$.
\end{theorem}

The function $\erfc$ in the theorem statement above is known as the complementary error function and is defined by
\begin{equation}
\label{seconeexp_erfc}
	\erfc(z) = \frac{2}{\sqrt{\pi}} \int_z^\infty e^{-s^2}\,ds,
\end{equation}
where the contour of integration is the horizontal line starting at $s=z$ and extending to the right to $s = z+\infty$. Information about the zeros of this function can be found in \cite{fettis:erfczeros}.

\begin{remark}
Theorem \ref{seconeexp_maintheo} gives us precise asymptotics for individual zeros of the scaled partial sums $p_{n-1}(r_n z)$ near the corner of the limit curve $S$ located at $z=1$. Details of this are given in Section \ref{corscasec_width}, where we use that information to verify part (b) of the Modified Saff-Varga Width Conjecture (see Section \ref{introsec_width}) for this class of functions.
\end{remark}

\subsection{Definitions and Preliminaries}
\label{seconeexpsub_defsandprelims}

We will repeat here several relevant definitions from Section \ref{seccurve1_defsandprelims}.

Let $\gamma$ be an admissible contour for the function
\begin{equation}
\label{seconeexp_phidef}
	\varphi(z) := \left.\left(z^\lambda - 1 - \lambda \log z\right)\right/\lambda
\end{equation}
and define
\begin{equation}
\label{seconeexp_fdef}
	F_n(z) = \frac{r_n^{-a} (\log r_n)^{-b}}{2\pi i} \int_{\gamma} \left(e^{1/\lambda}s\right)^{-n}\! f(r_n s) \frac{ds}{s-z}
\end{equation}
for $z \notin \gamma$, $z \neq 0$. Just as in Section \ref{seccurve1_defsandprelims},
\begin{equation}
\label{seconeexp_fnexplicit}
	F_n(z) = \frac{1}{r_n^a (\log r_n)^b (e^{1/\lambda} z)^n} \times \begin{cases}
		-p_{n-1}(r_n z) & \text{for } z \text{ outside } \gamma, \\
		f(r_n z) - p_{n-1}(r_n z) & \text{for } z \neq 0 \text{ inside } \gamma.
		\end{cases}
\end{equation}
In particular we have
\[
	F_n^+(z) = F_n^-(z) + \frac{f(r_n z)}{r_n^a (\log r_n)^b (e^{1/\lambda}z)^n}, \qquad z \in \gamma,
\]
where $F_n^+$ (resp.\ $F_n^-$) refers to the continuous extension of $F_n$ from inside (resp.\ outside) $\gamma$ onto $\gamma$.

Note that
\begin{align}
\frac{f(r_n z)}{r_n^a (\log r_n)^b (e^{1/\lambda}z)^n} &\sim z^a \left(z^\lambda e^{1-z^\lambda}\right)^{-n/\lambda} \nonumber \\
&= z^a e^{n \varphi(z)} \label{seconeexp_integrandasymp}
\end{align}
as $n \to \infty$ for $|{\arg z}| \leq \theta$. Additionally, a straightforward calculation shows that $\varphi(1) = \varphi'(1) = 0$ and $\varphi''(1) = \lambda$, so
\[
	\varphi(s) = \frac{\lambda}{2}(s-1)^2 + O\!\left((s-1)^3\right)
\]
in a neighborhood of $s=1$.  The inverse function theorem ensures the existence of a neighborhood $V$ of the origin, a neighborhood $U \subset U_\gamma$ of $s=1$, and a biholomorphic map $\psi \colon V \to U$ which satisfies
\[
	(\varphi \circ \psi)(x) = x^2
\]
for $x \in V$. This function $\psi$ maps a segment of the imaginary axis onto the path of steepest descent of the function $\re \varphi(z)$ going through $z=1$ with $\psi(0) = 1$ and we make the choice that $\psi'(0) = \sqrt{2/\lambda}$.

The following definitions are unique to this chapter.

Let $\gamma_\theta = \gamma \cap \{z \in \C : |{\arg z}| \leq \theta\}$ and define
\begin{equation}
\label{seconeexp_gdef}
	G_n(z) = \frac{1}{2\pi i}\int_{\gamma_\theta} e^{n \varphi(s)} \frac{ds}{s-z},
\end{equation}
where $\varphi$ is as in \eqref{seconeexp_phidef}.  Plemelj's formula implies that
\[
	G_n^+(z) = G_n^-(z) + e^{n\varphi(z)}, \qquad z \in \gamma_\theta,
\]
where $G_n^+$ and $G_n^-$ refer to the continuous extensions of $G_n$ from the left and right of $\gamma_\theta$ onto $\gamma_\theta$, respectively.  Based on the asymptotic \eqref{seconeexp_integrandasymp} and the fact that the saddle point of the function $\varphi(s)$ is located at $s=1$, we expect that $F_n(z) \approx G_n(z)$ for $z \approx 1$ as $n \to \infty$.  Something to this effect is shown in Lemma \ref{seconeexp_fngnapproxlemma}.

Just as in \cite[p.\ 189]{mclaughlin:exprh} we define
\[
	h(\zeta) = \frac{1}{2\pi i} \int_{-\infty}^{\infty} e^{-u^2} \frac{du}{u-\zeta}, \qquad \zeta \in \C \setminus \R
\]
and
\[
	P_n(z) = h\!\left(-i\sqrt{n} \psi^{-1}(z)\right), \qquad z \in U \setminus \gamma_\theta.
\]
Plemelj's formula implies that
\[
	h^+(x) = h^-(x) + e^{-x^2}, \qquad x \in \R,
\]
where $\R$ is given the usual orientation from $-\infty$ to $+\infty$, and setting $z = \psi(ix/\sqrt{n})$ yields
\[
	P_n^+(z) = P_n^-(z) + e^{n\varphi(z)}, \qquad z \in U \cap \gamma_\theta.
\]
Here $+$ and $-$ indicate approaching the contour $\gamma_\theta$ from the left and from the right, respectively.

\subsection{Proof of Theorem \ref{seconeexp_maintheo}}
\label{seconeexpsub_proof}

Choose $\epsilon > 0$ such that $\overline{B_{2\epsilon}(1)} \subset U$ and define
\[
	m(z) = \begin{cases}
		G_n(z) & \text{for } z \in \C \setminus \left( \gamma_\theta \cup \overline{B_{2\epsilon}(1)} \right), \\
		G_n(z) - P_n(z) & \text{for } z \in B_{2\epsilon}(1) \setminus \gamma_\theta.
	\end{cases}
\]
The jumps for $G_n(z)$ and $P_n(z)$ cancel each other out as $z$ moves across $\gamma_\theta$ in $B_{2\epsilon}(1)$, so $m$ is analytic on $B_{2\epsilon}(1)$.  If we define the contours
\begin{equation}
\label{seconeexpeq_biggammadef}
	\Gamma_1 = \partial B_{2\epsilon}(1), \qquad \Gamma_2 = \gamma_\theta \setminus B_{2\epsilon}(1), \qquad \Gamma = \Gamma_1 \cup \Gamma_2,
\end{equation}
where $\Gamma_1$ is oriented in the counterclockwise direction and $\Gamma_2$ inherits its orientation from $\gamma_\theta$ (and thus $\gamma$), then the function $m$ uniquely solves the following Riemann-Hilbert problem.

\begin{figure}[!htb]
	\centering
	\begin{tabular}{cc}
		\includegraphics[width=0.25\textwidth]{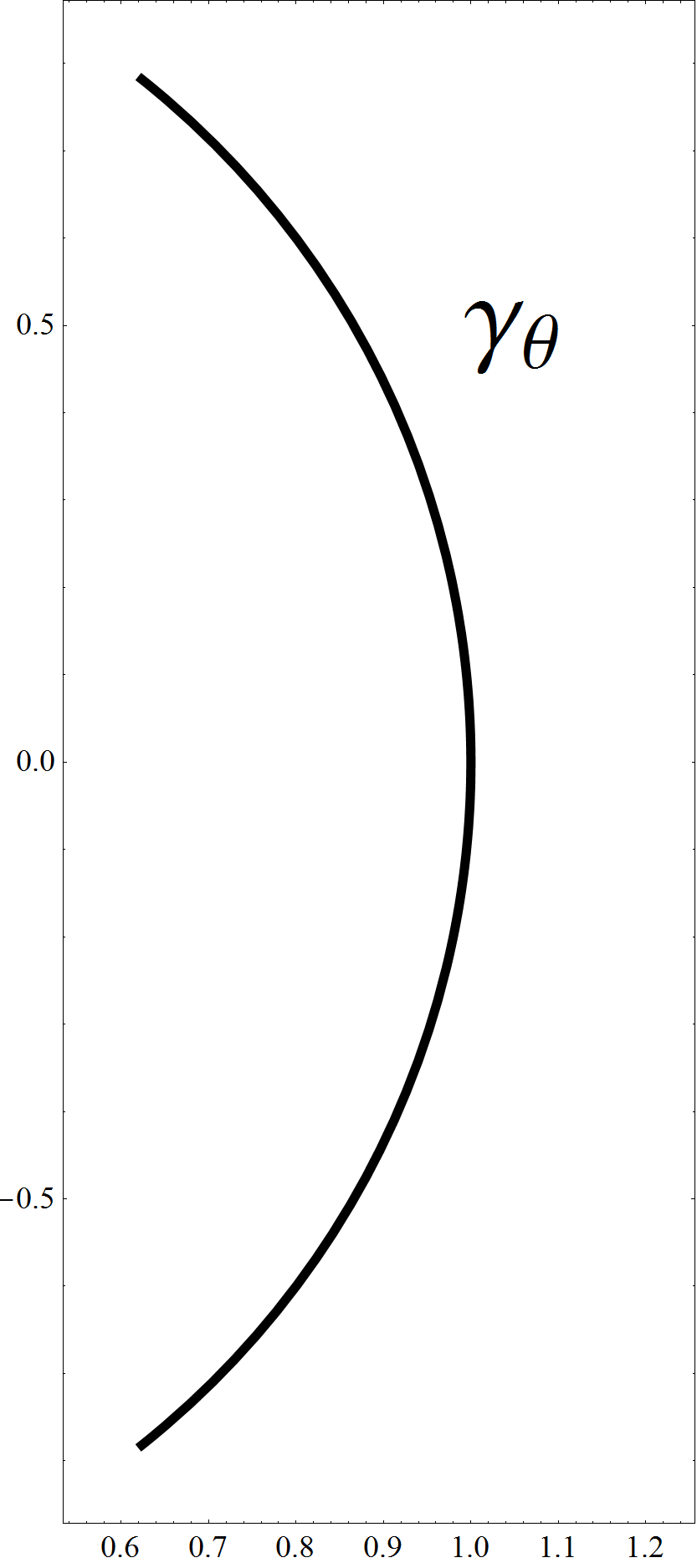}
		& \includegraphics[width=0.25\textwidth]{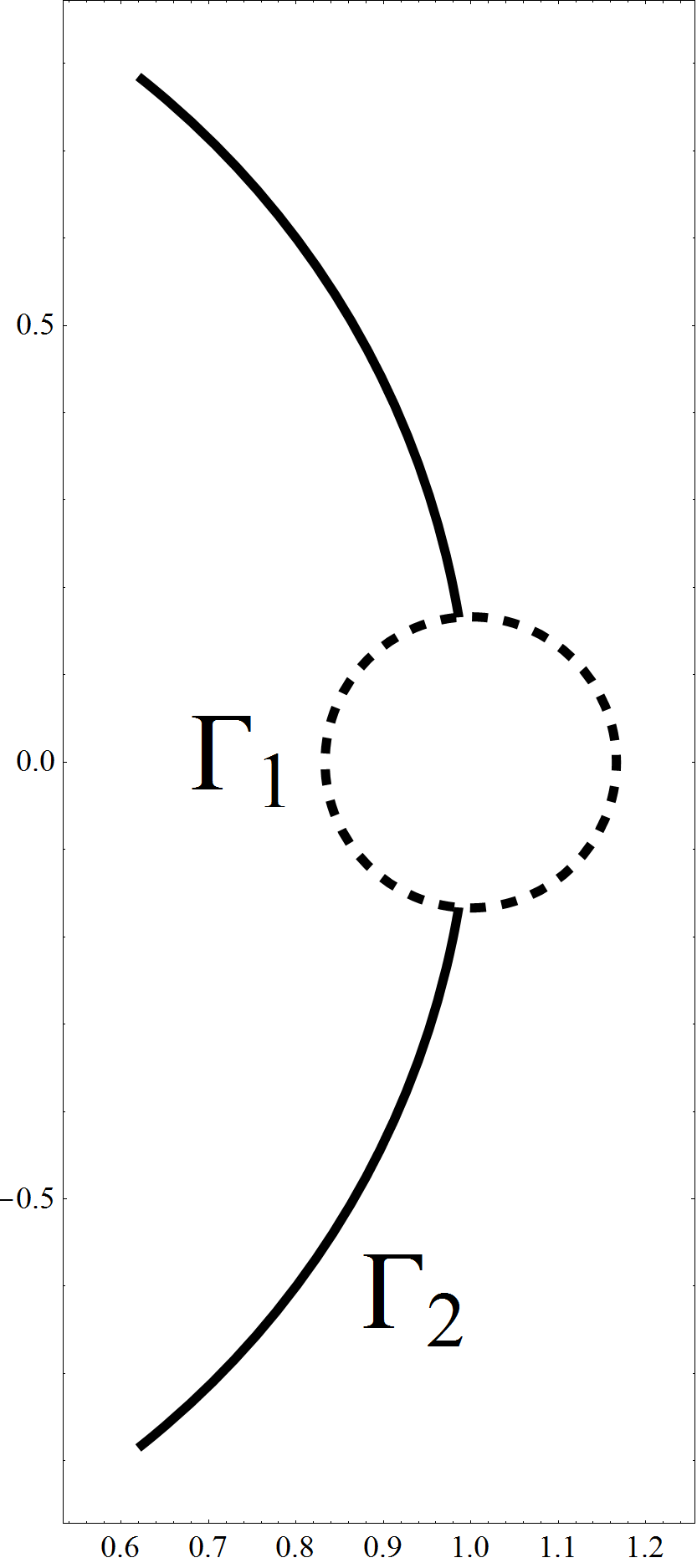}
	\end{tabular}
	\caption[Schematic representations of the curve {$\gamma_\theta$} and the curve {$\Gamma$} in \eqref{seconeexpeq_biggammadef}.]{Schematic representations of the curve $\gamma_\theta$ (left) and the curve $\Gamma$ in \eqref{seconeexpeq_biggammadef} (right). In the plot of $\Gamma$, the curve $\Gamma_1$ is dashed and the curve $\Gamma_2$ is solid black.}
	\label{seconeexpfig_biggamma1}
\end{figure}

\begin{rhp}
\label{seconeexp_mainrhp}
Seek an analytic function $M \colon \C \setminus \Gamma \to \C$ such that
\begin{enumerate}
	\item $M^+(z) = M^-(z) - P_n(z)$ for $z \in \Gamma_1 \setminus \Gamma_2$,
	\item $M^+(z) = M^-(z) + e^{n\varphi(z)}$ for $z \in \Gamma_2$ except at endpoints,
	\item $M(z) \to 0$ as $|z| \to \infty$,
	\item if $c$ is an endpoint of either arc of $\Gamma_2$ then $M(z) = O(|z-c|^q)$ as $z \to c$ with $z \in \C \setminus \Gamma_2$ for some $q > -1$.
\end{enumerate}
\end{rhp}

Plemelj's formula then yields
\begin{align}
m(z) &= \frac{1}{2\pi i} \int_\Gamma \Bigl[m^+(s) - m^-(s)\Bigr] \frac{ds}{s-z} \nonumber \\
	&= -\frac{1}{2\pi i} \int_{\Gamma_1} P_n(s) \frac{ds}{s-z} + \frac{1}{2\pi i} \int_{\Gamma_2} e^{n\varphi(s)} \frac{ds}{s-z} \label{seconeexp_mint}.
\end{align}

\begin{lemma}
\label{seconeexp_pintlemma}
\[
	\int_{\Gamma_1} P_n(s) \frac{ds}{s-z} = O(n^{-1/2})
\]
uniformly for $z \in B_\epsilon(1)$ as $n \to \infty$.
\end{lemma}

\begin{proof}
There exists a constant $C_1$ such that
\[
	|h(\zeta)| \leq C_1 |\zeta|^{-1}
\]
for $\zeta \notin \R$.  Setting $\zeta = -i\sqrt{n}\psi^{-1}(s)$ yields
\[
	|P_n(s)| \leq C_1 n^{-1/2} |\psi^{-1}(s)|^{-1} = C_1 n^{-1/2} |\varphi(s)|^{-1/2}
\]
for $s \in U \setminus \gamma_\theta$.  Thus if $s \in \Gamma_1$ then $|\varphi(s)| \geq C_2$ for some constant $C_2 > 0$, so
\[
	\left| \int_{\Gamma_1} P_n(s) \frac{ds}{s-z} \right| \leq C_1 C_2^{-1/2} \epsilon^{-1} n^{-1/2}
\]
for $z \in B_\epsilon(1)$.
\end{proof}

\begin{lemma}
\label{seconeexp_int2lemma}
There is a constant $c > 0$ such that
\[
	\int_{\Gamma_2} e^{n\varphi(s)} \frac{ds}{s-z} = O(e^{-c n})
\]
uniformly for $z \in B_\epsilon(1)$ as $n \to \infty$.
\end{lemma}

\begin{proof}
Recalling Definition \ref{seccurve1_admisscontour}, since $\gamma_\theta = \gamma \cap \{z \in \C \colon |{\arg z}| \leq \theta\}$ and $\Gamma_2 = \gamma_\theta \setminus B_{2\epsilon}(1)$ there exists a constant $c > 0$ such that $\re \varphi(s) < -c$ for $s \in \Gamma_2$. Further, if $z \in B_{\epsilon}(1)$ and $s \in \Gamma_2$ then $|s-z| > \epsilon$, so that
\begin{align*}
\left| \int_{\Gamma_2} e^{n\varphi(s)} \frac{ds}{s-z} \right| &\leq \int_{\Gamma_2} e^{n\re \varphi(s)} \frac{|ds|}{|s-z|} \\
	&< \epsilon^{-1} \len(\Gamma_2) e^{-c n}.
\end{align*}
\end{proof}

Combining Lemmas \ref{seconeexp_pintlemma} and \ref{seconeexp_int2lemma} yields
\[
	m(z) = o(1),
\]
and hence, by the definition of $m$,
\begin{equation}
\label{seconeexp_gnpnapprox}
	G_n(z) = P_n(z) + o(1)
\end{equation}
uniformly for $z \in B_\epsilon(1)$ as $n \to \infty$.  Now set $z = 1 + w/\sqrt{n}$, where $w$ is restricted to a compact subset of $\re w < 0$.

\begin{lemma}
\label{seconeexp_fngnapproxlemma}
\[
	\lim_{n \to \infty} F_n\!\left(1+w/\sqrt{n}\right) - G_n\!\left(1+w/\sqrt{n}\right) = 0
\]
uniformly for $w$ restricted to compact subsets of $\re w < 0$.
\end{lemma}

\begin{proof}
In this proof we will write $z = 1+w/\sqrt{n}$ as a shorthand, keeping in mind the implicit dependence of $z$ on $w$ and $n$.

Split the integral for $F_n$ into the two pieces
\begin{equation}
\label{seconeexp_lem3fnsplit}
	F_n(z) = \frac{r_n^{-a} (\log r_n)^{-b}}{2\pi i} \left( \int_{\gamma_\theta} (e^{1/\lambda}s)^{-n} f(r_ns) \frac{ds}{s-z} + \int_{\gamma \setminus \gamma_\theta} (e^{1/\lambda}s)^{-n} f(r_ns) \frac{ds}{s-z} \right).
\end{equation}
As in the previous lemma the second integral here is uniformly exponentially decreasing.

By the asymptotic assumption on $f$ in \eqref{seconeexp_fgrowth}, for $|{\arg z}| \leq \theta$ we can write
\begin{equation}
\label{seconeexp_deltadef}
	f(z) = z^a (\log z)^b \exp(z^\lambda) \bigl[ 1+\delta(z) \bigr],
\end{equation}
where $\delta(z) \to 0$ uniformly as $|z| \to \infty$ in this sector.  This implies
\[
	\frac{f(r_n s)}{r_n^a (\log r_n)^b (e^{1/\lambda}s)^n} = s^a e^{n \varphi(s)} \left(1 + \frac{\log s}{\log r_n}\right)^b \bigl[ 1 + \delta(r_n s) \bigr]
\]
for $s \in \gamma_\theta$. Then define
\begin{equation}
\label{seconeexp_deltatdef}
	\tilde{\delta}(r_n, s) = \left(1 + \frac{\log s}{\log r_n}\right)^b \bigl[ 1 + \delta(r_n s) \bigr] - 1
\end{equation}
and write the integrand of the first integral in \eqref{seconeexp_lem3fnsplit} as
\[
	e^{n\varphi(s)} + e^{n\varphi(s)} \left(s^a - 1\right) + s^a e^{n\varphi(s)} \tilde{\delta}(r_n, s).
\]
Recalling the definition of $G_n$ in \eqref{seconeexp_gdef}, it follows that \eqref{seconeexp_lem3fnsplit} can be rewritten as
\begin{align}
F_n(z) &= G_n(z) + \frac{1}{2\pi i} \int_{\gamma_\theta} e^{n\varphi(s)} \left(s^a - 1\right) \frac{ds}{s-z} + \frac{1}{2\pi i} \int_{\gamma_\theta} s^a e^{n\varphi(s)}\tilde{\delta}(r_n, s) \frac{ds}{s-z} \nonumber \\
&\qquad + O(e^{-cn})
\label{seconeexp_fngnequiv}
\end{align}
for some constant $c > 0$.  We will show that both of these remaining integrals tend to $0$ uniformly.

The contour $\gamma_\theta$ passes through the point $s=1$ vertically, so by assumption there exists a positive constant $C_2$ such that $|s-z| \geq C_1 n^{-1/2}$.  For $n$ large enough $z \notin \gamma_\theta$, and in that case
\[
	\left|\frac{s-1}{s-z}\right| \leq 1 + \left| \frac{1-z}{s-z} \right| \leq 1 + C_1^{-1} n^{1/2} |1-z| \leq C_2
\]
for some constant $C_2$. Hence
\begin{align*}
\left| \int_{\gamma_\theta} e^{n\varphi(s)} \left(s^a - 1\right) \frac{ds}{s-z} \right| &\leq \int_{\gamma_\theta} e^{n\re \varphi(s)} \left| \frac{s^a - 1}{s-1} \right| \left| \frac{s-1}{s-z} \right| |ds| \\
	&\leq C_2 \int_{\gamma_\theta} e^{n\re \varphi(s)} \left| \frac{s^a - 1}{s-1} \right| |ds|,
\end{align*}
which tends to zero as $n \to \infty$.

Split the second integral in \eqref{seconeexp_fngnequiv} like
\[
	\int_{\gamma_\theta} = \int_{\gamma_\theta \cap B_\epsilon(1)} + \int_{\gamma_\theta \setminus B_\epsilon(1)}.
\]
The integral over $\gamma_\theta \setminus B_\epsilon(1)$ decreases exponentially.  Let $s = \psi(it)$ and let
\[
	-i\psi^{-1}(\gamma_\theta \cap B_\epsilon(1)) = (-\alpha_1,\alpha_2),
\]
where $\alpha_1, \alpha_2 > 0$, so that
\begin{align*}
&\left| \int_{\gamma_\theta \cap B_\epsilon(1)} s^a e^{n\varphi(s)}\tilde{\delta}(r_n, s) \frac{ds}{s-z} \right| \\
&\qquad = \left| \int_{-\alpha_1}^{\alpha_2} e^{-nt^2} \tilde{\delta}(r_n, \psi(it)) \frac{\psi(it)^a \psi'(it)}{\psi(it) - z}\,dt \right| \\
&\qquad \leq C_1^{-1} n^{1/2} \sup_{-\alpha_1 < t < \alpha_1} \left|\tilde{\delta}(r_n, \psi(it)) \psi(it)^a\psi'(it)\right| \int_{-\alpha_1}^{\alpha_2} e^{-nt^2} \,dt \\
&\qquad < C_1^{-1} \sqrt{\pi} \sup_{-\alpha_1 < t < \alpha_1} \left|\tilde{\delta}(r_n, \psi(it)) \psi(it)^a\psi'(it)\right|,
\end{align*}
which tends to $0$ as $n \to \infty$ by the assumption on $\delta$ and, by extension, $\tilde\delta$.

Combining the above estimates with \eqref{seconeexp_fngnequiv} it follows that
\[
	F_n(z) = G_n(z) + o(1)
\]
uniformly as $n \to \infty$.
\end{proof}

As a consequence of this lemma, equation \eqref{seconeexp_gnpnapprox} implies that
\begin{equation}
\label{seconeexp_fnpnasymp}
	F_n\!\left(1+w/\sqrt{n}\right) = P_n\!\left(1+w/\sqrt{n}\right) + o(1)
\end{equation}
uniformly as $n \to \infty$.

Following the argument in \cite[p.\ 194]{mclaughlin:exprh}, it can be shown that
\[
	h(\zeta) = \frac{1}{2}e^{-\zeta^2} \erfc(-i\zeta)
\]
on $\im \zeta > 0$.  Setting
\[
	\zeta = -i\sqrt{n}\psi^{-1}(z) = -i\sqrt{n\varphi(z)}
\]
for an appropriately chosen branch of the square root yields an expression for $P_n$,
\[
	P_n(z) = \frac{1}{2} e^{n\varphi(z)} \erfc\!\left(-\sqrt{n\varphi(z)}\right),
\]
valid for $z \in U$ to the left of $\gamma_\theta$.  Since $2 - \erfc(x) = \erfc(-x)$ this can be rewritten as
\[
	P_n(z) = e^{n\varphi(z)} - \frac{1}{2} e^{n\varphi(z)} \erfc\!\left(\sqrt{n\varphi(z)}\right).
\]
It is straightforward to show that
\[
	\lim_{n \to \infty} n \varphi(1+w/\sqrt{n}) = \frac{\lambda}{2} w^2
\]
uniformly, so
\[
	P_n(1+w/\sqrt{n}) = e^{\lambda w^2/2} - \frac{1}{2} e^{\lambda w^2/2} \erfc\!\left(w \sqrt{\lambda/2}\,\right) + o(1)
\]
uniformly as $n \to \infty$.  By substituting this into \eqref{seconeexp_fnpnasymp} it follows that
\begin{equation}
\label{seconeexp_fnerfasymp1}
F_n(1+w/\sqrt{n}) = e^{\lambda w^2/2} - \frac{1}{2} e^{\lambda w^2/2} \erfc\!\left(w \sqrt{\lambda/2}\,\right) + o(1)
\end{equation}
uniformly as $n \to \infty$.

For $n$ large enough
\[
	F_n(1+w/\sqrt{n}) = \frac{1}{r_n^a (\log r_n)^b} \left(\frac{f(r_n(1+w/\sqrt{n}))}{e^{n/\lambda} (1+w/\sqrt{n})^{n}} - \frac{p_{n-1}(r_n(1+w/\sqrt{n}))}{e^{n/\lambda} (1+w/\sqrt{n})^{n}} \right)
\]
by \eqref{seconeexp_fnexplicit}.  The asymptotic assumption \eqref{seconeexp_fgrowth} grants us the uniform estimate
\[
	\frac{f(r_n(1+w/\sqrt{n}))}{r_n^a (\log r_n)^b e^{n/\lambda} (1+w/\sqrt{n})^{n}} = e^{\lambda w^2/2} + o(1),
\]
and substituting this into the above formula yields
\[
	F_n(1+w/\sqrt{n}) = e^{\lambda w^2/2} - e^{\lambda w^2/2} \frac{p_{n-1}(r_n(1+w/\sqrt{n}))}{f(r_n(1+w/\sqrt{n}))}(1+o(1)) + o(1)
\]
uniformly as $n \to \infty$.  Substituting this into \eqref{seconeexp_fnerfasymp1} then yields the expression
\[
\frac{p_{n-1}(r_n(1+w/\sqrt{n}))}{f(r_n(1+w/\sqrt{n}))} (1+o(1)) = \frac{1}{2} \erfc\!\left(w \sqrt{\lambda/2}\,\right) + o(1),
\]
which holds uniformly as $n \to \infty$.  Theorem \ref{seconeexp_maintheo} follows immediately from this asymptotic.

\subsection{Aside: An Alternate Riemann-Hilbert Problem}

The proof in the last section proceeded by solving a Riemann-Hilbert problem (RHP \ref{seconeexp_mainrhp}) which connected $G_n(z)$ to $P_n(z)$, giving an explicit expression for the error between them. We then proved that $G_n(z) \approx F_n(z)$ which, after some calculations, essentially concluded the proof. It is instructive, however, to note that we could have instead solved a Riemann-Hilbert problem that connected $F_n(z)$ directly to $P_n(z)$, thereby constructing a function which connects the global behavior of $F_n(z)$ to this relationship between $G_n(z)$ and the local parametrix $P_n(z)$.

We will detail this alternate proof in this section. It is substantially more complicated than the proof given in the last section, and we must even assume something stronger about the asymptotic behavior of $f$.

Let $\Delta$ be the circle centered at $z=1$ which subtends an angle of $\theta$ from the origin.  Denote by $\sigma_1, \sigma_2$ the points where $\Delta$ intersects the line of steepest descent of the function $\re \varphi(z)$ which passes through the point $z=1$.  Note that by symmetry $\sigma_1 = \overline{\sigma_2}$ and $\re \varphi(\sigma_1) = \re \varphi(\sigma_2)$.  Further, $\re \varphi(\sigma_1) < 0$.

\begin{condition}
\label{seconeexp_dfgrowth}
There exists a constant $\nu < -\re \varphi(\sigma_1)$ such that if $z$ is restricted to any compact subset of $\{z \in \C : z \neq 0 \text{ and } |{\arg z}| \leq \theta\}$ then
\[
	\frac{f'(r_n z)}{f(r_n z)} = O(e^{\nu n})
\]
uniformly in $z$ as $n \to \infty$.
\end{condition}

Assuming that this condition holds, we now introduce our alternate definition for $m(z)$ as well as the Riemann-Hilbert problem it solves.

Choose $\epsilon > 0$ such that $\overline{B_{2\epsilon}(1)} \subset U$ and define 
\[
	m(z) = \begin{cases}
		F_n(z) & \text{for } z \in \C \setminus \left( \gamma \cup \overline{B_{2\epsilon}(1)} \right), \\
		G_n(z) - P_n(z) & \text{for } z \in B_{2\epsilon}(1) \setminus \gamma.
	\end{cases}
\]
The jumps for $G_n(z)$ and $P_n(z)$ cancel each other out as $z$ moves across $\gamma$ in $B_{2\epsilon}(1)$, so $m$ is analytic on $B_{2\epsilon}(1)$.  If we define the contours
\begin{equation}
\label{seconeexpeq_biggammadef2}
	\Gamma_1 = \partial B_{2\epsilon}(1), \qquad \Gamma_2 = \gamma \setminus B_{2\epsilon}(1), \qquad \Gamma = \Gamma_1 \cup \Gamma_2,
\end{equation}
where $\Gamma_1$ is oriented in the counterclockwise direction and $\Gamma_2$ inherits its orientation from $\gamma$, then the function $m$ uniquely solves the following Riemann-Hilbert problem.

\begin{figure}[!htb]
	\centering
	\includegraphics[width=0.6\textwidth]{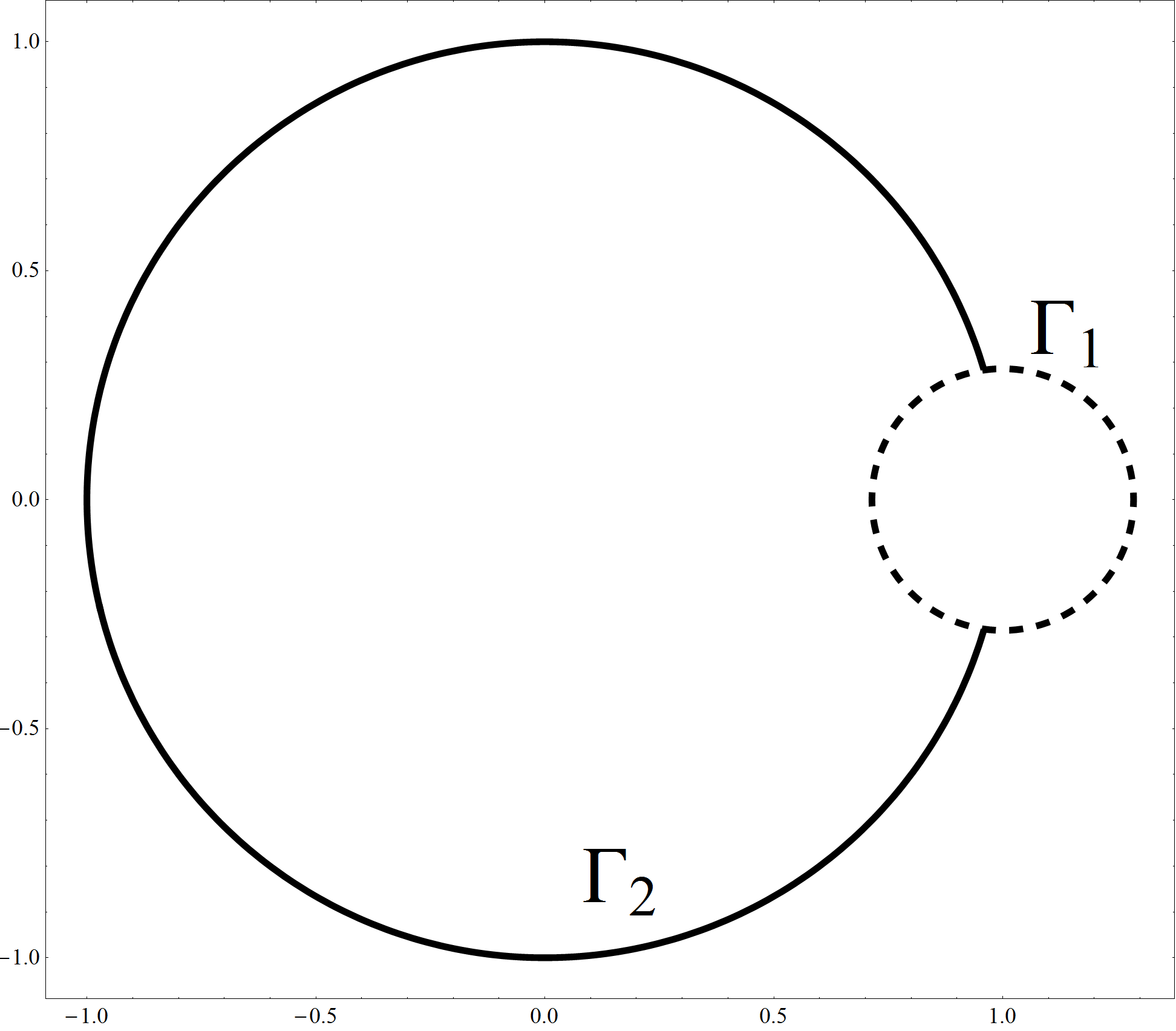}
	\caption[Schematic representation of the curve {$\Gamma$} in \eqref{seconeexpeq_biggammadef2}.]{Schematic representation of the curve $\Gamma$ in \eqref{seconeexpeq_biggammadef2}. The curve $\Gamma_1$ is dashed and the curve $\Gamma_2$ is solid black.}
	\label{seconeexpfig_biggamma2}
\end{figure}

\begin{rhp}
\label{seconeexp_mainrhp2}
Seek an analytic function $M \colon \C \setminus \Gamma \to \C$ such that
\begin{enumerate}
 \item $M^+(z) = M^-(z) - P_n(z) + G_n(z) - F_n(z)$ for $z \in \Gamma_1 \setminus \Gamma_2$,
 \item $M^+(z) = M^-(z) + \frac{f(r_n z)}{r_n^a (\log r_n)^b (e^{1/\lambda}z)^n}$ for $z \in \Gamma_2$ except at endpoints,
\item $M(z) \to 0$ as $|z| \to \infty$,
\item if $c$ is an endpoint of $\Gamma_2$ then $M(z) = O(|z-c|^q)$ as $z \to c$ with $z \in \C \setminus \Gamma_2$ for some $q > -1$.
\end{enumerate}
\end{rhp}

Plemelj's formula grants the integral representation
\begin{align}
m(z) &= \frac{1}{2\pi i} \int_\Gamma \Bigl[m^+(s) - m^-(s)\Bigr] \frac{ds}{s-z} \nonumber \\
	&= -\frac{1}{2\pi i} \int_{\Gamma_1} P_n(s) \frac{ds}{s-z} + \frac{1}{2\pi i} \int_{\Gamma_1} G_n(s)\frac{ds}{s-z} - \frac{1}{2\pi i} \int_{\Gamma_1} F_n(s)\frac{ds}{s-z} \nonumber \\
	&\qquad + \frac{r_n^{-a} (\log r_n)^{-b}}{2\pi i} \int_{\Gamma_2} (e^{1/\lambda}s)^{-n} f(r_n s) \frac{ds}{s-z}.
\label{seconeexp_mint2}
\end{align}
The following lemmas will show that, as $n \to \infty$, each of these integrals tends to zero uniformly as long as $z$ is bounded away from $\Gamma$.

\begin{lemma}
\label{seconeexp_gintlemma}
\[
	\int_{\Gamma_1} G_n(s)\frac{ds}{s-z} = O(n^{-1/2})
\]
uniformly for $z \in B_\epsilon(1)$ as $n \to \infty$.
\end{lemma}

\begin{proof}
For $z \in B_\epsilon(1)$
\begin{align*}
	\left| \int_{\Gamma_1} G_n(s)\frac{ds}{s-z} \right| &\leq 4\pi \epsilon \left\| \frac{G_n(s)}{s-z} \right\|_{L^\infty(\Gamma_1)} \\
	&\leq 4\pi \| G_n(s) \|_{L^\infty(\Gamma_1)}.
\end{align*}
Let $\Gamma_1^+$ and $\Gamma_1^-$ denote the closures of the parts of $\Gamma_1$ lying to the left and to the right of $\gamma_\theta$, respectively.  Then from the above it follows that
\begin{equation}
\label{seconeexp_gintest}
	\left| \int_{\Gamma_1} G_n(s)\frac{ds}{s-z} \right| \leq 4\pi \left( \| G_n(s) \|_{L^\infty(\Gamma_1^+)} + \| G_n(s) \|_{L^\infty(\Gamma_1^-)} \right).
\end{equation}

Define $s_1,s_2$ to be the points where $\Gamma_1$ intersects $\gamma_\theta$. Depending on whether $s$ approaches $s_j$ from the left or the right,
\[
	G_n(s_j) = \pm \frac{1}{2} e^{n\varphi(s_j)} + \frac{1}{2\pi i} \operatorname{P.V.} \int_{\gamma_\theta} e^{n\varphi(t)} \frac{dt}{t-s_j}.
\]
Note that the first term here decays exponentially.

Deform the contour $\gamma_\theta$ in a small neighborhood $A$ of $s_j$ to be a straight line passing through $s_j$.  Choose this neighborhood small enough so that $\gamma_\theta$ still lies entirely below the saddle point at $s=1$ on the surface $\re \varphi(s)$ except where it passes through $s=1$.  Then
\[
	\operatorname{P.V.} \int_{\gamma_\theta} e^{n\varphi(t)} \frac{dt}{t-s_j} = \int_{\gamma_\theta \cap A} \frac{e^{n\varphi(t)} - e^{n\varphi(s_j)}}{t-s_j}\,dt + \int_{\gamma_\theta \setminus A} e^{n\varphi(t)} \frac{dt}{t-s_j}.
\]
A straightforward application of the Laplace method to the second integral here yields
\[
	\int_{\gamma_\theta \setminus A} e^{n\varphi(t)} \frac{dt}{t-s_j} = O(n^{-1/2}).
\]
Taylor's theorem grants the estimate
\begin{align*}
\left| e^{n\varphi(t)} - e^{n\varphi(s_j)} \right| &\leq |t-s_j| \sup_{\tau \in \gamma_\theta \cap A} \left| n\varphi'(\tau) e^{n\varphi(\tau)} \right| \\
&\leq |t-s_j| n e^{n(\re \varphi(s_j)+c)} \sup_{\tau \in \gamma_\theta} |\varphi'(\tau)|,
\end{align*}
where $0 < c < -\re \varphi(s_j)$, and thus it follows that
\[
	\left| \int_{\gamma_\theta \cap A} \frac{e^{n\varphi(t)} - e^{n\varphi(s_j)}}{t-s_j}\,dt \right| < \const. \cdot ne^{n(\re \varphi(s_j) + c)},
\]
and this tends to $0$.  Combining these facts,
\begin{equation}
\label{seconeexp_gest1}
	G_n(s_j) = O(n^{-1/2})
\end{equation}
as $n \to \infty$.

Now suppose $s \in \Gamma_1^+ \setminus \{s_1,s_2\}$.  Then $e^{n\varphi(t)}/(t-s)$ is analytic in a neighborhood of $\gamma_\theta$.  Deform $\gamma_\theta$ near $s_1$ and $s_2$ so that it stays a small positive distance away from $\Gamma_1^+$, and in such a way that $\gamma_\theta$ is unchanged in the disk $B_\epsilon(1)$.  Split the integral for $G_n(s)$ into the pieces
\[
	G_n(s) = \frac{1}{2\pi i}\int_{\gamma_\theta \setminus B_\epsilon(1)} e^{n\varphi(t)} \frac{ds}{t-s} + \frac{1}{2\pi i}\int_{\gamma_\theta \cap B_\epsilon(1)} e^{n\varphi(t)} \frac{ds}{t-s}.
\]
After this deformation, the first integral is bounded by
\[
	\left| \int_{\gamma_\theta \setminus B_\epsilon(1)} e^{n\varphi(t)} \frac{ds}{t-s} \right| \leq C e^{-c n},
\]
where $C > 0$ and $c > 0$ are constants independent of $s$.  In the second integral let $t = \psi(iu)$ and define $-i\psi^{-1}(\gamma_\theta \cap B_\epsilon(1)) = (-\alpha_1,\alpha_2)$, so that
\begin{align}
\left| \int_{\gamma_\theta \cap B_\epsilon(1)} e^{n\varphi(t)} \frac{ds}{t-s} \right| &= \left| \int_{-\alpha_1}^{\alpha_2} e^{-nu^2} \frac{i\psi'(iu)}{\psi(iu)-s}\,du \right| \nonumber \\
&\leq \sup_{u \in (\alpha_1,\alpha_2)} \left| \frac{\psi'(iu)}{\psi(iu)-s} \right| \int_{-\alpha_1}^{\alpha_2} e^{-nu^2}\,du \nonumber \\
&\leq \epsilon^{-1} \sqrt{\pi/n} \sup_{u \in (\alpha_1,\alpha_2)} |\psi'(iu)|. \label{seconeexp_gest2}
\end{align}
An identical process will yield the same bound for $s \in \Gamma_1^- \setminus \{s_1,s_2\}$.

Combining \eqref{seconeexp_gest1} and \eqref{seconeexp_gest2} in \eqref{seconeexp_gintest} results in the estimate
\[
	\int_{\Gamma_1} G_n(s) \frac{ds}{s-z} = O(n^{-1/2})
\]
uniformly for $z \in B_\epsilon(1)$ as $n \to \infty$, as desired.
\end{proof}

\begin{lemma}
\label{seconeexp_gamma2lemma}
There exists a constant $c>0$ such that
\[
	r_n^{-a} (\log r_n)^{-b} \int_{\Gamma_2} (e^{1/\lambda}s)^{-n} f(r_n s) \frac{ds}{s-z} = O(e^{-cn})
\]
uniformly for $z \in B_\epsilon(1)$ as $n \to \infty$.
\end{lemma}

\begin{proof}
Let $\Gamma_2'$ denote the part of $\Gamma_2$ for which $|{\arg s}| \leq \theta$ and let $\Gamma_2''$ denote the part for which $\theta < |{\arg s}|$.  Split the integral into the two parts
\[
	\int_{\Gamma_2} = \int_{\Gamma_2'} + \int_{\Gamma_2''}.
\]

For $|{\arg z}| \leq \theta$ we can write
\[
	f(z) = z^a (\log z)^b \exp(z^\lambda) \bigl[ 1 + \delta(z) \bigr],
\]
where $\delta(z) \to 0$ uniformly as $|z| \to \infty$, so for $s \in \Gamma_2'$
\[
	\frac{f(r_n s)}{r_n^a (\log r_n)^b (e^{1/\lambda}s)^n} = s^a e^{n \varphi(s)} \left(1 + \frac{\log s}{\log r_n}\right)^b \bigl[ 1 + \delta(r_n s) \bigr].
\]
If $s \in \Gamma_2'$ then there is a constant $d > 0$ such that $\re \varphi(s) < -d$.  The quantities $s^a$, $\log s/\log r_n$, and $\delta(r_n s)$ are uniformly bounded for $s \in \Gamma_2'$, and the quantities $r_n^a$ and $(\log r_n)^b$ grow subexponentially, so if $z \in B_\epsilon(1)$ there are positive constants $C_1$ and $d'$ such that
\[
	\left| \int_{\Gamma_2'} (e^{1/\lambda}s)^{-n} f(r_n s) \frac{ds}{s-z} \right| \leq \len(\Gamma_2') \cdot \epsilon^{-1} \cdot C_1 e^{-d'n}.
\]

For $|{\arg z}| > \theta$
\[
	|f(z)| \leq C_2 \exp(\mu|z|^\lambda)
\]
for some constant $C_2$.  If $s \in \Gamma_2''$ then $|s| = 1$, so
\[
	\left| \frac{f(r_n s)}{(e^{1/\lambda}s)^n} \right| \leq C_2 \exp[(\mu - 1)n/\lambda],
\]
and, since $|s-z| \geq \epsilon$,
\[
	\left| \int_{\Gamma_2''} (e^{1/\lambda}s)^{-n} f(r_n s) \frac{ds}{s-z} \right| \leq \len(\Gamma_2'') \cdot \epsilon^{-1} \cdot C_2 \exp[(\mu - 1)n/\lambda].
\]
Combining this with the above estimate yields the desired result.
\end{proof}

\begin{lemma}
\label{seconeexp_fnintlemma}
\[
	\int_{\Gamma_1} F_n(s)\frac{ds}{s-z} = O(n^{-1/2})
\]
uniformly for $z \in B_\epsilon(1)$ as $n \to \infty$.
\end{lemma}

\begin{proof}
Split the integral for $F_n$ into the two pieces
\[
	F_n(s) = \frac{r_n^{-a} (\log r_n)^{-b}}{2\pi i} \left( \int_{\gamma \setminus \gamma_\theta} (e^{1/\lambda}t)^{-n} f(r_nt) \frac{dt}{t-s} + \int_{\gamma_\theta} (e^{1/\lambda}t)^{-n} f(r_nt) \frac{dt}{t-s} \right)
\]
and denote by $F_n^1(s)$ and $F_n^2(s)$ the left and right terms, respectively.

If $s \in \Gamma_1$ and $t \in \gamma\setminus\gamma_\theta$ then $|t-s| \geq C_1$ for some constant $C_1 > 0$ since
\[
	s \in \overline{B_{2\epsilon}(1)} \subset U \subset \{z \in \C : |{\arg z}| \leq \theta\}
\]
and $U$ is open.  There exists a constant $C_2$ such that
\[
	|f(z)| \leq C_2 \exp(\mu|z|^\lambda)
\]
for $|{\arg z}| \geq \theta$, and just as in the proof of Lemma \ref{seconeexp_gamma2lemma} it can be shown that
\[
	\left| \int_{\gamma \setminus \gamma_\theta} (e^{1/\lambda}t)^{-n} f(r_nt) \frac{dt}{t-s} \right| \leq \len(\gamma \setminus \gamma_\theta) \cdot C_1^{-1} \cdot C_2 \exp[(\mu-1)n/\lambda].
\]
It follows that there are positive constants $C_3$ and $c$ such that
\[
	\left| \int_{\Gamma_1} F_n^1(s) \frac{ds}{s-z} \right| \leq C_3 e^{-cn}.
\]

Now consider the integral over $\gamma_\theta$.  For $|{\arg z}| \leq \theta$ we can write
\begin{equation}
\label{seconeexp_deltadef2}
	f(z) = z^a (\log z)^b \exp(z^\lambda) \bigl[ 1+\delta(z) \bigr],
\end{equation}
where $\delta(z) \to 0$ uniformly as $|z| \to \infty$ in this sector.  This implies
\[
	\frac{f(r_n t)}{r_n^a (\log r_n)^b (e^{1/\lambda}t)^n} = t^a e^{n \varphi(t)} \left(1 + \frac{\log t}{\log r_n}\right)^b \bigl[ 1 + \delta(r_n t) \bigr]
\]
for $t \in \gamma_\theta$, and so
\begin{align*}
	&\int_{\Gamma_1} F^2_n(s) \frac{ds}{s-z} \\
	&\qquad = \frac{1}{2\pi i}\int_{\Gamma_1} \frac{1}{s-z} \int_{\gamma_\theta} t^a e^{n\varphi(t)} \frac{dt}{t-s} \,ds + \frac{1}{2\pi i}\int_{\Gamma_1} \frac{1}{s-z} \int_{\gamma_\theta} t^a e^{n\varphi(t)} \tilde{\delta}(r_n, t) \frac{dt}{t-s} \,ds,
\end{align*}
where
\begin{equation}
\label{seconeexp_deltatdef2}
	\tilde{\delta}(r_n, t) = \left(1 + \frac{\log t}{\log r_n}\right)^b \bigl[ 1 + \delta(r_n t) \bigr] - 1.
\end{equation}
The first integral in this expression can be estimated using the method in Lemma \ref{seconeexp_gintlemma} while the second requires a little more care.  Actually the proof will go through just as before except for the estimates at the points $s_j$, which we will detail here.

Name the inner integral
\[
	g_n(s) = \frac{1}{2\pi i}\int_{\gamma_\theta} t^a e^{n\varphi(t)} \tilde{\delta}(r_n, t) \frac{dt}{t-s}.
\]
Depending on whether $s$ approaches $s_j$ from the left or the right,
\[
	g_n(s_j) = \pm \frac{1}{2} s_j^a e^{n\varphi(s_j)}\tilde{\delta}(r_n, s_j) + \frac{1}{2\pi i} \operatorname{P.V.} \int_{\gamma_\theta} t^a e^{n\varphi(t)} \tilde{\delta}(r_n, t) \frac{dt}{t-s_j}.
\]
The first term here decays exponentially.  Deform the contour $\gamma_\theta$ in a small neighborhood $A$ of $s_j$ to be a straight line passing through $s_j$.  Choose this neighborhood small enough so that $\gamma_\theta$ still lies entirely below the saddle point at $s=1$ on the surface $\re \varphi(s)$ except where it passes through $s=1$.  Then
\begin{align*}
	&\operatorname{P.V.} \int_{\gamma_\theta} t^a e^{n\varphi(t)} \tilde{\delta}(r_n, t) \frac{dt}{t-s_j} \\
	&\qquad = \int_{\gamma_\theta \cap A} \frac{t^a e^{n\varphi(t)} \tilde{\delta}(r_n, t) - s_j^a e^{n\varphi(s_j)}\tilde{\delta}(r_n, s_j)}{t-s_j}\,dt + \int_{\gamma_\theta \setminus A} t^a e^{n\varphi(t)} \tilde{\delta}(r_n, t) \frac{dt}{t-s_j}.
\end{align*}
For the second integral, the Laplace method yields
\[
	\int_{\gamma_\theta \setminus A} t^a e^{n\varphi(t)} \tilde{\delta}(r_n, t) \frac{dt}{t-s_j} = o(n^{-1/2}).
\]
Taylor's theorem implies that
\begin{align*}
	&\left| t^a e^{n\varphi(t)} \tilde{\delta}(r_n, t) - s_j^a e^{n\varphi(s_j)}\tilde{\delta}(r_n, s_j) \right| \\
	&\qquad \leq |t-s_j| \sup_{\tau \in \gamma_\theta \cap A} \left|a\tau^{a-1}e^{n\varphi(\tau)}\tilde{\delta}(r_n, \tau) + n\varphi'(\tau)\tau^ae^{n\varphi(\tau)}\tilde{\delta}(r_n, \tau) + \tau^ae^{n\varphi(\tau)}\tilde{\delta}_\tau(r_n, \tau) \right| \\
	&\qquad \leq |t-s_j| \left(\sup_{\tau \in \gamma_\theta \cap A} \left|a\tau^{a-1}e^{n\varphi(\tau)}\tilde{\delta}(r_n, \tau) + n\varphi'(\tau)\tau^ae^{n\varphi(\tau)}\tilde{\delta}(r_n, \tau) \right| \right. \\
	&\hspace{3cm} + \left. \sup_{\tau \in \gamma_\theta \cap A} \left| \tau^ae^{n\varphi(\tau)}\tilde{\delta}_\tau(r_n, \tau) \right| \right).
\end{align*}
The first supremum here decays exponentially.  For the second,
\[
	\sup_{\tau \in \gamma_\theta \cap A} \left| \tau^a e^{n\varphi(\tau)}\tilde{\delta}_\tau(r_n, \tau) \right| \leq e^{n(\re \varphi(s_j)+c')} \sup_{\tau \in \gamma_\theta} \left| \tau^a \tilde{\delta}_\tau(r_n, \tau) \right|,
\]
where $0 < c' < -\re \varphi(s_j)$.  By choosing $A$ smaller it can be shown that this estimate holds for any fixed $c' > 0$ small enough.  Differentiating the formula for $\tilde\delta$ yields
\[
	\tilde{\delta}_\tau(r_n, \tau) = \left(\frac{b}{\tau\log(r_n \tau)} + \frac{r_n}{1+\delta(r_n \tau)} \right) \left( \tilde\delta(r_n,\tau) + 1 \right) \delta'(r_n \tau),
\]
and, from \eqref{seconeexp_deltadef2},
\[
	\delta'(r_n \tau) = \left[ \frac{f'(r_n \tau)}{f(r_n \tau)} - \frac{1}{r_n \tau} \left( a + \frac{b}{\log(r_n \tau)} + n\tau^\lambda \right) \right] \bigl[ 1+\delta(r_n \tau) \bigr].
\]
After substituting this into the previous expression, an appeal to Condition \ref{seconeexp_dfgrowth} grants an estimate
\[
	\sup_{\tau \in \gamma_\theta \cap A} \left| \tau^ae^{n\varphi(\tau)}\tilde\delta_\tau(r_n,\tau) \right| \leq C_4 r_n e^{n (\re \varphi(s_j) + c' + \nu)},
\]
where $C_4 > 0$ is a constant independent of $n$.  In addition to taking $c'$ as small as we like, by choosing $U_\gamma$, $U$, and $\epsilon$ slightly larger the quantity $\re \varphi(s_j)$ can be made as close to $\re \varphi(\sigma_j)$ as desired. Arrangements can thus be made so that the quantity $\re \varphi(s_j) + c' + \nu$ is negative.  It follows that
\[
	\left| \int_{\gamma_\theta \cap A} \frac{t^a e^{n\varphi(t)} \tilde{\delta}(r_n, t) - s_j^a e^{n\varphi(s_j)}\tilde{\delta}(r_n, s_j)}{t-s_j}\,dt \right| \leq C_5 e^{-c''n}
\]
for some positive constants $C_5$ and $c''$, and combining this with the above Laplace method estimate yields
\[
	g_n(s_j) = o(n^{-1/2})
\]
as $n \to \infty$.

The remainder of the proof proceeds exactly as in Lemma \ref{seconeexp_gintlemma}.
\end{proof}

Indeed, Lemmas \ref{seconeexp_gintlemma}, \ref{seconeexp_gamma2lemma}, \ref{seconeexp_fnintlemma}, and \ref{seconeexp_pintlemma} imply that
\[
	m(z) = o(1),
\]
and then, by the definition of $m$,
\[
	G_n(z) = P_n(z) + o(1)
\]
uniformly for $z \in B_\epsilon(1)$ as $n \to \infty$.  Now set $z = 1 + w/\sqrt{n}$, where $w$ is restricted to a compact subset of $\re w < 0$.  By Lemma \ref{seconeexp_fngnapproxlemma}
\begin{equation}
\label{fnpnasymp}
	F_n(1+w/\sqrt{n}) = P_n(1+w/\sqrt{n}) + o(1)
\end{equation}
uniformly as $n \to \infty$.

The remainder of the proof of the theorem proceeds exactly as in the previous section.

\section{Unbalanced Growth in Two Directions: Destruction of the Scaling Limit}
\label{secunbalcauchy_sec}

In Section \ref{curscatwosec_sec} the scaling limit at the arcs of the limit curve for the partial sums of a function with two directions of maximal growth had to be modified depending on the balance of the constants $\re a$ and $\re b$. When $\re a > \re b$ one limit holds, when $\re a < \re b$ another, and when $\re a = \re b$ a sort of transitional limit holds.

The behavior of the scaling limit at the corner of the limit curve is starkly different. It does not go through any such change when $\re a = \re b$. In a sense the geometry of the zeros of the partial sums doesn't change that much when $\re a - \re b$ is only slightly negative compared to when it is slightly positive---the zeros still lie outside the curve---and it turns out that the corner scaling limit isn't sensitive to the change that does happen there, which is that the rate at which the zeros approach the arcs of the limit curve begins to depend on both $a$ and $b$ (see Theorem \ref{curscatwotheo_maintheo}, compare the definitions of $z_n^1(w)$ and $z_n^2(w)$).

There is a second bifurcation past this one which we have noticed previously in Chapter \ref{chap_limitcurves} (Theorem \ref{seccurvetwo_maintheo}), namely that the zeros approach the limit curve from the exterior when $\re a - \re b > -\lambda/2$ and from the interior when $\re a - \re b < -\lambda/2$. We will see that this bifurcation is so severe that the scaling limit at the corner of the limit curve is completely destroyed when $\re a - \re b \leq -\lambda/2$.

For examples of this behavior see Sections \ref{applicationssec_expints} and \ref{applicationssec_paraboliccylinder}.

\subsection{Statement of the Result}

Let $a,b,A \in \C$, $0 < \lambda < \infty$, $\mu < 1$, and $\zeta \in \C$ with $|\zeta| = 1$, $\zeta \neq 1$. Let $\theta \in (0,\pi)$ be small enough so that the sectors $|{\arg z}| \leq \theta$ and $|{\arg(z/\zeta)}| \leq \theta$ are disjoint.  We suppose that $f$ is an entire function with the asymptotic behavior
\begin{equation}
\label{secunbalcauchy_fgrowth}
f(z) = \begin{cases}
	z^a \exp\!\left(z^\lambda\right) \bigl[1+o(1)\bigr] & \text{for } |{\arg z}| \leq \theta, \\
	A(z/\zeta)^b \exp\!\left((z/\zeta)^\lambda\right) \bigl[1+o(1)\bigr] & \text{for } |{\arg(z/\zeta)}| \leq \theta, \\
	O\!\left(\exp\!\left(\mu |z|^\lambda\right)\right) & \text{otherwise}
	\end{cases}			
\end{equation}
as $|z| \to \infty$, with each estimate holding uniformly in its sector. For this $f$, let
\[
	p_n(z) = \sum_{k=0}^{n} \frac{f^{(k)}(0)}{k!} z^k
\]
and define
\begin{equation}
\label{secunbalcauchy_rndef}
	r_n = \left(\frac{n}{\lambda}\right)^{1/\lambda}.
\end{equation}

\begin{theorem}
\label{secunbalcauchy_maintheo}
\[
	\lim_{n \to \infty} \frac{p_{n-1}(r_n(1+w/\sqrt{n}))}{f(r_n(1+w/\sqrt{n}))} = \frac{1}{2} \erfc\!\left(w\sqrt{\lambda/2}\,\right)
\]
uniformly for $w$ restricted to any compact subset of $\re w < 0$ if and only if
\[
	\re b - \re a < \frac{\lambda}{2}.
\]
\end{theorem}

\begin{remark}
Just as with Theorem \ref{seconeexp_maintheo} in the case of one direction of maximal growth, Theorem \ref{secunbalcauchy_maintheo} gives us precise asymptotics for individual zeros of the scaled partial sums $p_{n-1}(r_n z)$ near the corner of the limit curve $S$ located at $z=1$ in the case where $\re b - \re a < \lambda/2$. Details of this are given in Section \ref{corscasec_width}, where we use that information to verify part (b) of the Modified Saff-Varga Width Conjecture (see Section \ref{introsec_width}) for this class of functions.
\end{remark}

\subsection{Definitions and Preliminaries}

We will repeat here several relevant definitions from Section \ref{seccurvetwo_defsandprelims}.

Let $\gamma$ be an admissible contour for the function
\begin{equation}
\label{secunbalcauchy_phidef}
	\varphi(z) := \left.\left(z^\lambda - 1 - \lambda \log z\right)\right/\lambda
\end{equation}
and define
\begin{equation}
\label{secunbalcauchy_fndef}
	F_n(z) = \frac{r_n^{-a}}{2\pi i} \int_{\gamma} \left(e^{1/\lambda}s\right)^{-n}\! f(r_n s) \frac{ds}{s-z}
\end{equation}
for $z \notin \gamma$, $z \neq 0$. Just as in Section \ref{seccurvetwo_defsandprelims},
\begin{equation}
\label{secunbalcauchy_fnexplicit}
	F_n(z) = \frac{1}{r_n^a (e^{1/\lambda} z)^n} \times \begin{cases}
		-p_{n-1}(r_n z) & \text{for } z \text{ outside } \gamma, \\
		f(r_n z) - p_{n-1}(r_n z) & \text{for } z \neq 0 \text{ inside } \gamma.
		\end{cases}
\end{equation}

\subsection{Proof of Theorem \ref{secunbalcauchy_maintheo}}

As in the proof of Theorem \ref{seccurvetwo_maintheo} in Section \ref{seccurvetwo_maintheoproof} we will split $\gamma$ into three parts. Call $\gamma_1$ the part of $\gamma$ in the sector $|{\arg z}| \leq \theta$, call $\gamma_2$ the part of $\gamma$ in the sector $|{\arg(z/\zeta)}| \leq \theta$, and call $\gamma_3$ the part of $\gamma$ outside of either of those sectors. Then divide the integral in \eqref{secunbalcauchy_fndef} into the three parts
\[
	\int_\gamma = \int_{\gamma_1} + \int_{\gamma_2} + \int_{\gamma_3}.
\]

Using a method identical to the proof of Lemma \ref{seccurve1_gminusgtlem} it can be shown that
\begin{equation}
\label{secunbalcauchy_g3asymp}
	\int_{\gamma_3} \left(e^{1/\lambda} s\right)^{-n}\! f(r_n s) \frac{ds}{s-z} = O\!\left(e^{(\mu-1)n/\lambda}\right)
\end{equation}
as $n \to \infty$ uniformly for $z$ restricted to any sector $|{\arg z}| \leq \theta - \epsilon$ with $\epsilon > 0$. Also, Lemma \ref{seccurvetwo_g2lemma} grants the asymptotic
\begin{equation}
\label{secunbalcauchy_g2asymp}
	\int_{\gamma_2} \left(e^{1/\lambda} s\right)^{-n}\! f(r_n s) \frac{ds}{s-z} = \frac{iA\zeta^{1-n} r_n^b}{\zeta-z} \sqrt\frac{2\pi}{\lambda n} + o(r_n^b n^{-1/2})
\end{equation}
as $n \to \infty$ uniformly for $z \in \C \setminus \tilde N_\epsilon$ with $\epsilon > 0$.

Combining equations \eqref{secunbalcauchy_g3asymp} and \eqref{secunbalcauchy_g2asymp} in the definition of $F_n$ yields the estimate
\begin{equation}
\label{secunbalcauchy_fnerror}
	F_n(z) = \frac{r_n^{-a}}{2\pi i} \int_{\gamma_1} \left(e^{1/\lambda}\right)^{-n}\! f(r_n s) \frac{ds}{s-z} + \frac{A\zeta^{1-n} r_n^{b-a}}{(\zeta-z) \sqrt{2\pi \lambda n}} + o(r_n^{b-a} n^{-1/2})
\end{equation}
as $n \to \infty$ uniformly for $z$ restricted to any sector $|{\arg z}| \leq \theta - \epsilon$ with $\epsilon > 0$.

Using essentially the same technique used in Section \ref{seconeexpsub_proof} to prove Theorem \ref{seconeexp_maintheo} it can be shown that
\[
	\left. \frac{r_n^{-a}}{2\pi i} \int_{\gamma_1} \left(e^{1/\lambda}\right)^{-n}\! f(r_n s) \frac{ds}{s-z} \right|_{z=1+w/\sqrt{n}} = e^{\lambda w^2/2} - \frac{1}{2} e^{\lambda w^2/2} \erfc\!\left(w\sqrt{\lambda/2}\right) + o(1)
\]
and that
\[
	F_n(1+w/\sqrt{n}) = e^{\lambda w^2/2} - e^{\lambda w^2/2} \frac{p_{n-1}(r_n(1+w/\sqrt{n}))}{f(r_n(1+w/\sqrt{n}))} \bigl[ 1 + o(1) \bigr] + o(1)
\]
as $n \to \infty$ uniformly for $w$ restricted to compact subsets of $\re w < 0$. Substituting these into \eqref{secunbalcauchy_fnerror} produces the estimate
\begin{align*}
&\frac{p_{n-1}(r_n(1+w/\sqrt{n}))}{f(r_n(1+w/\sqrt{n}))} \bigl[ 1 + o(1) \bigr] \\
&\qquad = \frac{1}{2} \erfc\!\left(w\sqrt{\lambda/2}\right) - \frac{A\zeta^{1-n} r_n^{b-a}}{(\zeta-1-w/\sqrt{n}) \sqrt{2\pi \lambda n}} + o(r_n^{b-a} n^{-1/2}) + o(1)
\end{align*}
as $n \to \infty$ uniformly for $w$ restricted to compact subsets of $\re w < 0$. It follows that the formula
\[
	\lim_{n \to \infty} \frac{p_{n-1}(r_n(1+w/\sqrt{n}))}{f(r_n(1+w/\sqrt{n}))} = \frac{1}{2} \erfc\!\left(w\sqrt{\lambda/2}\,\right)
\]
holds if and only if $r_n^{b-a} n^{-1/2} = o(1)$, and since $r_n = (n/\lambda)^{1/\lambda}$ this is equivalent to the requirement that
\[
	\re b - \re a < \frac{\lambda}{2}.
\]
This completes the proof of Theorem \ref{secunbalcauchy_maintheo}.

\section{Generalization to More Directions of Maximal Growth}
\label{corscasec_many}

Unlike the scaling limit at the arcs of the limit curve (Section \ref{curscamanysec_sec}), generalizing the scaling limit at the corner of the curve is relatively straightforward.

Let $a,b_1,\ldots,b_m,A_1,\ldots,A_m \in \C$, $0 < \lambda < \infty$, $\mu < 1$, and $\zeta_1,\ldots,\zeta_m \in \C$ with $|\zeta_k| = 1$, $\zeta_k \neq 1$ for all $k = 1,\ldots,m$ and $\zeta_j \neq \zeta_k$ for $j \neq k$. Let $\theta \in (0,\pi)$ be small enough so that all of the sectors $|{\arg z}| \leq \theta$, $|{\arg(z/\zeta_k)}| \leq \theta$, $k=1,\ldots,m$, are disjoint.  We suppose that $f$ is an entire function with the asymptotic behavior
\begin{equation}
\label{secunbalcauchymany_fgrowth}
f(z) = \begin{cases}
	z^a \exp\!\left(z^\lambda\right) \bigl[1+o(1)\bigr] & \text{for } |{\arg z}| \leq \theta, \\
	A_1(z/\zeta_1)^{b_1} \exp\!\left((z/\zeta_1)^\lambda\right) \bigl[1+o(1)\bigr] & \text{for } |{\arg(z/\zeta_1)}| \leq \theta, \\
	\qquad \vdots \\
	A_m(z/\zeta_m)^{b_m} \exp\!\left((z/\zeta_m)^\lambda\right) \bigl[1+o(1)\bigr] & \text{for } |{\arg(z/\zeta_m)}| \leq \theta, \\
	O\!\left(\exp\!\left(\mu |z|^\lambda\right)\right) & \text{otherwise}
	\end{cases}			
\end{equation}
as $|z| \to \infty$, with each estimate holding uniformly in its sector. For this $f$, let $p_n(z)$, $r_n$, and $F_n(z)$ be defined as in Section \ref{secunbalcauchy_sec}. For convenience of notation, define $b_0 = a$, $A_0 = 1$, and $\zeta_0 = 1$.

For each new direction of maximal growth of $f$, equation \eqref{secunbalcauchy_fnerror} gains a corresponding term and error term. Explicitly, for $f$ as defined above, the equation becomes
\[
	F_n(z) = \frac{r_n^{-a}}{2\pi i} \int_{\gamma_1} \left(e^{1/\lambda}\right)^{-n}\! f(r_n s) \frac{ds}{s-z} + \sum_{k=1}^{m} \left[\frac{A_k\zeta_k^{1-n} r_n^{b_k-a}}{(\zeta_k-z) \sqrt{2\pi \lambda n}} + o(r_n^{b_k-a} n^{-1/2}) \right]
\]
as $n \to \infty$ uniformly for $z$ restricted to any sector $|{\arg z}| \leq \theta - \epsilon$ with $\epsilon > 0$. Continuing the argument in the proof of Theorem \ref{secunbalcauchy_maintheo} presents the conclusion that the desired scaling limit exists if and only if $r_n^{b_k-a} n^{-1/2} = o(1)$, $k = 1,\ldots,m$. In other words, we have the following result.

\begin{theorem}
\label{secunbalcauchymany_maintheo}
\[
	\lim_{n \to \infty} \frac{p_{n-1}(r_n(1+w/\sqrt{n}))}{f(r_n(1+w/\sqrt{n}))} = \frac{1}{2} \erfc\!\left(w\sqrt{\lambda/2}\,\right)
\]
uniformly for $w$ restricted to any compact subset of $\re w < 0$ if and only if
\[
	\re b_k - \re a < \frac{\lambda}{2}, \quad k = 1,\ldots,m.
\]
\end{theorem}

\begin{remark}
Theorem \ref{secunbalcauchymany_maintheo} gives us precise asymptotics for individual zeros of the scaled partial sums $p_{n-1}(r_n z)$ near the corner of the limit curve $S$ located at $z=1$ in the case where $f$ has multiple directions of maximal exponential growth as long as $\re b_k - \re a < \lambda/2$, $k=1,\ldots,m$. Details of this are given in Section \ref{corscasec_width}, where we use that information to verify part (b) of the Modified Saff-Varga Width Conjecture (see Section \ref{introsec_width}) for this class of functions.
\end{remark}

\section{Verification of the Modified Saff-Varga Width Conjecture at the Exceptional Argument $\arg z = 0$}
\label{corscasec_width}

The theorems in this chapter allow us to verify part (b) of the Modified Saff-Varga Width Conjecture (see Section \ref{introsec_width}) at the exceptional argument $\arg z = 0$ unconditionally when $f$ has a single direction of maximal exponential growth and on the condition that $\re b_k - \re a < \lambda/2$, $k = 1,\ldots,m$ when $f$ has two or more directions of maximal exponential growth.

Under these conditions Theorems \ref{seconeexp_maintheo}, \ref{secunbalcauchy_maintheo}, and \ref{secunbalcauchymany_maintheo} imply that if $w_0$ is any solution of the equation
\begin{equation}
\label{corscalwidtheq_erfceq}
	\erfc\!\left(w\sqrt{\lambda/2}\,\right) = 0
\end{equation}
then $p_{n-1}(z)$ has a zero $z_0$ of the form
\[
	z_0 = r_n + r_nw_0/\sqrt{n} + o\!\left(r_n n^{-1/2}\right)
\]
as $n \to \infty$. It follows that
\[
	z_0 - r_n \sim r_n w_0/\sqrt{n}
\]
as $n \to \infty$, and hence that, for any fixed $\epsilon > 0$, $z_0$ lies in the disk
\[
	|z - r_n| \leq r_n n^{-1/2 + \epsilon}
\]
for $n$ large enough. Since equation \eqref{corscalwidtheq_erfceq} has infinitely-many solutions, the number of zeros of $p_{n-1}(z)$ in any such disk tends to infinity as $n \to \infty$. Setting $\rho_n = r_n$, this is precisely the condition in part (b) of the Modified Saff-Varga Width Conjecture with $k=2$ at the exceptional argument $\arg z = 0$.

%% file: chap_applications.tex
\graphicspath{{images/chap_applications/}}

\chapter{Applications}
\label{chap_applications}

In this chapter we will apply the results from Chapters \ref{chap_limitcurves}, \ref{chap_curvescaling}, and \ref{chap_posfinite} to several common special functions. These functions were chosen to illustrate different behaviors of the zeros of the partial sums.

\section{The Sine and Cosine Functions}
\label{applicationssec_sincos}

Each of the functions
\[
	\sin z = \frac{e^{iz} - e^{-iz}}{2i} \qquad \text{and} \qquad \cos z = \frac{e^{iz} + e^{-iz}}{2}
\]
are entire of order $1$ and have two directions of maximal exponential growth, one as $|z| \to \infty$ with $\epsilon < {\arg z} < \pi - \epsilon$ and one as $|z| \to \infty$ with $-\pi+\epsilon < {\arg z} < -\epsilon$. Indeed, if $\theta \in (0,\pi/2)$ then
\[
	\begin{drcases} \left| \frac{\mp 2i \sin(\pm iz)}{e^z} - 1 \right| \\ \left| \frac{2 \cos(\pm iz)}{e^z} - 1 \right| \end{drcases} \leq e^{-2|z|\cos\theta}, \qquad |{\arg z}| \leq \theta
\]
and
\[
	\begin{drcases} | {\mp} 2i \sin(\pm iz) | \\ | 2 \cos(\pm iz) | \end{drcases} \leq 2 e^{|z|\cos\theta}, \qquad \theta \leq |{\arg z}| \leq \pi-\theta.
\]
So, for $\theta \in(0,\pi/2)$ and $\mu = \cos\theta$ we have
\[
	\mp 2i \sin(\pm iz) = \begin{cases}
	e^z \,[1+o(1)] & \text{for } |{\arg z}| \leq \theta, \\
	-e^{-z} \,[1+o(1)] & \text{for } |{\arg -z}| \leq \theta, \\
	O\!\left(e^{\mu |z|}\right) & \text{otherwise}
	\end{cases}
\]
and
\[
	2 \cos(\pm iz) = \begin{cases}
	e^z \,[1+o(1)] & \text{for } |{\arg z}| \leq \theta, \\
	e^{-z} \,[1+o(1)] & \text{for } |{\arg -z}| \leq \theta, \\
	O\!\left(e^{\mu |z|}\right) & \text{otherwise}
	\end{cases}
\]
as $|z| \to \infty$ uniformly in each of these sectors.

In the notation of the asymptotic assumption on the functions $f$ we have considered in the thesis (see, e.g., \eqref{seccurvetwo_fgrowth}), for each of the rotated and scaled functions $\mp 2i\sin(\pm iz)$ and $2\cos(\pm iz)$ and for $\theta \in (0,\pi/2)$ we have $a = b = 0$, $\lambda = 1$, $\mu = \cos\theta$, and $\zeta = -1$. For the sine functions we have $A=-1$ and for the cosines we have $A=1$.

\begin{remark}
Applying the results in this thesis to the functions $-2i\sin(-iz)$ and $2\cos(-iz)$ gives us information about the zeros of their partial sums in the upper half-plane, and applying them to the function $2i\sin(iz)$ and $2\cos(-iz)$ gives information about the lower half-plane. This is a very useful trick; by applying the results repeatedly to rotated versions of a function with multiple directions of maximal growth we can obtain information about its partial sums in all of its maximal growth sectors. This technique is used throughout this chapter.
\end{remark}

Applying Theorems \ref{seccurvetwo_maintheo}, \ref{curscatwotheo_maintheo}, and \ref{secunbalcauchy_maintheo} to these sine and cosine functions yields the following collections of results.

\begin{theorem}
Let
\[
	p_n[\sin](z) = \sum_{k=0}^{\lfloor (n-1)/2 \rfloor} \frac{(-1)^k z^{2k+1}}{(2k+1)!}
\]
denote the $n^\th$ partial sum of the Maclaurin series for $\sin z$ and let
\[
	S = \{ z \in \C : |z\exp(1-z)| = 1,\ |z| \leq 1, \text{ and } \re z > 0\}.
\]

The limit points of the zeros of the scaled partial sums $p_{n-1}[\sin](n z)$ which do not lie on the real axis are precisely the points of the set $iS \cup -iS$.

Let $\xi \in S$, $\xi \neq 1$ and define
\[
	\tau = \im(\xi - 1 - \log \xi),
\]
\[
	\tau_n \equiv \tau n \pmod{2\pi}, \quad -\pi < \tau_n \leq \pi,
\]
and
\[
	z_n(w) = \xi \left( 1 + \frac{\log n}{2(1-\xi)n} - \frac{w-i\tau_n}{(1-\xi)n} \right).
\]
Then
\begin{equation}
\label{applicationseq_sincurve}
	\frac{p_{n-1}[\sin](\pm in z_n(w))}{\sin(\pm in z_n(w))} = 1 - \left( \frac{1}{1-\xi} - \frac{(-1)^n}{1+\xi} \right) \frac{e^{-w}}{\sqrt{2\pi}} + o(1)
\end{equation}
as $n \to \infty$ uniformly on compact subsets of the $w$-plane.

Additionally,
\begin{equation}
\label{applicationseq_sincorner}
	\lim_{n \to \infty} \frac{p_{n-1}[\sin](\pm i (n+w\sqrt{n}))}{\sin(\pm i(n+w\sqrt{n}))} = \frac{1}{2} \erfc\!\left(\frac{w}{\sqrt{2}}\right)
\end{equation}
uniformly on compact subsets of $\re w < 0$.
\end{theorem}

\begin{theorem}
Let
\[
	p_n[\cos](z) = \sum_{k=0}^{\lfloor n/2 \rfloor} \frac{(-1)^k z^{2k}}{(2k)!}
\]
denote the $n^\th$ partial sum of the Maclaurin series for $\cos z$ and let
\[
	S = \{ z \in \C : |z\exp(1-z)| = 1,\ |z| \leq 1, \text{ and } \re z > 0\}.
\]

The limit points of the zeros of the scaled partial sums $p_{n-1}[\cos](n z)$ which do not lie on the real axis are precisely the points of the set $iS \cup -iS$.

Let $\xi \in S$, $\xi \neq 1$ and define
\[
	\tau = \im(\xi - 1 - \log \xi),
\]
\[
	\tau_n \equiv \tau n \pmod{2\pi}, \quad -\pi < \tau_n \leq \pi,
\]
and
\[
	z_n(w) = \xi \left( 1 + \frac{\log n}{2(1-\xi)n} - \frac{w-i\tau_n}{(1-\xi)n} \right).
\]
Then
\[
	\frac{p_{n-1}[\cos](\pm in z_n(w))}{\cos(\pm in z_n(w))} = 1 - \left( \frac{1}{1-\xi} + \frac{(-1)^n}{1+\xi} \right) \frac{e^{-w}}{\sqrt{2\pi}} + o(1)
\]
as $n \to \infty$ uniformly on compact subsets of the $w$-plane.

Additionally,
\[
	\lim_{n \to \infty} \frac{p_{n-1}[\cos](\pm i (n+w\sqrt{n}))}{\cos(\pm i(n+w\sqrt{n}))} = \frac{1}{2} \erfc\!\left(\frac{w}{\sqrt{2}}\right)
\]
uniformly on compact subsets of $\re w < 0$.
\end{theorem}

\begin{remark}
Due to the appearance of $(-1)^n$ in the scaling limits corresponding to the arcs of the limit curve we actually get two different limits if we restrict $n$ to run through only even or only odd integers. For example, the scaling limit for the sine function yields
\[
	\lim_{n \to \infty} \frac{p_{2n-1}[\sin](\pm i2n z_{2n}(w))}{\sin(\pm i2n z_{2n}(w))} = 1 - \left( \frac{1}{1-\xi} - \frac{1}{1+\xi} \right) \frac{e^{-w}}{\sqrt{2\pi}}
\]
and
\[
	\lim_{n \to \infty} \frac{p_{2n}[\sin](\pm i(2n+1) z_{2n+1}(w))}{\sin(\pm i(2n+1) z_{2n+1}(w))} = 1 - \left( \frac{1}{1-\xi} + \frac{1}{1+\xi} \right) \frac{e^{-w}}{\sqrt{2\pi}},
\]
each converging uniformly on compact subsets of the $w$-plane.
\end{remark}

\begin{figure}[!htb]
	\centering
	\begin{minipage}[c]{0.5\textwidth}
		\includegraphics[width=\textwidth]{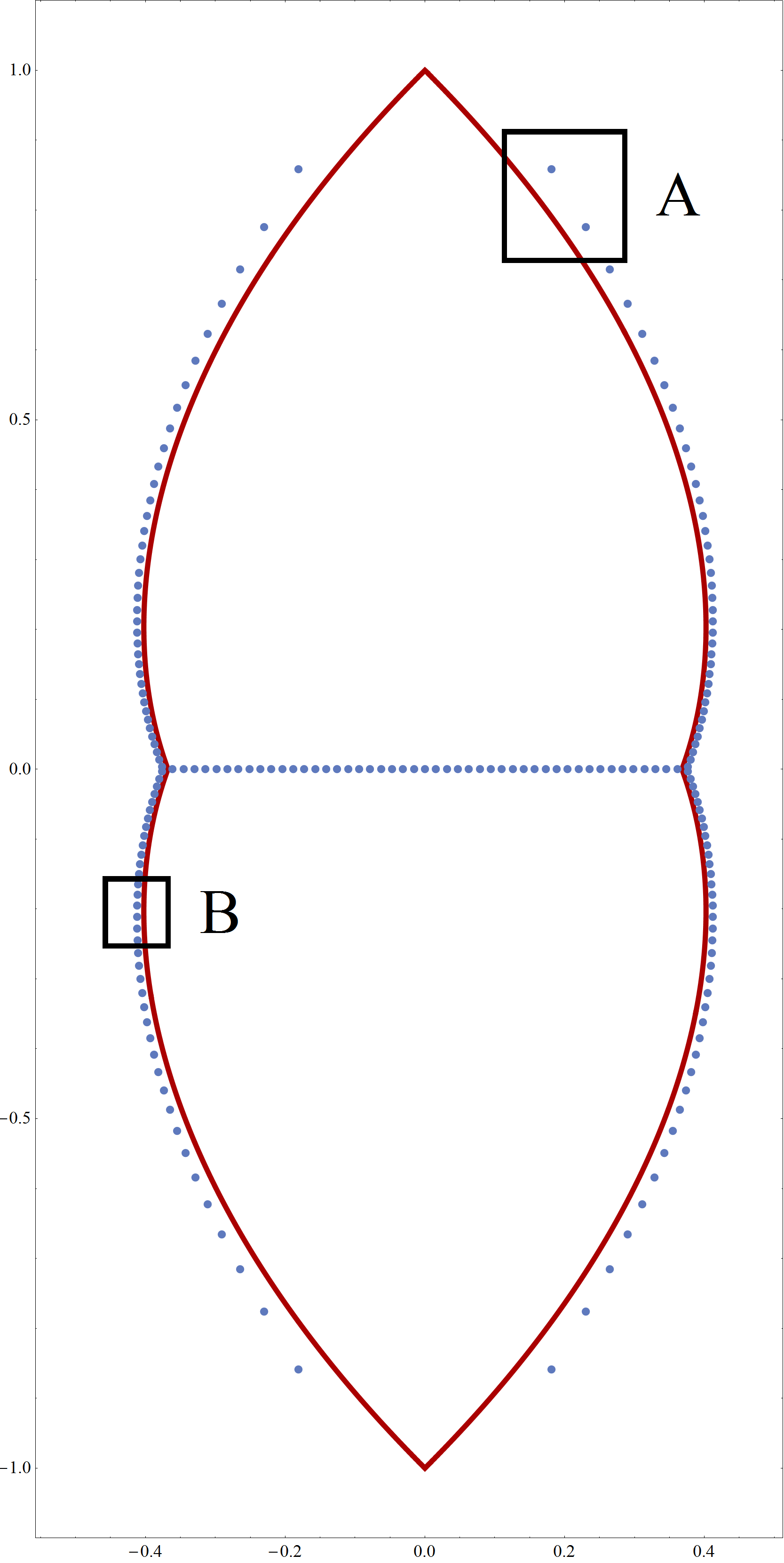}
	\end{minipage}
	\begin{minipage}[c]{0.44\textwidth}
		\begin{tabular}{c}
			\includegraphics[width=\textwidth]{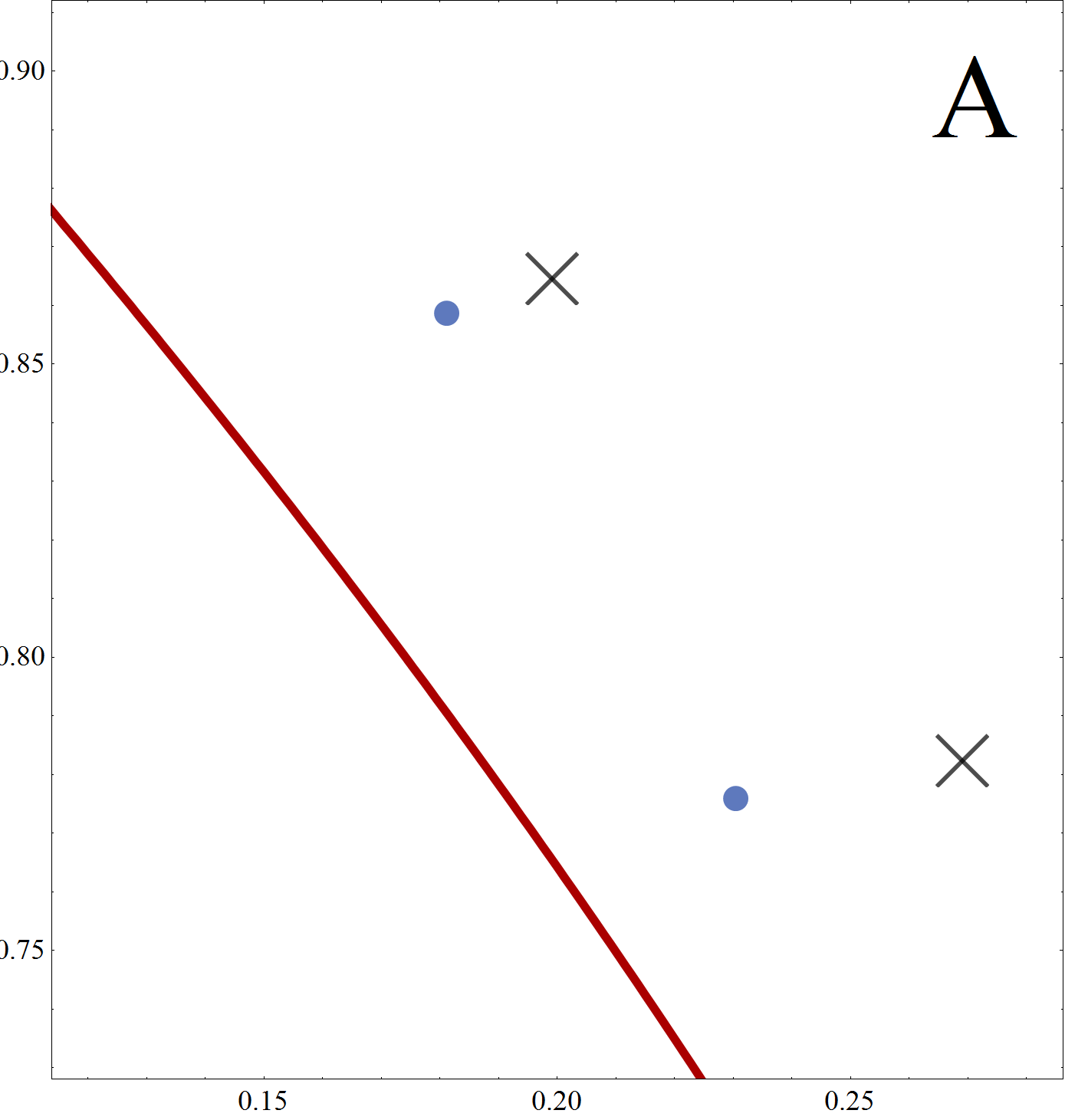} \\
			\includegraphics[width=\textwidth]{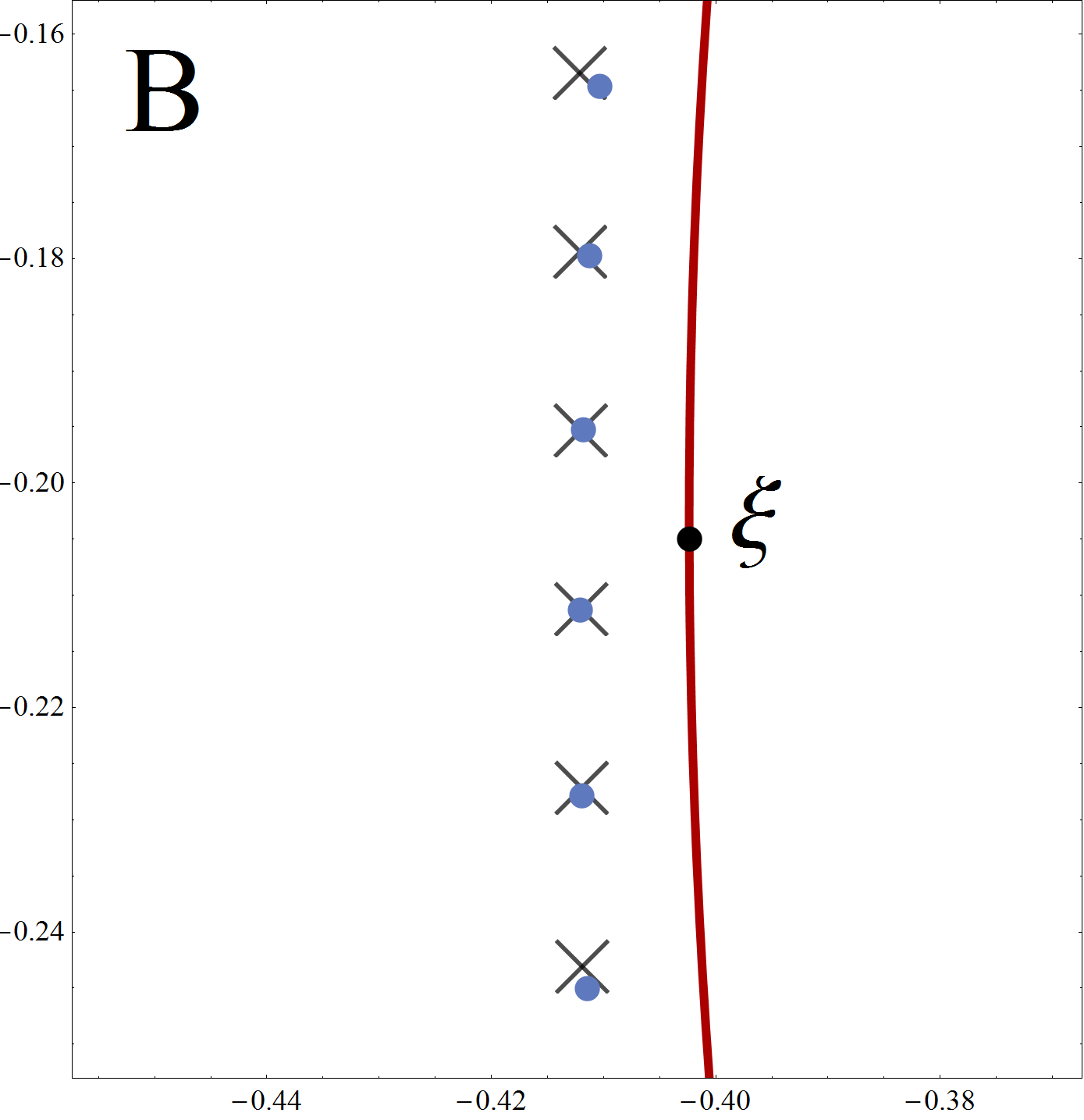}
		\end{tabular}
	\end{minipage}
	\caption[The zeros of {$p_{n-1}[\sin](nz)$} for {$n=200$} and their limit curve.]{The zeros of $p_{n-1}[\sin](nz)$ for $n=200$ in blue and their limit curve in red. Region A, magnified in the top-right, shows the approximations for the zeros in that region, represented as black crosses, which are given by the corner scaling limit in \eqref{applicationseq_sincorner}. Region B, magnified in the bottom-right, shows the approximations for the zeros in that region, represented as black crosses, which are given by the curve scaling limit in \eqref{applicationseq_sincurve}. The point $\xi$ which is used in \eqref{applicationseq_sincurve} to obtain the approximations in Region B is shown in black.}
	\label{applicationsfig_sin}
\end{figure}

\begin{remark}
The approximations in Region B in Figure \ref{applicationsfig_sin} appear much better than those in Region A. Indeed, since the sine function is so close to an exponential in those regions we expect from what is known about the zeros of the partial sums of the exponential function that the absolute error of the approximations is on the order of $(\log n)^2/n^2$ in Region B and on the order of $1/n$ in Region A. In fact, based on the asymptotic expansion in equation \eqref{introeq_rhzapprox} we expect that the two approximations in Region A have absolute errors of approximately $0.02$ and $0.04$, respectively, which agrees with what is shown in the plot.
\end{remark}

\section{Bessel Functions of the First Kind}
\label{applicationssec_bessel}

The Bessel functions of the first kind are defined by
\[
	J_\nu(z) = \left(\frac{z}{2}\right)^\nu \sum_{k=0}^{\infty} \frac{(-z^2/4)^k}{\Gamma(\nu+k+1) k!}
\]
for $\nu \in \C$, where a suitable branch cut is chosen for the factor $(z/2)^\nu$. The series in this definition converges for all $z \in \C$, so the function $(2/z)^\nu J_\nu(z)$ is entire. As such we define the new function
\[
	\mathbf J_\nu(z) = \left(\frac{2}{z}\right)^\nu J_\nu(z).
\]

According to NIST's Digital Library of Mathematical Functions \cite[eq.~10.17.3]{nist:dlmf} the Bessel function obeys the asymptotic
\[
	\mathbf J_\nu(z) = \frac{1}{\sqrt{\pi}} \left(\frac{2}{z}\right)^{\nu+1/2} \left[ \cos \omega \bigl[ 1 + o(1) \bigr] + O\!\left(\frac{\sin\omega}{\omega} \right) \right]
\]
as $|z| \to \infty$ uniformly in any sector $|{\arg z}| \leq \pi - \epsilon$ with $\epsilon > 0$, where
\[
	\omega = z - \frac{\pi\nu}{2} - \frac{\pi}{4}.
\]
As $\mathbf J_\nu(z)$ is even, the same asymptotic is valid for $\mathbf J_\nu(-z)$ in $|{\arg z}| \leq \pi - \epsilon$. It follows that for any $\theta \in (0,\pi/2)$ there is a $\mu < 1$ such that
\[
	\mathbf J_\nu(\pm iz) = 2^\nu \sqrt\frac{2}{\pi} \times \begin{cases}
	z^{-\nu-1/2} e^z \,[1+o(1)] & \text{for } |{\arg z}| \leq \theta, \\
	(-z)^{-\nu-1/2} e^{-z} \,[1+o(1)] & \text{for } |{\arg -z}| \leq \theta, \\
	O\!\left(e^{\mu |z|}\right) & \text{otherwise}
	\end{cases}
\]
as $|z| \to \infty$ uniformly in each of these sectors.

From this information we deduce that the function $\mathbf J_\nu$ has two directions of maximal exponential growth. In the notation of the asymptotic assumption on the functions $f$ we have considered in the thesis (see, e.g., \eqref{seccurvetwo_fgrowth}) we have $a = b = -\nu-1/2$, $A=1$ (after rescaling), $\lambda = 1$, and $\zeta = -1$. Applying Theorems \ref{seccurvetwo_maintheo}, \ref{curscatwotheo_maintheo}, and \ref{secunbalcauchy_maintheo} to these functions yields the following collection of results.

\begin{theorem}
Let
\[
	p_n[\mathbf J_\nu](z) = \sum_{k=0}^{\lfloor n/2 \rfloor} \frac{(-z^2/4)^k}{\Gamma(\nu+k+1) k!}
\]
denote the $n^\th$ partial sum of the Maclaurin series for $\mathbf J_\nu(z)$ and let
\[
	S = \{ z \in \C : |z\exp(1-z)| = 1,\ |z| \leq 1, \text{ and } \re z > 0\}.
\]

The limit points of the zeros of the scaled partial sums $p_{n-1}[\mathbf J_\nu](n z)$ which do not lie on the real axis are precisely the points of the set $iS \cup -iS$.

Let $\xi \in S$, $\xi \neq 1$ and define
\[
	\tau = \im(\xi - 1 - \log \xi),
\]
\[
	\tau_n \equiv \tau n \pmod{2\pi}, \quad -\pi < \tau_n \leq \pi,
\]
and
\[
	z_n(w) = \xi \left( 1 + \frac{\log n}{2(1-\xi)n} - \frac{w-i\tau_n}{(1-\xi)n} \right).
\]
Then
\begin{equation}
\label{applicationseq_besselcurve}
	\frac{p_{n-1}[\mathbf J_\nu](\pm in z_n(w))}{\mathbf J_\nu(\pm in z_n(w))} = 1 - \left( \frac{1}{1-\xi} + \frac{(-1)^n}{1+\xi} \right) \frac{\xi^{\nu+1/2} e^{-w}}{\sqrt{2\pi}} + o(1)
\end{equation}
as $n \to \infty$ uniformly on compact subsets of the $w$-plane.

Additionally,
\begin{equation}
\label{applicationseq_besselcorner}
	\lim_{n \to \infty} \frac{p_{n-1}[\mathbf J_\nu](\pm i (n+w\sqrt{n}))}{\mathbf J_\nu(\pm i(n+w\sqrt{n}))} = \frac{1}{2} \erfc\!\left(\frac{w}{\sqrt{2}}\right)
\end{equation}
uniformly on compact subsets of $\re w < 0$.
\end{theorem}

\begin{remark}
Just as in the case of the sine and cosine functions, the scaling limit corresponding to the arcs of the limit curve gives two different limits if we restrict $n$ to run through only even or only odd integers.
\end{remark}

\begin{figure}[!htb]
	\centering
	\begin{minipage}[c]{0.5\textwidth}
		\includegraphics[width=\textwidth]{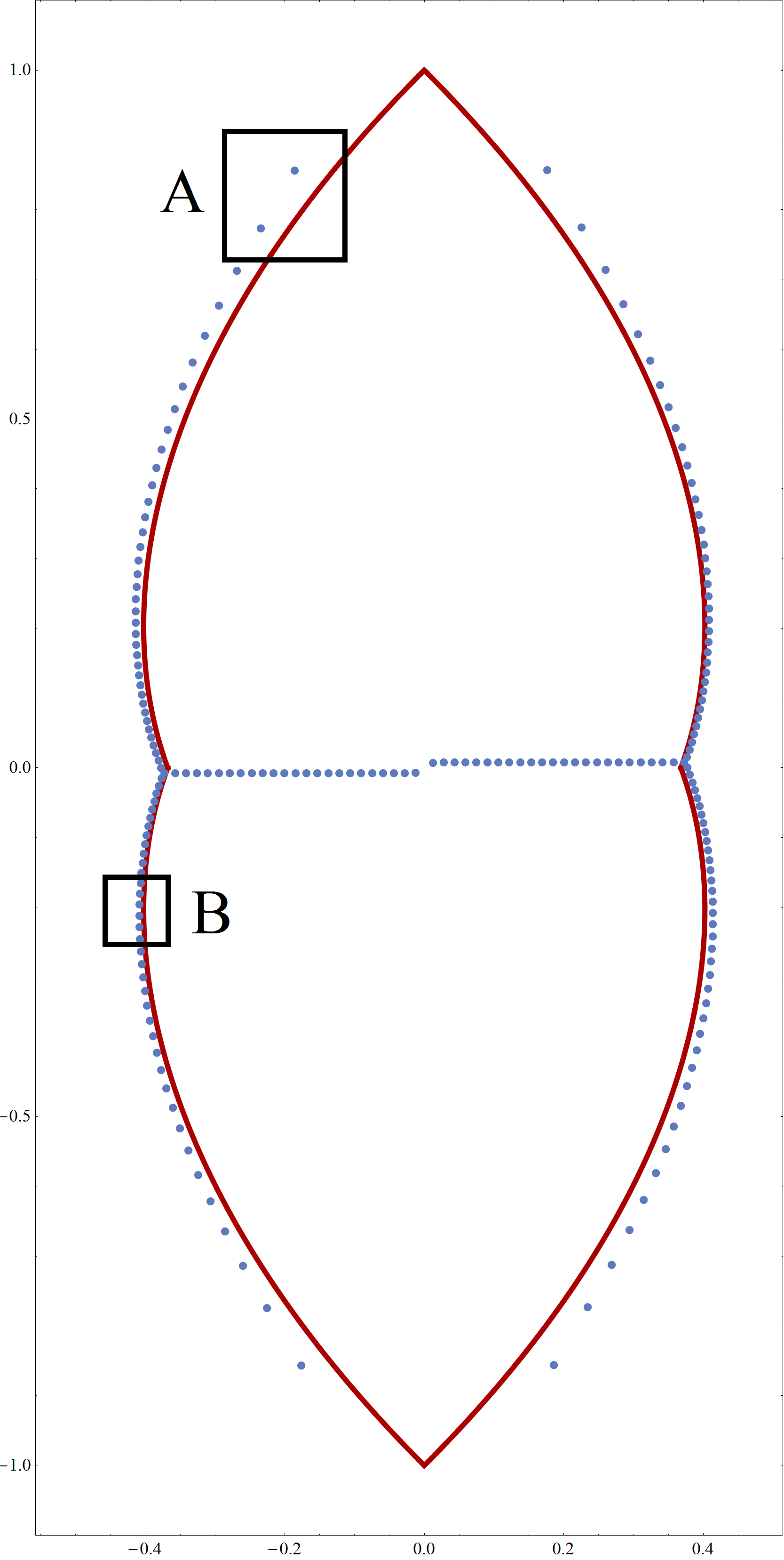}
	\end{minipage}
	\begin{minipage}[c]{0.44\textwidth}
		\begin{tabular}{c}
			\includegraphics[width=\textwidth]{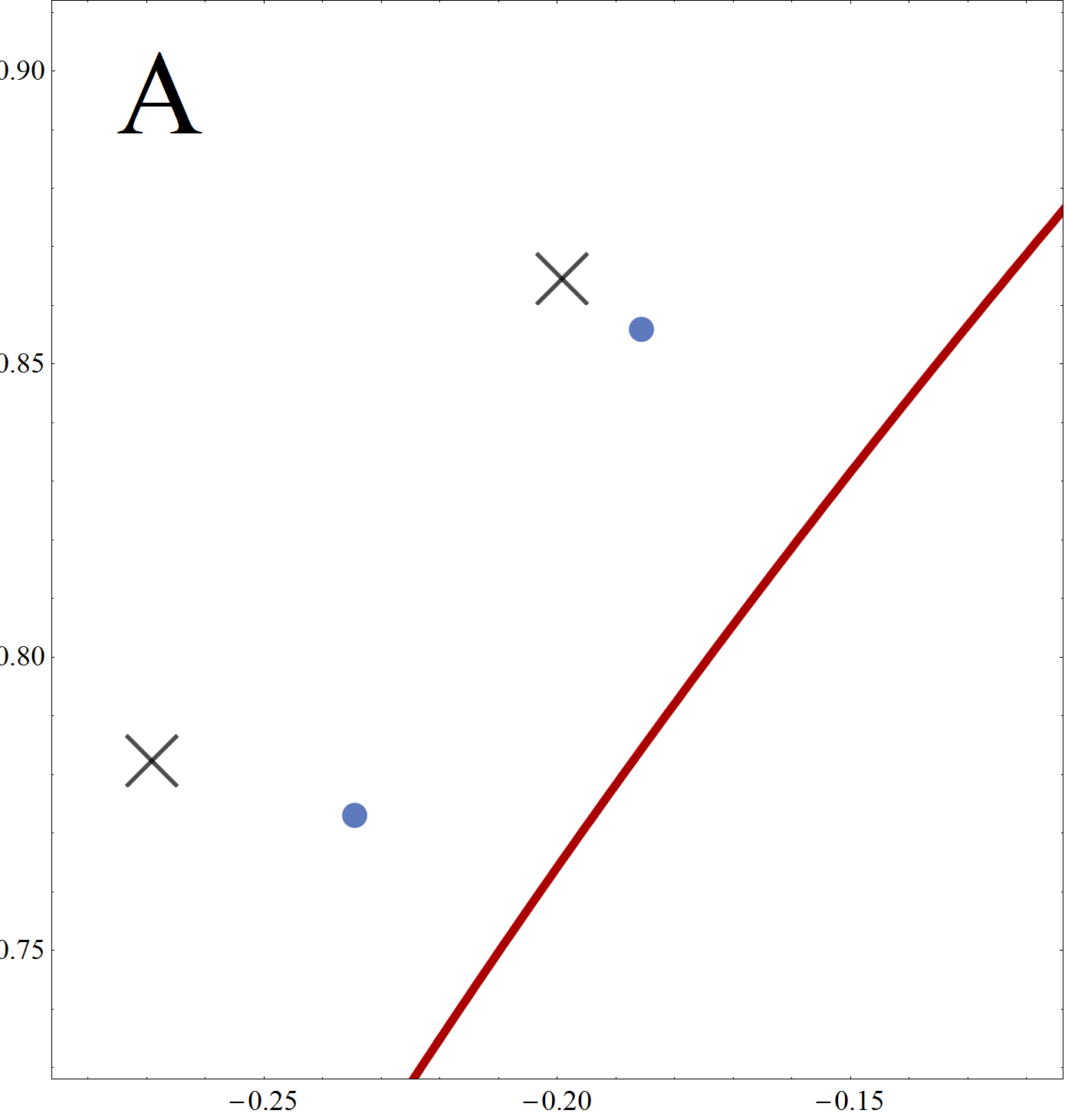} \\
			\includegraphics[width=\textwidth]{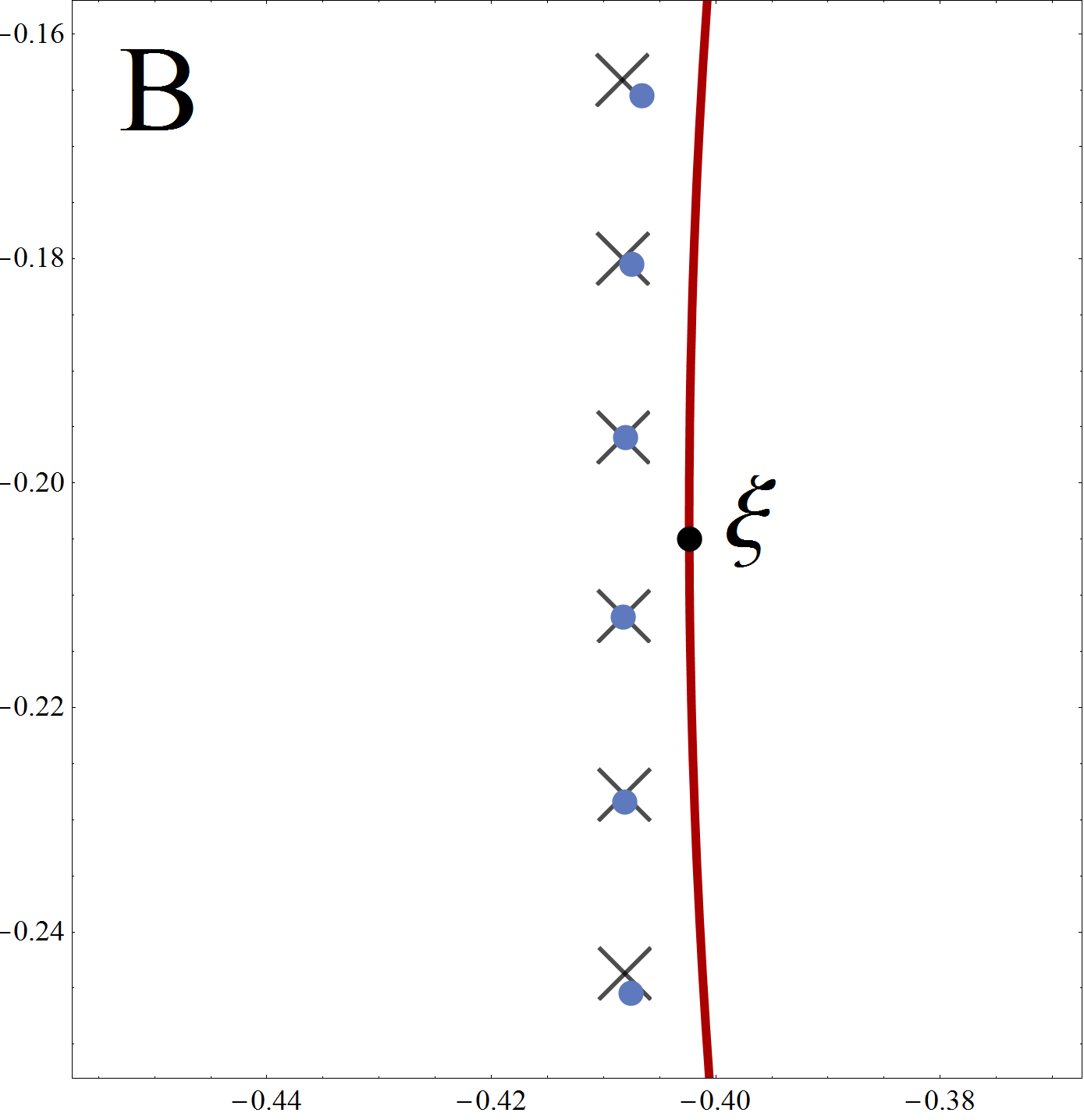}
		\end{tabular}
	\end{minipage}
	\caption[The zeros of {$p_{n-1}[\mathbf J_\nu](nz)$} for {$\nu = i$} and {$n=200$} and their limit curve.]{The zeros of $p_{n-1}[\mathbf J_\nu](nz)$ for $\nu = i$ and $n=200$ in blue and their limit curve in red. Region A, magnified in the top-right, shows the approximations for the zeros in that region, represented as black crosses, which are given by the corner scaling limit in \eqref{applicationseq_besselcorner}. Region B, magnified in the bottom-right, shows the approximations for the zeros in that region, represented as black crosses, which are given by the curve scaling limit in \eqref{applicationseq_besselcurve}. The point $\xi$ which is used in \eqref{applicationseq_besselcurve} to obtain the approximations in Region B is shown in black.}
	\label{applicationsfig_bessel}
\end{figure}

\section{Confluent Hypergeometric Functions}
\label{applicationssec_confluent}

In this section we will consider the functions
\[
	\M(\alpha,\beta,z) = \frac{1}{\Gamma(\alpha)} \sum_{k=0}^{\infty} \frac{\Gamma(k+\alpha)}{\Gamma(k+\beta) k!} z^k,
\]
where $\alpha,\beta \in \C$ and $\alpha \neq 0,-1,-2,\ldots$. The series converges for all $z \in \C$, so these functions are entire. They are related to the usual hypergeometric ${}_1F_1$ functions by
\[
	\frac{{}_1F_1(\alpha,\beta,z)}{\Gamma(\beta)} = \M(\alpha,\beta,z)
\]
for fixed $\beta \neq 0,-1,2,\ldots$, as well as in the limit $\beta \to -m$, $m = 0,1,2,\ldots$.

The Digital Library of Mathematical Functions (or the DLMF) gives the following asymptotic for $\M$:
\[
	\M(\alpha,\beta,z) = \frac{z^{\alpha-\beta} e^z}{\Gamma(\alpha)} \bigl[ 1 + o(1) \bigr] + \frac{e^{\pm i \pi \alpha} z^{-\alpha}}{\Gamma(\beta-\alpha)} \bigl[ 1 + o(1) \bigr]
\]
as $|z| \to \infty$ uniformly in any sector $-\pi/2 + \epsilon \leq \pm {\arg z} \leq 3\pi/2 - \epsilon$ with $\epsilon > 0$ for appropriate choices of branches for $z^{\alpha-\beta}$ and $z^{-\alpha}$ and an appropriate determination of $\arg z$ \cite[eq.~13.7.2]{nist:dlmf}. It follows that for any $\theta \in (0,\pi/2)$ there exists a constant $\mu < 1$ such that
\[
	\Gamma(\alpha) \M(\alpha,\beta,z) = \begin{cases}
		z^{\alpha-\beta} e^z \,[1+o(1)] & \text{for } |{\arg z}| \leq \theta, \\
		O\!\left(e^{\mu |z|}\right) & \text{for } |{\arg z}| > \theta
	\end{cases}
\]
as $|z| \to \infty$ uniformly in each of these sectors.

From the above information we deduce that the function $\M$ has a single direction of maximal exponential growth. In the notation of the asymptotic assumption on the functions $f$ we have considered in the thesis (see, e.g., \eqref{seccurve1_fgrowth}) we have $a = \alpha-\beta$, $b=0$, and $\lambda = 1$. Applying Theorems \ref{seccurve1_maintheo}, \ref{curscaonetheo_maintheo}, and \ref{seconeexp_maintheo} to these functions yields the following collection of results.

\begin{theorem}
Let
\[
	p_n[\M](z) = \frac{1}{\Gamma(\alpha)} \sum_{k=0}^{n} \frac{\Gamma(k+\alpha)}{\Gamma(k+\beta) k!} z^k
\]
denote the $n^\th$ partial sum of the Maclaurin series for $\M(\alpha,\beta,z)$ and let
\[
	S = \{ z \in \C : |z\exp(1-z)| = 1,\ |z| \leq 1, \text{ and } \re z > 0\}.
\]

The limit points of the zeros of the scaled partial sums $p_{n-1}[\M](n z)$ in the right half-plane are precisely the points of $S$.

Let $\xi \in S$, $\xi \neq 1$ and define
\[
	\tau = \im(\xi - 1 - \log \xi),
\]
\[
	\tau_n \equiv \tau n \pmod{2\pi}, \quad -\pi < \tau_n \leq \pi,
\]
and
\[
	z_n(w) = \xi \left( 1 + \frac{\log n}{2(1-\xi)n} - \frac{w-i\tau_n}{(1-\xi)n} \right).
\]
Then
\begin{equation}
\label{applicationseq_confluentcurve}
	\lim_{n \to \infty} \frac{p_{n-1}[\M](n z_n(w))}{\M(\alpha,\beta,n z_n(w))} = 1 - \frac{e^{-w}}{\xi^{\alpha-\beta}(1-\xi)\sqrt{2\pi}}
\end{equation}
uniformly on compact subsets of the $w$-plane.

Additionally,
\begin{equation}
\label{applicationseq_confluentcorner}
	\lim_{n \to \infty} \frac{p_{n-1}[\M](n+w\sqrt{n})}{\M(\alpha,\beta,n+w\sqrt{n})} = \frac{1}{2} \erfc\!\left(\frac{w}{\sqrt{2}}\right)
\end{equation}
uniformly on compact subsets of $\re w < 0$.
\end{theorem}

\begin{figure}[!htb]
	\centering
	\includegraphics[width=0.9\textwidth]{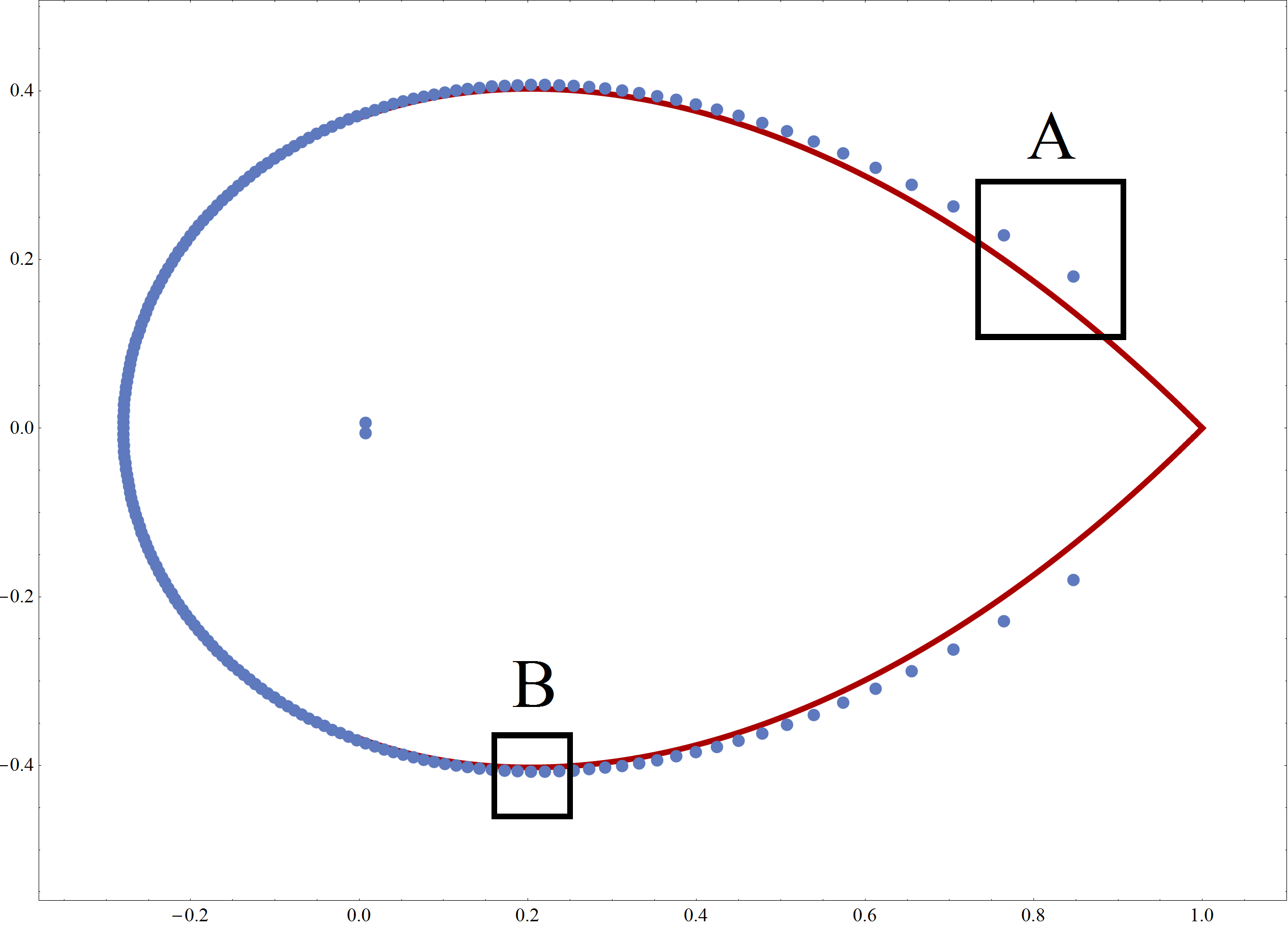} \\
	\vspace{0.1cm}
	\begin{tabular}{cc}
		\includegraphics[width=0.45\textwidth]{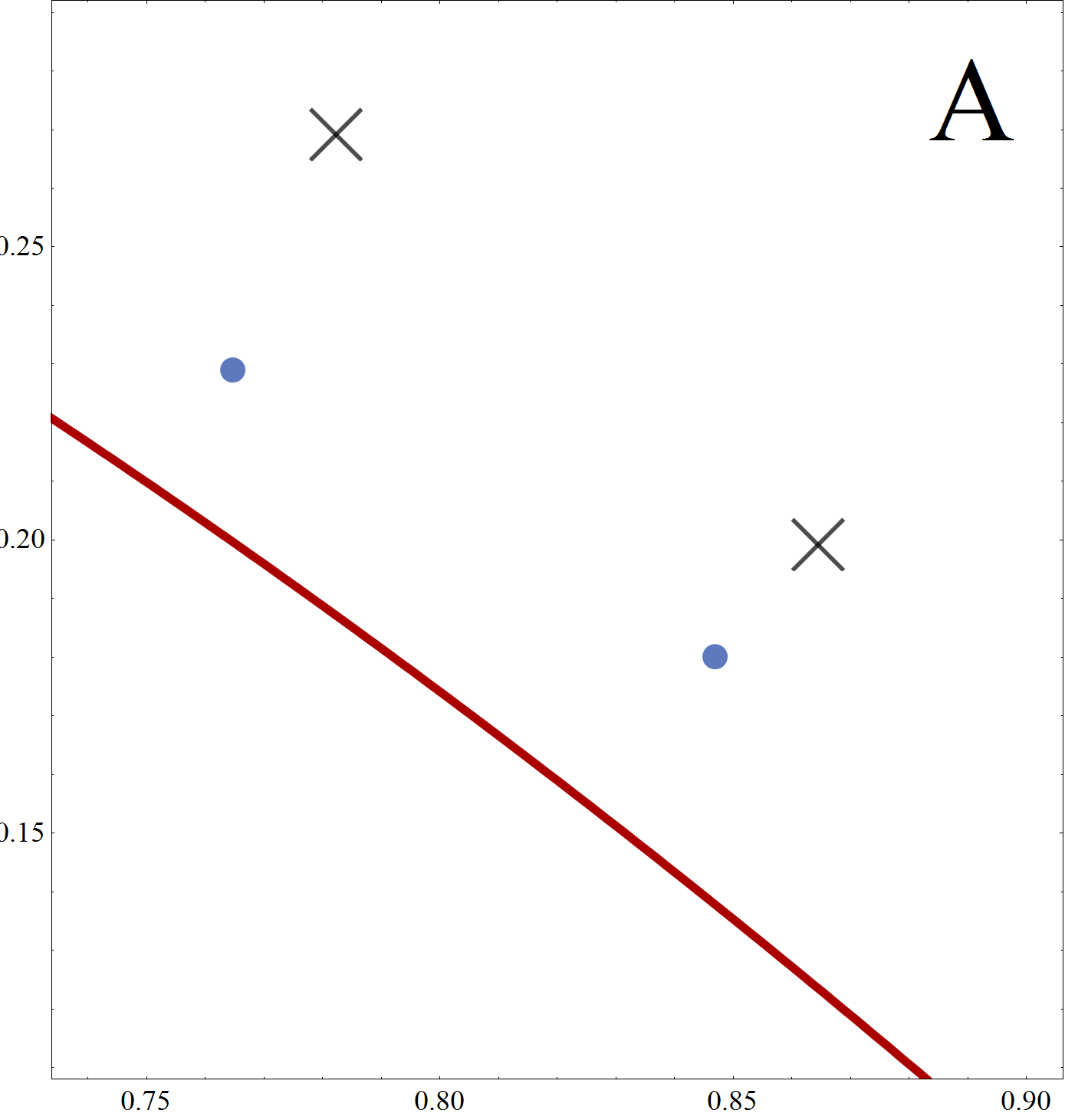}
		& \includegraphics[width=0.45\textwidth]{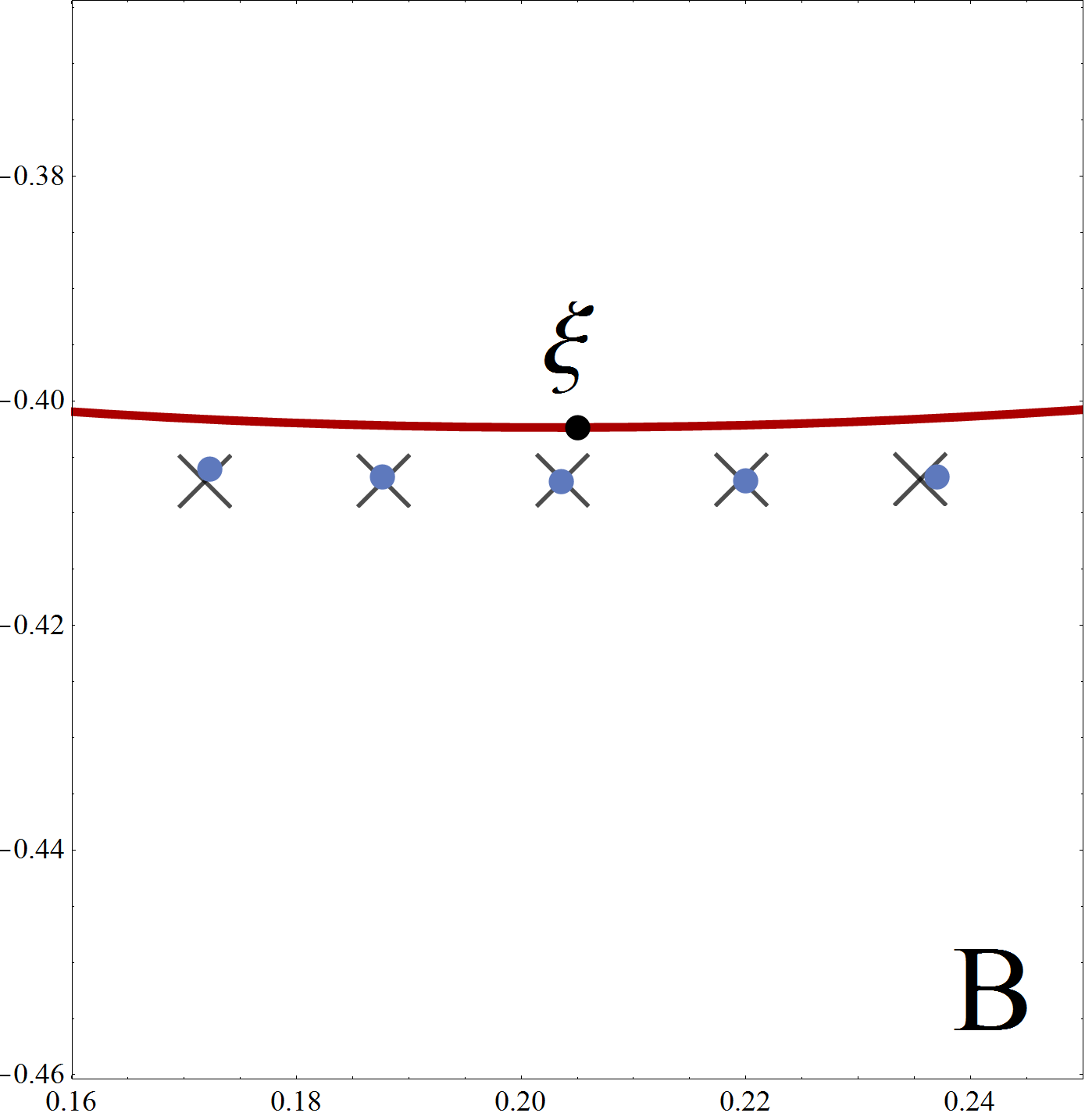}
	\end{tabular}
	\caption[The zeros of {$p_{n-1}[\M](nz)$} for {$\alpha = -1/2$}, {$\beta = -5/2$}, and {$n=200$} and their limit curve in the right half-plane.]{The zeros of $p_{n-1}[\M](nz)$ for $\alpha = -1/2$, $\beta = -5/2$, and $n=200$ in blue and their limit curve in the right half-plane in red. Region A, magnified in the bottom-left, shows the approximations for the zeros in that region, represented as black crosses, which are given by the corner scaling limit in \eqref{applicationseq_confluentcorner}. Region B, magnified in the bottom-right, shows the approximations for the zeros in that region, represented as black crosses, which are given by the curve scaling limit in \eqref{applicationseq_confluentcurve}. The point $\xi$ which is used in \eqref{applicationseq_confluentcurve} to obtain the approximations in Region B is shown in black.}
	\label{applicationsfig_confluent}
\end{figure}

\section{Exponential Integrals}
\label{applicationssec_expints}

Let $-1 \leq r < 1$ and let $g\colon [r,1] \to \C\cup\{\infty\}$ be a measurable, integrable function satisfying
\[
	g(t) = (t-r)^p g_1(t-r) = (1-t)^q g_2(1-t),
\]
where
\begin{enumerate}[label=(\arabic*)]
\item $\re p > -1$ and $\re q > -1$,
\item $g_1(0)$ and $g_2(0)$ are both finite and nonzero, and
\item in a neighborhood of $t=0$, both $g_1'(t)$ and $g_2'(t)$ exist and are bounded.
\end{enumerate}

Define the function
\begin{equation}
\label{applicationseq_expintdef}
	f(z) = \int_r^1 e^{zt} g(t)\,dt.
\end{equation}
Under the above conditions this $f$ is an entire function of order $1$.

By Watson's lemma (see Theorem \ref{watsonslemma} or \cite[Secs.\ 2.2 and 2.3]{miller:aaa}) the function $f$ obeys the asymptotics
\[
	f(z) \sim g_2(0) \Gamma(q+1) z^{-q-1} e^z
\]
and
\[
	f(-z) \sim g_1(0) \Gamma(p+1) z^{-p-1} e^{-rz}
\]
as $|z| \to \infty$ uniformly in any sector $|{\arg z}| \leq \theta$ with $0 \leq \theta < \pi/2$. Additionally, if $|{\arg z}| \leq \pi/2 - \theta$ then
\begin{align*}
|f(\pm i z)| &\leq \int_r^1 e^{\mp t \im z} |g(t)|\,dt \\
	&= \int_r^1 e^{\mp |z| t \sin \arg z} |g(t)|\,dt \\
	&\leq \int_r^1 e^{|zt| \cos\theta} |g(t)|\,dt \\
	&\leq e^{|z|\cos\theta} \int_r^1 |g(t)|\,dt.
\end{align*}
It follows that for any fixed $\theta \in (0,\pi/2)$, if $-1 < r < 1$ then
\begin{equation}
\label{applicationseq_expintasymprleq1}
	[g_2(0) \Gamma(q+1)]^{-1} f(z) = \begin{cases}
		z^{-q-1} e^z \,[1+o(1)] & \text{for } |{\arg z}| \leq \theta, \\
		O\!\left(e^{\mu |z|}\right) & \text{for } |{\arg z}| > \theta
	\end{cases}
\end{equation}
and if $r = -1$ then
\begin{equation}
\label{applicationseq_expintasympreq1}
	[g_2(0) \Gamma(q+1)]^{-1} f(z) = \begin{cases}
	z^{-q-1} e^z \,[1+o(1)] & \text{for } |{\arg z}| \leq \theta, \\
	A (-z)^{-p-1} e^{-z} \,[1+o(1)] & \text{for } |{\arg -z}| \leq \theta, \\
	O\!\left(e^{\mu |z|}\right) & \text{otherwise}
	\end{cases}
\end{equation}
as $|z| \to \infty$ uniformly in each sector, where $\mu < 1$ and
\[
	A = \frac{g_1(0)\Gamma(p+1)}{g_2(0)\Gamma(q+1)}.
\]
Applying Theorems \ref{seccurve1_maintheo}, \ref{curscaonetheo_maintheo}, \ref{seconeexp_maintheo}, \ref{seccurvetwo_maintheo}, \ref{curscatwotheo_maintheo}, and \ref{secunbalcauchy_maintheo} to this function yields the following collection of results.

\begin{theorem}
\label{applicationsthm_expint}
Let
\[
	p_n[f](z) = \sum_{k=0}^{n} \frac{z^k}{k!} \int_r^1 t^k g(t)\,dt
\]
denote the $n^\th$ partial sum of the Maclaurin series for $f(z)$ and let
\[
	S = \{ z \in \C : |z\exp(1-z)| = 1,\ |z| \leq 1, \text{ and } \re z > 0\}.
\]

If $-1 < r < 1$ then the limit points of the zeros of the scaled partial sums $p_{n-1}[f](n z)$ in the right half-plane are precisely the points of $S$. If $r = -1$, the limit points which do not lie on the imaginary axis are precisely the points of the set $S \cup -S$.

Let $\xi \in S$, $\xi \neq 1$ and define
\[
	\tau = \im(\xi - 1 - \log \xi).
\]
Define the sequence $\tau_n$ by the conditions
\[
	\tau n \equiv \tau_n \pmod{2\pi}, \qquad -\pi < \tau_n \leq \pi
\]
and let
\[
	z_n^1(w) = \xi \left(1 + \frac{\log n}{2(1-\xi)n} - \frac{w-i\tau_n}{(1-\xi) n}\right).
\]
Define the sequence $\sigma_n$ by the conditions
\[
	\pi n \equiv \sigma_n \pmod{2\pi}, \qquad -\pi < \sigma_n \leq \pi
\]
and let
\[
	z_n^2(w) = \xi \left[ 1 + \left(p-q+\frac{1}{2}\right)\frac{\log n}{(1-\xi) n} - \frac{w - i\sigma_n - i\tau_n}{(1-\xi) n} \right].
\]

If $-1 < r < 1$ or if $r = -1$ and $\re q < \re p$ then
\begin{equation}
\label{applicationseq_expintcurve1}
	\lim_{n \to \infty} \frac{p_{n-1}[f](n z_n^1(w))}{f(n z_n^1(w))} = 1 - \frac{\xi^{q+1} e^{-w}}{(1-\xi)\sqrt{2\pi}}
\end{equation}
as $n \to \infty$ uniformly on compact subsets of the $w$-plane. When $r = -1$, if $\re q > \re p$ then
\begin{equation}
\label{applicationseq_expintcurve2}
	\lim_{n \to \infty} \frac{p_{n-1}[f](n z_n^2(w))}{f(n z_n^2(w))} = 1 - \frac{A \xi^{q+1} e^{-w}}{(1+\xi)\sqrt{2\pi}}
\end{equation}
and if $\re q = \re p$ then
\[
	\frac{p_{n-1}[f](n z_n^1(w))}{f(n z_n^1(w))} = 1 - \left(\frac{1}{1-\xi} + \frac{A(-1)^{n} n^{q-p}}{1+\xi} \right) \frac{\xi^{q+1} e^{-w}}{\sqrt{2\pi}} + o(1)
\]
as $n \to \infty$. Both of these limits are uniform with respect to $w$ on compact subsets of the $w$-plane.

Additionally, if $-1 < r < 1$ then
\begin{equation}
\label{applicationseq_expintcorner}
	\lim_{n \to \infty} \frac{p_{n-1}[f](n+w\sqrt{n})}{f(n+w\sqrt{n})} = \frac{1}{2} \erfc\!\left(\frac{w}{\sqrt{2}}\right)
\end{equation}
uniformly on compact subsets of $\re w < 0$. This is also true when $r = -1$ if and only if
\[
	\re p - \re q + \frac{1}{2} > 0.
\]
\end{theorem}

\begin{figure}[!htb]
	\centering
	\includegraphics[width=0.95\textwidth]{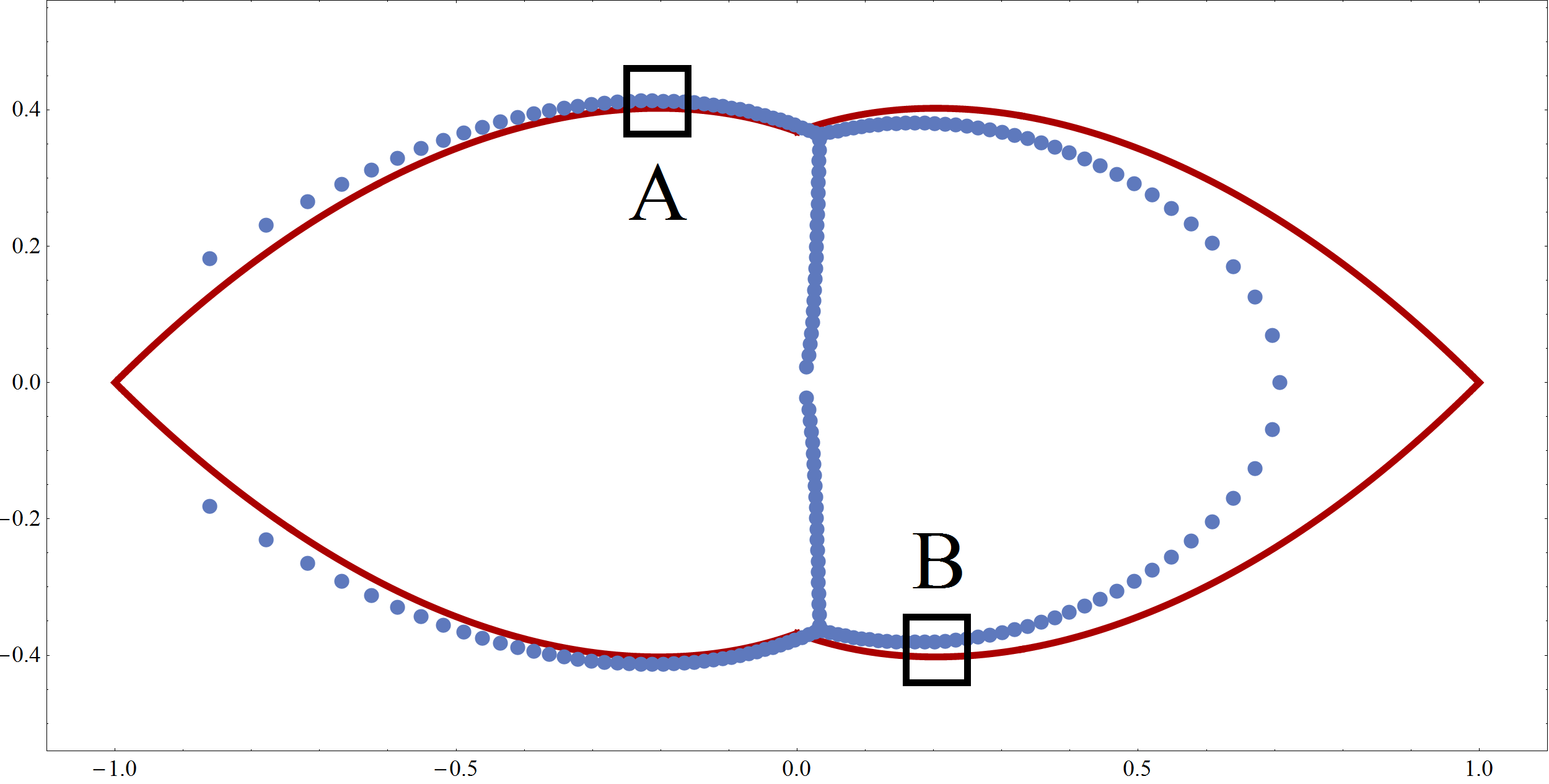} \\
	\vspace{0.1cm}
	\begin{tabular}{cc}
		\includegraphics[width=0.45\textwidth]{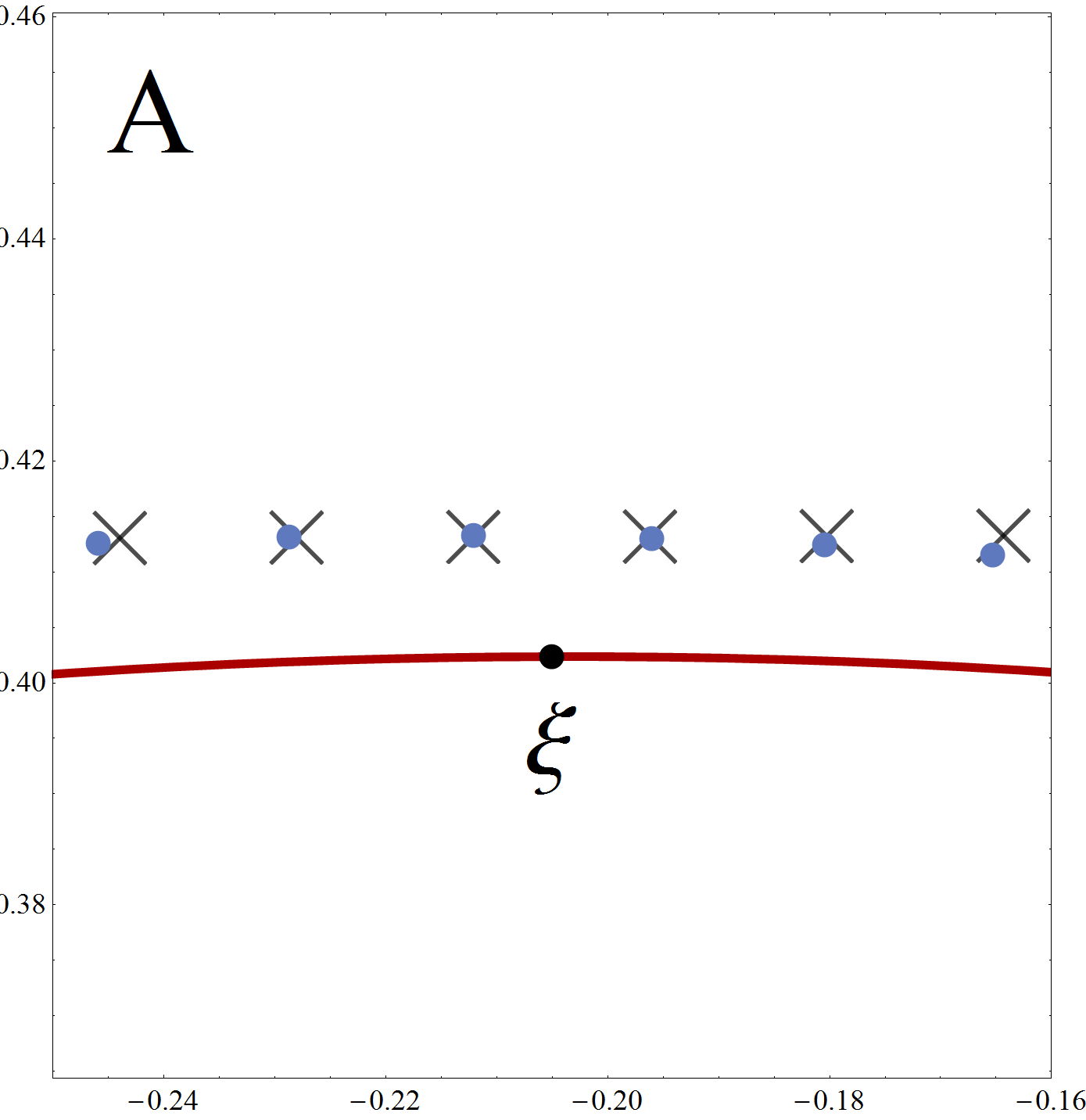}
		& \includegraphics[width=0.45\textwidth]{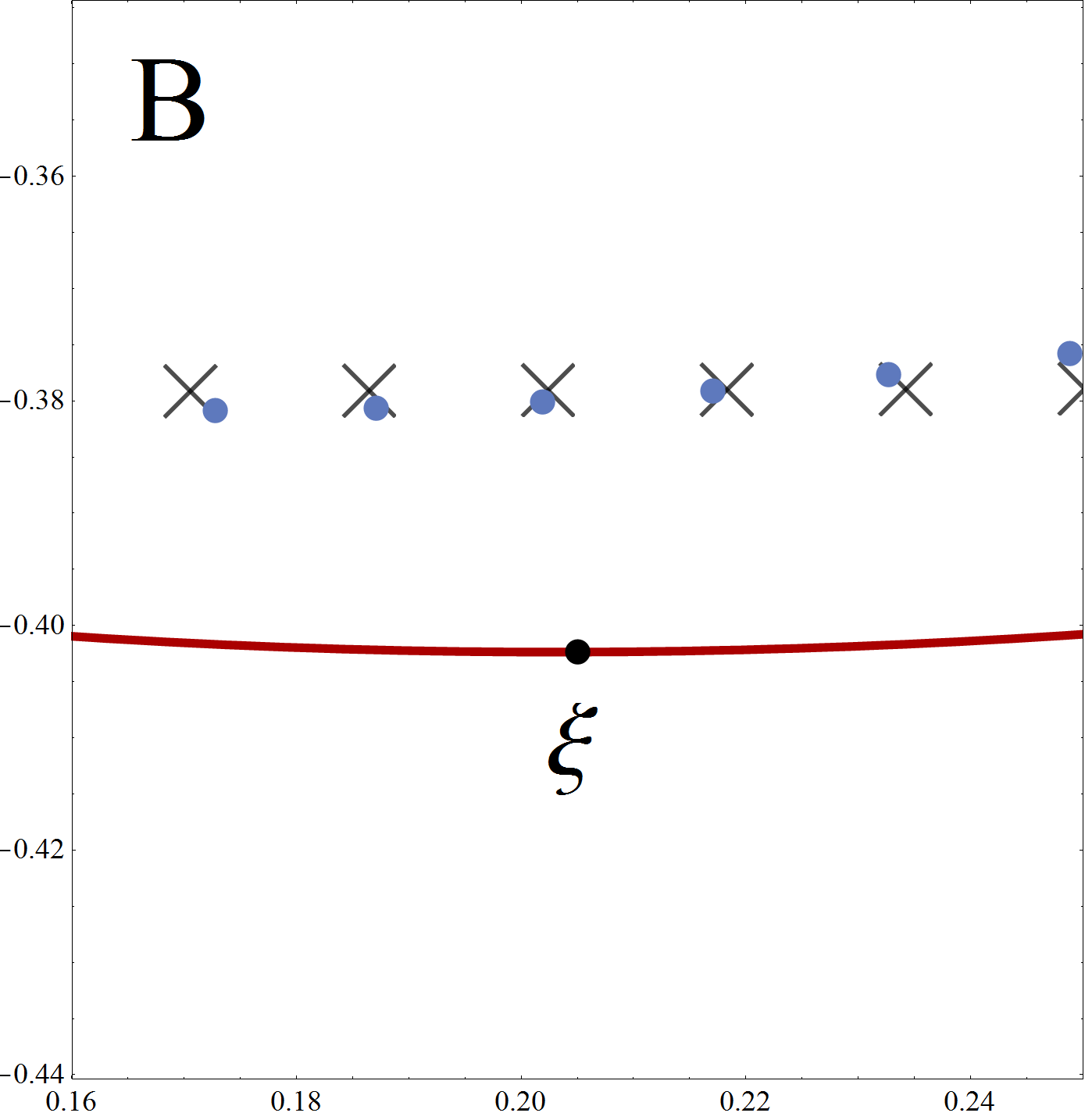}
	\end{tabular}
	\caption[The zeros of {$p_{n-1}[f](nz)$} for {$g(t) = (1-t)^3$}, {$r=-1$}, and {$n=200$} and their limit curve.]{The zeros of $p_{n-1}[f](nz)$ for $g(t) = (1-t)^3$, $r=-1$, and $n=200$ in blue and their limit curve in red. The magnified regions A and B show the approximations for the zeros in those regions, represented as black crosses, which are given by the curve scaling limits in \eqref{applicationseq_expintcurve1} and \eqref{applicationseq_expintcurve2}, respectively. Note that to apply limit \eqref{applicationseq_expintcurve1} we applied Theorem \ref{applicationsthm_expint} to the reflected function $f(-z)$. The points $\xi$ which are used in the respective applications of \eqref{applicationseq_expintcurve1} and \eqref{applicationseq_expintcurve2} are shown in black. The corner scaling limit \eqref{applicationseq_expintcorner} does not exist for this function $f$ so we do not get approximations for the zeros which approach the point $z=1$.}
	\label{applicationsfig_expint}
\end{figure}

\section{Airy Functions}
\label{applicationssec_airy}

The Airy functions of the first and second kind are entire functions defined for $z \in \C$ by the integrals
\[
	\Ai(z) = \frac{1}{2\pi i} \int_{e^{-i\pi/3} \infty}^{e^{i\pi/3} \infty} \exp\!\left(\frac{1}{3}t^3 - zt\right)dt
\]
and
\[
	\Bi(z) = \frac{1}{2\pi} \int_{-\infty}^{e^{i\pi/3} \infty} \exp\!\left(\frac{1}{3}t^3 - zt\right)dt + \frac{1}{2\pi} \int_{-\infty}^{e^{-i\pi/3} \infty} \exp\!\left(\frac{1}{3}t^3 - zt\right)dt,
\]
respectively.

The DLMF gives the following uniform asymptotics for $\Ai$ \cite[eqns.\ 9.7.5 and 9.7.9]{nist:dlmf}: for any fixed $\epsilon > 0$,
\[
	\Ai(z) \sim \frac{z^{-1/4}}{2\sqrt{\pi}} \exp\!\left(-\frac{2}{3} z^{3/2}\right)
\]
as $|z| \to \infty$ with $|{\arg z}| \leq \pi - \epsilon$ and
\[
	\Ai(-z) = \frac{z^{-1/4}}{\sqrt{\pi}} \left[ \cos \omega \bigl[1+o(1)\bigr] + O\!\left(\frac{\sin \omega}{\omega}\right) \right]
\]
as $|z| \to \infty$ with $|{\arg z}| \leq 2\pi/3 - \epsilon$, where
\[
	\omega = \frac{2}{3} z^{3/2} - \frac{\pi}{4}.
\]
It follows that for any $\epsilon > 0$ there is some $\mu < 2/3$ such that
\[
	\Ai(z) = \frac{1}{2\sqrt{\pi}} \times \begin{cases}
		z^{-1/4} \exp\!\left(-\frac{2}{3} z^{3/2}\right) [1+o(1)] & \text{for } |{\arg z}| \leq \pi - \epsilon, \\
		O\!\left(\exp\!\left(\mu|z|^{3/2}\right)\right) & \text{for } |{\arg z}| \geq \pi-\epsilon
	\end{cases}
\]
as $|z| \to \infty$ uniformly in each sector. The quantity $\re(-z^{3/2})$ is maximal on the rays $\arg z = \pm 2\pi/3$, so we conclude that $\Ai$ has two directions of maximal exponential growth.

We could now apply Theorems \ref{seccurvetwo_maintheo}, \ref{curscatwotheo_maintheo}, and \ref{secunbalcauchy_maintheo} directly. However, in each of these results we restricted ourselves to sectors which are symmetric about the rays of maximal growth of the function. This was done to simplify the notation, and not because some aspect of the analysis requires it. So, I hope the reader will permit me to extend the generality of the results just slightly in this direction without additional comment toward rigorous justification. In particular, we will consider the function $\Ai(z)$ to have two half-open sectors of maximal growth, one $0 \leq \arg z < \pi-\epsilon$ and the other $-\pi + \epsilon < \arg z \leq 0$, even though the rays of maximal growth ($\arg z = 2\pi/3$ and $-2\pi/3$, respectively) do not bisect these sectors. With this in mind, the appropriately extended versions of Theorems \ref{seccurvetwo_maintheo}, \ref{curscatwotheo_maintheo}, and \ref{secunbalcauchy_maintheo} give the following collection of results.

\begin{theorem}
Let $p_n[\Ai](z)$ denote the $n^\th$ partial sum of the Maclaurin series for $\Ai(z)$, let
\[
	S = \left\{ z \in \C : \left|z^{3/2}\exp\!\left(1-z^{3/2}\right)\right| = 1,\ |z| \leq 1, \text{ and } {-}2\pi/3 \leq \arg z < \pi/3 \right\},
\]
and define
\[
	S_+ = S, \qquad S_- = \overline S,
\]
and
\[
	r_n = \left(\frac{2n}{3}\right)^{2/3}.
\]

The limit points of the zeros of the scaled partial sums $p_{n-1}[\Ai]((3/2)^{2/3} r_n z)$ which do not lie on the ray $\arg z = \pi$ are precisely the points of the set
\[
	\left[ (3/2)^{2/3} e^{i2\pi/3} S_+ \right] \cup \left[ (3/2)^{2/3} e^{-i2\pi/3} S_- \right].
\]

Let $\xi \in S_\pm$, $\xi \neq 1$ and define
\[
	\tau = \im\!\left(\xi^{3/2} - 1 - \frac{3}{2}\log \xi\right),
\]
\[
	\tau_n \equiv \frac{\tau n}{3/2} \pmod{2\pi}, \quad -\pi < \tau_n \leq \pi,
\]
and
\[
	z_n(w) = \xi \left( 1 + \frac{\log n}{2(1-\xi^{3/2})n} - \frac{w-i\tau_n}{(1-\xi^{3/2})n} \right).
\]
Then
\begin{equation}
\label{applicationseq_aicurve}
	\frac{p_{n-1}[\Ai]\!\left((3/2)^{2/3} e^{\pm i2\pi/3} r_n z_n(w)\right)}{\Ai\bigl((3/2)^{2/3} e^{\pm i2\pi/3} r_n z_n(w)\bigr)} = 1 - \left( \frac{1}{1-\xi} - \frac{e^{\mp i2\pi n/3}}{e^{\pm i2\pi/3} - \xi} \right) \frac{\xi^{1/4} e^{-w}}{\sqrt{3\pi}} + o(1)
\end{equation}
as $n \to \infty$ uniformly on compact subsets of the $w$-plane.

Additionally,
\begin{equation}
\label{applicationseq_aicorner}
	\lim_{n \to \infty} \frac{p_{n-1}[\Ai]\bigl((3/2)^{2/3} e^{\pm i2\pi/3} r_n (1+w/\sqrt{n})\bigr)}{\Ai\bigl((3/2)^{2/3} e^{\pm i2\pi/3} r_n (1+w/\sqrt{n})\bigr)} = \frac{1}{2} \erfc\!\left(w \sqrt{3/4}\right)
\end{equation}
uniformly on compact subsets of $\re w < 0$.
\end{theorem}

\begin{figure}[!htb]
	\centering
	\includegraphics[width=0.8\textwidth]{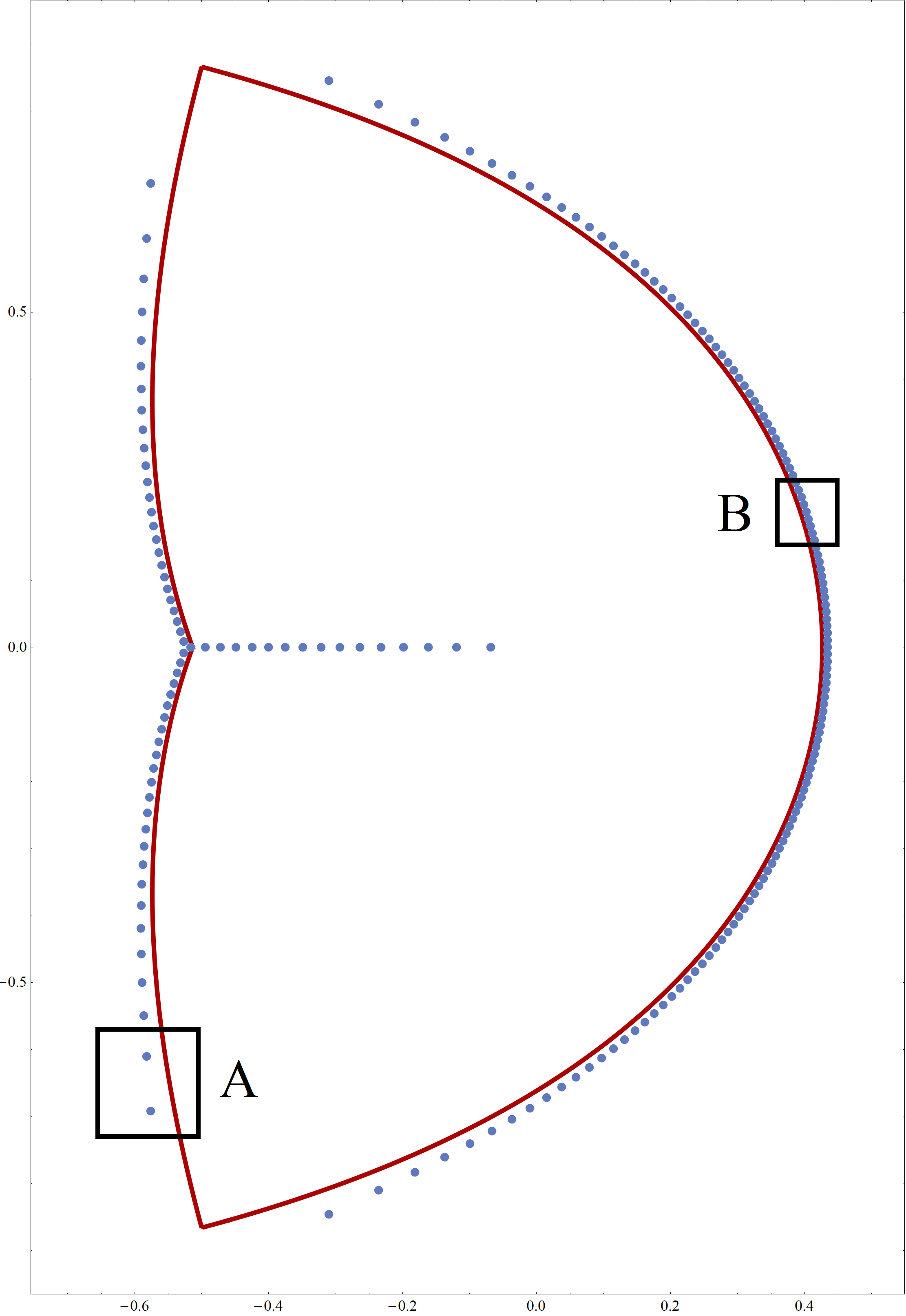}
	\caption[The zeros of {$p_{n-1}[\Ai]((3/2)^{2/3} r_n z)$} for {$n=200$} and their limit curve.]{The zeros of $p_{n-1}[\Ai]((3/2)^{2/3} r_n z)$ for $n=200$ in blue and their limit curve in red.}
	\label{applicationsfig_ai1}
\end{figure}

\begin{figure}[!htb]
	\centering
	\begin{tabular}{cc}
		\includegraphics[width=0.47\textwidth]{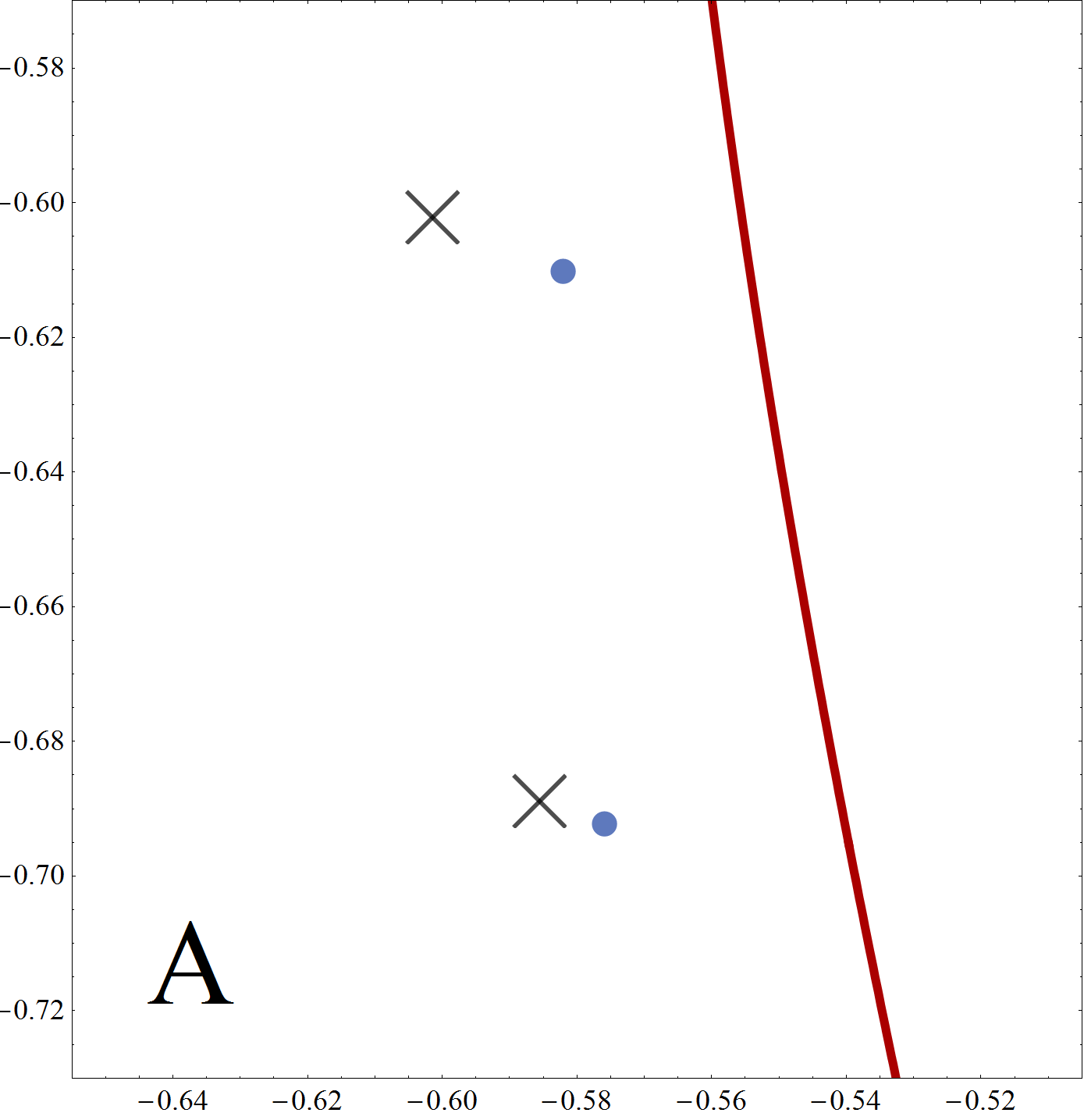}
		& \includegraphics[width=0.47\textwidth]{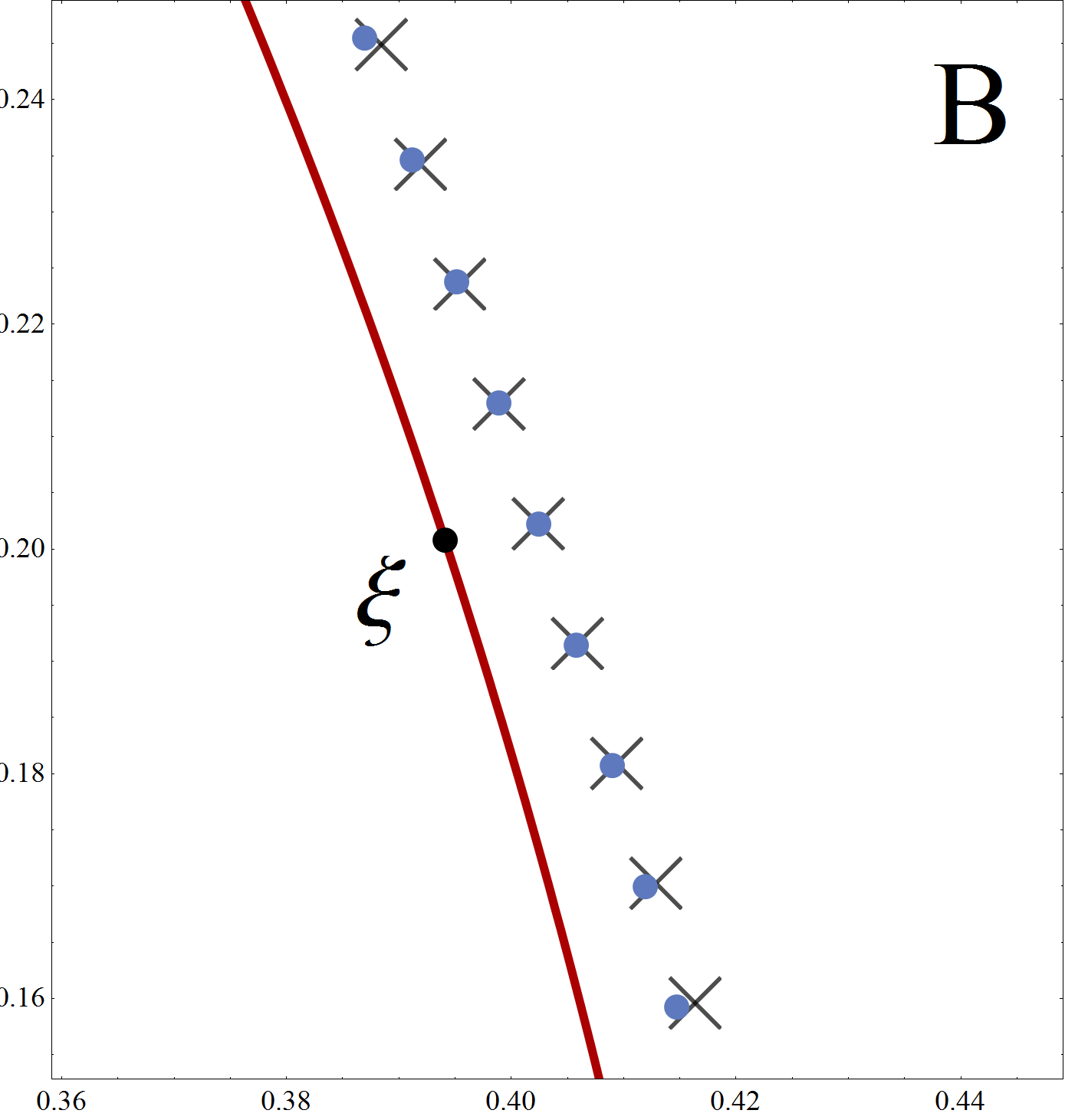}
	\end{tabular}
	\caption[Magnifications of the regions labeled in Figure \ref{applicationsfig_ai1}.]{Magnifications of the regions labeled in Figure \ref{applicationsfig_ai1}. Region A shows the approximations for the zeros in that region, represented as black crosses, which are given by the corner scaling limit in \eqref{applicationseq_aicorner}. Region B shows the approximations for the zeros in that region, represented as black crosses, which are given by the curve scaling limit in \eqref{applicationseq_aicurve}. The point $\xi$ which is used in \eqref{applicationseq_aicurve} to obtain the approximations in Region B is shown in black.}
	\label{applicationsfig_ai2}
\end{figure}

\begin{remark}
Due to the appearance of $e^{\pm i2\pi n/3}$ in the scaling limit corresponding to the arcs of the limit curve, the ratio $p_{n-1}[\Ai](\cdots)/\Ai(\cdots)$ may tend to one of three different limits if we restrict $n$ to run through any one of the three residue classes modulo $3$.
\end{remark}

The DLMF gives the following uniform asymptotics for $\Bi$ \cite[eqns.\ 9.7.7, 9.7.11, and 9.7.13]{nist:dlmf}: for any fixed $\epsilon > 0$,
\[
	\Bi(z) \sim \frac{z^{-1/4}}{\sqrt{\pi}} \exp\!\left(\frac{2}{3} z^{3/2}\right)
\]
as $|z| \to \infty$ with $|{\arg z}| \leq \pi/3 - \epsilon$,
\[
	\Bi(-z) = \frac{z^{-1/4}}{\sqrt{\pi}} \left[ -\sin \omega \bigl[1+o(1)\bigr] + O\!\left(\frac{\cos \omega}{\omega}\right) \right]
\]
as $|z| \to \infty$ with $|{\arg z}| \leq 2\pi/3 - \epsilon$, and
\[
	\Bi\!\left(e^{\pm i \pi/3} z\right) = e^{\pm i \pi/6} z^{-1/4} \sqrt\frac{2}{\pi} \left[ \cos\!\left(\omega \mp \frac{i}{2}\log 2 \right) \!\bigl[1+o(1)\bigr] + O\!\left(\frac{\sin\!\left(\omega \mp \frac{i}{2}\log 2 \right)}{\omega}\right)\right]
\]
as $|z| \to \infty$ with $|{\arg z}| \leq 2\pi/3 - \epsilon$, where
\[
	\omega = \frac{2}{3} z^{3/2} - \frac{\pi}{4}.
\]
It follows that for any $\epsilon > 0$ there is some $\mu < 2/3$ such that
\[
	\sqrt{\pi} \Bi(z) \sim \begin{cases}
		z^{-1/4} \exp\!\left(\frac{2}{3} z^{3/2}\right) & \text{for } |{\arg z}| \leq \pi/3 - \epsilon, \\
		\frac{1}{2} e^{\mp i\pi/3} \left(e^{\pm i2\pi/3} z\right)^{-1/4} \exp\!\left[\frac{2}{3} \!\left(e^{\pm i2\pi/3} z\right)^{3/2}\right]\! & \text{for } \!\left|{\arg e^{\pm i2\pi/3} z}\right| \leq \pi/3 - \epsilon, \\
		O\!\left(\exp\!\left(\mu|z|^{3/2}\right)\right) & \text{otherwise}
	\end{cases}
\]
as $|z| \to \infty$ uniformly in each sector. Evidently the $\Bi$ function has three directions of maximal exponential growth. The discussion in Section \ref{seccurvemany_sec} and Theorems \ref{curscamanytheo_aeqbeqb} and \ref{secunbalcauchymany_maintheo} yield the following collection of results.

\begin{theorem}
Let $p_n[\Bi](z)$ denote the $n^\th$ partial sum of the Maclaurin series for $\Bi(z)$, let
\[
	S = \left\{ z \in \C : \left|z^{3/2}\exp\!\left(1-z^{3/2}\right)\right| = 1,\ |z| \leq 1, \text{ and } |{\arg z}| < \pi/3 \right\},
\]
and define
\[
	r_n = \left(\frac{2n}{3}\right)^{2/3}.
\]

The limit points of the zeros of the scaled partial sums $p_{n-1}[\Bi]((3/2)^{2/3} r_n z)$ which do not lie on the rays $\arg z = \pm \pi/3, \pi$ are precisely the points of the set
\[
	S \cup e^{i2\pi/3} S \cup e^{-i2\pi/3} S.
\]

Let $\xi \in S$, $\xi \neq 1$ and define
\[
	\tau = \im\!\left(\xi^{3/2} - 1 - \frac{3}{2}\log \xi\right),
\]
\[
	\tau_n \equiv \frac{\tau n}{3/2} \pmod{2\pi}, \quad -\pi < \tau_n \leq \pi,
\]
and
\[
	z_n(w) = \xi \left( 1 + \frac{\log n}{2(1-\xi^{3/2})n} - \frac{w-i\tau_n}{(1-\xi^{3/2})n} \right).
\]
Then
\begin{align*}
&\frac{p_{n-1}[\Bi]\!\left((3/2)^{2/3} r_n z_n(w) \right)}{\Bi\bigl((3/2)^{2/3} r_n z_n(w)\bigr)} \\
&\qquad = 1 - \left( \frac{1}{1-\xi} - \frac{e^{-i2\pi n/3}}{2(e^{i2\pi/3} - \xi)} - \frac{e^{i2\pi n/3}}{2(e^{-i2\pi/3} - \xi)} \right) \frac{\xi^{1/4} e^{-w}}{\sqrt{3\pi}} + o(1)
\end{align*}
and
\begin{align}
&\frac{p_{n-1}[\Bi]\!\left((3/2)^{2/3} e^{\pm i2\pi/3} r_n z_n(w) \right)}{\Bi\bigl((3/2)^{2/3} e^{\pm i2\pi/3} r_n z_n(w)\bigr)} \nonumber \\
&\qquad = 1 - \left( \frac{1}{1-\xi} - \frac{2e^{\pm i2\pi n/3}}{e^{\mp i2\pi/3} - \xi} + \frac{e^{\mp i2\pi n/3}}{e^{\pm i2\pi/3} - \xi} \right) \frac{\xi^{1/4} e^{-w}}{\sqrt{3\pi}} + o(1)
\label{applicationseq_bicurve}
\end{align}
as $n \to \infty$ uniformly on compact subsets of the $w$-plane.

Additionally, for fixed $k \in \{-1,0,1\}$
\begin{equation}
\label{applicationseq_bicorner}
	\frac{p_{n-1}[\Bi]\!\left((3/2)^{2/3} e^{i2\pi k/3} r_n z_n(w) \right)}{\Bi\bigl((3/2)^{2/3} e^{i2\pi k/3} r_n z_n(w)\bigr)} = \frac{1}{2} \erfc\!\left(w \sqrt{3/4}\right)
\end{equation}
uniformly on compact subsets of $\re w < 0$.
\end{theorem}

\begin{figure}[!htb]
	\centering
	\includegraphics[width=0.93\textwidth]{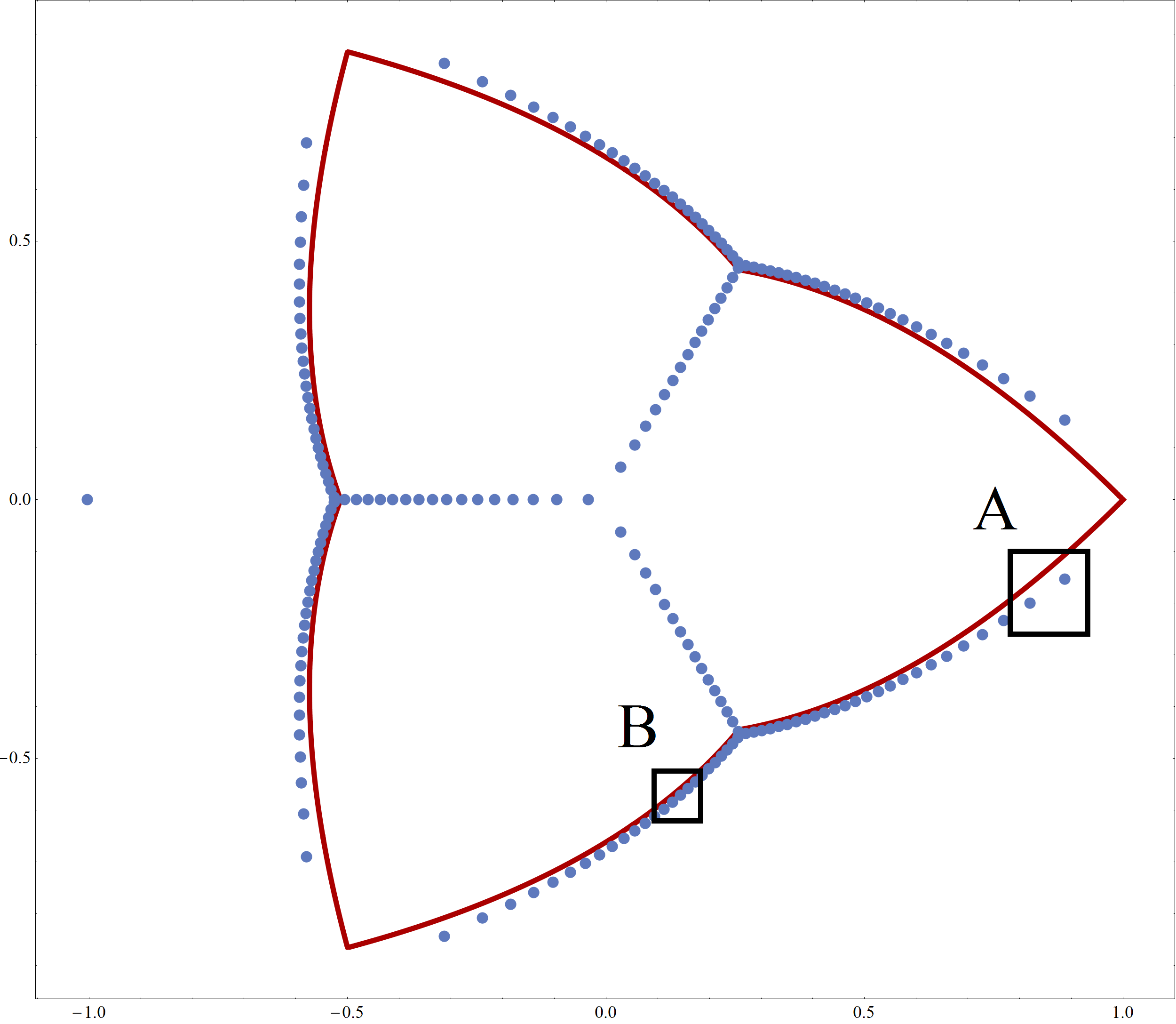}
	\caption[The zeros of {$p_{n-1}[\Bi]((3/2)^{2/3} r_n z)$} for {$n=200$} and their limit curve.]{The zeros of $p_{n-1}[\Bi]((3/2)^{2/3} r_n z)$ for $n=200$ in blue and their limit curve in red.}
	\label{applicationsfig_bi1}
\end{figure}

\begin{figure}[!htb]
	\centering
	\begin{tabular}{cc}
		\includegraphics[width=0.47\textwidth]{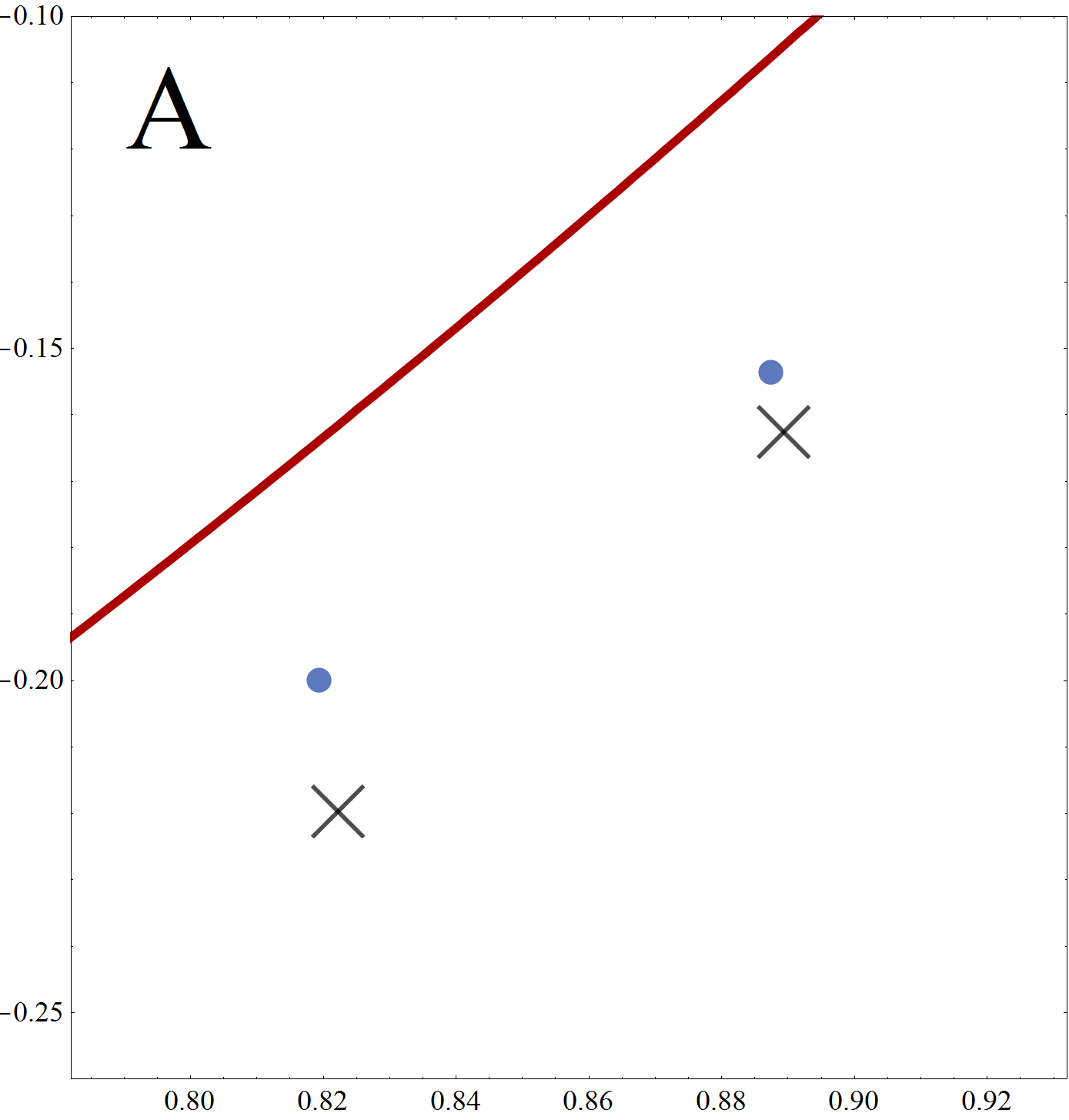}
		& \includegraphics[width=0.47\textwidth]{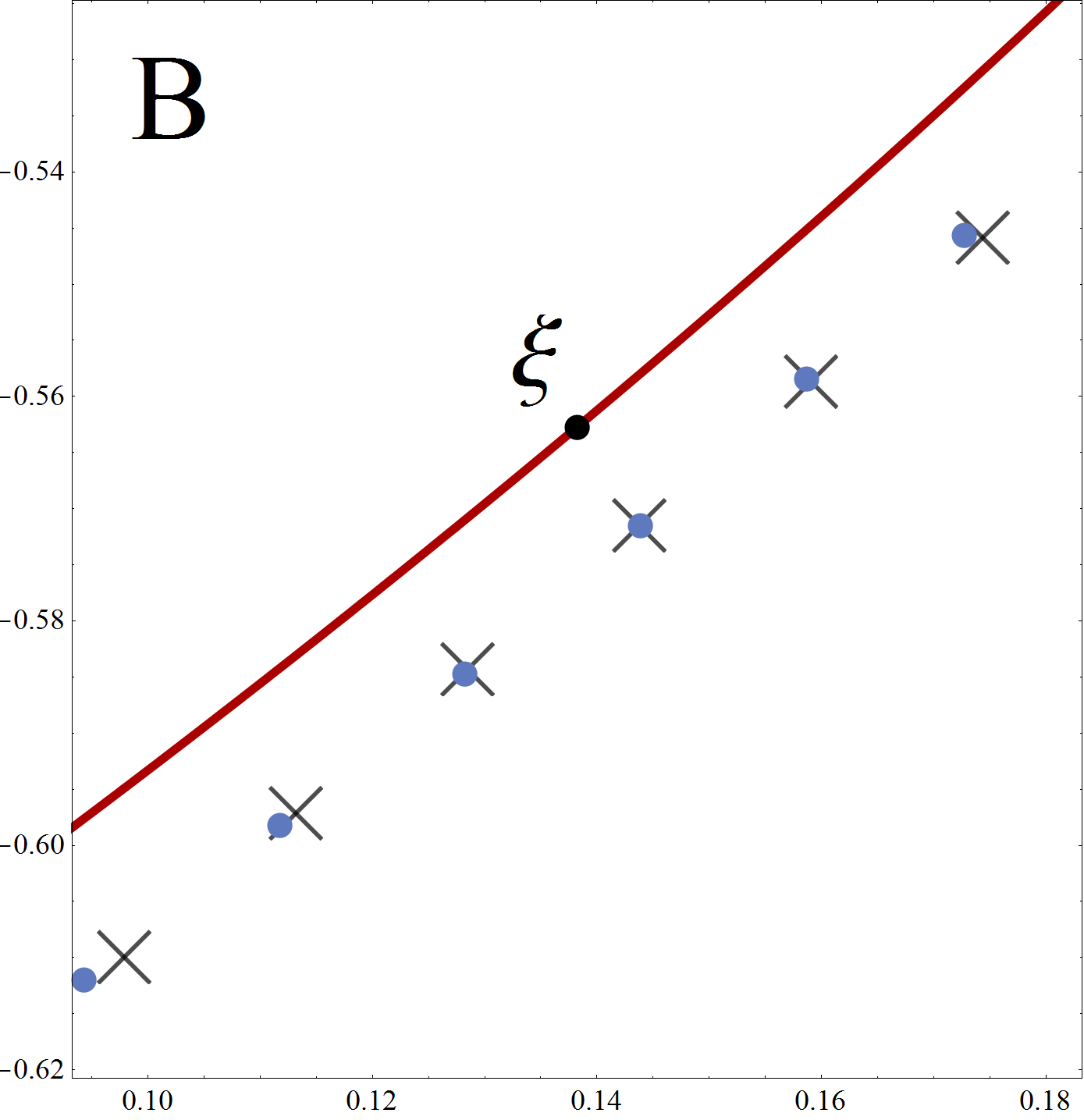}
	\end{tabular}
	\caption[Magnifications of the regions labeled in Figure \ref{applicationsfig_bi1}.]{Magnifications of the regions labeled in Figure \ref{applicationsfig_bi1}. Region A shows the approximations for the zeros in that region, represented as black crosses, which are given by the corner scaling limit in \eqref{applicationseq_bicorner}. Region B shows the approximations for the zeros in that region, represented as black crosses, which are given by the curve scaling limit in \eqref{applicationseq_bicurve}. The point $\xi$ which is used in \eqref{applicationseq_bicurve} to obtain the approximations in Region B is shown in black.}
	\label{applicationsfig_bi2}
\end{figure}

\section{Parabolic Cylinder Functions}
\label{applicationssec_paraboliccylinder}

The parabolic cylinder function $U(a,z)$ is an entire function which is defined for $a,z \in \C$ as the solution of the differential equation
\[
	\frac{d^2 u}{dz^2} - \left(\frac{1}{4}z^2 + a\right) u = 0
\]
identified by the asymptotic behavior
\[
	U(a,z) \sim z^{-a-1/2} e^{-z^2/4}
\]
as $|z| \to \infty$ with $|\arg z| \leq 3\pi/4 - \epsilon$ and
\[
	U(a,z) = z^{-a-1/2} e^{-z^2/4} \,\bigl[1+o(1)\bigr] \pm \frac{i\sqrt{2\pi}}{\Gamma(a+1/2)} e^{\mp i\pi a} z^{a-1/2} e^{z^2/4} \,\bigl[1+o(1)\bigr]
\]
as $|z| \to \infty$ with $\pi/4+\epsilon \leq \pm {\arg z} \leq 5\pi/4 - \epsilon$, where in the latter appropriate choices for the branches of $z^{-a-1/2}$ and $z^{a-1/2}$ and an appropriate determination of $\arg$ is made \cite[eqns.\ 12.9.1 and 12.9.3]{nist:dlmf}. These functions are sometimes written using the alternate notation
\[
	D_\nu(z) = U(-\nu-1/2,z),
\]
as in \cite[sec.\ 19.3]{aands:handbook}.

From the above asymptotic it follows that $U$ has three directions of maximal exponential growth. Indeed, for any $\epsilon > 0$ there is a $\mu < 1/4$ such that
\[
	U(a,z) \!\sim\! \begin{cases}
		\frac{\sqrt{2\pi}}{\Gamma(a+1/2)} (-z)^{a-1/2} \exp\!\left[\frac{1}{4}(-z)^2\right] & \text{for } |{\arg -z}| \leq \pi/4 - \epsilon, \\
		e^{i\pi (a/2+1/4)} (iz)^{-a-1/2} \exp\!\left[\frac{1}{4}(iz)^2\right] & \text{for } -\pi/4 + \epsilon < {\arg iz} < \pi/2 - \epsilon, \\
		e^{-i\pi (a/2+1/4)} (z/i)^{-a-1/2} \exp\!\left[\frac{1}{4}(z/i)^2\right]\!\! & \text{for } -\pi/2 + \epsilon < {\arg(z/i)} < \pi/4 - \epsilon, \\
		O\!\left(\exp\!\left(\mu|z|^2\right)\right) & \text{otherwise}
	\end{cases}
\]
as $|z| \to \infty$ uniformly in each sector. As with the $\Ai$ function we will consider the half-open sectors $0 \leq \pm \arg z < 3\pi/4$ to be maximal growth sectors even though the maximal growth rays $\arg z = \pm \pi/2$ do not bisect them. With this in mind, the discussion in Section \ref{seccurvemany_sec} and Theorems \ref{curscamanytheo_agebb}, \ref{curscamanytheo_aeqbgeb}, \ref{curscamanytheo_aeqbeqb}, \ref{curscamanytheo_beqbgea} and \ref{secunbalcauchymany_maintheo} yield the following collection of results.

\begin{theorem}
Let $p_n[U](z)$ denote the $n^\th$ partial sum of the Maclaurin series for $U(a,z)$, let
\[
	S = \left\{ z \in \C : \left|z^2\exp\!\left(1-z^2\right)\right| = 1,\ |z| \leq 1, \text{ and } -\pi/2 \leq {\arg z} < \pi/4 \right\},
\]
and define
\[
	S_{\pi/4} = S \cap \{z \in \C : |{\arg z}| < \pi/4\}, \qquad S_+ = S, \qquad S_- = \overline S
\]
and
\[
	r_n = \left(\frac{n}{2}\right)^{1/2}.
\]

The limit points of the zeros of the scaled partial sums $p_{n-1}[U](2 r_n z)$ which do not lie on the rays $\arg z = \pm 3\pi/4$ are precisely the points of the set
\[
	iS_+ \cup -iS_- \cup -S_{\pi/4}.
\]

Define
\[
	\tau = \im\!\left(\xi^2 - 1 - 2\log \xi\right),
\]
\[
	\tau_n \equiv \frac{\tau n}{2} \pmod{2\pi}, \quad -\pi < \tau_n \leq \pi,
\]
\[
	\sigma_n \equiv \frac{\pi n}{2} \pmod{2\pi}, \quad -\pi < \sigma_n \leq \pi,
\]
\[
	z_n^1(w) = \xi \left( 1 + \frac{\log n}{2(1-\xi^2)n} - \frac{w-i\tau_n}{(1-\xi^2)n} \right),
\]
\[
	z_n^2(w) = \xi \left( 1 + \frac{(2a+1)\log n}{2(1-\xi^2)n} - \frac{w-i\sigma_n-i\tau_n}{(1-\xi^2)n} \right),
\]
and
\[
	z_n^\pm(w) = \xi\left( 1 + \frac{(1-2a)\log n}{2(1-\xi^2)n} - \frac{w \mp i\sigma_n-i\tau_n}{(1-\xi^2)n} \right).
\]

Let $\xi \in S_{\pi/4}$, $\xi \neq 1$. If $\re a > 0$ then
\[
	\lim_{n \to \infty} \frac{p_{n-1}[U](-2 r_n z_n^1(w))}{U(a,-2 r_n z_n^1(w))} = 1 - \frac{\xi^{1/2-a} e^{-w}}{2\sqrt{\pi}(1-\xi)},
\]
if $\re a < 0$ then
\begin{align}
&\frac{p_{n-1}[U](-2 r_n z_n^2(w))}{U(a,-2 r_n z_n^2(w))} \nonumber \\
&\qquad = 1 - \left( \frac{\exp\{i\pi(a/2+3/4)\}}{i-\xi} + \frac{(-1)^n \exp\{-i\pi(a/2-1/4)\}}{i+\xi} \right) \frac{\Gamma(a+1/2) e^{-w}}{2^{a+3/2} \pi \xi^{a-1/2}} \nonumber \\
&\qquad\qquad {}+ o(1)
\label{applicationseq_paraboliccurve}
\end{align}
as $n \to \infty$, and if $\re a = 0$ then
\begin{align*}
&\frac{p_{n-1}[U](-2 r_n z_n^1(w))}{U(a,-2 r_n z_n^1(w))} \\
	&\qquad = 1 - \left[ \frac{1}{1-\xi} + \frac{\Gamma(a+1/2) r_n^{-2a}}{4^a \sqrt{2\pi}} \left( \frac{i^{1-n} \exp\{i\pi(a/2+1/4)\}}{i-\xi} \right.\right. \\
	&\qquad\qquad\qquad\qquad\qquad \left.\left. {}- \frac{(-i)^{1-n} \exp\{-i\pi(a/2+1/4)\}}{i+\xi} \right) \right] \frac{e^{-w}}{2\sqrt{\pi} \xi^{a-1/2}} + o(1)
\end{align*}
as $n \to \infty$, with each limit holding uniformly on compact subsets of the $w$-plane.

Let $\xi \in S_\pm$, $\xi \neq 1$. If $\re a > 0$ then
\[
	\lim_{n \to \infty} \frac{p_{n-1}[U](\pm 2ir_nz_n^\pm(w))}{U(a,\pm 2ir_nz_n^\pm(w))} = 1 - i\exp\!\left\{\pm i\pi\!\left(\tfrac{a}{2} + \tfrac{1}{4}\right)\!\right\} \!\frac{2^{a-1/2} \xi^{a+1/2} e^{-w}}{\Gamma(a+1/2) (i\mp\xi)},
\]
if $\re a < 0$ then
\[
	\frac{p_{n-1}[U](\pm 2ir_nz_n^1(w))}{U(a,\pm 2ir_nz_n^1(w))} = 1 - \left( \frac{1}{1-\xi} + \frac{(-1)^n \exp\{\pm i\pi(a+1/2)\}}{1+\xi} \right) \frac{\xi^{a+1/2} e^{-w}}{2\sqrt{\pi}} + o(1)
\]
as $n \to \infty$, and if $\re a = 0$ then
\begin{align*}
&\frac{p_{n-1}[U](\pm 2ir_nz_n^1(w))}{U(a,\pm 2ir_nz_n^1(w))} \\
&\qquad = 1 - \left( \frac{1}{1-\xi} + \frac{(-1)^n \exp\{\pm i\pi(a+1/2)\}}{1+\xi} \right. \\
&\qquad\qquad\qquad\qquad \left. {}\pm \exp\!\left\{\pm i\pi \left(\tfrac{a}{2}+\tfrac{1}{4}\right)\right\} \!\frac{4^a \sqrt{2\pi} (\pm i)^{1-n} r_n^{2a}}{\Gamma(a+1/2)(i\mp\xi)} \right) \frac{\xi^{a+1/2} e^{-w}}{2\sqrt{\pi}} + o(1)
\end{align*}
as $n \to \infty$, with each limit holding uniformly on compact subsets of the $w$-plane.

Additionally,
\[
	\lim_{n \to \infty} \frac{p_{n-1}[U](-2r_n(1+w/\sqrt{n}))}{U(a,-2r_n(1+w/\sqrt{n}))} = \frac{1}{2} \erfc(w)
\]
if and only if $\re a > -1/2$ and
\begin{equation}
\label{applicationseq_paraboliccorner}
	\lim_{n \to \infty} \frac{p_{n-1}[U](\pm 2ir_n(1+w/\sqrt{n}))}{U(a,\pm 2ir_n(1+w/\sqrt{n}))} = \frac{1}{2} \erfc(w)
\end{equation}
if and only if $\re a < 1/2$, with each limit holding uniformly on compact subsets of $\re w < 0$.
\end{theorem}

\begin{figure}[!htb]
	\centering
	\includegraphics[width=0.9165\textwidth]{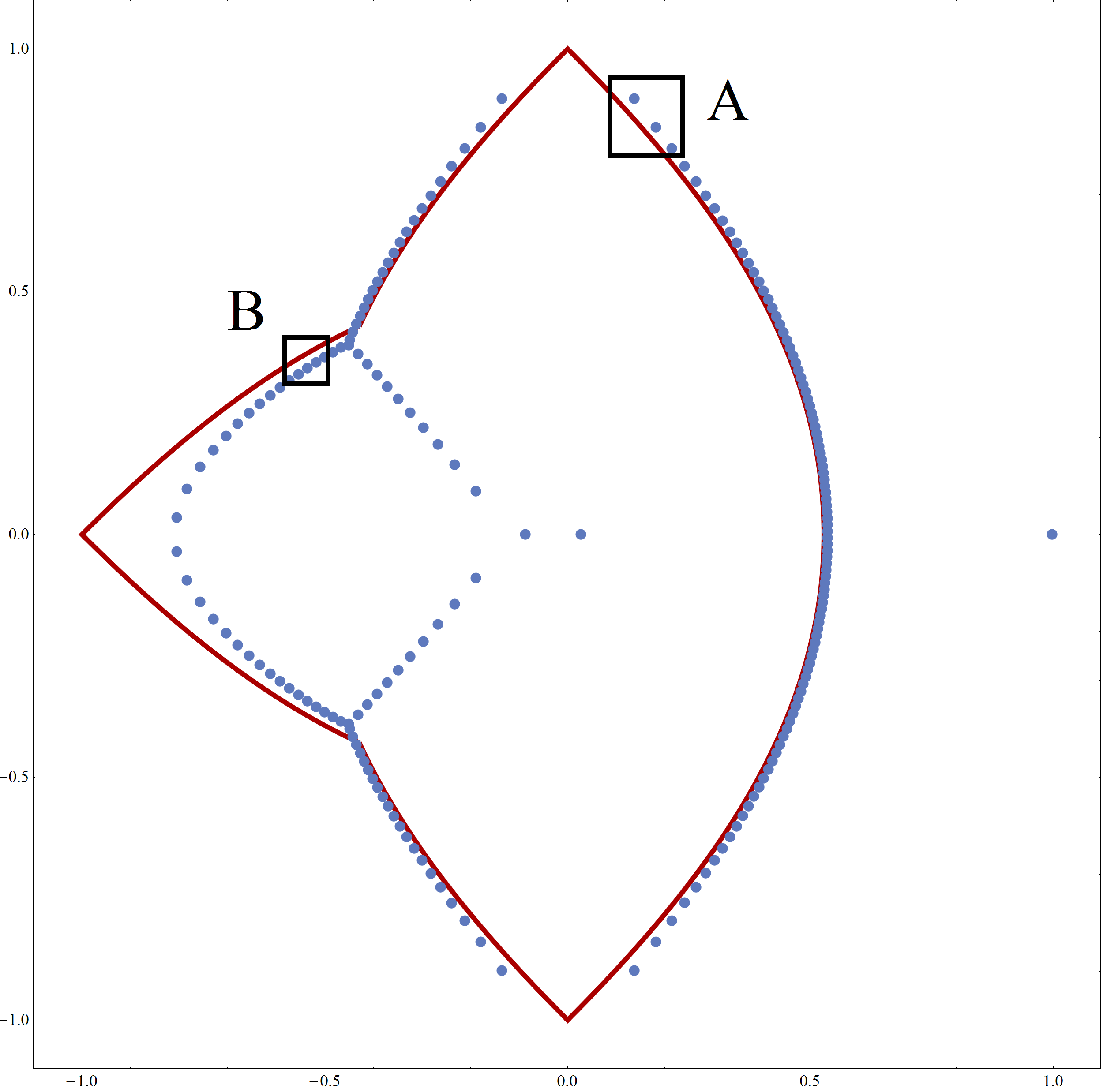}
	\caption[The zeros of {$p_{n-1}[U](2 r_n z)$} for {$a = -2$} and {$n=200$} and their limit curve.]{The zeros of $p_{n-1}[U](2 r_n z)$ for $a = -2$ and $n=200$ in blue and their limit curve in red.}
	\label{applicationsfig_parabolic1}
\end{figure}

\begin{figure}[!htb]
	\centering
	\begin{tabular}{cc}
		\includegraphics[width=0.47\textwidth]{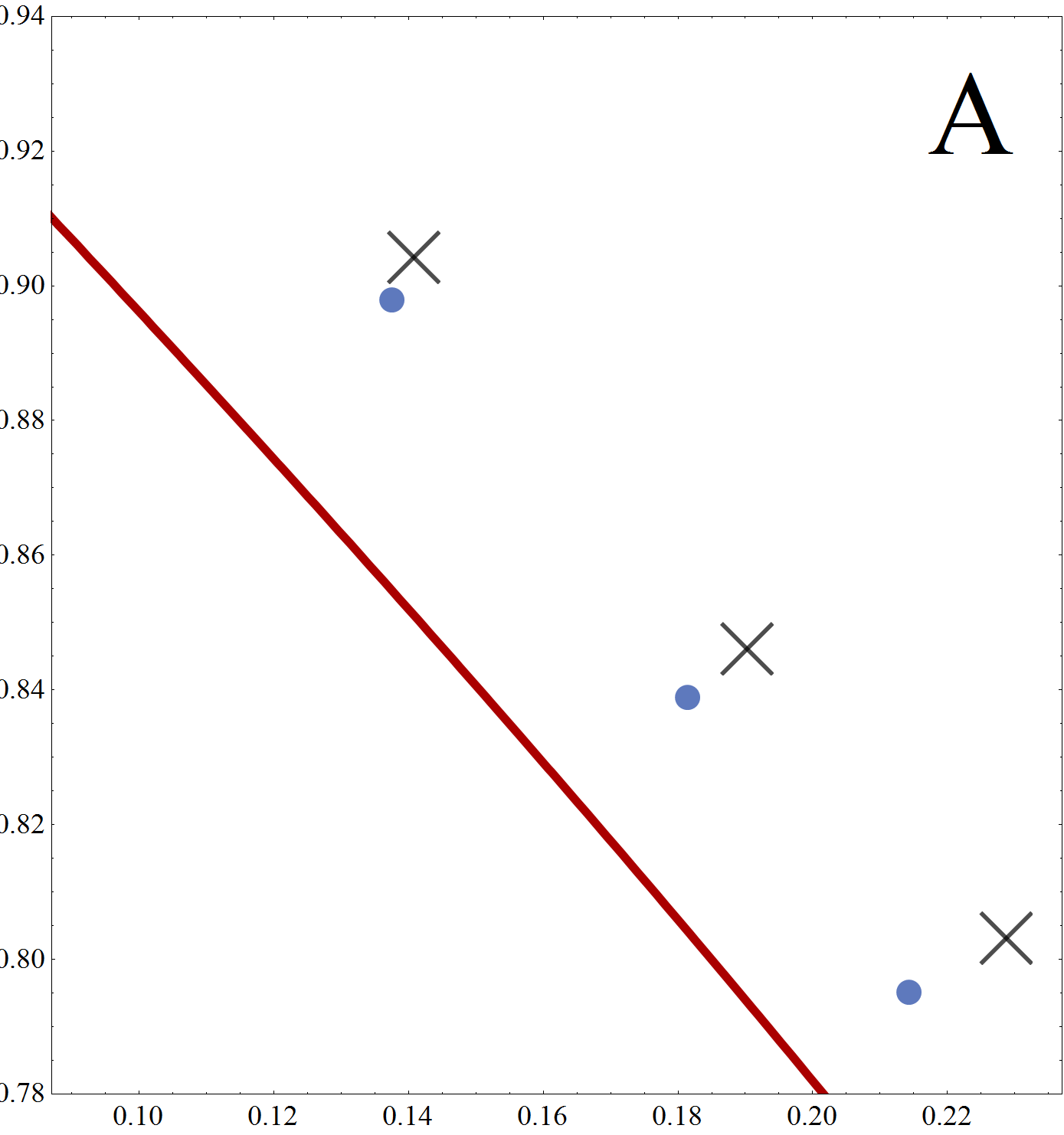}
		& \includegraphics[width=0.47\textwidth]{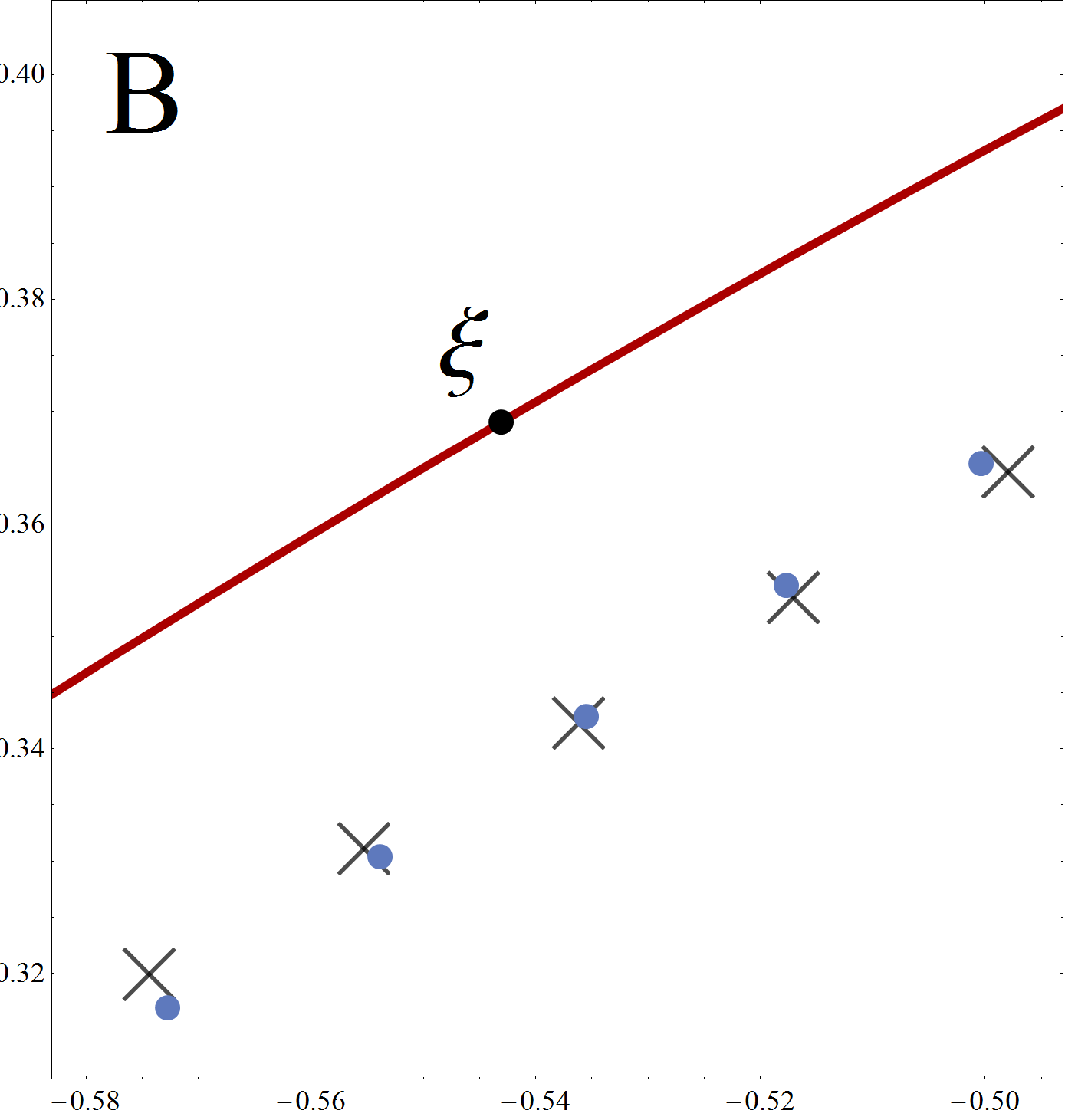}
	\end{tabular}
	\caption[Magnifications of the regions labeled in Figure \ref{applicationsfig_parabolic1}.]{Magnifications of the regions labeled in Figure \ref{applicationsfig_parabolic1}. Region A shows the approximations for the zeros in that region, represented as black crosses, which are given by the corner scaling limit in \eqref{applicationseq_paraboliccorner}. Region B shows the approximations for the zeros in that region, represented as black crosses, which are given by the curve scaling limit in \eqref{applicationseq_paraboliccurve}. The point $\xi$ which is used in \eqref{applicationseq_paraboliccurve} to obtain the approximations in Region B is shown in black.}
	\label{applicationsfig_parabolic2}
\end{figure}





%% file: chap_conclusion.tex
\graphicspath{{images/chap_conclusion/}}

\chapter{Conclusion}
\label{chap_conclusion}

In this thesis it has been shown that if an entire function has simple exponential growth in a finite number of directions in the complex plane---growth like $z^a \exp(z^\lambda)$ for some $a \in \C$ and some $\lambda > 0$---and is exponentially smaller elsewhere, then the partial sums of its power series have scaling limits in these exponential growth directions which yield information about the asymptotic behavior of their zeros. In particular these scaling limits describe enough about the trajectories of the zeros to verify that the Saff-Varga Width Conjecture (see Section \ref{introsec_width}) holds for these entire functions.

In this chapter we will give some context for these results and discuss an open problem relating to them.

\section{Big Picture and the Width Conjecture}

In his thesis \cite{rosen:thesis} Rosenbloom considered entire functions
\[
	f(z) := \sum_{k=0}^{\infty} a_k z^k
\]
for which some subsequence of $f(|a_n|^{-1/n}z)^{1/n}$ converges to an analytic function in some domain (see Theorem \ref{introthm_rosen2}). For example, Rosenbloom tells us that if
\begin{equation}
\label{concleq_rosenasymp}
	\lim_{n \to \infty} f(|a_n|^{-1/n}ez)^{1/n} = e^{z}
\end{equation}
in some domain then the zeros of the scaled partial sums $p_n[f](|a_n|^{-1/n}ez)$ converge precisely to the classical Szeg\H{o} curve $|ze^{1-z}| = 1$ in that domain.

These results can be thought of as the ``first-order'' theory for the zeros of partial sums of power series for entire functions. Given information about the limiting behavior of $f(|a_n|^{-1/n}z)^{1/n}$, Rosenbloom deduces the limiting behavior of the zeros of the scaled partial sums $p_n[f](|a_n|^{-1/n}ez)$.

The assumptions made in this thesis are stronger than Rosenbloom's. Whereas Rosenbloom would only assume something like \eqref{concleq_rosenasymp}, I would instead assume that
\begin{equation}
\label{concleq_myasymp}
	f(nz) \sim (nz)^a (\log n)^b e^{nz}
\end{equation}
as $n \to \infty$ for $z$ in some domain. Rosenbloom-type results can indeed be deduced from this---i.e.\ the proper scaling of the zeros (they will scale like $n$) as well as the fact that the zeros of the scaled partial sums $p_n[f](nz)$ will accumulate on the classical Szeg\H{o} curve $|ze^{1-z}| = 1$. However, under this stronger assumption we can go one step further than Rosenbloom and deduce not only how quickly the zeros approach the limit curve but also information about their geometry as they do so. In this way the results in this thesis can be considered part of the ``second-order'' theory of the zeros.

The scaling limits in Chapter \ref{chap_curvescaling} tell us that the zeros of the scaled partial sums approach their limit curve at a rate of approximately $\log n/n$, that they are (locally) separated from each other by a distance of approximately $1/n$, and that they approximately lie on straight lines parallel to their limit curve. Similarly, the scaling limits in Chapter \ref{chap_posfinite} tell us that the zeros approach the convex corner of their limit curve at a rate of approximately $1/\sqrt{n}$, that they are (locally) separated from each other by a distance of approximately $1/\sqrt{n}$, and that they approximately lie on rays with arguments $\pm 3\pi/4$ originating at the corner of the curve.

This second-order information is detailed enough to imply the validity of the Saff-Varga Width Conjecture for the class of functions considered in this thesis. However, while substantial, this class of functions is smaller than the one considered by Rosenbloom. A verification of the Width Conjecture for Rosenbloom's class would mark a significant step forward in the theory. It would be interesting to see whether the techniques used here can be generalized to that case.

A proof of the Width Conjecture in the general setting still eludes us.

\section{Open Problem: Asymptotics in the Transitional Regions}

In this thesis we did not study regions of the limit curve which lie on the boundary between two sectors of maximal exponential growth.

For example, consider the example of $f(x) = \sin x$ as in Section \ref{applicationssec_sincos}. This function behaves exponentially as $\im z \to \pm \infty$, and the dividing line between these two exponential growth directions is the real axis. The zeros of $f$ lie on this line, and so the partial sums of $f$ have corresponding zeros there as well. This can been seen clearly in Figure \ref{conclusionfig_sin}. This line of zeros on the real axis extends outward until it intersects the limit curves from the upper and lower half-planes at the points $z = \pm 1/e$. At the intersection at $z=1/e$ magnified in the figure we can see that the zeros of the partial sums come together in a three-spoke junction, one ray emanating from the point $z=1/e$ to the left along the real axis and the other two into the upper and lower half-planes with arguments of $\pm(\arctan e^{-1} - \pi/2)$.

\begin{figure}[!htb]
	\centering
	\begin{minipage}[c]{0.5\textwidth}
		\includegraphics[width=\textwidth]{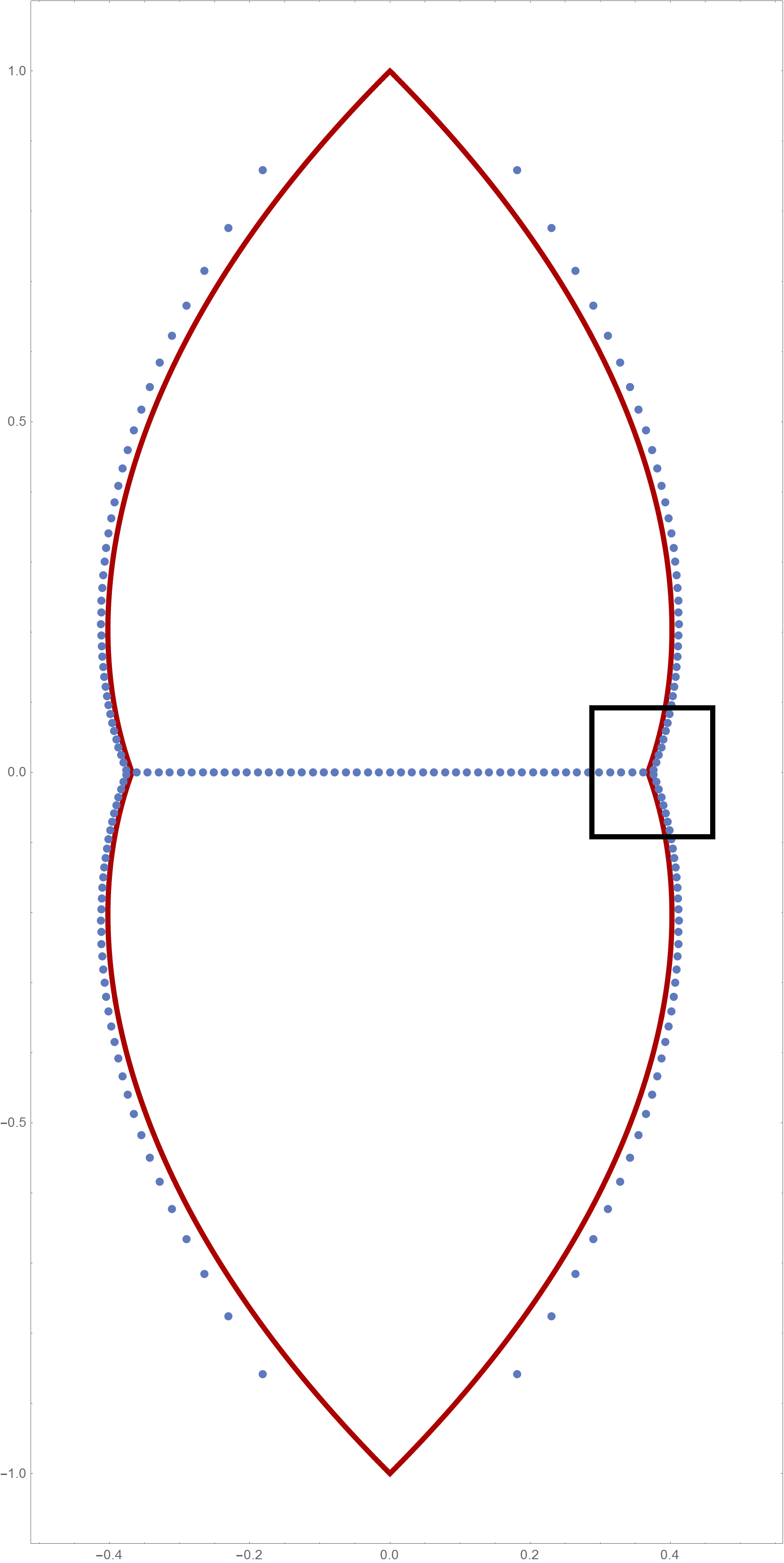}
	\end{minipage}
	\begin{minipage}[c]{0.44\textwidth}
		\includegraphics[width=\textwidth]{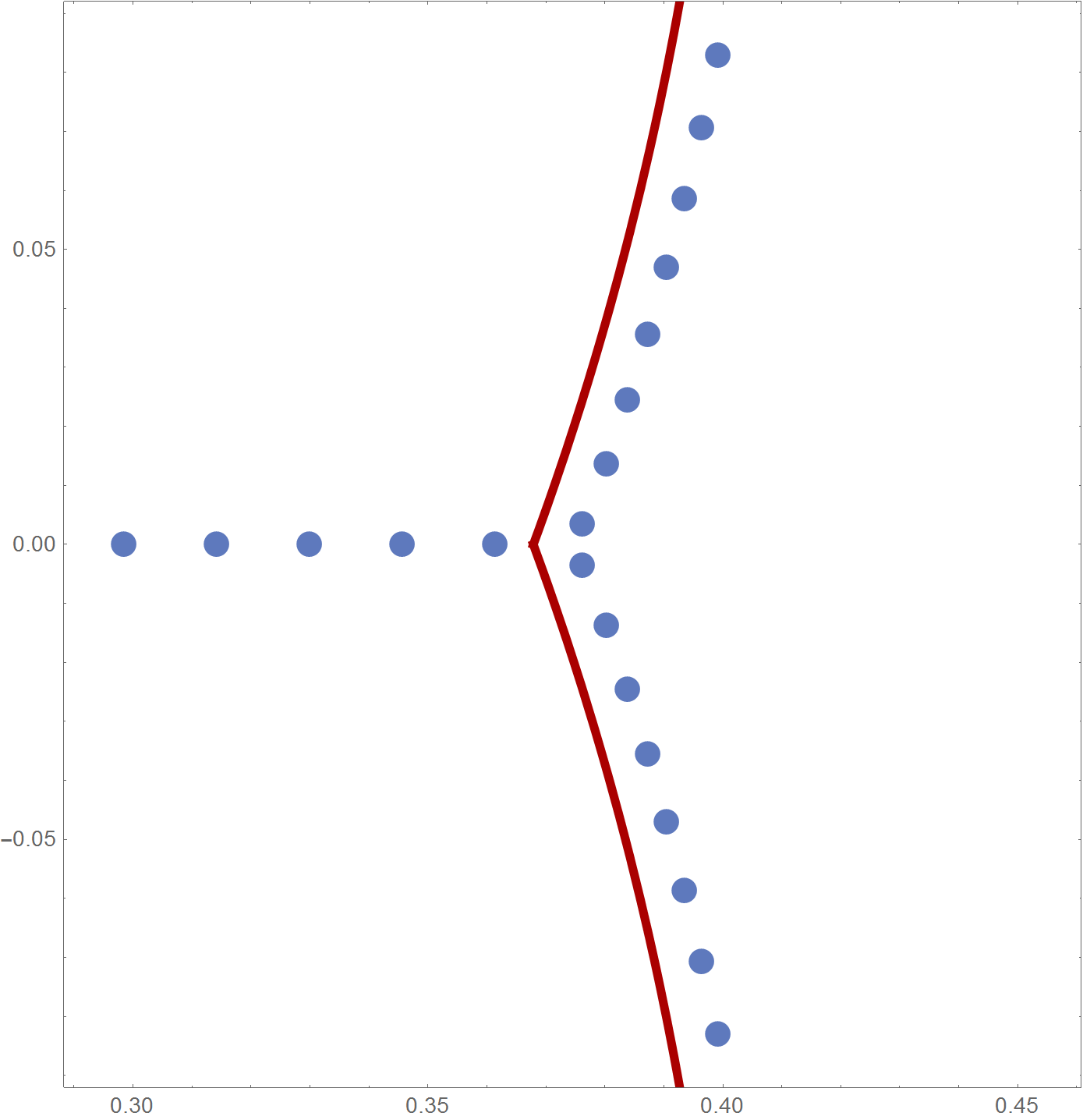}
	\end{minipage}
	\caption[The zeros of {$p_{n-1}[\sin](nz)$} for {$n=200$} and their limit curve.]{The zeros of $p_{n-1}[\sin](nz)$ for $n=200$ in blue and their limit curve in red. The boxed region, magnified on the right, shows the zeros which are near this corner of the limit curve.}
	\label{conclusionfig_sin}
\end{figure}

This three-spoke geometry seems to appear wherever two exponential growth sectors meet. Further examples can be seen in Figures \ref{applicationsfig_bessel}, \ref{applicationsfig_expint}, \ref{applicationsfig_ai1}, \ref{applicationsfig_bi1}, and \ref{applicationsfig_parabolic1}. We know so far that the partial sums have universal $\erfc$ scaling limits which capture the asymptotics near the convex corners of the limit curve and universal $1-De^{-w}$ scaling limits which capture asymptotics near the smooth arcs, so it seems reasonable to expect that there is some analogous universal scaling limit which captures the asymptotics at these three-spoke intersections.

This phenomenon has been studied in the case where $f$ is an exponential sum of the form
\[
	f(z) = \sum_{k=0}^{n} \exp(\omega_k z),
\]
where $\omega_k \in \C$, by Bleher and Mallison. We refer the reader to their paper \cite{mallison:expsums}, and in particular their Theorem 7.1. Their result indicates that the universal behavior may involve a three-term exponential sum of the form
\[
	c_1 e^{a_1 z} + c_2 e^{a_2 z} + c_3 e^{a_3 z}.
\]

%% file: appendix_laplace.tex
\chapter{The Laplace Method}
\label{appendix_laplace}

The ``Laplace method'' is a statement concerning the analytic behavior, usually in regards to the growth or decay, of integrals of a certain type.  The most common version of it looks something like this:

\begin{center}
\fbox{\begin{minipage}{0.8\textwidth}
If $g\colon\R \to \R$ has a global maximum at $t=t_0$ and $g''(t_0) < 0$ then
\[
	\int_{-\infty}^{\infty} e^{\lambda g(t)}\,dt \sim e^{\lambda g(t_0)} \sqrt{\frac{2\pi}{-\lambda g''(t_0)}}
\]
as $\lambda \to \infty$ with $\lambda > 0$.
\end{minipage}}
\end{center}

\noindent Throughout this thesis I use several variants of this result. For example, more general versions of the above would allow $\lambda$ to be complex or would replace the real integral with a complex contour integral. A full description of these tools can be found in a book on asymptotic analysis (such as Miller's \cite{miller:aaa}). Instead, my goal in this appendix is to investigate a few simple archetypal situations in order to help a starting student gain a basic understanding of the topic, and hopefully to give them enough of an idea of it that they can follow the more advanced maneuvers in the body of the thesis.

The no-assembly-required statement above is very useful for applications and can even give correct results in situations that don't quite satisfy the hypotheses.  While we will prove this statement, the main idea behind the Laplace method is much simpler, and in my opinion much more powerful:

\begin{center}
\fbox{\begin{minipage}{0.8\textwidth}
\textbf{The Laplace method.}\par

An integral can be approximated by approximating the integrand near its largest point.
\end{minipage}}
\end{center}

\noindent This perspective can get you pretty far, I think.  I'll focus in the succeeding sections on investigating how it can inform our decisions when attempting to approximate various types of integrals which depend in some way on a parameter.

\section{Something Like a Laplace Transform}

Consider the Laplace transform of a function $f$,
\[
	F(\lambda) = \int_0^\infty f(t)e^{-\lambda t}\,dt,
\]
where $\lambda>0$.  Suppose we want to estimate this integral for large $\lambda$.  If all we know about $f$ is that it's bounded, say $|f(t)| \leq M$, then we can at least say that
\begin{align*}
|F(\lambda)| &\leq \int_0^\infty |f(t)| e^{-\lambda t}\,dt \\
	&\leq M \int_0^\infty e^{-\lambda t}\,dt \\
	&= \frac{M}{\lambda},
\end{align*}
so $F(\lambda) = O(\lambda^{-1})$.  It stands to reason that if we know more about $f$ then we can probably get a more accurate picture of the behavior of $F$.

Let's hand wave a little.  If $f(t)$ is a ``nice'' function that doesn't grow very quickly as $t \to \infty$ then the quantity $f(t)e^{-\lambda t}$ will be very small for large values of $t$.  As $\lambda$ grows, the quantity will even be very small for not-so-large values of $t$.  In a sense, the values of $t$ near $t=0$ are the only ones for which the $e^{-\lambda t}$ factor doesn't make the integrand (and its contributions to the integral) negligible.

More precisely (but still informally), if $\lambda t \to \infty$ then $f(t)e^{-\lambda t} \to 0$, so we only care about the values of $t$ for which $\lambda t$ is bounded.  We could chose any bound, so let's say we only care about all $t$ for which $\lambda t < 1$.

Letting $\lambda t = u$ the integral becomes
\[
	F(\lambda) = \frac{1}{\lambda} \int_0^\infty f\!\left(\frac{u}{\lambda}\right) e^{-u}\,du,
\]
and, as mentioned, we only really care about the values $u < 1$.  When $\lambda$ is much larger than $u$ we have $u/\lambda \approx 0$, so if $f$ is sufficiently nice we might expect that
\[
	F(\lambda) \approx \frac{1}{\lambda} \int_0^\infty f(0) e^{-u}\,du = \frac{f(0)}{\lambda}
\]
for large $\lambda$.

So what do we mean when we say that $f$ needs to be ``sufficiently nice''?  In order to ensure that $f(u/\lambda) \approx f(0)$ when $u/\lambda \approx 0$ we'll at least need to assume that $f$ is continuous from the right at $0$.  Of course $u/\lambda$ is not small over the entire range of integration, so we'll also need to assume that the ``tail'' of $f$, i.e.\ the set of values of $f(t)$ when $t$ is large, does not contribute too much to the total size of the integral.  We can formalize this as follows.

\begin{theorem}
Fix $0 < T \leq \infty$ and let $f \colon [0,T] \to \C \cup \{\infty\}$ be a measurable function satisfying
\[
	\int_0^T |f(t)| e^{-\ell t}\,dt < \infty
\]
for some constant $\ell \geq 0$.  Suppose $f$ is right-continuous at $0$ with $f(0) \neq 0$.  Then
\[
	F(\lambda) := \int_0^T f(t) e^{-\lambda t}\,dt \sim \int_0^\infty f(0) e^{-\lambda t} = \frac{f(0)}{\lambda}
\]
as $\lambda \to \infty$ with $\lambda > 0$.
\label{babywatson}
\end{theorem}

\begin{proof}
The fundamental idea behind the proof of this result (and many of the following ones as well) is that we will extract the ``dominant'' contribution from the integral and leave behind something which is relatively much smaller.  To wit, write
\[
	f(t) = f(0) + (f(t) - f(0)),
\]
so that
\begin{align*}
F(\lambda) &= \int_0^T f(0) e^{-\lambda t}\,dt + \int_0^T (f(t) - f(0)) e^{-\lambda t}\,dt \\
	&= \frac{f(0)}{\lambda} - \frac{f(0)}{\lambda e^{T\lambda}} + \int_0^T (f(t) - f(0)) e^{-\lambda t}\,dt.
\end{align*}
Since the second term is
\[
	\frac{f(0)}{\lambda e^{T\lambda}} = o\!\left(\frac{1}{\lambda}\right)
\]
it only remains to show that
\[
	\int_0^T (f(t) - f(0)) e^{-\lambda t}\,dt = o\!\left(\frac{1}{\lambda}\right)
\]
as well, for this will imply that
\[
	F(\lambda) = \frac{f(0)}{\lambda} + o\!\left(\frac{1}{\lambda}\right),
\]
allowing us to conclude the result.

Let's take advantage of the continuity of $f$.  For any $\epsilon > 0$ we can find a $0 < \delta < T$ such that
\[
	|f(t) - f(0)| < \epsilon
\]
for all $0 \leq t \leq \delta$.  We then have
\begin{align}
\left|\int_0^\delta (f(t) - f(0)) e^{-\lambda t}\,dt\right| &\leq \int_0^\delta |f(t) - f(0)| e^{-\lambda t}\,dt \nonumber \\
	&< \epsilon \int_0^\delta e^{-\lambda t}\,dt \nonumber \\
	&< \epsilon \int_0^\infty e^{-\lambda t}\,dt \nonumber \\
	&= \frac{\epsilon}{\lambda}. \label{1_1_phihead}
\end{align}
To estimate the contribution to the integral from the remaining interval $(\delta,T)$ we will separate $f(t)$ and $f(0)$,
\[
	\left|\int_\delta^T (f(t) - f(0)) e^{-\lambda t}\,dt\right| \leq \int_\delta^T |f(t)|  e^{-\lambda t}\,dt + \int_\delta^T |f(0)|  e^{-\lambda t}\,dt,
\]
and treat each piece individually.  For the first we apply the assumption that $f e^{-\ell t}$ is integrable to obtain the estimate
\begin{align}
\int_\delta^T |f(t)|  e^{-\lambda t}\,dt &= \int_\delta^T |f(t)| e^{(\ell-\lambda)t} e^{-\ell t}\,dt \nonumber \\
	&\leq e^{(\ell-\lambda)\delta} \int_\delta^T |f(t)| e^{-\ell t}\,dt
\label{1_1_phitail2}
\end{align}
for $\lambda \geq \ell$.  Since $\lambda e^{(\ell-\lambda)\delta} \to 0$ as $\lambda \to \infty$ we can therefore find an $M > 0$ such that
\begin{equation}
	\int_\delta^T |f(t)|  e^{-\lambda t}\,dt < \frac{\epsilon}{\lambda}
\label{1_1_phitail1}
\end{equation}
for all $\lambda > M$.  Handling the second integral is similar; we only need to note that
\begin{align*}
\int_\delta^T |f(0)| e^{-\lambda t}\,dt &\leq \int_\delta^\infty |f(0)|  e^{-\lambda t}\,dt \\
	&= \frac{|f(0)|}{\lambda e^{\delta \lambda}} \\
	&< \frac{\epsilon}{\lambda}
\end{align*}
for $\lambda$ large enough.  Combining this and \eqref{1_1_phitail1} with the estimate in \eqref{1_1_phihead} we may conclude that
\[
	\left|\int_0^T (f(t)-f(0)) e^{-\lambda t}\,dt\right| < \frac{3\epsilon}{\lambda}
\]
for all $\lambda$ large enough.  Since $\epsilon$ was arbitrary, this is equivalent to the statement
\[
	\int_0^T (f(t)-f(0)) e^{-\lambda t}\,dt = o\!\left(\frac{1}{\lambda}\right),
\]
which is what was to be shown.
\end{proof}

The assumption that $f(0) \neq 0$ is definitely necessary in Theorem \ref{babywatson}---after all, a conclusion like ``$F(\lambda) \sim 0$'' wouldn't make sense with our definition of ``$\sim$''.  However, in the course of the proof we actually showed that if $f$ satisfies all of the other assumptions and $f(0) = 0$ then
\[
	\int_0^T f(t)e^{-\lambda t}\,dt = o\!\left(\frac{1}{\lambda}\right)
\]
as $\lambda \to \infty$.  If we go back and examine that part of the proof in detail (when we integrate $f(t) - f(0)$ over the the interval $(0,\delta)$) we get a better idea of just what kind of information we need about $f$: we need to know more about how $f(t) \to f(0)$ as $t \to 0$.  There are myriad conditions we could consider, but instead of going into full generality we'll just consider two regimes which tend to be very useful.

\subsection{Integrands Asymptotic to Powers of $t$}

In this first regime, $f(t)$ behaves like a series of powers of $t$ as $t \to 0$.  The result treating this case is known as Watson's lemma.

\begin{theorem}[Watson's lemma]
Fix $0 < T \leq \infty$ and let $f \colon [0,T] \to \C \cup \{\infty\}$ be a measurable function satisfying
\[
	\int_0^T |f(t)| e^{-\ell t}\,dt < \infty
\]
for some constant $\ell \geq 0$.  Suppose that
\[
	f(t) \dsim \sum_{k=0}^{\infty} a_k t^{b_k}
\]
as $t \to 0^+$, where $-1 < \re b_k < \re b_{k+1}$ for all $k$.  Then
\[
	F(\lambda) := \int_0^T f(t)e^{-\lambda t}\,dt \dsim \sum_{k=0}^{\infty} \frac{a_k \Gamma(b_k+1)}{\lambda^{b_k+1}}
\]
as $\lambda \to \infty$ with $\lambda > 0$.
\label{watsonslemma}
\end{theorem}

The proof will have three main steps.  First, since the asymptotic series only approximates $f(t)$ well near $t=0$, we will remove the tail of the integral in $F$.  Then, having only the head of the integral, we'll approximate $f$.  Once we've done that we'll reattach appropriate tails to any remaining integrals.  This kind of argument will reappear in various forms all throughout this appendix.

\begin{proof}
Fix any integer $n \geq 1$.  By assumption we can find an $M > 0$ and a $\delta$ with $0 < \delta < T$ such that
\begin{equation}
	\left|f(t) - \sum_{k=0}^{n-1} a_k t^{b_k} \right| \leq M t^{b_n}
\label{watson_phiapprox}
\end{equation}
for all $0 \leq t \leq \delta$.  Using this $\delta$ we split the integral in $F$ like
\begin{equation}
	F(\lambda) = \int_0^\delta f(t)e^{-\lambda t}\,dt + \int_\delta^T f(t)e^{-\lambda t}\,dt.
\label{watson_split}
\end{equation}

The second integral in \eqref{watson_split} will be negligible---this is identical to what happened in the proof of Theorem \ref{babywatson}.  Its interval of integration, $(\delta,T)$, is bounded away from the main contribution coming from $t=0$, so the exponential factor in the integrand, $e^{-\lambda t}$, will dominate and cause the integral to decay exponentially as $\lambda \to \infty$.  Indeed, repeating the calculations in \eqref{1_1_phitail2} we find that
\begin{equation}
	\int_\delta^T f(t)e^{-\lambda t}\,dt = O\!\left(e^{-\delta\lambda}\right)
\label{watson_tail1}
\end{equation}
as $\lambda \to \infty$.  We will appeal to this kind of heuristic to inform our rigorous calculations whenever we need to remove a tail from or attach a tail to an integral.

Next we'll attempt to use our asymptotic approximation for $f$ in the first integral in \eqref{watson_split}.  To start we write
\[
	\int_0^\delta f(t)e^{-\lambda t}\,dt = \sum_{k=0}^{n-1} a_k \int_0^\delta t^{b_k} e^{-\lambda t}\,dt + \int_0^\delta \left(f(t) - \sum_{k=0}^{n-1} a_k t^{b_k}\right)e^{-\lambda t}\,dt.
\]
Using \eqref{watson_phiapprox} we estimate the remainder integral like
\begin{align*}
\left|\int_0^\delta \left(f(t) - \sum_{k=0}^{n-1} a_k t^{b_k}\right)e^{-\lambda t}\,dt\right| &\leq \int_0^\delta \left|f(t) - \sum_{k=0}^{n-1} a_k t^{b_k}\right|e^{-\lambda t}\,dt \\
	&\leq M \int_0^\delta t^{b_n} e^{-\lambda t}\,dt \\
	&< M \int_0^\infty t^{b_n} e^{-\lambda t}\,dt \\
	&= \frac{M \Gamma(b_n+1)}{\lambda^{b_n+1}},
\end{align*}
so that
\begin{equation}
	\int_0^\delta f(t)e^{-\lambda t}\,dt = \sum_{k=0}^{n-1} a_k \int_0^\delta t^{b_k} e^{-\lambda t}\,dt + O\!\left(\frac{1}{\lambda^{b_n+1}}\right).
\label{watson_headest}
\end{equation}

Now we'll attach the tails onto the $n$ integrals appearing in \eqref{watson_headest}.  We write
\begin{align*}
\int_0^\delta t^{b_k} e^{-\lambda t}\,dt &= \int_0^\infty t^{b_k} e^{-\lambda t}\,dt - \int_\delta^\infty t^{b_k} e^{-\lambda t}\,dt \\
	&= \frac{\Gamma(b_k+1)}{\lambda^{b_k+1}} - \int_\delta^\infty t^{b_k} e^{-\lambda t}\,dt,
\end{align*}
and note that the remainder integral decays exponentially:
\begin{align*}
0 < \int_\delta^\infty t^{b_k} e^{-\lambda t}\,dt &= e^{-\delta\lambda}\int_0^\infty (s+\delta)^{b_k} e^{-\lambda s}\,ds \\
	&< e^{-\delta\lambda} \int_0^\infty (s+\delta)^{b_k} e^{-s}\,ds,
\end{align*}
where the inequality holds for $\lambda > 1$.  Thus
\[
	\int_0^\delta t^{b_k} e^{-\lambda t}\,dt = \frac{\Gamma(b_k+1)}{\lambda^{b_k+1}} + O\!\left(e^{-\delta\lambda}\right)
\]
as $\lambda \to \infty$.  Substituting this in \eqref{watson_headest} yields
\begin{align}
\int_0^\delta f(t)e^{-\lambda t}\,dt &= \sum_{k=0}^{n-1} a_k \left[\frac{\Gamma(b_k+1)}{\lambda^{b_k+1}} + O\!\left(e^{-\delta\lambda}\right)\right] + O\!\left(\frac{1}{\lambda^{b_n+1}}\right) \nonumber \\
	&= \sum_{k=0}^{n-1} \frac{a_k \Gamma(b_k+1)}{\lambda^{b_k+1}} + O\!\left(e^{-\delta\lambda}\right) + O\!\left(\frac{1}{\lambda^{b_n+1}}\right).
\label{watson_headest2}
\end{align}
Note that we only added together finitely many terms of the form $O(e^{-\delta \lambda})$.  Even though the implied constants in each $O(\cdots)$ are different, the result is still $O(e^{-\delta \lambda})$.

Combining \eqref{watson_tail1} and \eqref{watson_headest2} in \eqref{watson_split} we find that
\begin{align*}
\int_0^T f(t)e^{-\lambda t}\,dt &= \sum_{k=0}^{n-1} \frac{a_k \Gamma(b_k+1)}{\lambda^{b_k+1}} + O\!\left(e^{-\delta\lambda}\right) + O\!\left(\frac{1}{\lambda^{b_n+1}}\right) \\
&= \sum_{k=0}^{n-1} \frac{a_k \Gamma(b_k+1)}{\lambda^{b_k+1}} + O\!\left(\frac{1}{\lambda^{b_n+1}}\right)
\end{align*}
as $\lambda \to \infty$.  Since $n$ was arbitrary, this completes the proof.
\end{proof}

\begin{remark}
Since
\[
	\int_0^\infty t^b e^{-\lambda t}\,dt = \frac{\Gamma(b+1)}{\lambda^{b+1}}
\]
provided $\re b > -1$ and $\lambda > 0$, another form of the conclusion of Watson's lemma is that
\[
	F(\lambda) := \int_0^T f(t)e^{-\lambda t}\,dt \dsim \sum_{k=0}^{\infty} a_k \int_0^\infty t^{b_k} e^{-\lambda t}\,dt
\]
as $\lambda \to \infty$ with $\lambda > 0$.  One informal interpretation of this is that $\int \sum \dsim \sum \int$ under the hypotheses of Watson's lemma.  That is to say, the series and the integral can be exchanged in this case \textit{as long as both series are interpreted as asymptotic series}.  This a sort of asymptotic dominated convergence result.
\end{remark}

\subsection{Integands Asymptotic to Powers of Logarithms and Powers of $t$}

Now we will consider the case when the asymptotic series for $f(t)$ as $t \to 0$ involves powers of \textit{logarithms} of $t$ in addition to powers of $t$, say
\[
	f(t) \dsim \sum_{k=0}^{\infty} a_k (-\log t)^{b_k} t^{c_k}
\]
as $t \to 0^+$.  An astute practitioner would try to substitute this series into the integral for $F(\lambda)$ but would quickly run into integrals like
\[
	\int_0^\delta (-\log t)^b t^c e^{-\lambda t}\,dt,
\]
which have no closed form.  Apparently we need to investigate their asymptotic behavior.

The result below is due to Erd\'elyi, and an alternate proof, along with other interesting related results, can be found in the paper \cite{erdelyi:generalasymp}.

\begin{theorem}
Suppose that $0 < \delta < 1$, $\re a > 0$, and $b \in \R$.  Then
\[
	\int_0^\delta (-\log t)^b t^{a-1} e^{-\lambda t}\,dt \dsim \lambda^{-a} \sum_{k=0}^{\infty} (-1)^k \binom{b}{k} \Gamma^{(k)}(a) (\log \lambda)^{b-k}
\]
as $\lambda \to \infty$ with $\lambda > 0$.\newnot{symbol:binom}
\label{logthm}
\end{theorem}

\begin{proof}
Fix an integer $n \geq 0$.  Taylor's theorem with remainder gives us the identity
\begin{align}
(1+x)^b &= \sum_{k=0}^{n} \binom{b}{k} x^k + (b-n) \binom{b}{n} x^{n+1} \int_0^1 (1+xt)^{b-n-1} (1-t)^n\,dt \nonumber \\
	&= \sum_{k=0}^{n} \binom{b}{k} x^k + C_{b,n} x^{n+1} g(b,n,x),
\label{binomialthm}
\end{align}
valid for $x > -1$.

Turning to the integral in question, the substitution $s = \lambda t$ gives us
\begin{align*}
\int_0^\delta (-\log t)^b t^{a-1} e^{-\lambda t}\,dt &= \lambda^{-a} \int_{0}^{\delta \lambda} (\log \lambda - \log s)^b s^{a-1} e^{-s}\,ds \\
	&= \lambda^{-a} (\log \lambda)^b \int_0^{\delta\lambda} \left(1-\frac{\log s}{\log \lambda}\right)^b s^{a-1} e^{-s}\,ds.
\end{align*}
Setting $x = -\log s/\log \lambda$ in \eqref{binomialthm} and inserting it into the above expression yields
\begin{align}
&\lambda^a (\log\lambda)^{-b} \int_0^\delta (-\log t)^b t^{a-1} e^{-\lambda t}\,dt \nonumber \\
	&\qquad {}= \sum_{k=0}^{n} (-1)^k \binom{b}{k} (\log\lambda)^{-k} \int_0^{\delta\lambda} (\log s)^k s^{a-1} e^{-s}\,ds \nonumber \\
	&\qquad\qquad {}+ C_{b,n} (\log\lambda)^{-n-1} \int_0^{\delta\lambda} (-\log s)^{n+1} g\!\left(b,n,-\frac{\log s}{\log\lambda}\right)s^{a-1}e^{-s}\,ds.
\label{logthmerror1}
\end{align}

To aid us in estimating the last integral in the above expression we will assume for the moment that $n$ and $\lambda$ are large enough to have $n > b-1$ and $\delta\lambda > 1$.  For $0 < s < 1$ we have
\[
	g\!\left(b,n,-\frac{\log s}{\log\lambda}\right) = \int_0^1 \left(1-\frac{t\log s}{\log \lambda}\right)^{b-n-1}(1-t)^n\,dt < \int_0^1 (1-t)^n\,dt = \frac{1}{n+1},
\]
and for $1 < s < \delta\lambda$
\begin{align*}
g\!\left(b,n,-\frac{\log s}{\log\lambda}\right) &< \left(1-\frac{\log s}{\log \lambda}\right)^{b-n-1} \int_0^1 (1-t)^n\,dt \\
	&< \frac{(-\log\delta)^{b-n-1}}{n+1}.
\end{align*}
On the whole interval $0 < s < \delta\lambda$ we have $g(b,n,-\log s/\log\lambda) \geq 0$, so we can conclude that
\begin{align*}
&\left|\int_0^{\delta\lambda} (-\log s)^{n+1} g\!\left(b,n,-\frac{\log s}{\log\lambda}\right)s^{a-1}e^{-s}\,ds\right| \\
&\qquad\quad {}< \frac{1+(-\log\delta)^{b-n-1}}{n+1} \int_0^{\delta\lambda} |\log s|^{n+1} s^{a-1} e^{-s}\,ds \\
&\qquad\quad {}< \frac{1+(-\log\delta)^{b-n-1}}{n+1} \int_0^\infty |\log s|^{n+1} s^{a-1} e^{-s}\,ds
\end{align*}
and therefore, combining this with \eqref{logthmerror1}, that
\begin{align}
&\lambda^a (\log\lambda)^{-b} \int_0^\delta (-\log t)^b t^{a-1} e^{-\lambda t}\,dt \nonumber \\
	&\qquad {}= \sum_{k=0}^{n} (-1)^k \binom{b}{k} (\log\lambda)^{-k} \int_0^{\delta\lambda} (\log s)^k s^{a-1} e^{-s}\,ds + O\!\left((\log\lambda)^{-n-1}\right).
\label{logthmnotails}
\end{align}

The last thing we need to do is attach the tails onto the integrals in the sum.  For each $k=0,...,n$ they can be expressed as
\begin{align*}
\int_0^{\delta\lambda} (\log s)^k s^{a-1} e^{-s}\,ds &= \int_0^\infty (\log s)^k s^{a-1} e^{-s}\,ds - \int_{\delta\lambda}^\infty (\log s)^k s^{a-1} e^{-s}\,ds \\
	&= \Gamma^{(k)}(a) - \int_{\delta\lambda}^\infty (\log s)^k s^{a-1} e^{-s}\,ds,
\end{align*}
and the error in each case is exponentially small:
\begin{align*}
\left|\int_{\delta\lambda}^\infty (\log s)^k s^{a-1} e^{-s}\,ds\right| &\leq \int_{\delta\lambda}^\infty (\log s)^k s^{\re a-1} e^{-s}\,ds \\
	&< \int_{\delta\lambda}^\infty s^{\re a+k-1} e^{-s}\,ds \\
	&= \lambda^{\re a+k} e^{-\delta\lambda} \int_0^\infty (\delta+r)^{\re a+k-1} e^{-\lambda r}\,dr \\
	&< \lambda^{\re a+k} e^{-\delta\lambda} \int_0^\infty (\delta+r)^{\re a+k-1} e^{-r}\,dr,
\end{align*}
where we made the substitution $s=\lambda(\delta+r)$ in the third line.  Thus
\[
	\int_0^{\delta\lambda} (\log s)^k s^{a-1} e^{-s}\,ds = \Gamma^{(k)}(a) + O\!\left(\lambda^{\re a+k} e^{-\delta\lambda}\right)
\]
for each $k$, and so from \eqref{logthmnotails} we get
\begin{align}
&\int_0^\delta (-\log t)^b t^{a-1} e^{-\lambda t}\,dt \nonumber \\
	&\qquad {}= \lambda^{-a} \sum_{k=0}^{n} (-1)^k \binom{b}{k} \Gamma^{(k)}(a) (\log\lambda)^{b-k} + O\!\left(\lambda^{-a}(\log\lambda)^{b-n-1}\right).
\label{logthmfinal}
\end{align}

We assumed that $n > b-1$, while we really want \eqref{logthmfinal} to hold for all $n = 0,1,2,\ldots$.  Actually, this follows easily.  If $b-1 \geq 0$ then set $n=\lfloor b-1 \rfloor + 1$ and fix $m \in \{0,1,\ldots,n-1\}$.  By \eqref{logthmfinal} we indeed have
\begin{align*}
&\int_0^\delta (-\log t)^b t^{a-1} e^{-\lambda t}\,dt \\
	&\qquad {}= \lambda^{-a} \left(\sum_{k=0}^{m} + \sum_{k=m+1}^{n} \right) (-1)^k \binom{b}{k} \Gamma^{(k)}(a) (\log\lambda)^{b-k} + O\!\left(\lambda^{-a}(\log\lambda)^{b-n-1}\right) \\
	&\qquad {}= \lambda^{-a} \sum_{k=0}^{m} (-1)^k \binom{b}{k} \Gamma^{(k)}(a) (\log\lambda)^{b-k} + O\!\left(\lambda^{-a}(\log\lambda)^{b-m-1}\right),
\end{align*}
which concludes the proof.
\end{proof}

\section{More Complicated Exponents}

Now we'll investigate how to deal with integrals of the form
\[
	F(\lambda) = \int_a^b f(t)e^{\lambda g(t)}\,dt.
\]
Just as in the previous section, if $f(t)$ is nice enough then the $\lambda \to \infty$ behavior of the integrand will essentially be dominated by the exponential factor $e^{\lambda g(t)}$. In that special case we had $ g(t) = -t$, and the maximum of this quantity occurred at the left endpoint of the interval of integration.  For general $ g$ the situation can be much more complicated:
\begin{itemize}
\item $ g(t)$ can have a global maximum at $t=a$, $t=b$, or both;
\item $ g(t)$ can have one or more global maxima inside the interval $(a,b)$;
\item any such global maximum may or may not be a critical point of $ g(t)$ with high-order degeneracy.
\end{itemize}
In any case, though, the method of attack remains the same: approximate the integrand near its largest points.  Here that means we'll be approximating $f$ and $ g$ near any global maxima of $ g$.

One principle that simplifies our work is that \textit{the contribution of each maximum of $ g$ can be estimated independently}.  Let's prove a small lemma toward this.

\begin{lemma}
Fix $-\infty \leq a < b \leq \infty$ and let $f \colon [a,b] \to \C \cup \{\infty\}$ and $ g \colon [a,b] \to \R$ be measurable functions satisfying
\[
	\int_a^b |f(t)| e^{\ell g(t)}\,dt < \infty
\]
for some constant $\ell \geq 0$.  If $ g(t) < M$ for $I \subset [a,b]$ then
\[
	\int_I f(t)e^{\lambda  g(t)}\,dt = O\!\left(e^{M\lambda}\right)
\]
for $\lambda \geq \ell$.
\label{laplacelemma}
\end{lemma}

\begin{proof}
The calculation is pretty similar to the ones we have done before:
\begin{align*}
\left|\int_I f(t) e^{\lambda g(t)}\,dt\right| &\leq \int_I |f(t)| e^{\lambda g(t)}\,dt \\
&= \int_I |f(t)| e^{(\lambda-\ell) g(t)} e^{\ell g(t)}\,dt \\
&\leq e^{(\lambda-\ell)M} \int_I |f(t)| e^{\ell g(t)}\,dt \\
&\leq e^{(\lambda-\ell)M} \int_a^b |f(t)| e^{\ell g(t)}\,dt
\end{align*}
for $\lambda \geq \ell$.
\end{proof}

So, suppose that $ g\colon [a,b] \to \R$ is a continuous function that has global maxima at the points $t_1,\ldots,t_n \in [a,b]$, i.e.\ that 
\[
	 g(t_1) = \cdots =  g(t_n) >  g(t), \quad t \in [a,b] \setminus \{ t_1,\ldots,t_n\}.
\]
We can then find disjoint intervals $I_1,\ldots,I_n \subset [a,b]$ with $t_j \in I_j$, $j=1,\ldots,n$ and a constant $\delta > 0$ such that $ g(t) <  g(t_1) - \delta$ for all $t \in [a,b] \setminus (I_1 \cup \cdots \cup I_n)$.  So, if $f$ is a function such that the hypotheses of Lemma \ref{laplacelemma} are satisfied, taking $M =  g(t_1) - \delta$ and $I = [a,b] \setminus (I_1 \cup \cdots \cup I_n)$ in the lemma tells us that
\begin{align}
\int_a^b f(t)e^{\lambda  g(t)}\,dt &= \sum_{k=1}^{n} \int_{I_k} f(t)e^{\lambda  g(t)}\,dt + \int_{[a,b] \setminus (I_1 \cup \cdots \cup I_n)} f(t)e^{\lambda  g(t)}\,dt \nonumber \\
	&= \sum_{k=1}^{n} \int_{I_k} f(t)e^{\lambda  g(t)}\,dt + O\!\left(e^{( g(t_1)-\delta) \lambda}\right)
\label{laplaceindependence}
\end{align}
as $\lambda \to \infty$, $\lambda > 0$.

We are now free to estimate each integral in the sum on its own.  Of course, we're banking on the fact that the leading-order behavior of these integrals is asymptotically larger than the $O(\cdots)$ term, but this will indeed be the case: it will turn out that their behavior will involve a factor of $e^{\lambda g(t_1)}$ multiplied by some other subexponential factors.  As such, the integrals in the sum will, in fact, be \textit{exponentially} larger than the $O(\cdots)$ term.

In light of these observations we can simplify the discussion in this section and only consider exponent functions $g$ with a single global maximum.  We'll focus on just two common situations, illustrating different methods in each case.  It might be instructive, after seeing both methods, to try applying each of them in the other case.

\subsection{A Simple Maximum on the Boundary}

In this section we'll suppose that the exponent function $ g\colon[a,b] \to \R$ has a maximum at the left endpoint $t=a$ of the interval we're integrating over.  We'll also suppose that $ g'(a) < 0$, which means that this maximum isn't a critical point of $g$.  Techniques to handle a maximum at a critical point of $g$ can be found in the next section.

One thing that I should point out about the theorems in this section and the next is that we won't explicitly assume that $g$ has a unique global maximum at $t=a$, but rather that $g(t)$ is forever bounded below $g(a)$ by some fixed amount as soon as $t > a$.  This does imply that $g(t)$ has a unique global maximum at $t=a$, but the two concepts aren't equivalent.  For example, the function defined by
\[
	g(t) = \begin{cases}
				-t & 0 \leq t < 1, \\
				-t^{-1} & t \geq 1
			\end{cases}
\]
has a unique global maximum at $t=0$, but, after dropping below $g(0)$, $g(t)$ eventually rises again and gets arbitrarily close to $g(0)$ as $t \to \infty$.  This $g$ is therefore not one of the ones considered in the following theorems.  We'll revisit this example after the proof.

\begin{theorem}
Fix $-\infty < a < b \leq \infty$ and let $f \colon [a,b] \to \R \cup \{\infty\}$ and $ g \colon [a,b] \to \R$ be measurable functions satisfying
\[
	\int_a^b |f(t)| e^{\ell g(t)}\,dt < \infty
\]
for some constant $\ell \geq 0$.  Suppose that $g'(a)$ exists with $g'(a) < 0$, and that for any $\delta > 0$ we can find an $\eta(\delta) > 0$ such that $g(t) < g(a) - \eta(\delta)$ whenever $a+\delta < t < b$.  Suppose also that
\[
	f(t) \sim (t-a)^p
\]
as $t \to a^+$ with $p > -1$.  Then
\[
	F(\lambda) := \int_a^b f(t) e^{\lambda g(t)}\,dt \sim \frac{\Gamma(p+1)}{(- g'(a)\lambda)^{p+1}} \,e^{g(a)\lambda}
\]
as $\lambda \to \infty$ with $\lambda > 0$.
\label{laplaceboundary}
\end{theorem}

The idea behind the statement of this theorem and its proof is to try to get the leading order asymptotic for $F(\lambda)$ using the least amount of information about $g$ as possible.  By only assuming that $g'(a)$ exists we limit ourselves to very weak bounds for $g(t)$ as $t \to a^+$, and are therefore limited to upper and lower bounds for $F(\lambda)$ which are within multiplicative factors of $1\pm\epsilon$ of the desired asymptotic.  We then perform a final limiting step taking $\epsilon \to 0$ to conclude the result.

\begin{proof}
To start let's make the change of variables $s = t-a$ in the integral and rename $ h(s) =  g(s+a)$, so that
\[
	F(\lambda) = \int_0^{b-a} f(s+a) e^{\lambda h(s)}\,ds.
\]

By assumption $ h'(0)$ exists and is negative.  If we define
\[
	H(s) =  h(s) -  h(0) -  h'(0)s
\]
then $H(0) = H'(0) = 0$, so that
\[
	\frac{H(s) - H(0)}{s-0} = \frac{H(s)}{s} \to 0
\]
as $s \to 0$.  It follows that for all $\epsilon$ with $0 < \epsilon < - h'(0)$ we can find a $\delta$ with $0 < \delta < b-a$ such that
\[
	\left| h(s) -  h(0) -  h'(0)s\right| \leq \epsilon s
\]
and
\[
	(1-\epsilon)s^p < f(s+a) < (1+\epsilon)s^p
\]
for all $0 \leq s \leq \delta$.  We then have
\begin{align}
&(1-\epsilon)e^{\lambda h(0)} \int_0^\delta s^p e^{\lambda( h'(0)-\epsilon)s}\,ds \nonumber \\
	&\hspace{2cm} < \int_0^\delta f(s+a) e^{\lambda h(s)}\,ds \nonumber \\
		&\hspace{4cm} < (1+\epsilon)e^{\lambda h(0)} \int_0^\delta s^p e^{\lambda( h'(0)+\epsilon)s}\,ds.
\label{lapendnotails}
\end{align}
By assumption we know that $h(s) = g(t) < g(a)-\eta(\delta) = h(0) - \eta(\delta)$ for $s = t-a > \delta$, so by Lemma \ref{laplacelemma} we have
\begin{align*}
&\int_0^\delta f(s+a) e^{\lambda h(s)}\,ds = F(\lambda) + O\!\left(e^{(h(0)-\eta(\delta))\lambda}\right), \\
&\int_0^\delta s^p e^{\lambda( h'(0)-\epsilon)s}\,ds = \frac{\Gamma(p+1)}{[-(h'(0)-\epsilon)\lambda]^{p+1}} + O\!\left(e^{\delta (h'(0)-\epsilon) \lambda}\right), \\
&\int_0^\delta s^p e^{\lambda( h'(0)+\epsilon)s}\,ds = \frac{\Gamma(p+1)}{[-(h'(0)+\epsilon)\lambda]^{p+1}} + O\!\left(e^{\delta (h'(0)+\epsilon) \lambda}\right).
\end{align*}
We can then write
\[
	\frac{(1-\epsilon)\Gamma(p+1)}{[-(h'(0)-\epsilon)\lambda]^{p+1}} + O\!\left(e^{-\mu \lambda}\right) < e^{-h(0)\lambda} F(\lambda) < \frac{(1+\epsilon)\Gamma(p+1)}{[-(h'(0)+\epsilon)\lambda]^{p+1}} + O\!\left(e^{-\mu \lambda}\right)
\]
as $\lambda \to \infty$ for some $\mu > 0$ which does not depend on $\lambda$, and hence
\[
	\frac{(1-2\epsilon)\Gamma(p+1)}{[-(h'(0)-\epsilon)\lambda]^{p+1}} < e^{-h(0)\lambda} F(\lambda) < \frac{(1+2\epsilon)\Gamma(p+1)}{[-( h'(0)+\epsilon)\lambda]^{p+1}}
\]
for $\lambda > 0$ large enough.  Since $\epsilon$ was arbitrary, it follows that
\[
	F(\lambda) \sim \frac{\Gamma(p+1)}{(- h'(0)\lambda)^{p+1}} \,e^{ h(0)\lambda} = \frac{\Gamma(p+1)}{(- g'(a)\lambda)^{p+1}} \,e^{ g(a)\lambda}
\]
as $\lambda \to \infty$, which is what was to be shown.
\end{proof}

Another example of this method can be found in \cite[sec.\ 4.2]{debruijn:ama} for the case when $g$ has a critical point in the interior of the interval (like in the next section).

Let's revisit the exponent function example we mentioned before the theorem, namely
\[
	g(t) = \begin{cases}
				-t & 0 \leq t < 1, \\
				-t^{-1} & t \geq 1.
			\end{cases}
\]
If we take
\[
	f(t) = \begin{cases}
				1 & 0 \leq t < 1, \\
				t^{-2} & t \geq 1
			\end{cases}
\]
then all hypotheses of Theorem \ref{laplaceboundary} are satisfied except for one: $g$ does have a unique global maximum at $t=0$, but the necessary quantity $\eta(\delta)$ doesn't exist for any $\delta > 0$.  If we naively tried to apply the result we might reason that since $f(t) \sim 1$ and $g(t) \sim -t$ as $t \to 0^+$ we should have
\[
	\int_0^\infty f(t) e^{\lambda g(t)}\,dt \stackrel{?}{\sim} \int_0^\infty e^{-\lambda t}\,dt = \frac{1}{\lambda}
\]
as $\lambda \to \infty$.  It turns out that this isn't quite correct, and that in this case we actually do get non-negligible contributions coming from the tail of the integral $\int_\delta^\infty \cdots \,dt$.

If we split the integral at $t=1$ and write
\[
	\int_0^\infty f(t) e^{\lambda g(t)}\,dt = \int_0^1 e^{-\lambda t}\,dt + \int_1^\infty t^{-2} e^{-\lambda/t}\,dt
\]
then make the change of variables $t = 1/u$ in the second integral we get
\[
	\int_1^\infty t^{-2} e^{-\lambda/t}\,dt = \int_0^1 e^{-\lambda u}\,du.
\]
Apparently the tail actually makes a contribution of $1/\lambda$ to the integral as well.  Consequently, the correct asymptotic is
\[
	\int_0^\infty f(t) e^{\lambda g(t)}\,dt \sim \frac{2}{\lambda},
\]
which is twice our naive estimate.

It's worth noting that it is possible for the naive guess to be off by much more than just a constant factor.  Take, for example,
\[
	g(t) = \begin{cases}
				-t & 0 \leq t < 1, \\
				-e^{-t} & t \geq 1,
			\end{cases}
\]
with the same $f$ as before.  Here, splitting the integral and making the change of variables $e^{-t} = u$ yields
\begin{align*}
\int_0^\infty f(t) e^{\lambda g(t)}\,dt &= \int_0^1 e^{-\lambda t}\,dt + \int_1^\infty t^{-2} \exp\!\left(-\lambda e^{-t}\right) dt \\
&= \int_0^1 e^{-\lambda t}\,dt + \int_0^{1/e} (-\log u)^{-2} u^{-1} e^{-\lambda u} du.
\end{align*}
By using Theorem \ref{logthm} we can show that the second integral is asymptotic to $1/\log\lambda$, so in fact
\[
	\int_0^\infty f(t) e^{\lambda g(t)}\,dt \sim \frac{1}{\log\lambda}
\]
as $\lambda \to \infty$.  This is very far off from the naive guess of $1/\lambda$.

Point being, if an $\eta$ function doesn't exist, one must take special care to also estimate the asymptotic contribution coming from the tail of the integral.

\begin{remark}
Theorem \ref{laplaceboundary} is also true when $p \in \mathbb C$ with $\re p > -1$, though the proof above doesn't really generalize to complex $p$.
\end{remark}

\subsection{A Critical Point in the Interior}

In this section we'll suppose that $g$ has a maximum at some $t_0 \in (a,b)$.  The method we'll use requires that we know a little more about $f$ and $g$ than in the previous section, but in return we can quantify relative error of the resulting asymptotic for $F(\lambda)$.

\begin{theorem}
Fix $-\infty \leq a < b \leq \infty$ and let $f \colon [a,b] \to \C \cup \{\infty\}$ and $g \colon [a,b] \to \R$ be measurable functions satisfying
\[
	\int_a^b |f(t)| e^{\ell g(t)}\,dt < \infty
\]
for some constant $\ell \geq 0$.  Suppose that $g'(t_0) = 0$ and $g''(t_0) < 0$, that $ g'''(t_0)$ exists, and that for any $\delta > 0$ we can find an $\eta(\delta) > 0$ such that $g(t) < g(a) - \eta(\delta)$ for all $t \in (a,b)$ with $|t-t_0| > \delta$.  Suppose also that
\[
	f(t) = (t-t_0)^{2n} + O\!\left((t-t_0)^{2n+1}\right)
\]
as $t \to t_0$ for some integer $n \geq 0$.  Then
\[
	F(\lambda) := \int_a^b f(t) e^{\lambda g(t)}\,dt = \Gamma\!\left(n+\frac{1}{2}\right) \left(\frac{2}{-g''(t_0)\lambda}\right)^{n+1/2} e^{g(t_0)\lambda} \left[1 + O\!\left(\lambda^{-1/2}\right)\right]
\]
as $\lambda \to \infty$ with $\lambda > 0$.
\label{laplacemethod}
\end{theorem}

\begin{proof}
Our assumptions imply that $g$ has a maximum at $t=t_0$, and so we infer that $g'(t_0) = 0$.  To simplify things a little we'll make the translation $t=s+t_0$ and rename $h(s) = g(t)$, so that
\[
	F(\lambda) = \int_{a-t_0}^{b-t_0} f(s+t_0) e^{\lambda h(s)}\,ds.
\]
The new exponent function $h(s)$ has a maximum at $s=0$.  The main idea for the proof is that we want to replace $h(s)$ and $f(s+t_0)$ with simple approximations near this maximum, keeping track throughout of the errors we introduce in doing so.

By making an argument similar to the one in the beginning of the proof of Theorem \ref{laplaceboundary} (or simply by appealing to Taylor's theorem) it can be deduced from the assumptions that
\begin{align}
h(t) &= h(0) + \frac{h''(0)}{2} s^2 + \frac{h'''(0)}{6} s^3 + o\!\left(s^3\right) \nonumber \\
	&= h(0) + \frac{h''(0)}{2} s^2 + O\!\left(s^3\right)
\label{laphestimate}
\end{align}
as $s \to 0$.

We wish to use $h(0) + h''(0) s^2/2$ as an approximation for $h(s)$ near $s=0$, and near $s=0$ the error in doing so is $O(s^3)$.  The key step in the proof is that we will use this $O(s^3)$ error to find the size of the interval about $s=0$ which contains all of the asymptotic information about the integral.  For now we just note that because of the the asymptotic in \eqref{laphestimate} we can find a $\delta > 0$ such that $(-\delta,\delta) \subset (a-t_0,b-t_0)$ and
\begin{equation}
	h(t) < h(0) + \frac{h''(0)}{4}s^2
\label{lapdeltadef}
\end{equation}
for all $|s| < \delta$.  We then have
\begin{equation}
	F(\lambda) = \int_{-\delta}^{\delta} f(s+t_0) e^{\lambda h(s)}\,ds + O\!\left(e^{(h(0)-\eta(\delta))\lambda}\right).
\label{lapdeltaest}
\end{equation}

Now, suppose that $\varphi$ is some function with $-A|x| \leq \varphi(x) \leq A|x|$ for all $x \in [-B,B]$.  Then $|\varphi(x)| \leq AB$ and
\begin{align*}
\left| \frac{e^{\varphi(x)} - 1}{x} \right| & = \left| \frac{\varphi(x)}{x} \sum_{k=0}^{\infty} \frac{|\varphi(x)|^k}{(k+1)!} \right| \\
	&\leq A \sum_{k=0}^{\infty} \frac{|\varphi(x)|^k}{(k+1)!} \\
	&\leq A e^{|\varphi(x)|} \\
	&\leq A e^{AB}
\end{align*}
for all $x \in [-B,B]$.  This implies that $e^{\varphi(x)} = 1 + O(x)$.  Stated differently, we have shown that
\begin{equation}
	e^{O(x)} = 1 + O(x) \quad \text{if } x = O(1).
\label{explemma}
\end{equation}

How can we use this?  As stated above we want to write
\begin{align*}
\exp(\lambda h(s)) &= \exp\!\left[\lambda\left(h(0) + \frac{h''(0)}{2} s^2 + O\!\left(s^3\right)\right)\right] \\
	&= \exp\!\left[\lambda\left(h(0) + \frac{h''(0)}{2} s^2\right)\right] \exp\!\left[ O\!\left(\lambda s^3\right) \right],
\end{align*}
and ideally we would then like to simplify the exponential with the $O(\cdots)$ term, say like
\[
	\exp\!\left[O\!\left(\lambda s^3\right)\right] = 1 + O\!\left(\lambda s^3\right).
\]
Well, \eqref{explemma} tells us that this is true as long as $\lambda s^3 = O(1)$ or, equivalently, $s = O(\lambda^{-1/3})$.  This is precisely the range of $s$ which contains all of the asymptotic information about the integral.  Specifically we will show that the contributions coming from $|s| > \lambda^{-1/3}$ are exponentially small relative to those from $|s| < \lambda^{-1/3}$.

Returning to \eqref{lapdeltaest}, we split the integral like
\[
	F(\lambda) = \left(\int_{|s| < \lambda^{-1/3}} + \int_{\lambda^{-1/3} < |s| < \delta} \right) f(s+t_0) e^{\lambda h(s)}\,ds + O\!\left(e^{(h(0)-\eta(\delta))\lambda}\right).
\]
By taking $\delta$ smaller if necessary we can ensure that $f(s+t_0)$ is bounded on $(-\delta,\delta)$, say by $|f(s+t_0)| \leq M_0$, and then using \eqref{laphestimate} we get the estimate
\begin{align*}
\left| \int_{\lambda^{-1/3} < |s| < \delta} f(s+t_0) e^{\lambda h(s)}\,ds \right| &\leq M_0 \int_{\lambda^{-1/3} < |s| < \delta} e^{\lambda h(s)}\,ds \\
	&< M_0\int_{\lambda^{-1/3} < |s| < \delta} \exp\!\left[\lambda \left(h(0) + \frac{h''(0)}{4} s^2\right)\right]\,ds \\
	&< M_0 \exp\!\left[\lambda \left(h(0) + \frac{h''(0)}{4} \lambda^{-2/3}\right)\right] \int_{\lambda^{-1/3} < |s| < \delta} ds \\
	&< 2\delta M_0 \exp\!\left[h(0)\lambda + \frac{h''(0)}{4} \lambda^{1/3}\right].
\end{align*}
This error dominates the previous error term in \eqref{lapdeltaest}, so we end up with
\begin{equation}
	F(\lambda) = \int_{|s| < \lambda^{-1/3}} f(s+t_0) e^{\lambda h(s)}\,ds + O\!\left(e^{h(0)\lambda + h''(0)\lambda^{1/3}/4}\right).
\label{lapcubicerror}
\end{equation}

With $|s| < \lambda^{-1/3}$ it is now true that
\begin{align*}
	&f(s+t_0) = s^{2n} + O\!\left(s^{2n}\right), \\
	&h(s) = h(0) + \frac{h''(0)}{2} s^2 + O\!\left(s^3\right),
\end{align*}
and
\[
	\exp\!\left[O\!\left(\lambda s^3\right)\right] = 1 + O\!\left(\lambda s^3\right),
\]
so we can insert these approximations into the integral.  We note that
\begin{align*}
\left[s^{2n}+O\!\left(s^{2n+1}\right)\right] \left[1+O\!\left(\lambda s^3\right)\right] &= s^{2n} + O\!\left(s^{2n+1}\right) + O\!\left(\lambda s^{2n+3}\right) + O\!\left(\lambda s^{2n+4}\right) \\
&= s^{2n} + O\!\left(s^{2n+1}\right) + O\!\left(\lambda s^{2n+3}\right)
\end{align*}
in this situation, so the integral splits into the three pieces
\begin{align*}
&e^{-\lambda h(0)} \int_{|s| < \lambda^{-1/3}} f(s+t_0) e^{\lambda h(s)}\,ds \\
	&\qquad {}= \int_{|s| < \lambda^{-1/3}} \left[s^{2n}+O\!\left(s^{2n+1}\right)\right] e^{\lambda h''(0) s^2/2 + O(\lambda s^3)}\,ds \\
	&\qquad {}= \int_{|s| < \lambda^{-1/3}} \left[s^{2n}+O\!\left(s^{2n+1}\right)\right] e^{\lambda h''(0) s^2/2} \left[1+O\!\left(\lambda s^3\right)\right]\,ds \\
	&\qquad {}= \int_{|s| < \lambda^{-1/3}} s^{2n} e^{\lambda h''(0) s^2/2}\,ds + \int_{|s| < \lambda^{-1/3}} O\!\left(s^{2n+1}\right) e^{\lambda h''(0) s^2/2}\,ds \\
	&\qquad\qquad\quad {}+ \lambda \int_{|s| < \lambda^{-1/3}} O\!\left(s^{2n+3}\right) e^{\lambda h''(0) s^2/2}\,ds.
\end{align*}
The two integrals involving $O(\cdots)$ factors can be handled simply; if we suppose the quantity represented by $O(s^k)$ is bounded by $M_1 |s|^k$ then
\begin{align*}
\left| \int_{|s| < \lambda^{-1/3}} O\bigl(s^k\bigr) e^{\lambda h''(0) s^2/2}\,ds \right| &\leq M_1 \int_{|s| < \lambda^{-1/3}} |s|^k e^{\lambda h''(0) s^2/2}\,ds \\
&= 2M_1 \int_0^{\lambda^{-1/3}} s^k e^{\lambda h''(0) s^2/2}\,ds \\
&< 2M_1 \int_0^\infty s^k e^{\lambda h''(0) s^2/2}\,ds \\
&= M_1 \Gamma\!\left(\frac{k+1}{2}\right) \left(\frac{2}{-h''(0)\lambda}\right)^{(k+1)/2},
\end{align*}
so taking $k=2n+1$ and $k=2n+3$ gives us the respective estimates of $O(\lambda^{-n-1})$ and $O(\lambda^{-n-2})$ for the integrals.  Finally after combining like $O(\cdots)$ terms we are left with
\[
	e^{-\lambda h(0)} \int_{|s| < \lambda^{-1/3}} f(s+t_0) e^{\lambda h(s)}\,ds = \int_{|s| < \lambda^{-1/3}} s^{2n} e^{\lambda h''(0) s^2/2}\,ds + O\!\left(\lambda^{-n-1}\right).
\]
Finally Lemma \ref{laplacelemma} tells us that
\[
	\int_{|s| > \lambda^{-1/3}} s^{2n} e^{\lambda h''(0) s^2/2}\,ds = O\!\left(e^{h''(0) \lambda^{1/3}/2}\right),
\]
so
\begin{equation}
	e^{-\lambda h(0)} \int_{|s| < \lambda^{-1/3}} f(s+t_0) e^{\lambda h(s)}\,ds = \Gamma\!\left(n+\frac{1}{2}\right) \left(\frac{2}{-h''(0)\lambda}\right)^{n+1/2} + O\!\left(\lambda^{-n-1}\right).
\label{lapwithtail}
\end{equation}

Upon substituting \eqref{lapwithtail} into \eqref{lapcubicerror} we observe that the new error term \linebreak $O(\lambda^{-n-1}e^{h(0)\lambda})$ dominates the previous one, leaving us with
\begin{align*}
	F(\lambda) &= \Gamma\!\left(n+\frac{1}{2}\right) \left(\frac{2}{-h''(0)\lambda}\right)^{n+1/2} e^{h(0)\lambda} + O\!\left(\lambda^{-n-1} e^{h(0)\lambda}\right) \\
	&= \Gamma\!\left(n+\frac{1}{2}\right) \left(\frac{2}{-h''(0)\lambda}\right)^{n+1/2} e^{h(0)\lambda} \left[1 + O\!\left(\lambda^{-1/2}\right)\right] \\
	&= \Gamma\!\left(n+\frac{1}{2}\right) \left(\frac{2}{-g''(t_0)\lambda}\right)^{n+1/2} e^{g(t_0)\lambda} \left[1 + O\!\left(\lambda^{-1/2}\right)\right]
\end{align*}
as desired.
\end{proof}

%% file: main.bbl
\providecommand{\bysame}{\leavevmode\hbox to3em{\hrulefill}\thinspace}
\providecommand{\MR}{\relax\ifhmode\unskip\space\fi MR }
\providecommand{\MRhref}[2]{%
  \href{http://www.ams.org/mathscinet-getitem?mr=#1}{#2}
}
\providecommand{\href}[2]{#2}
\begin{thebibliography}{10}

\bibitem{aands:handbook}
M.~Abramowitz and I.~A. Stegun, \emph{Handbook of {M}athematical {F}unctions
  with {F}ormulas, {G}raphs, and {M}athematical {T}ables}, National Bureau of
  Standards Applied Mathematics Series, vol.~55, For sale by the Superintendent
  of Documents, U.S. Government Printing Office, Washington, D.C., 1964.

\bibitem{asv:ke}
N.~Anderson, E.~B. Saff, and R.~S. Varga, \emph{On the {E}nestr\"om-{K}akeya
  theorem and its sharpness}, Linear Algebra Appl. \textbf{28} (1979), 5--16.

\bibitem{andrievskiiblatt:discrepancybook}
V.~V. Andrievskii and H.-P. Blatt, \emph{Discrepancy of {S}igned {M}easures and
  {P}olynomial {A}pproximation}, Springer Monographs in Mathematics,
  Springer-Verlag, New York, 2002.

\bibitem{acv:angulardistribution}
V.~V. Andrievskii, A.~J. Carpenter, and R.~S. Varga, \emph{Angular distribution
  of zeros of the partial sums of {$e^z$} via the solution of inverse
  logarithmic potential problem}, Comput. Methods Funct. Theory \textbf{6}
  (2006), no.~2, 447--458.

\bibitem{mallison:expsums}
P.~Bleher and R.~{Mallison, Jr.}, \emph{Zeros of sections of exponential sums},
  Int. Math. Res. Not. (2006), Art. ID 38937, 49.

\bibitem{buckholtz:copapprox}
J.~D. Buckholtz, \emph{Concerning an approximation of {C}opson}, Proc. Amer.
  Math. Soc. \textbf{14} (1963), 564--568.

\bibitem{buckholtz:expcharacter}
\bysame, \emph{A characterization of the exponential series}, Amer. Math.
  Monthly \textbf{73} (1966), no.~4, part II, 121--123.

\bibitem{carlson:entiresector}
F.~Carlson, \emph{Sur les fonctions enti\`eres}, Ark. Mat. Astr. Fys.
  \textbf{35A} (1948), no.~14, 18.

\bibitem{cvw:expasympi}
A.~J. Carpenter, R.~S. Varga, and J.~Waldvogel, \emph{Asymptotics for the zeros
  of the partial sums of {$e^z$}. {I}}, Rocky Mountain J. Math. \textbf{21}
  (1991), no.~1, 99--120.

\bibitem{debruijn:ama}
N.~G. de~Bruijn, \emph{Asymptotic {M}ethods in {A}nalysis}, third ed., Dover
  Publications Inc., New York, 1981.

\bibitem{dieu:expsections}
J.~Dieudonn\'e, \emph{Sur les z\'eros des polynomes-sections de $e^x$}, Bull.
  Sci. Math. \textbf{70} (1935), 333--351.

\bibitem{nist:dlmf}
\emph{{NIST Digital Library of Mathematical Functions}}, http://dlmf.nist.gov/,
  Release 1.0.10 of 2015-08-07, Online companion to \cite{olver:nhmf}.

\bibitem{dupuy:partialsumsanswer}
T.~Dupuy, \emph{Answer to: {R}oots of truncations of $e^x - 1$}, MathOverflow,
  URL:
  \htmladdnormallink{http://mathoverflow.net/a/16500}{http://mathoverflow.net/a/16500}
  (version: 2014-08-21).

\bibitem{edrei:paderem}
A.~Edrei, \emph{The {P}ad\'e table of functions having a finite number of
  essential singularities}, Pacific J. Math. \textbf{56} (1975), no.~2,
  429--453.

\bibitem{esv:sections}
A.~Edrei, E.~B. Saff, and R.~S. Varga, \emph{Zeros of {S}ections of {P}ower
  {S}eries}, Lecture {N}otes in {M}athematics, vol. 1002, Springer-Verlag,
  Berlin, 1983.

\bibitem{erdelyi:generalasymp}
A.~Erd{\'e}lyi, \emph{General asymptotic expansions of {L}aplace integrals},
  Arch. Rational Mech. Anal. \textbf{7} (1961), 1--20.

\bibitem{fettis:erfczeros}
H.~E. Fettis, J.~C. Caslin, and K.~R. Cramer, \emph{Complex zeros of the error
  function and of the complementary error function}, Math. Comp. \textbf{27}
  (1973), 401--407.

\bibitem{gakhov:bvp}
F.~D. Gakhov, \emph{Boundary {V}alue {P}roblems}, Translation edited by I. N.
  Sneddon, Pergamon Press, Oxford-New York-Paris; Addison-Wesley Publishing
  Co., Inc., Reading, Mass.-London, 1966.

\bibitem{iverson:zeros}
K.~E. Iverson, \emph{The zeros of the partial sums of {$e^z$}}, Math. Tables
  and Other Aids to Computation \textbf{7} (1953), 165--168.

\bibitem{norfolk:binom}
S.~Janson and T.~S. Norfolk, \emph{Zeros of sections of the binomial
  expansion}, Electron. Trans. Numer. Anal. \textbf{36} (2009/10), 27--38.

\bibitem{jentzsch2}
R.~Jentzsch, \emph{Fortgesetzte {U}ntersuchungen \"uber die {A}bschnitte von
  {P}otenzreihen}, Acta Math. \textbf{41} (1916), no.~1, 253--270.

\bibitem{jentzsch}
\bysame, \emph{Untersuchungen zur {T}heorie der {F}olgen analytischer
  {F}unktionen}, Acta Math. \textbf{41} (1916), no.~1, 219--251.

\bibitem{kappert:sincos}
M.~Kappert, \emph{On the zeros of the partial sums of {$\cos(z)$} and
  {$\sin(z)$}}, Numer. Math. \textbf{74} (1996), no.~4, 397--417.

\bibitem{mclaughlin:exprh}
T.~Kriecherbauer, A.~B.~J. Kuijlaars, K.~D. T.-R. McLaughlin, and P.~D. Miller,
  \emph{Locating the zeros of partial sums of {$e^z$} with {R}iemann-{H}ilbert
  methods}, Integrable {S}ystems and {R}andom {M}atrices, Contemp. Math., vol.
  458, Amer. Math. Soc., Providence, RI, 2008, pp.~183--195.

\bibitem{levin:entirelec}
B.~Ya. Levin, \emph{Lectures on {E}ntire {F}unctions}, Translations of
  Mathematical Monographs, vol. 150, American Mathematical Society, Providence,
  RI, 1996, In collaboration with and with a preface by Yu. Lyubarskii, M.
  Sodin and V. Tkachenko, Translated from the Russian manuscript by Tkachenko.

\bibitem{marden:geom}
M.~Marden, \emph{Geometry of {P}olynomials}, Second edition. Mathematical
  Surveys, No. 3, American Mathematical Society, Providence, R.I., 1966.

\bibitem{miller:aaa}
P.~D. Miller, \emph{Applied {A}symptotic {A}nalysis}, Graduate Studies in
  Mathematics, vol.~75, American Mathematical Society, Providence, RI, 2006.

\bibitem{miller:rhnotes}
\bysame, \emph{Lecture notes on the analysis of {R}iemann-{H}ilbert problems},
  Department of Mathematics, University of Michigan, Fall 2015.

\bibitem{musk:singularintegrals}
N.~I. Muskhelishvili, \emph{Singular {I}ntegral {E}quations}, Dover
  Publications, Inc., New York, 1992, Boundary {P}roblems of {F}unction
  {T}heory and {T}heir {A}pplication to {M}athematical {P}hysics, Translated
  from the second (1946) Russian edition and with a preface by J. R. M. Radok,
  Corrected reprint of the 1953 English translation.

\bibitem{newriv:expzeros}
D.~J. Newman and T.~J. Rivlin, \emph{The zeros of the partial sums of the
  exponential function}, J. Approx. Theory \textbf{5} (1972), 405--412.

\bibitem{newriv:expzeroscorrect}
\bysame, \emph{Correction to: ``{T}he zeros of the partial sums of the
  exponential function''}, J. Approx. Theory \textbf{16} (1976), no.~4,
  299--300.

\bibitem{norfolk:widthconj}
T.~S. Norfolk, \emph{Some observations on the {S}aff-{V}arga width conjecture},
  Rocky Mountain J. Math. \textbf{21} (1991), no.~1, 529--538.

\bibitem{norfolk:1f1}
\bysame, \emph{On the zeros of the partial sums to {${}_1F_1(1;b;z)$}}, J.
  Math. Anal. Appl. \textbf{218} (1998), no.~2, 421--438.

\bibitem{norfolk:transforms}
\bysame, \emph{Asymptotics of the partial sums of a set of integral
  transforms}, Numer. Algorithms \textbf{25} (2000), no.~1-4, 279--291,
  Mathematical journey through analysis, matrix theory and scientific
  computation (Kent, OH, 1999).

\bibitem{olver:nhmf}
F.~W.~J. Olver, D.~W. Lozier, R.~F. Boisvert, and C.~W. Clark (eds.),
  \emph{{NIST Handbook of Mathematical Functions}}, Cambridge University Press,
  New York, NY, 2010, Print companion to \cite{nist:dlmf}.

\bibitem{zhel:lindelof}
I.~Ostrovskii and N.~Zheltukhina, \emph{The asymptotic zero distribution of
  sections and tails of classical {L}indel\"of functions}, Math. Nachr.
  \textbf{283} (2010), no.~4, 573--587.

\bibitem{pritskervarga:weighted}
I.~E. Pritsker and R.~S. Varga, \emph{The {S}zeg{\H o} curve, zero distribution
  and weighted approximation}, Trans. Amer. Math. Soc. \textbf{349} (1997),
  no.~10, 4085--4105.

\bibitem{rosen:thesis}
P.~C. Rosenbloom, \emph{On sequences of polynomials, especially sections of
  power series}, Ph.D. thesis, Stanford University, 1944, Abstracts in Bull.
  Amer. Math. Soc. \textbf{48} (1942), 839; \textbf{49} (1943), 689.

\bibitem{rosen:distrib}
\bysame, \emph{Distribution of zeros of polynomials}, Lectures on Functions of
  a Complex Variable (W.~Kaplan, ed.), The University of Michigan Press, Ann
  Arbor, 1955, pp.~265--285.

\bibitem{sv:overconvergence}
E.~B. Saff and R.~S. Varga, \emph{Geometric overconvergence of rational
  functions in unbounded domains}, Pacific J. Math. \textbf{62} (1976), no.~2,
  523--549.

\bibitem{sv:zerofree}
\bysame, \emph{Zero-free parabolic regions for sequences of polynomials}, SIAM
  J. Math. Anal. \textbf{7} (1976), no.~3, 344--357.

\bibitem{sv:openprobs}
\bysame, \emph{Some open problems concerning polynomials and rational
  functions}, Pad\'e and rational approximation ({P}roc. {I}nternat. {S}ympos.,
  {U}niv. {S}outh {F}lorida, {T}ampa, {F}la., 1976), Academic Press, New York,
  1977, pp.~483--488.

\bibitem{szego:jentzsch}
G.~Szeg\H{o}, \emph{\"{U}ber die {N}ullstellen von {P}olynomen, die in einem
  {K}reis gleichm\"{a}{\ss}ig konvergieren}, Sitzungsber. Ber. Math. Ges.
  \textbf{21} (1922), 59--64.

\bibitem{szego:exp}
\bysame, \emph{{\"U}ber eine {E}igenschaft der {E}xponentialreihe},
  Sitzungsber. Ber. Math. Ges. \textbf{23} (1924), 50--64.

\bibitem{szego:collected1}
\bysame, \emph{Collected {P}apers. {V}ol. 1}, Contemporary Mathematicians,
  Birkh\"auser, Boston, Mass., 1982, 1915--1927, Edited by R. Askey, Including
  commentaries and reviews by G. P{\'o}lya, P. C. Rosenbloom, Askey, L. E.
  Payne, T. Kailath and B. M. McCoy.

\bibitem{varga:strips}
R.~S. Varga, \emph{Semi-infinite and infinite strips free of zeros}, Univ. e
  Politecnico Torino. Rend. Sem. Mat. \textbf{11} (1952), 289--296.

\bibitem{vc:expasympii}
R.~S. Varga and A.~J. Carpenter, \emph{Asymptotics for the zeros of the partial
  sums of {$e^z$}. {II}}, Computational {M}ethods and {F}unction {T}heory
  ({V}alpara\'\i so, 1989), Lecture Notes in Math., vol. 1435, Springer,
  Berlin, 1990, pp.~201--207.

\bibitem{vc:sincosasympi}
\bysame, \emph{Zeros of the partial sums of {$\cos(z)$} and {$\sin(z)$}. {I}},
  Numer. Algorithms \textbf{25} (2000), no.~1-4, 363--375.

\bibitem{vc:sincosasympii}
\bysame, \emph{Zeros of the partial sums of {$\cos(z)$} and {$\sin(z)$}. {II}},
  Numer. Math. \textbf{90} (2001), no.~2, 371--400.

\bibitem{vc:sincosasympiii}
\bysame, \emph{Zeros of the partial sums of {$\cos(z)$} and {$\sin(z)$}.
  {III}}, Appl. Numer. Math. \textbf{60} (2010), no.~4, 298--313.

\bibitem{vargas:limitcurves}
A.~R. Vargas, \emph{Limit curves for zeros of sections of exponential
  integrals}, Constr. Approx. \textbf{40} (2014), no.~2, 219--239.

\bibitem{zhel:mlsectails}
N.~Zheltukhina, \emph{Asymptotic zero distribution of sections and tails of
  {M}ittag-{L}effler functions}, C. R. Math. Acad. Sci. Paris \textbf{335}
  (2002), no.~2, 133--138.

\end{thebibliography}
